\documentclass[12pt]{amsart}

\usepackage{amsmath,amsthm,amssymb,latexsym,xcolor,mathrsfs,enumerate,array, multicol,makecell,caption,float,longtable}

\usepackage[hidelinks]{hyperref}
\usepackage[backend=biber,
style=alphabetic,]{biblatex}
\renewbibmacro{in:}{}
\DeclareFieldFormat*{title}{#1}
\usepackage{amsfonts}
\usepackage{bbm}

\usepackage[title]{appendix}

\allowdisplaybreaks

\numberwithin{equation}{section}
\numberwithin{equation}{subsection}

\renewcommand*{\theequation}{%
  \ifnum\value{subsection}=0 %
    \thesection
  \else
    \thesubsection
  \fi
  .\arabic{equation}%
}

\let\oldtocsection=\tocsection
\let\oldtocsubsection=\tocsubsection
\let\oldtocsubsubsection=\tocsubsubsection
\renewcommand{\tocsection}[2]{\hspace{0em}\oldtocsection{#1}{#2}}
\renewcommand{\tocsubsection}[2]{\hspace{1em}\oldtocsubsection{#1}{#2}}
\renewcommand{\tocsubsubsection}[2]{\hspace{2em}\oldtocsubsubsection{#1}{#2}}

\title[Twisted $2k$th moments of primitive Dirichlet $L$-functions]{Twisted $2k$th moments of primitive Dirichlet $L$-functions: beyond the diagonal}
\author{Siegfred Baluyot}
\address{American Institute of Mathematics \\
600 East Brokaw Road San Jose, CA 95112}
\email{\href{mailto:sbaluyot@aimath.org}{sbaluyot@aimath.org}}
\author{Caroline L. Turnage-Butterbaugh}
\address{Carleton College \\ 1 North College Street Northfield, MN 57707}
\email{\href{mailto:cturnageb@carleton.edu}{cturnageb@carleton.edu}}


\newcommand{\ordp}{\text{ord}_p}
\DeclareMathOperator{\re}{\text{\upshape{Re}}}

\newcolumntype{C}{>{$}c<{$}} 

\newtheorem{theorem}{Theorem}[section]
\newtheorem{conjecture}[theorem]{Conjecture}
\newtheorem{lemma}[theorem]{Lemma}

\newtheorem{prop}[theorem]{Proposition}
\newtheorem*{glh}{Generalized Lindel\"{o}f Hypothesis (GLH)}

\begin{document}

\begin{abstract}
We study the family of Dirichlet $L$-functions of all even primitive characters of conductor at most $Q$, where $Q$ is a parameter tending to $\infty$. For an arbitrary positive integer $k$, we approximate the twisted $2k$th moment of this family by using Dirichlet polynomial approximations of $L^k(s,\chi)$ of length $X$, with $Q<X<Q^2$. Assuming the Generalized Lindel\"{o}f Hypothesis, we prove an asymptotic formula for these approximations of the twisted moments. Our result agrees with the prediction of Conrey, Farmer, Keating, Rubinstein, and Snaith for this family of $L$-functions, and provides the first rigorous evidence beyond the diagonal terms for their conjectured asymptotic formula for the general $2k$th moment of this family.
\end{abstract}

\maketitle

\vspace{-0.1in}
\begingroup
  \hypersetup{hidelinks}
  \tableofcontents
\endgroup
\newpage

\section{Historical overview and motivation}\label{sec: HistoricalOverview}

In recent decades, there has been much interest and measured progress in the study of moments of $L$-functions. The program has its beginnings in the study of the $2k$th moment
\[
M_k(T):=\int_{0}^{T}\left|\zeta\left(\tfrac{1}{2}+it\right) \right|^{2k}\,dt
\]
of the Riemann zeta-function $\zeta(s)$, where $k$ is any positive real number. A great deal of effort has been made to understand $M_k(T)$ for different values of $k$ as $T\to \infty$, yet asymptotic formulas for $M_k(T)$ have remained stubbornly out of reach in all but a few cases. In 1918, Hardy and Littlewood \cite{HL} showed that $M_1(T)\sim T\log T$ as $T\to\infty$, and in 1926 Ingham \cite{Ingham} showed that $M_2(T) \sim (2\pi^2)^{-1}T\log^4T$ as $T\to\infty$. To date, an asymptotic formula is not known to hold for any other $M_k(T)$. Historically, the original motivation for studying $M_k(T)$ has been to prove the Lindel\"{o}f Hypothesis (LH), which asserts that\footnote{Here and throughout this paper, we employ Vinogradov notation and use $f \ll g$ to mean $f=O(g)$.} for any $\varepsilon>0$, $\zeta(1/2+it)\ll t^\varepsilon$ as $t\rightarrow \infty$. In fact, if one could show that $M_k(T)\ll T^{1+\varepsilon}$ for all positive integers $k$ and arbitrarily small $\varepsilon>0$, then LH would follow~\cite[Theorem~13.2]{Titchmarsh}. Proving an asymptotic formula for $M_k(T)$ for any integer $k\ge 3$ is now considered an important problem in its own right.

A folklore conjecture predicts that if $k$ is a positive real number, then, for some unspecified constant $c_k$, we have $M_k(T) \sim c_kT(\log T)^{k^2}$ as $T \to \infty$. In support of this conjecture, it is now known due to the work of many authors that
\[
T(\log T)^{k^2} \ll M_k(T) \ll T(\log T)^{k^2} ,
\]
where the lower bound holds for any real $k\ge 0$, and the upper bound holds unconditionally for $0\le k \le 2$ and conditionally on the Riemann Hypothesis for $k>2$ (see \cite{Ram2}, \cite{Ram1}, \cite{HBlower}, \cite{Soundzeta}, \cite{RadSound}, \cite{Harperzeta}, \cite{BettinChandeeRadziwill}, \cite{BettinChandeeRadziwill}, \cite{HeapRadziwillSound}), and \cite{SoundHeap}). The problem of finding an asymptotic formula for $M_k(T)$ for $k\geq 3$ is so intractable that, up until recently, there had been no viable guess for the exact value of the coefficient $c_k$ in the conjecture $M_k(T) \sim c_kT(\log T)^{k^2}$ for any integer $k\geq 3$. In 1993, Conrey and Ghosh~\cite{ConreyGhoshtalk,ConreyGhosh} predicted the exact value of $c_3$. Later, Conrey and Gonek~\cite{ConreyGonek} used a different approach to conjecture the exact values of both $c_3$ and $c_4$. Both approaches involve heuristic number-theoretic arguments, and the predicted values of $c_3$ agree. Recently, Ng~\cite{Ng} has made the heuristic argument of Conrey and Gonek rigorous, and used it to prove an asymptotic formula for $M_3(T)$ under the assumption of an additive divisor conjecture.

A breakthrough was made in the late 90's when Keating and Snaith \cite{KeatingSnaithRMTzeta} modeled $M_k(T)$ via characteristic polynomials of large random matrices. Doing so allowed them to conjecture the exact value of $c_k$ for all complex $k$ with $\re(k)\ge-1/2$. Remarkably, their predictions agree with the Conrey-Ghosh-Gonek conjectures for $c_3$ and $c_4$ . Later, Diaconu, Goldfeld, and Hoffstein~\cite{DGH} used the theory of multiple Dirichlet series to conjecture the value of $c_k$ for all natural numbers $k$. Despite the differences between these approaches, all the conjectures agree.

Keating and Snaith \cite{KeatingSnaithRMTLfunctions2, KeatingSnaithRMTLfunctions} have made analogous predictions for various families of $L$-functions. One family that has received much attention in the literature is the family of all primitive Dirichlet $L$-functions of modulus $q$. Precisely, let $\chi\bmod q$ be a primitive Dirichlet character, and let
\[
L(s,\chi) = \sum_{n=1}^{\infty}\frac{\chi(n)}{n^s}=\prod_{p}\left(1-\frac{\chi(p)}{p^s}\right)^{-1}, \quad \re(s)>1
\]
be its associated Dirichlet $L$-function. In 1931, Paley \cite{Paley} showed that $\sum_{\chi}|L(1/2,\chi)|^2\sim (\phi^2(q)/q)\log q $ as $q\to \infty$, where the sum is over all characters modulo $q$. The work of Heath-Brown \cite{HB4th} shows 
$$
\sideset{}{^*}\sum_{\chi \bmod q}|L(\tfrac{1}{2},\chi)|^4\sim \frac{\phi^*(q)}{2\pi^2}\prod_{p|q}\frac{(1-\tfrac{1}{p})^3}{(1+\tfrac{1}{p})}(\log q)^4, \quad q \to \infty
$$
with some restrictions on $q$, where $*$ is used to indicate that the sum is over primitive characters and $\phi^*(q)$ is the number of primitive characters modulo $q$. Soundararajan \cite{Sound4th} improved the result to hold for all $q$. Young  \cite{Young4th} showed that this asymptotic formula holds with a power savings error term when the modulus $q$ is prime. Progress for this family is at the same level as that of the zeta-function, and asymptotic expressions have only been obtained for the second and fourth moments. Likewise, sharp lower and upper bounds for the $2k$th moments can be computed; see \cite{RSlower},  \cite{Soundzeta}, \cite{HBupper}, \cite{Harperzeta}, and \cite{SoundHeap}.

By averaging over all $q\le Q$, Huxley \cite{Huxley} used the large sieve inequality to obtain upper bounds of the predicted order of magnitude for $\sum_{q\le Q}\sum_{\chi\bmod q}^*|L(1/2,\chi)|^{2k}$ with $k=3,4$. A recent innovation of Conrey, Iwaniec, and Soundararajan \cite{CISAsymptoticLargeSieve} allowed them to prove an asymptotic formula for the sixth moment averaged over all $q$, albeit with an additional small averaging over the critical line \cite{CIS6th}. Their method, called the \textit{asymptotic large sieve}, was later refined by Chandee and Li \cite{ChandeeLi8Dirichlet} in the context of the eighth moment with the same additional averaging. The asymptotic large sieve has also been used to study the zeros of primitive Dirichlet $L$-functions (see \cite{CISGaps}, \cite{CISCriticalZeros}, \cite{ChandeeLeeLiuRadziwill}) and the twisted second moment \cite{CIS}. (See Section~\ref{sec: outline} for a more detailed discussion on the asymptotic large sieve.)

Inspired by the discovery of Keating and Snaith, Conrey, Farmer, Keating, Rubinstein, and Snaith \cite{CFKRS} used random matrix theory as a guide to formulate a heuristic, which we refer to as ``the CFKRS recipe" or simply ``the recipe," that predicts precise asymptotic formulas for integral moments of various families of $L$-functions. For the family of primitive Dirichlet $L$-functions, the CFKRS recipe leads to the conjecture
\[
\sum_{q\le Q}\,\sideset{}{^*}\sum_{\chi \bmod q}\left|L\left(\tfrac{1}{2},\chi \right)\right|^{2k} \sim c_k\sum_{q\le Q} \,\sideset{}{^*}\sum_{\chi \bmod q}\prod_{p| q}\Bigg(\sum_{m=0}^{\infty}\frac{\binom{m+k-1}{k-1}^2}{p^m}\Bigg)^{-1}(\log q)^{k^2}, \qquad Q \to \infty
\]
for all positive integers $k$, with an explicit value of $c_k$. More generally, the CFKRS recipe predicts an asymptotic formula for 
\begin{equation}\label{eqn: moment}
\sum_{q\le Q}\,\sideset{}{^*}\sum_{\chi \bmod q}\prod_{\alpha \in A}L\left(\tfrac{1}{2}+\alpha,\chi \right)\prod_{\beta \in B}L\left(\tfrac{1}{2}+\beta,\overline{\chi} \right),
\end{equation}
where $A,B$ are finite multisets of small complex numbers, which we refer to as ``shifts." These shifts allowed Conrey et al.~\cite{CFKRS} to write the conjecture as a combinatorial sum that reveals some underlying structure in the asymptotic formula. Within each term in the sum, the shifts appear in an arrangement that involves element exchanges between the multisets $A$ and $B$. Thus each term in the conjectured asymptotic formula can be described as having $\ell$ ``swaps," where $\ell$ is the number of elements exchanged by each multiset with the other. Each $\ell$-swap term may contain  leading order terms, lower order terms, or both. We precisely state the conjecture in the context of our main theorem in Conjecture \ref{con: conjecture} below.

The CFKRS recipe arrives at the conjecture by assuming that certain terms are negligible in the calculation of the moment. While this leads to the  ``final simple answer that should emerge" \cite[page 35]{CFKRS}, the heuristic does not indicate how or why those terms can be ignored. Recently, Conrey and Keating~\cite{CK1}, \cite{CK2}, \cite{CK3}, \cite{CK4}, \cite{CK5} have developed a new approach to this problem for $\zeta(s)$ using Dirichlet polynomial approximations. They estimate the moments 
\begin{equation*}
\int_{T}^{2T}\prod_{\alpha \in A} \zeta(\tfrac{1}{2}+\alpha+it)\prod_{\beta \in B} \zeta(\tfrac{1}{2}+\beta-it)\, dt
\end{equation*}
by approximating the product over $\alpha \in A$ by a Dirichlet polynomial of length $X$ and doing the same for the product over $\beta \in B$. One of their early observations suggests that the size of $X$ determines the values of $\ell$ for which the $\ell$-swap terms contribute at most $o(T)$ to the conjectured asymptotic formula. In particular, they predict that if $X < T/(2\pi)$ then all but the zero-swap term contribute $o(T)$. Similarly, if $T/\pi < X< T^2/(4\pi^2)$ then all but the zero- and one-swap terms should contribute $o(T)$,  if $T^2/\pi^2 < X < T^3/(8\pi^3)$ then all but the zero-, one-, and two-swap terms should contribute $o(T)$, and so on.

This prediction reveals the difficulty in obtaining asymptotic formulas for higher moments of $L$-functions. Historically, the approach to calculating moments has been to  use the approximate functional equation, and this is in fact the approach used in the CFKRS recipe.  For low moments (with $k=1,2$, say), only the so-called ``diagonal" terms from the approximate functional equation contribute to the main term. On the other hand, the previously mentioned conjectures of Conrey et al.~ and Conrey and Keating indicate that high moments have the more delicate and challenging feature that some of the ``off-diagonal" terms actually contribute to the main term. In order to extract these contributions, more sophisticated techniques are needed.

\section{Main result}
We are interested in understanding the twisted $2k$th moment of all primitive Dirichlet $L$-functions of modulus $q$, averaged over all moduli $q\le Q$. To state the result precisely, we must introduce a bit of notation. In Section \ref{sec: notation}, we give a more comprehensive overview of the notation used in this article, with clarifying examples. 

For a finite multiset $A=\{\alpha_1,\alpha_2,\dots,\alpha_r\}$ of complex numbers $\alpha_i$, we define $\tau_A(m)$ for positive integers $m$ by
\begin{equation*}
\tau_A(m) := \sum_{m_1\cdots m_r =m}m_1^{-\alpha_1}\cdots m_r^{-\alpha_r},
\end{equation*}
where the sum is over all positive integers $m_1,\dots,m_r$ such that $m_1\cdots m_r =m$. Thus, if $\chi$ is a Dirichlet character, then 
$$
\sum_{m=1}^{\infty} \frac{\tau_A(m)\chi(m)}{m^s} = \prod_{\alpha\in A }L(s+\alpha,\chi)
$$
for all $s$ such that the left-hand side converges absolutely, where the product on the right-hand side is over all $\alpha \in A$, counted with multiplicity. For any multiset $A$ and  $s\in \mathbb{C}$,  we define $A_s$ to be the multiset $A$ with $s$ added to each element. In other words, if $A=\{\alpha_1,\alpha_2,\dots,\alpha_r\}$, then
\begin{equation*}
A_s := \{\alpha_1+s,\alpha_2+s,\dots,\alpha_r+s\}.
\end{equation*}
If $A$ and $B$ are multisets, then we let $A\cup B$ denote the multiset sum of $A$ and $B$ and $A\smallsetminus B$ denote the multiset difference. We write $A^{-}$ to denote the multiset $A$ with each element multiplied by $-1$.

In this paper, we study the moments \eqref{eqn: moment} with twists $\chi(h)\overline{\chi}(k)$ using Dirichlet polynomial approximations. Thus the main object that we are interested in is
\begin{equation}\label{eqn: S(h,k)def}
\begin{split}
\mathcal{S}(h,k):=\sum_{q=1}^{\infty}W\left(\frac{q}{Q}\right) \sideset{}{^\flat}\sum_{\chi \bmod q} \chi(h)\overline{\chi}(k) \sum_{m=1}^{\infty}\frac{\tau_A(m)\chi(m)}{\sqrt{m}}V\left(\frac{m}{X}\right) \sum_{n=1}^{\infty}\frac{\tau_B(n)\overline{\chi}(n)}{\sqrt{n}}V\left(\frac{n}{X}\right),
\end{split}
\end{equation}
where $W$ is a smooth, nonnegative function that is compactly supported on $(0,\infty)$, the symbol $\flat$ denotes that the sum is over all even, primitive characters modulo $q$, and $V$ is a smooth, nonnegative function that is compactly supported on $[0,\infty)$ and satisfies $V(0)>0$. Note that the length of the $m$-sum, as well as the $n$-sum,  is of the same order of magnitude as $X$. Note also that we use the symbol $k$ in \eqref{eqn: S(h,k)def} for the twist $\overline{\chi}(k)$. This $k$ should not be interpreted as the same $k$ we use when we refer to the $2k$th moment.

In order to  state the asymptotic formula for $\mathcal{S}(h,k)$ that is predicted by the CFKRS recipe, we define
\begin{equation}\label{eqn: I_l(h,k)def}
\begin{split}
\mathcal{I}_\ell(h,k) &:= \sum_{\substack{q=1 \\ (q,hk)=1} }^{\infty} W\left( \frac{q}{Q}\right)\sideset{}{^\flat}\sum_{\chi \bmod q} \frac{1}{(2\pi i )^2} \int_{(\varepsilon)} \int_{(\varepsilon)} X^{s_1+s_2} \widetilde{V}(s_1) \widetilde{V}(s_2)   \\
&\hspace{.5in}\times \sum_{\substack{U\subseteq A, V\subseteq B \\ |U|=|V|=\ell}} \prod_{\alpha\in U}  \frac{\mathscr{X} (\tfrac{1}{2}+\alpha +s_1 )}{ q^{\alpha+s_1} } \prod_{\beta\in V}  \frac{\mathscr{X} (\tfrac{1}{2}+\beta+s_2 )}{ q^{\beta+s_2} }    \\
&\hspace{.5in}\times   \sum_{\substack{1\leq m,n<\infty \\ mh=nk\\ (mn,q)=1}} \frac{\tau_{A_{s_1}\smallsetminus U_{s_1} \cup (V_{s_2})^{-}} (m) \tau_{B_{s_2}\smallsetminus V_{s_2} \cup (U_{s_1})^{-} } (n)  }{\sqrt{mn}} \,ds_2\,ds_1 ,
\end{split}
\end{equation}
where $\varepsilon>0$ is an arbitrarily small constant, 
\begin{equation*}
\widetilde{V}(s) : = \int_0^{\infty} V(x) x^{s-1}\,dx
\end{equation*}
is the Mellin transform of $V$, and 
\begin{equation*}
\mathscr{X} (s) := {\pi}^{s-\frac{1}{2}} \frac{\Gamma(\frac{1}{2}-\frac{1}{2}s)}{\Gamma(\frac{1}{2}s)}.
\end{equation*}
Here, the sum over $U,V$ should be interpreted as taking into account the multiplicity of the elements in $A$ and $B$. The sum $\mathcal{I}_\ell(h,k)$ is precisely the sum of all the $\ell$-swap terms from the recipe prediction. We call these terms the ``$\ell$-swap terms" because the multiset $A_{s_1}\smallsetminus U_{s_1} \cup (V_{s_2})^{-}$ results from taking the set $A_{s_1}$ and replacing the $\ell$ elements of $U_{s_1}$ with the negatives of the $\ell$ elements in $V_{s_2}$. Similarly, $B_{s_2}\smallsetminus V_{s_2} \cup (U_{s_1})^{-}$ results from taking the multiset $B_{s_2}$ and replacing the $\ell$ elements of $V_{s_2}$ with the negatives of the $\ell$ elements in $U_{s_1}$. Thus, we are swapping $\ell$ elements from $A_{s_1}$ with $\ell$ elements from $(B_{s_2})^-$. In particular, $\mathcal{I}_0(h,k)$ is the zero-swap term, $\mathcal{I}_1(h,k)$ is the sum of the one-swap terms, and so on.  We remark that the $m,n$-sum should be interpreted as its analytic continuation, which we write explicitly in \eqref{eqn: I_l(h,k)def2} below.

In Section \ref{sec: recipe}, we show how to derive the following conjecture for the asymptotic behavior of  $\mathcal{S}(h,k)$ using the CFKRS recipe.
\begin{conjecture}\label{con: conjecture}
Let $A$ and $B$ be finite multisets of complex numbers $\ll 1/\log Q$, where $Q$ is a large parameter. Define $\mathcal{S}(h,k)$ by \eqref{eqn: S(h,k)def}. Then, for all $X>0$,
\[
\mathcal{S}(h,k) \sim \sum_{\ell=0}^{\min\{|A|,|B|\}}\mathcal{I}_\ell(h,k), \qquad \text{as } Q \to \infty.
\]
\end{conjecture}

Towards this conjecture, we prove the following theorem. 
\begin{theorem}\label{thm: main}
Let $Q$ be a large parameter and $X=Q^{\eta}$ with $1<\eta<2$. Let $A$ and $B$ be finite multisets of complex numbers $\ll 1/\log Q$, and define $\mathcal{S}(h,k)$ by \eqref{eqn: S(h,k)def}. Then, assuming the Generalized Lindel\"{o}f Hypothesis, we have
\begin{equation}\label{eqn: main}
\mathcal{S}(h,k) = \mathcal{I}_0(h,k) + \mathcal{I}_1(h,k) + \mathcal{E}(h,k),
\end{equation}
where the error term $\mathcal{E}(h,k)$ satisfies, for arbitrarily small $\epsilon>0$,
\begin{equation}\label{eqn: mainbound}
\sum_{h,k\leq Q^{\vartheta}} \frac{\lambda_h \overline{\lambda}_k}{\sqrt{hk}} \mathcal{E}(h,k) \ll_{\epsilon,|A|,|B|,V,W} Q^{1+\frac{\vartheta}{2}+\frac{\eta}{2}+\epsilon} + Q^{\frac{5}{2}-\frac{\eta}{2}+\vartheta+\epsilon}
\end{equation}
uniformly for $0<\vartheta < 2-\eta$ and arbitrary complex numbers $\lambda_h$ such that $\lambda_h\ll_{\varepsilon} h^{\varepsilon}$ for arbitrarily small $\varepsilon>0$.
\end{theorem}

Theorem \ref{thm: main} proves that, under GLH, the zero- and one-swap terms conjectured by the CFKRS recipe are correct. This provides the first rigorous evidence beyond the diagonal terms for the conjecture of Conrey et al.~\cite{CFKRS} for the general $2k$th moment of this family. 

While the recipe provides a detailed prediction for the asymptotic formula satisfied by \eqref{eqn: S(h,k)def}, at present it seems difficult to rigorously prove all the steps involved. We thus approach the problem in a different way using the asymptotic large sieve, which in recent years has become one of the primary tools for studying moments of primitive Dirichlet $L$-functions. Our general strategy in proving Theorem~\ref{thm: main} is based on the approach of Conrey, Iwaniec, and Soundararajan~\cite{CIS}, who applied the asymptotic large sieve to study the twisted second moment. Thus, our work is similar to theirs in many respects. However, there are crucial differences due to the generality of our situation and the intricacy of the predicted asymptotic formula that we aim to prove.
 
The crux of the proof is to uncover the one-swap terms and then show that they match the prediction in Conjecture~\ref{con: conjecture}. The difficulty here is that while Conjecture~\ref{con: conjecture} tells us what the one-swap terms should look like, and the asymptotic large sieve gives us a general idea of where we might find them, neither gives any indication on how to extract the one-swap terms from the asymptotic formula that results from using the asymptotic large sieve. We achieve this through delicate and deliberate contour integration by breaking the predicted one-swap terms into several residues (Section~\ref{subsec: one-swap1}), doing the same for one of the main terms brought about by the use of the asymptotic large sieve (Section~\ref{subsec: one-swap2}), and then matching these residues to show that they are asymptotically equal via Euler product identities (Section~\ref{sec: matchresidues}).

\

\noindent{\bf Remarks} 
\begin{itemize}
\item The main terms in \eqref{eqn: main} are of size about $Q^2$. If we also assume that $\vartheta < (\eta-1)/2$, then the right hand side of \eqref{eqn: mainbound} is $\ll Q^{2-\delta}$ for some $\delta>0$.
\item It can be shown using \eqref{eqn: mellinrapiddecay}, \eqref{eqn: Stirlingchi}, and \eqref{eqn: I_l(h,k)def2} below that, if $A,B$ are as in Theorem~\ref{thm: main}, then $\mathcal{I}_{\ell}(h,k)\ll Q^{2-2\ell\varepsilon+\delta} X^{2\varepsilon} (hk)^{\delta}$ for arbitrarily small $\delta>0$. Hence, if $X=Q^{\eta}$ with $\eta<\ell$, then $\mathcal{I}_{\ell}(h,k) \ll Q^{2-\delta}(hk)^{\delta}$ for some $\delta>0$. Thus Theorem~\ref{thm: main} is consistent with Conjecture~\ref{con: conjecture} when $X=Q^{\eta}$ with $1<\eta<2$.
\item We assume the Generalized Lindel\"{o}f Hypothesis (GLH) in a few key places, which we identify throughout the course of the proof. In each of these instances, there may be a large number of zeta-functions or $L$-functions that we need to bound. If the cardinalities of $A$ and $B$ are not too large, then it may be possible to carry out these estimations unconditionally. 
\item For convenience, we have only considered even primitive characters. For odd characters, some parts of the arguments are simpler, while in other parts only small changes are needed. The conclusion of the theorem for odd primitive characters is the same except that we must replace the function $\mathscr{X}(s)$ with $\pi^{s-\frac{1}{2}} \Gamma(\frac{2-s}{2})/\Gamma(\frac{s+1}{2})$ in the definition of $\mathcal{I}_\ell(h,k)$. We describe the changes to the proof carefully in Section~\ref{sec: outline}.
\item The terms $\mathcal{I}_0(h,k)$ and $\mathcal{I}_1(h,k)$ are both holomorphic functions of the shifts $\alpha \in A$ and $\beta \in B$. We prove this fact at the end of Section \ref{sec: proofofthm}. We may use  \eqref{eqn: vandermonde1swap} as a convenient way to evaluate $\mathcal{I}_1(h,k)$ when some of the elements in $A\cup B$ have multiplicity greater than 1. In particular, we can use \eqref{eqn: vandermonde1swap} to evaluate $\mathcal{I}_1(h,k)$ when all the shifts $\alpha \in A$ and $\beta \in B$ are 0.
\end{itemize}

The one-swap terms have also been found for other families of $L$-functions. Hamieh and Ng \cite{HamiehNg} do this for the $2k$th moments of $\zeta(s)$ under the assumption of an additive divisor conjecture by making some of the arguments in the work of Conrey and Keating \cite{CK3} rigorous. In our situation, we do not need to assume an analogous divisor conjecture because we are able to leverage the asymptotic large sieve. On the other hand, we must assume GLH because the factors $\tau_A$ and $\tau_B$ are unchanged when applying the asymptotic large sieve and thus give rise to a potentially large number of $L$-functions.  Conrey and Rodgers \cite{ConreyRodgers} have found the one-swap terms for the family of quadratic Dirichlet $L$-functions. They also do not need to assume any divisor conjecture because they are able to use the Poisson summation method of Soundararajan \cite{SoundPoisson}. As in our situation, they also need to assume GLH to bound large numbers of $L$-function factors.

Analogous results have been proved unconditionally in the function field setting. Andrade and Keating~\cite{AndradeKeating} used the CFKRS recipe to predict the asymptotic formulas for moments of $L$-functions associated with hyperelliptic curves of genus $g$ over a fixed finite field, where $g$ is a parameter going to infinity. Florea~\cite{FloreaThesis} has recovered the one-swap terms for this family. Moreover, Bui, Florea, and Keating~\cite{BuiFloreaKeating} have found the one-swap terms for the 2-level density of zeros of this family. In this setting, the Poisson summation method is the primary tool for studying moments of $L$-functions  (see also \cite{Florea4th}, \cite{Florea23}, \cite{BuiFloreaKeatingRoditty-Gershon}, and \cite{BuiFloreaKeating2}). For a different family over function fields, Sawin~\cite{Sawin} has formulated a heuristic that recovers the CFKRS prediction, which he then confirms under the assumption of a conjecture on the vanishing of certain cohomology groups.

In order to extract the two-swap terms predicted by Conjecture~\ref{con: conjecture}, the discussion at the end of Section~\ref{sec: HistoricalOverview} suggests that we must work with a Dirichlet polynomial approximation of length $X>Q^2$. In this situation, the predicted two-swap terms are of size about $Q^2$. Without any additional input, the asymptotic large sieve does not seem effective when $X>Q^2$ because it no longer reduces the moduli of the character sums for such $X$ (see Section \ref{sec: outline} for more details). In fact, the predicted two-swap terms should be hidden inside the term $\mathcal{E}(h,k)$ in \eqref{eqn: main}, and thus we no longer expect the left-hand side of \eqref{eqn: mainbound} to be $\ll Q^{2-\delta}$ when $X>Q^2$. This limitation of the asymptotic large sieve is analogous to the limitation of the Poisson summation method in evaluating high moments of the family of quadratic Dirichlet $L$-functions.

With some additional work, we may be able to use our result to study the sixth moment of primitive Dirichlet $L$-functions. There could also be potential applications to studying gaps between zeros of Dirichlet $L$-functions.

\

\noindent\emph{Outline of the article.} In Section \ref{sec: notation}, we give a comprehensive list of all the notation used in the article. In Section \ref{sec: recipe}, we use the CFKRS recipe to derive Conjecture~\ref{con: conjecture}. We give a detailed outline of the proof of Theorem~\ref{thm: main} in Section \ref{sec: outline}. The remaining sections are devoted to proving the theorem. In Section~\ref{sec: diagonal}, we examine the diagonal terms to extract the zero-swap term. We study the off-diagonal terms in Sections~\ref{sec: Lsum}-\ref{sec: Ur}, where we extract the one-swap terms. Finally, in Section \ref{sec: proofofthm}, we complete the proof of Theorem \ref{thm: main} and prove the holomorphy of $\mathcal{I}_0(h,k)$ and $\mathcal{I}_1(h,k)$.

\

\noindent{\bf Acknowledgments.}\, Work on this project began in the summer of 2020 at the American Institute of Mathematics as part of the NSF Focused Research Group ``Averages of $L$-functions and Arithmetic Stratification"  supported by NSF DMS-1854398 FRG. We are grateful to  Brian Conrey for suggesting this problem, for many helpful discussions, and for all the support and encouragement. We also thank David Farmer, Alexandra Florea, and Brad Rodgers for a number of useful comments that improved the exposition. The second author thanks the American Institute of Mathematics for providing a focused research environment in February and March 2022, during which the manuscript was prepared. The first author is supported by NSF DMS-1854398 FRG, and the second author is partially supported by NSF DMS-1902193 and NSF DMS-1854398 FRG.

\section{Notation, conventions, and preliminaries}\label{sec: notation}
In this section, we collect our commonly used notation for the reader's convenience. We also list a number of technical assumptions and basic facts that we use throughout the paper. The reader may choose to skip this section and only refer to it when needed.

We employ standard notation in analytic number theory and use $\int_{(c)}$ to denote integrals along the line from $c-i\infty$ to $c+i\infty$. We let $\varepsilon>0$ denote an arbitrarily small constant whose value may change from one line to the next. We also sometimes use $\epsilon>0$ to denote an arbitrarily small constant, except that the value of $\epsilon$ remains the same all throughout. This distinction between $\varepsilon$ and $\epsilon$ will often be harmless, and we will use $\epsilon$ only when the situation requires more concreteness, such as when dealing with integrals like
$$
\int_{(\epsilon)}\int_{(\epsilon/2)} \Gamma(w)\Gamma(z)\Gamma(w-z)\,dz\,dw.
$$
The symbol $\varepsilon$ may sometimes depend on $\epsilon$, but only when the concreteness of $\epsilon$ is no longer required. When at least one of $\varepsilon$ or $\epsilon$ is present, in some fashion, in an inequality or error term, we allow implied constants to depend on $\varepsilon$ or $\epsilon$ without necessarily indicating so in the notation. We sometimes indicate the dependence of implied constants on variables by the use of subscripts: for example, $Y\ll_b Z$ or $Y=O_b(Z)$ means that the implied constant may depend on $b$.

The symbol $p$ always denotes a prime number. We use $\ordp(m)$ to denote the exponent of $p$ in the prime factorization of $m$. For example, ord$_3(72)=2$ and ord$_5(84)=0$. We let $\phi$ be the Euler totient function, and $\mu$ the M\"{o}bius function. If $h$ and $k$ are positive integers that are present in some form in an equation or inequality, then we use $H$ to denote $h/(h,k)$ and $K$ to denote $k/(h,k)$.

For a multiset $E=\{\xi_1,\xi_2,\dots,\xi_j\}$ of complex numbers, we define $\tau_E(m)$ for positive integers $m$ by
\begin{equation}\label{eqn: taudef}
\tau_E(m) := \sum_{m_1\cdots m_j =m}m_1^{-\xi_1}\cdots m_j^{-\xi_j},
\end{equation}
where the sum is over all positive integers $m_1,\dots,m_j$ such that $m_1\cdots m_j =m$. Thus, for example, if $\xi_1=\cdots=\xi_j=0$, then $\tau_E(m)$ is the $j$-fold divisor function. If $E$ is empty, then we define $\tau_E(1)=1$ and $\tau_E(m)=0$ for all other $m$. It follows that if $E$ is a finite multiset of complex numbers, then
$$
\sum_{m=1}^{\infty} \frac{\tau_E(m)}{m^s} = \prod_{\xi\in E }\zeta(s+\xi)
$$
for all $s$ such that the left-hand side converges absolutely, where $\zeta(s)$ is the Riemann zeta-function and the product on the right-hand side is over all $\xi\in E$, counted with multiplicity. We define $\tau_E(p^{-1})$ to be zero for every multiset $E$. If $r$ is a real number such that each element of $E$ has real part $\geq r$, then \eqref{eqn: taudef} and the divisor bound imply
\begin{equation}\label{eqn: divisorbound}
\tau_E(m)\ll_{\varepsilon} m^{-r+\varepsilon}.
\end{equation}

If $E$ is a multiset of complex numbers and $s\in \mathbb{C}$, then we define $E_s$ to be the multiset $E$ with $s$ added to each element. In other words, if $E=\{\xi_1,\xi_2,\dots,\xi_j\}$, then
\begin{equation*}
E_s := \{\xi_1+s,\xi_2+s,\dots,\xi_j+s\}.
\end{equation*}
It follows immediately from this definition and \eqref{eqn: taudef} that
\begin{equation}\label{eqn: taufactoringidentity}
\tau_{E_s}(m) = m^{-s}\tau_E(m).
\end{equation}
If $E$ is a multiset, then we let $|E|$ denote its cardinality, counting multiplicity. If $D$ and $E$ are multisets, then we let $D\cup E$ denote the multiset sum of $D$ and $E$, which means that the multiplicity of each element in $D\cup E$ is exactly the sum of the multiplicity of the element in $D$ and its multiplicity in $E$. Similarly, we define $D\smallsetminus E$ to be the multiset difference, which is the multiset with each element having multiplicity equal to its multiplicity as an element of $D$ minus its multiplicity as an element of $E$ if this difference is nonnegative, and equal to zero otherwise. Thus, for example, if $A=\{\alpha_1,\alpha_2,\dots,\alpha_j\}$ is a multiset of complex numbers, $\alpha=\alpha_1$, and $\beta$ and $s$ are complex numbers, then \eqref{eqn: taudef} implies
\begin{equation*}
\tau_{A_s \smallsetminus \{\alpha+s\} \cup \{-\beta-s\}} (m) = \sum_{m_1\cdots m_j=m} m_1^{\beta+s}m_2^{-\alpha_2-s}m_3^{-\alpha_3-s}\cdots m_j^{-\alpha_j-s}
\end{equation*}
for every positive integer $m$. For most of our proofs, we will be dealing with sets instead of multisets, and in most cases $D\smallsetminus E$ and $D\cup E$ reduce to ordinary set difference and set union, respectively.

The letter $Q$ denotes a parameter tending to $\infty$, and $\vartheta \in (0,1)$ is a parameter. We define $X=Q^{\eta}$ with $\eta$ a parameter satisfying $1<\eta<2$. The quantities $C$ and $Y$, which satisfy $C\geq 1$ and $Y\geq XQ^{\vartheta}$ and are introduced in Sections~\ref{sec: outline} and \ref{sec: Ur}, respectively, are positive parameters that we will choose to be powers of $Q$ at the end of the proof of Theorem~\ref{thm: main}. The sequence $\lambda_1,\lambda_2,\dots$ is an arbitrary sequence of complex numbers such that $\lambda_h \ll_{\varepsilon} h^{\varepsilon}$ for all positive integers $h$. We use this sequence only to prove the property \eqref{eqn: mainbound} of $\mathcal{E}(h,k)$. In Section~\ref{sec: U2}, we use the symbol $\delta$ to denote the reciprocal of an arbitrarily large power of $Q$, say
\begin{equation}\label{eqn: deltadef}
\delta=Q^{-99}.
\end{equation}
In many places in the same section and in other sections, we also use the symbol $\delta$ as an index of a product, but this will not cause confusion.

We let $A$ and $B$ be arbitrary fixed finite multisets of complex numbers. We usually denote elements of $A$ by $\alpha$ and elements of $B$ by $\beta$. We assume that $\alpha,\beta\ll 1/\log Q$ for all $\alpha\in A$ and $\beta\in B$, with the implied constant arbitrary but fixed. For convenience, we let $C_0>0$ be a fixed arbitrary constant and assume all throughout our proof of Theorem~\ref{thm: main} that if $A=\{\alpha_1,\alpha_2,\dots,\alpha_j\}$ and $B=\{\beta_1,\beta_2,\dots,\beta_{\ell}\}$, then
\begin{equation}\label{eqn: orbitals}
\begin{split}
|\alpha_{\nu}| &= \frac{2^{\nu}C_0}{\log Q} \ \ \ \text{for }\nu=1,2,\dots,j, \text{ and}\\
|\beta_{\nu}| &= \frac{2^{j+\nu}C_0}{\log Q} \ \ \ \text{for }\nu=1,2,\dots,\ell.
\end{split}
\end{equation}
This ensures that we do not encounter double poles when dealing with expressions such as $\prod_{\alpha\in A,\beta\in B}\zeta(\alpha+\beta+s)$. A consequence of \eqref{eqn: orbitals} is that if $J_1,J_2$ are subsets of $\{1,2,\dots,j\}$ and $L_1,L_2$ are subsets of $\{1,2,\dots,\ell\}$ such that either $J_1\neq J_2$ or $L_1\neq L_2$, then
\begin{equation}\label{eqn: zetaalphaalphabound}
\zeta\bigg( 1 + \sum_{\nu\in J_1} \alpha_{\nu} +\sum_{\nu\in L_1}\beta_{\nu}  - \sum_{\nu\in J_2} \alpha_{\nu} - \sum_{\nu\in L_2}\beta_{\nu}\bigg) \ll \log Q.
\end{equation}
We will eliminate the assumption \eqref{eqn: orbitals} in Section~\ref{sec: proofofthm} and show that Theorem~\ref{thm: main} holds for arbitrary finite multisets $A$ and $B$ such that $\alpha,\beta\ll 1/\log Q$ for all $\alpha\in A$ and $\beta\in B$. The assumption \eqref{eqn: orbitals} is unnecessary in carrying out the Euler product evaluations in Lemmas~\ref{lem: 1swapeulerbound}, \ref{lem: Kfunctionaleqn}, and \ref{lem: U2eulerbound} and Subsection~\ref{sec: matchresidues}. For those calculations, we only need the elements of $A$ and $B$ to be arbitrarily small, and so the assumption that $\alpha,\beta\ll 1/\log Q$ for all $\alpha\in A$ and $\beta\in B$ suffices.

We define the Mellin transform of a function $f$ by
\begin{equation}\label{eqn: mellindef}
\widetilde{f}(s) : = \int_0^{\infty} f(x) x^{s-1}\,dx.
\end{equation}
We assume that $V$ is a fixed smooth function from $[0,\infty)$ to $[0,\infty)$ that has compact support. We suppose that $V(0)>0$, since otherwise the $m$-sum (or $n$-sum) in \eqref{eqn: S(h,k)def} tends to $0$ as $X\rightarrow \infty$ and is thus an invalid approximation of the product of $L$-functions. Without loss of generality, we may assume that $V(0)=1$ since we may normalize by dividing $V(x)$ by $V(0)$. Integrating by parts, we see from the definition \eqref{eqn: mellindef} of $\widetilde{V}$ that if $\re(s)>0$, then
\begin{equation}\label{eqn: mellinVIBP}
\widetilde{V}(s) = -\frac{1}{s}\int_0^{\infty}V'(x) x^s\,dx.
\end{equation}
The latter integral is holomorphic for $\re(s)>-1$ since $V'$ is bounded and compactly supported. It thus follows from \eqref{eqn: mellinVIBP} that $s=0$ is a simple pole of $\widetilde{V}$ and
\begin{equation}\label{eqn: mellinVresidue}
\underset{s=0}{\text{Res}}\ \widetilde{V}(s) = \lim_{s\rightarrow 0} s\widetilde{V}(s) = 1
\end{equation}
because $V(0)=1$. We may apply integration by parts again to the right-hand side of \eqref{eqn: mellinVIBP} to analytically continue $\widetilde{V}(s)$ to $\re(s)>-2$. Repeating this process indefinitely, we see that $\widetilde{V}(s)$ is meromorphic on all of $\mathbb{C}$ with possible poles only at the non-positive integers.

We assume that $W$ is a fixed smooth function from $(0,\infty)$ to $[0,\infty)$ that has compact support. This means that the support of $W$ is bounded away from $0$, and it follows immediately from \eqref{eqn: mellindef} and Morera's theorem that $\widetilde{W}(s)$ is an entire function. The definition \eqref{eqn: mellindef} and a repeated application of integration by parts shows that if $n$ is a positive integer, then
\begin{equation}\label{eqn: mellinrapiddecay}
\widetilde{V}(s),\widetilde{W}(s) \ll_n \frac{1}{|s|^n}
\end{equation}
as $s\rightarrow \infty$. We will repeatedly use this fact without mention to justify moving lines of integration.

We allow implied constants to depend on $\varepsilon$, $\epsilon$, the cardinalities $|A|$ and $|B|$, the implied constant in the assumption $\alpha,\beta\ll 1/{\log Q}$, or the functions $V$ and $W$ without necessarily indicating so in the notation. The implied constants never depend on the actual values of $\alpha,\beta$ nor on any of $Q,X,C,Y,h,k,\lambda_h,\lambda_k, \vartheta,\eta$.

We define $\mathscr{X}(s)$ by
\begin{equation}\label{eqn: scriptXdef}
\mathscr{X} (s) = {\pi}^{s-\frac{1}{2}} \frac{\Gamma(\frac{1}{2}-\frac{1}{2}s)}{\Gamma(\frac{1}{2}s)}.
\end{equation}
In other words, we write the functional equation of $\zeta(s)$ as $\zeta(s)=\mathscr{X}(s)\zeta(1-s)$. The poles of $\mathscr{X}$ are at the odd positive integers, and Stirling's formula implies \cite[(4.12.3)]{Titchmarsh}
\begin{equation}\label{eqn: Stirlingchi}
\mathscr{X}(s) \asymp (1+|s|)^{\frac{1}{2}-\re(s)}
\end{equation}
for $s$ in any fixed vertical strip such that $s$ is bounded away from the poles of $\mathscr{X}$. The relation $f\asymp g$ means $f \ll g$ and $f \gg g$. We will use \eqref{eqn: Stirlingchi} repeatedly without mention. We define $\mathcal{H}(z,w)$ by
\begin{equation}\label{eqn: Hdef}
\mathcal{H}(z,w)=\sqrt{\pi}\frac{\Gamma(\tfrac{1-w}{2})\Gamma(\tfrac{z}{2})\Gamma(\tfrac{w-z}{2})}{\Gamma(\tfrac{w}{2})\Gamma(\tfrac{1-z}{2})\Gamma(\tfrac{1-w+z}{2})}.
\end{equation}
It follows from this and the definition \eqref{eqn: scriptXdef} of $\mathscr{X}$ that
\begin{equation}\label{eqn: Hintermsofchifactor}
\mathcal{H}(z,w)=\mathscr{X}(w)\mathscr{X}(1-z)\mathscr{X}(1-w+z).
\end{equation}
This and \eqref{eqn: Stirlingchi} imply
\begin{equation}\label{eqn: Hbound}
\mathcal{H}(z,w) \asymp |w|^{\frac{1}{2}-\re(w)}|z|^{\re(z)-\frac{1}{2}} |w-z|^{\re(w-z)-\frac{1}{2}}
\end{equation}
for $w,z$ in any fixed vertical strip such that $w$, $z$, and $w-z$ are bounded away from the integers.

We will repeatedly use without mention the well-known fact that $\zeta(s)$ and the Dirichlet $L$-functions each have at most polynomial growth in fixed vertical strips. Oftentimes, this polynomial growth is offset by the rapid decay \eqref{eqn: mellinrapiddecay} of the Mellin transforms. However, there are certain points in our argument, particularly when estimating integrals involving a large number of zeta or $L(s,\chi)$ factors, where we will need to assume the following.
\begin{glh}
    The Lindel\"{o}f Hypothesis for $\zeta(s)$ holds and 
$$
L(\tfrac{1}{2}+it,\psi) \ll_{\varepsilon} (q(1+|t|))^{\varepsilon}
$$
for all real $t$ and all non-principal Dirichlet characters $\psi$ modulo $q$, where the implied constant depends only on $\varepsilon$.
\end{glh}
The Generalized Riemann Hypothesis implies GLH~\cite{ConreyGhoshLindeloff}. We will explicitly mention our assumption of GLH each time we use it.

For conciseness, we adopt the convention that any expression of a sum in $\Sigma$-notation that contains the symbol $\pm$ means a sum of two copies of that expression: one with the symbol $\pm$ replaced by $+$, the other with $\pm$ replaced by $-$, and both with $\mp$ replaced by the sign opposite that replacing $\pm$. For example,
$$
\sum_{\substack{d|q \\ d|(m\pm n)}} \psi(\mp d) f(\pm d) g(d)
$$
means the same as
$$
\sum_{\substack{d|q \\ d|(m+ n)}} \psi(- d)f( d) g(d) + \sum_{\substack{d|q \\ d|(m- n)}} \psi( d)f(-d) g(d).
$$
and
$\sum_{a} h(\pm a)$ means the same as $\sum_{a} h(a) +\sum_{a} h(-a)$. On the other hand, we use the typical interpretation of $\pm$ in expressions like
\[
\int_{0}^{\infty }  \frac{c |mh\pm e^{\xi}nk|}{gx Q} W\left( \frac{c |mh\pm e^{\xi}nk|}{gx Q}\right) x^{w-1} \,dx 
\]
and in definitions such as
\[
\ell := \frac{|mh\pm nk|}{d}.
\]

We end this section with two lemmas that we will apply in various sections.
\begin{lemma}\label{lem: Lemma2ofCIS}\cite[Lemma 2]{CIS}
If $(mn,q)=1$, then
\[
\sideset{}{^{\flat}}\sum_{\chi \bmod q} \chi(m)\overline{\chi(n)} = \frac{1}{2}\Bigg(\sum_{\substack{d|q\\d|(m\pm n)}}\phi(d)\mu\left(\frac{q}{d}\right) \Bigg),
\]
where the $\flat$ indicates that the sum is over all the even primitive characters. Here, we have adopted the previously mentioned convention that the right-hand side means a sum of two copies of itself: one with $\pm$ replaced by $+$, and the other with $\pm$ replaced by $-$.
\end{lemma}

\begin{lemma}\label{lem: sumstoEulerproducts}
If $f(m_1,m_2,\dots,m_j;p)$ is a complex-valued function such that
$$
f(m_1,m_2,\dots,m_j;p) = f(p^{\ordp(m_1)},p^{\ordp(m_2)},\dots,p^{\ordp(m_j)};p)
$$
for all positive integers $m_1,m_2,\dots,m_j$ and primes $p$, then
$$
\sum_{1 \leq m_1 ,m_2 ,\dots,m_j<\infty} \prod_p f(m_1,m_2,\dots,m_j;p) = \prod_p \sum_{0 \leq b_1 , b_2 ,\dots,b_j<\infty} f(p^{b_1},p^{b_2},\dots,p^{b_j};p)
$$
if absolute convergence holds for both sides.
\end{lemma}
\begin{proof}[Proof sketch]
This can be proved using a standard argument (see, for example, \cite[Theorem~11.7]{Apostol}) together with the fact that $\prod_{p>y} f(1,\dots,1;p)\rightarrow 1$ as $y\rightarrow \infty$.
\end{proof}

\section{The CFKRS recipe for conjecturing asymptotic formulas for moments}\label{sec: recipe}

In this section, we apply the heuristic of Conrey et~al.~\cite{CFKRS} to conjecture the asymptotic formula for the sum $\mathcal{S}(h,k)$ defined by \eqref{eqn: S(h,k)def}. We also make the definition \eqref{eqn: I_l(h,k)def} of $\mathcal{I}_{\ell}(h,k)$ more explicit by writing out the analytic continuation of the $m,n$-sum. Furthermore, we write the $q$-sum in \eqref{eqn: I_l(h,k)def} in terms of an integral in order to facilitate subsequent calculations. For a more detailed discussion on the CFKRS recipe and its applications to other families of $L$-functions, see \cite{CFKRS}.

We first apply Mellin inversion, interchange the order of summation, and observe that
\begin{equation*}
\sum_{m=1}^{\infty}\frac{\tau_A(m)\chi(m)}{m^{\frac{1}{2}+s_1}} \sum_{n=1}^{\infty}\frac{\tau_B(n)\bar\chi(n)}{n^{\frac{1}{2}+s_2}} = \prod_{\alpha\in A} L(\tfrac{1}{2}+\alpha+s_1,\chi) \prod_{\beta\in B} L(\tfrac{1}{2}+\beta+s_2,\overline{\chi})
\end{equation*}
by the definition \eqref{eqn: taudef} of $\tau_E$ to deduce from \eqref{eqn: S(h,k)def} that
\begin{equation}\label{eqn: towardstherecipe1}
\begin{split}
\mathcal{S}(h,k) = \frac{1}{(2\pi i )^2} \int_{(2)} \int_{(2)} X^{s_1+s_2} \widetilde{V}(s_1) \widetilde{V}(s_2) \sum_{q=1}^{\infty} W\left( \frac{q}{Q}\right) \sideset{}{^\flat}\sum_{\chi \bmod q} \chi(h) \overline{\chi}(k)\\
\times  \prod_{\alpha\in A} L(\tfrac{1}{2}+\alpha+s_1,\chi) \prod_{\beta\in B} L(\tfrac{1}{2}+\beta+s_2,\overline{\chi}) \,ds_2\,ds_1,
\end{split}
\end{equation}
where $\widetilde{V}$ is defined by~\eqref{eqn: mellindef}. We may move the lines of integration to $\re(s_1)=\re(s_2)=\varepsilon$ because of the rapid decay of $\widetilde{V}$ and the fact that $L(s,\chi)$ is entire for non-principal $\chi$. Now recall that if $\chi$ is an even primitive character of conductor $q$, then $L(s,\chi)$ satisfies the functional equation \cite[\S9]{Davenport}
\begin{equation*}
L(s,\chi) =  G(\chi)q^{-s} \mathscr{X} (s) L(1-s,\overline{\chi}),
\end{equation*}
where $G(\chi)= \sum_{n\bmod{q}}\chi(n) \exp(2\pi i n/q)$ is the Gauss sum and $\mathscr{X} (s)$ is defined by \eqref{eqn: scriptXdef}. Then we have the approximate functional equation
\begin{equation*}
L(s,\chi) \approx \sum_n \frac{\chi(n)}{n^s} + G(\chi)q^{-s} \mathscr{X} (s) \sum_n \frac{\overline{\chi}(n)}{n^{1-s}}.
\end{equation*}
We replace each $L(s,\chi)$ factor in \eqref{eqn: towardstherecipe1} with the right-hand side of its approximate functional equation, and then multiply out the resulting product. We formally discard all the resulting terms except for those that have the same number of $G(\chi)$ factors as $G(\overline{\chi})$ factors. For the remaining terms, we use the fact that $G(\chi)G(\overline{\chi})=q$ \cite[\S9]{Davenport}, and formally extend the sums from the approximate functional equations to $\infty$. We then write the sums in terms of the function $\tau_E$ defined by \eqref{eqn: taudef}, and use the approximation \cite[(4.3.4)]{CFKRS}
\begin{equation*}
\sideset{}{^\flat}\sum_{\chi \bmod q} \chi(hm) \overline{\chi}(kn) \approx \left\{ \begin{array}{cl} \displaystyle \sideset{}{^\flat}\sum_{\chi \bmod q} 1 & \text{if }hm=kn \text{ and } (hkmn,q)=1\\ \\ 0 & \text{else}, \end{array}\right.
\end{equation*}
which we expect to follow from the orthogonality of Dirichlet characters (see also Lemma~\ref{lem: Lemma2ofCIS}). This leads us to conjecture Conjecture~\ref{con: conjecture}.

We may put Conjecture \ref{con: conjecture} into a more explicit form by writing out the analytic continuation of the $m,n$-sum in \eqref{eqn: I_l(h,k)def}. We do this by formally writing it as an Euler product, multiplying it by
\begin{equation}\label{eqn: CFKRSzetafactors}
\prod_{\substack{ \gamma\in A_{s_1}\smallsetminus U_{s_1} \cup (V_{s_2})^{-} \\ \delta \in B_{s_2}\smallsetminus V_{s_2} \cup (U_{s_1})^{-} } } \zeta(1+\gamma+\delta),
\end{equation}
and then dividing it by the Euler product of \eqref{eqn: CFKRSzetafactors}. In other words, we claim that the definition \eqref{eqn: I_l(h,k)def} of $\mathcal{I}_{\ell}(h,k)$ with the $m,n$-sum written explicitly as its analytic continuation is
\begin{align}
\mathcal{I}_{\ell}(h,k) &= \sum_{\substack{U\subseteq A, V\subseteq B \\ |U|=|V|=\ell}} \sum_{\substack{q=1 \\ (q,hk)=1} }^{\infty} W\left( \frac{q}{Q}\right)\sideset{}{^\flat}\sum_{\chi \bmod q} \frac{1}{(2\pi i )^2} \int_{(\varepsilon)} \int_{(\varepsilon)} X^{s_1+s_2} \widetilde{V}(s_1) \widetilde{V}(s_2)   \notag\\
&\hspace{.25in}\times  \prod_{\alpha\in U}  \frac{\mathscr{X} (\tfrac{1}{2}+\alpha +s_1 )}{ q^{\alpha+s_1} } \prod_{\beta\in V}  \frac{\mathscr{X} (\tfrac{1}{2}+\beta+s_2 )}{ q^{\beta+s_2} } \prod_{\substack{\gamma \in A_{s_1}\smallsetminus U_{s_1} \cup (V_{s_2})^{-} \\ \delta \in B_{s_2}\smallsetminus V_{s_2} \cup (U_{s_1})^{-}  }} \zeta(1+\gamma+\delta)   \notag\\
&\hspace{.25in}\times   \prod_{p|q}\Bigg\{ \prod_{\substack{\gamma \in A_{s_1}\smallsetminus U_{s_1} \cup (V_{s_2})^{-} \\ \delta \in B_{s_2}\smallsetminus V_{s_2} \cup (U_{s_1})^{-}  }}\left( 1- \frac{1}{p^{1+\gamma+\delta}}\right) \Bigg\} \prod_{p | hk}\Bigg\{ \prod_{\substack{\gamma \in A_{s_1}\smallsetminus U_{s_1} \cup (V_{s_2})^{-} \\ \delta \in B_{s_2}\smallsetminus V_{s_2} \cup (U_{s_1})^{-}  }}\left( 1- \frac{1}{p^{1+\gamma+\delta}}\right) \notag\\
&\hspace{.5in}\times\sum_{ \substack{0\leq m,n<\infty\\ m+\ordp(h) = n+\ordp(k)}}\frac{\tau_{A_{s_1}\smallsetminus U_{s_1} \cup (V_{s_2})^{-}} (p^m) \tau_{B_{s_2}\smallsetminus V_{s_2} \cup (U_{s_1})^{-} } (p^n)  }{p^{m/2}p^{n/2} }\Bigg\} \notag\\
&\hspace{.25in}\times   \prod_{p\nmid qhk}\Bigg\{ \prod_{\substack{\gamma \in A_{s_1}\smallsetminus U_{s_1} \cup (V_{s_2})^{-} \\ \delta \in B_{s_2}\smallsetminus V_{s_2} \cup (U_{s_1})^{-}  }}\left( 1- \frac{1}{p^{1+\gamma+\delta}}\right) \notag\\
&\hspace{.5in}\times\sum_{m=0}^{\infty}\frac{\tau_{A_{s_1}\smallsetminus U_{s_1} \cup (V_{s_2})^{-}} (p^m) \tau_{B_{s_2}\smallsetminus V_{s_2} \cup (U_{s_1})^{-} } (p^m)  }{p^{m} }\Bigg\} \,ds_2\,ds_1. \label{eqn: I_l(h,k)def2}
\end{align}

We now prove our claim by showing that the Euler product in \eqref{eqn: I_l(h,k)def2} converges absolutely for $A,B$ satisfying $\alpha,\beta\ll 1/\log Q$ for all $\alpha\in A$ and $\beta\in B$. To do this, we make the following observations for such $A,B$. If Re$(s_1)=$Re$(s_2)=\varepsilon$, then
\begin{equation}\label{eqn: CFKRSeuler1}
\begin{split}
\prod_{\substack{\gamma \in A_{s_1}\smallsetminus U_{s_1} \cup (V_{s_2})^{-} \\ \delta \in B_{s_2}\smallsetminus V_{s_2} \cup (U_{s_1})^{-}  }}\left( 1- \frac{1}{p^{1+\gamma+\delta}}\right)
& = 1- \sum_{\substack{ \gamma\in A_{s_1}\smallsetminus U_{s_1} \cup (V_{s_2})^{-} \\ \delta \in B_{s_2}\smallsetminus V_{s_2} \cup (U_{s_1})^{-} } } \frac{1}{p^{1+\gamma+\delta}}  + O\left( \frac{1}{p^{1+\varepsilon}}\right) \\
& = 1 -\frac{\tau_{A_{s_1}\smallsetminus U_{s_1} \cup (V_{s_2})^{-}} (p ) \tau_{B_{s_2}\smallsetminus V_{s_2} \cup (U_{s_1})^{-} } (p )  }{ p } + O\left( \frac{1}{p^{1+\varepsilon}}\right),
\end{split}
\end{equation}
where the last equality follows from the definition \eqref{eqn: taudef} of $\tau_E$. Furthermore, \eqref{eqn: divisorbound} implies
\begin{equation}\label{eqn: CFKRSeuler2}
\sum_{m=2}^{\infty} \frac{\tau_{A_{s_1}\smallsetminus U_{s_1} \cup (V_{s_2})^{-}} (p^{m}) \tau_{B_{s_2}\smallsetminus V_{s_2} \cup (U_{s_1})^{-} } (p^{m})  }{p^{m}} \ll \frac{1}{p^{2-\varepsilon}}
\end{equation}
for Re$(s_1)=$Re$(s_2)=\varepsilon$. From this and \eqref{eqn: CFKRSeuler1}, we deduce that if $p\nmid qhk$, then the local factor in \eqref{eqn: I_l(h,k)def2} corresponding to $p$ is $1+O(p^{-1-\varepsilon})$. Hence the Euler product in \eqref{eqn: I_l(h,k)def2} converges absolutely.

We next prove an integral expression for the $q$-sum in \eqref{eqn: I_l(h,k)def2} in order to facilitate the proof of Theorem~\ref{thm: main}. We first observe that if Re$(s_1)=$Re$(s_2)=\varepsilon$ and $\alpha,\beta\ll 1/\log Q$ for all $\alpha\in A$ and $\beta\in B$, then
\begin{align}
& \prod_{p|q}\Bigg\{ \prod_{\substack{\gamma \in A_{s_1}\smallsetminus U_{s_1} \cup (V_{s_2})^{-} \\ \delta \in B_{s_2}\smallsetminus V_{s_2} \cup (U_{s_1})^{-}  }}\left( 1- \frac{1}{p^{1+\gamma+\delta}}\right) \Bigg\} \prod_{p | hk}\Bigg\{ \prod_{\substack{\gamma \in A_{s_1}\smallsetminus U_{s_1} \cup (V_{s_2})^{-} \\ \delta \in B_{s_2}\smallsetminus V_{s_2} \cup (U_{s_1})^{-}  }}\left( 1- \frac{1}{p^{1+\gamma+\delta}}\right)\notag \\
&\hspace{.5in}\times\sum_{ \substack{0\leq m,n<\infty\\ m+\ordp(h) = n+\ordp(k)}}\frac{\tau_{A_{s_1}\smallsetminus U_{s_1} \cup (V_{s_2})^{-}} (p^m) \tau_{B_{s_2}\smallsetminus V_{s_2} \cup (U_{s_1})^{-} } (p^n)  }{p^{m/2}p^{n/2} }\Bigg\} \notag \\
& = \prod_{p|qhk} O(1) \ll (qhk)^{\varepsilon}. \label{eqn: CFKRSeuler3}
\end{align}
Now Lemma~\ref{lem: Lemma2ofCIS} with $m=n=1$ implies
\begin{equation*}
\sideset{}{^\flat}\sum_{\chi \bmod q}1 = \frac{1}{2}\sum_{d|q} \phi(d) \mu \left( \frac{q}{d}\right) +O(1).
\end{equation*}
We insert this into \eqref{eqn: I_l(h,k)def2}. The total contribution of the $O(1)$ error term is at most $\ll_{\varepsilon} X^{\varepsilon}Q^{1+\varepsilon}(hk)^{\varepsilon}$ if we assume \eqref{eqn: orbitals}, since we have \eqref{eqn: zetaalphaalphabound}, \eqref{eqn: CFKRSeuler1}, \eqref{eqn: CFKRSeuler2}, and \eqref{eqn: CFKRSeuler3}. We then write $W(q/Q)$ as an integral using its Mellin transform. We take this integral to be along Re$(w)=2+\varepsilon$ to keep the $q$-sum absolutely convergent. Expressing the $q$-sum as an Euler product using Lemma~\ref{lem: sumstoEulerproducts}, we then deduce from \eqref{eqn: I_l(h,k)def2} that, if \eqref{eqn: orbitals} holds, then
\begin{equation}\label{eqn: IandIstar}
\mathcal{I}_{\ell}(h,k) = \mathcal{I}_{\ell}^*(h,k) + O(X^{\varepsilon}Q^{1+\varepsilon}(hk)^{\varepsilon}),
\end{equation}
where $\mathcal{I}_{\ell}^*(h,k)$ is defined by
\begin{align}
\mathcal{I}_{\ell}^*(h,k)
& = \sum_{\substack{U\subseteq A, V\subseteq B \\ |U|=|V|={\ell}}} \frac{1}{2(2\pi i )^3} \int_{(\varepsilon)} \int_{(\varepsilon)} \int_{(2+\varepsilon)} X^{s_1+s_2} Q^w \widetilde{V}(s_1) \widetilde{V}(s_2)\widetilde{W}(w)  \notag\\
&\hspace{.25in} \times \prod_{\alpha\in U}  \mathscr{X} (\tfrac{1}{2}+\alpha +s_1 ) \prod_{\beta\in V}  \mathscr{X} (\tfrac{1}{2}+\beta+s_2 ) \prod_{\substack{\gamma \in A_{s_1}\smallsetminus U_{s_1} \cup (V_{s_2})^{-} \\ \delta \in B_{s_2}\smallsetminus V_{s_2} \cup (U_{s_1})^{-}  }} \zeta(1+\gamma+\delta)  \notag \\
&\hspace{.25in}  \times \prod_{p|hk} \prod_{\substack{\gamma \in A_{s_1}\smallsetminus U_{s_1} \cup (V_{s_2})^{-} \\ \delta \in B_{s_2}\smallsetminus V_{s_2} \cup (U_{s_1})^{-}  }}\left( 1- \frac{1}{p^{1+\gamma+\delta}}\right) \notag\\
&\hspace{.5in}\times \sum_{ \substack{0\leq m,n<\infty\\ m+\ordp(h) = n+\ordp(k)}}\frac{\tau_{A_{s_1}\smallsetminus U_{s_1} \cup (V_{s_2})^{-}} (p^m) \tau_{B_{s_2}\smallsetminus V_{s_2} \cup (U_{s_1})^{-} } (p^n)  }{p^{m/2}p^{n/2} } \notag\\
&\hspace{.25in}  \times \prod_{p\nmid hk} \Bigg\{ \prod_{\substack{\gamma \in A_{s_1}\smallsetminus U_{s_1} \cup (V_{s_2})^{-} \\ \delta \in B_{s_2}\smallsetminus V_{s_2} \cup (U_{s_1})^{-}  }}\left( 1- \frac{1}{p^{1+\gamma+\delta}}\right) \times \bigg( 1 + \frac{p-2}{p^{w+\sum_{\alpha\in U} (\alpha+s_1) +\sum_{\beta\in V} (\beta + s_2)} }  \notag\\
&\hspace{.5in}  + \left( 1-\frac{1}{p}\right)^2 \frac{p^{2(1-w-\sum_{\alpha\in U} (\alpha+s_1) -\sum_{\beta\in V} (\beta + s_2))}}{1-p^{1-w-\sum_{\alpha\in U} (\alpha+s_1) -\sum_{\beta\in V} (\beta + s_2)}} \bigg) \notag\\
&\hspace{.5in}  + \prod_{\substack{\gamma \in A_{s_1}\smallsetminus U_{s_1} \cup (V_{s_2})^{-} \\ \delta \in B_{s_2}\smallsetminus V_{s_2} \cup (U_{s_1})^{-}  }} \left( 1- \frac{1}{p^{1+\gamma+\delta}}\right) \notag\\
&\hspace{.75in}\times \sum_{m=1}^{\infty} \frac{\tau_{A_{s_1}\smallsetminus U_{s_1} \cup (V_{s_2})^{-}} (p^m) \tau_{B_{s_2}\smallsetminus V_{s_2} \cup (U_{s_1})^{-} } (p^m)  }{p^m}   \Bigg\} \,dw \,ds_2\,ds_1. \label{eqn: Istardef}
\end{align}

\section{Initial setup and outline of the proof of Theorem~\ref{thm: main}}\label{sec: outline}

We may assume that $(q,mnhk)=1$ in the definition \eqref{eqn: S(h,k)def} of $\mathcal{S}(h,k)$ since otherwise the summand is zero. We may thus apply Lemma~\ref{lem: Lemma2ofCIS} to deduce from \eqref{eqn: S(h,k)def} that
\begin{equation}\label{eqn: applyLemma2ofCIS}
\mathcal{S}(h,k) = \frac{1}{2} \sum_{\substack{1\leq q<\infty \\ (q,hk)=1}}W\left(\frac{q}{Q}\right) \sum_{\substack{1\leq m,n<\infty \\ (mn,q)=1}} \frac{\tau_A(m) \tau_B(n) }{\sqrt{mn}} V\left( \frac{m}{X} \right)  V\left( \frac{n}{X} \right) \sum_{\substack{c,d\geq 1 \\ cd=q \\ d|mh\pm nk }} \phi(d)\mu(c).
\end{equation}
Let $C>0$ be a parameter that we will choose to be some power of $Q$ at the end of our proof of Theorem~\ref{thm: main}. We use the notation of \cite{CIS} and split the right-hand side of \eqref{eqn: applyLemma2ofCIS} to write
\begin{equation}\label{eqn: Ssplit}
\mathcal{S}(h,k)=\mathcal{L}(h,k)+\mathcal{D}(h,k)+\mathcal{U}(h,k),
\end{equation}
where $\mathcal{L}(h,k)$ is the sum of the terms with $c>C$, $\mathcal{D}(h,k)$ is the sum of the ``diagonal'' terms with $c\leq C$ and $mh=nk$,  and $\mathcal{U}(h,k)$ is the sum of the ``off-diagonal'' terms with $c\leq C$ and $mh\neq nk$. In other words, $\mathcal{L}(h,k)$, $\mathcal{D}(h,k)$, and $\mathcal{U}(h,k)$ are defined by
\begin{equation}\label{eqn: Lsum}
\mathcal{L}(h,k):= \frac{1}{2}\sum_{\substack{1\leq q<\infty \\ (q,hk)=1}} W\left( \frac{q}{Q}\right)  \sum_{\substack{1\leq m,n<\infty \\ (mn,q)=1}} \frac{\tau_A(m) \tau_B(n) }{\sqrt{mn}} V\left( \frac{m}{X} \right)  V\left( \frac{n}{X} \right)  \sum_{\substack{c>C, d \geq 1 \\ cd=q \\ d|mh\pm nk}}\phi(d)\mu(c),
\end{equation}
\begin{equation}\label{eqn: Dsum}
\mathcal{D}(h,k):= \frac{1}{2}\sum_{\substack{1\leq q<\infty \\ (q,hk)=1}} W\left( \frac{q}{Q}\right)  \sum_{\substack{1\leq m,n<\infty \\ (mn,q)=1\\ mh=nk}} \frac{\tau_A(m) \tau_B(n) }{\sqrt{mn}}V\left( \frac{m}{X} \right)  V\left( \frac{n}{X} \right)\sum_{\substack{1\leq c\leq C, d \geq 1 \\ cd=q \\ d|mh\pm nk}}\phi(d)\mu(c),
\end{equation}
and
\begin{equation}\label{eqn: Usum}
\mathcal{U}(h,k):= \frac{1}{2}\sum_{\substack{1\leq q<\infty \\ (q,hk)=1}} W\left( \frac{q}{Q}\right)  \sum_{\substack{1\leq m,n<\infty \\ (mn,q)=1\\ mh\neq nk}} \frac{\tau_A(m) \tau_B(n) }{\sqrt{mn}} V\left( \frac{m}{X} \right)  V\left( \frac{n}{X} \right)  \sum_{\substack{1\leq c\leq C, d \geq 1 \\ cd=q \\ d|mh\pm nk}}\phi(d)\mu(c),
\end{equation}
respectively. The purpose of splitting the $c$-sum this way is that we need the $c$-sum to be finite when we apply the asymptotic large sieve. 

For the rest of this section, we outline our strategy for estimating each of $\mathcal{L}(h,k)$, $\mathcal{D}(h,k)$, and $\mathcal{U}(h,k)$. The presentation in this section will be terse in comparison to the actual arguments.

We treat $\mathcal{D}(h,k)$ in Section~\ref{sec: diagonal}. There, we extend the $c$-sum in \eqref{eqn: Dsum} to $\infty$, apply Mellin inversion, and then write sums in terms of an Euler product to show that, up to an admissible error term, $\mathcal{D}(h,k)$ equals the zero-swap term $\mathcal{I}_0(h,k)$, which is defined by \eqref{eqn: I_l(h,k)def} with $\ell=0$.

We evaluate $\mathcal{L}(h,k)$ in Section~\ref{sec: Lsum}. As in the approach of \cite{CIS}, we detect the divisibility condition $d|mh\pm nk$ using character sums and split $\mathcal{L}(h,k)$ into
$$
\mathcal{L}^0(h,k) + \mathcal{L}^r(h,k),
$$
where $\mathcal{L}^0(h,k)$ is the contribution of the principal characters while $\mathcal{L}^r(h,k)$ is the rest of the sum. We use M\"{o}bius inversion to convert $\mathcal{L}^0(h,k)$ into a sum over $c\leq C$ and show later that it cancels with a term from our analysis of $\mathcal{U}(h,k)$. We bound
$$
\sum_{h,k\leq Q^{\vartheta}} \frac{\lambda_h\overline{\lambda_k}}{\sqrt{hk}} \mathcal{L}^r(h,k)
$$
by applying Mellin inversion and writing the $m,n$-sum in terms of Dirichlet $L$-functions. We use GLH to bound these $L$-functions, and then apply the large sieve. The role of $C$ here is to make the bound from applying the large sieve $\ll Q^{2-\varepsilon}$. Our use of GLH differs from the approach in \cite{CIS}, where they are able to apply the bound for the fourth moment because they have only a few $L$-functions in their setting. 

The analysis of $\mathcal{U}(h,k)$ forms the most difficult part of the proof, and is done in Sections \ref{sec: U(h,k)split}, \ref{sec: U2}, and \ref{sec: Ur}. The first step in our analysis of $\mathcal{U}(h,k)$ is to make a change of variables and switch from the divisor $d$ of $mh\pm nk$ to the ``complementary modulus'' $\ell$ given by
\begin{equation}\label{eqn: switchtocomplementary}
\ell = \frac{|mh\pm nk|}{d}.
\end{equation}
We then use character sums to detect the condition $\ell | mh\pm nk$ and arrive at (essentially)
\begin{equation*}
\begin{split}
\mathcal{U}(h,k)&\approx \frac{1}{2} \sum_{c=1}^C \mu(c)   \sum_{\substack{1\leq m,n<\infty \\ (mn,c )=1\\ mh\neq nk}} \frac{\tau_A(m) \tau_B(n) }{\sqrt{mn}} V\left( \frac{m}{X} \right)  V\left( \frac{n}{X} \right)\\
&\hspace{.25in} \times\sum_{\ell=1}^{\infty} \frac{1}{ \ell} \sum_{\psi \bmod \ell} \psi (mh) \overline{\psi}( \mp nk) \frac{|mh\pm nk|}{ \ell} W\left( \frac{c|mh\pm nk|}{\ell Q}\right)
\end{split}
\end{equation*}
(the unabridged version of this is \eqref{eqn: Uafterswitch} in Section~\ref{sec: U(h,k)split}). This technique of switching to the complementary modulus is at the heart of the \textit{asymptotic large sieve} due to Conrey, Iwaniec, and Soundararajan \cite{CISAsymptoticLargeSieve}; see also \cite{CIS6th} and \cite{ChandeeLi8Dirichlet}. The purpose of switching from the divisor $d$ to the complementary modulus \eqref{eqn: switchtocomplementary} is to reduce the moduli of the characters we use to detect the divisibility condition. This, in turn, leads to a tighter upper bound when applying the large sieve inequality. Indeed, the variable $d$ in \eqref{eqn: Usum} satisfies $d\asymp Q/c$ because $cd=q$ and $q\asymp Q$ by the support of $W$. Thus, $d$ can be of size $\asymp Q$ since $c$ may be $1$. On the other hand, the variable $\ell$ in \eqref{eqn: switchtocomplementary} can only be at most $\ll XCQ^{\vartheta-1}$ for $h,k\leq Q^{\vartheta}$ since $d\asymp Q/c$, $c\leq C$, and $m,n\ll X$ in \eqref{eqn: Usum} by the support of $V$. If $X\ll Q^{2-\varepsilon}$, then $XCQ^{\vartheta-1}$ is a factor of $Q^{\varepsilon}$ smaller than $Q$ for suitably small $C$ and $\vartheta$. This technique and the asymptotic large sieve have proven to be extremely useful in the study of the family of primitive Dirichlet $L$-functions (see, for example, \cite{CIS6th}, \cite{CISCriticalZeros}, \cite{ChandeeLeeLiuRadziwill}, and \cite{ChandeeLi8Dirichlet}).

After expressing $\mathcal{U}(h,k)$ in terms of character sums, we may split $\mathcal{U}(h,k)$ into
$$
\mathcal{U}^0(h,k) + \mathcal{U}^r(h,k),
$$
where $\mathcal{U}^0(h,k)$ is the contribution of the principal characters while $\mathcal{U}^r(h,k)$ is the rest of the sum. We bound
\begin{equation}\label{eqn: Ursumtobound}
\sum_{h,k\leq Q^{\vartheta}} \frac{\lambda_h\overline{\lambda_k}}{\sqrt{hk}} \mathcal{U}^r(h,k)
\end{equation}
in Section~\ref{sec: Ur} through a procedure similar to that in \cite{CIS}. In this method, we first make a change of variables to remove some of the dependencies of the summation variables $m,n,h,k$ on each other. We then apply Mellin inversion, write the sum in terms of an Euler product, and then move the lines of integration closer to zero so that the resulting exponent of $X$ in the integrand has small real part. The Euler product contains a potentially large number of $L$-function factors, and we use GLH to bound these $L$-functions. We split the integrals into dyadic parts, and bound the Mellin transforms carefully by treating each dyadic part differently. This technical step, which we carry out explicitly in \eqref{eqn: beforelargesieve}, is a bit more delicate than the estimations in \cite{CIS} because there are more variables of integration after we apply Mellin inversion. Finally, we apply the large sieve inequality to estimate the character sums. It is at this point that we see the effectiveness of using the complementary modulus \eqref{eqn: switchtocomplementary}. If the character sums involve characters of modulus $Q$, then the large sieve inequality alone may not be enough to show that \eqref{eqn: Ursumtobound} has order of magnitude smaller than that of the main term in the predicted asymptotic formula for $\mathcal{S}(h,k)$.

To evaluate the contribution $\mathcal{U}^0(h,k)$ of the principal characters, we first apply Mellin inversion on the function $W$ and write the $\ell$-sum as an Euler product using Lemma 6 of \cite{CIS} (Lemma~\ref{lem: Lemma6ofCIS} in Section~\ref{sec: U(h,k)split}). We then move the line of integration to write
$$
\mathcal{U}^0(h,k) = \mathcal{U}^1(h,k) + \mathcal{U}^2(h,k),
$$
where $\mathcal{U}^1(h,k)$ is the residue from the pole of the (analytic continuation of the) Euler product, while $\mathcal{U}^2(h,k)$ is the integral along the new line. The residue $\mathcal{U}^1(h,k)$ is equal to the negative of $\mathcal{L}^0(h,k)$ plus an admissible error term, and thus cancels $\mathcal{L}^0(h,k)$.

We analyze the integral $\mathcal{U}^2(h,k)$ in Section~\ref{sec: U2} to uncover the predicted one-swap terms. This is where we carry out the delicate contour integration mentioned below Theorem~\ref{thm: main}. To begin, we apply Proposition~2 of \cite{CIS} (stated as Proposition~\ref{prop: CISProp2} in Section~\ref{sec: U2}) and separate the variables $m$ and $n$ in $|mh\pm nk|$ by writing $|mh\pm nk|^w$ in terms of an integral of a meromorphic function. We then apply Mellin inversion on the function $V$ and express the sum as an Euler product. We determine the analytic continuation of this Euler product, and then move the lines of integration to suitable locations to express $\mathcal{U}^2(h,k)$ as a sum of several residues and error terms. We use the Lindel\"{o}f Hypothesis for $\zeta(s)$ to justify moving some of the lines of integration and to bound one of the error terms. We also carry out a similar analysis of the sum $\mathcal{I}_1(h,k)$ of the one-swap terms from Conjecture~\ref{con: conjecture}. We then find that each residue in the expression for $\mathcal{U}^2(h,k)$ can be matched with a residue in the expression for $\mathcal{I}_1(h,k)$ in such a way that corresponding residues are equal up to a negligible error term. This step requires proving identities involving several Euler products. These Euler product identities, in turn, are consequences of certain properties of the function $\tau_E$, the chief one being
\begin{equation*}
\begin{split}
\tau_{A\smallsetminus\{\alpha\}\cup \{-\beta\}} (p^j) \tau_{B\smallsetminus\{\beta\}} (p^{\ell}) + \tau_{A\smallsetminus\{\alpha\}} (p^j) \tau_{B\smallsetminus\{\beta\}\cup \{-\alpha\}} (p^{\ell}) -  \tau_{A\smallsetminus\{\alpha\}} (p^j) \tau_{B\smallsetminus\{\beta\}} (p^{\ell}) \\
= \tau_{A\smallsetminus\{\alpha\}\cup \{-\beta\}} (p^j) \tau_{B\smallsetminus\{\beta\} \cup\{-\alpha\}} (p^{\ell}) -p^{\alpha+\beta} \tau_{A\smallsetminus\{\alpha\}\cup \{-\beta\}}(p^{j-1}) \tau_{B\smallsetminus\{\beta\} \cup\{-\alpha\}} (p^{\ell-1})
\end{split}
\end{equation*}
(Lemma~\ref{lem: CK3identity} in Section~\ref{sec: U2}), which stems from the work of Conrey and Keating \cite{CK3} on moments of zeta. Conjecture \ref{con: conjecture}, predicted by the CFKRS recipe, plays a crucial role in the analysis of $\mathcal{U}^2(h,k)$, as it provides a clear answer to aim for in untangling $\mathcal{U}^2(h,k)$.

\

\noindent{\it Changes in the proof for the odd case.} We now describe the changes we need to make in our proof in order to handle the odd primitive characters. The version of Lemma~\ref{lem: Lemma2ofCIS} for odd primitive characters states that if $(mn,q)=1$, then
\begin{equation*}
\sideset{}{^{\text{odd}}}\sum_{\chi \bmod q}\chi(m)\overline{\chi(n)} = \frac{1}{2}\sum_{\substack{d|q \\ d|m-n}} \phi(d)\mu\left( \frac{q}{d}\right) - \frac{1}{2}\sum_{\substack{d|q \\ d|m+n}} \phi(d)\mu\left( \frac{q}{d}\right),
\end{equation*}
where the superscript ``odd'' indicates that the sum is over all the odd primitive characters. Thus, to handle the sum over the odd primitive characters, we change our convention about the symbol $\pm$ and have $-1$ multiplied to the copy that has $\pm$ replaced by $+$. A consequence of this sign change is that the analogues of $\mathcal{L}^0(h,k)$ and $\mathcal{U}^0(h,k)$ for odd primitive characters are zero. The main term in the asymptotic formula for the analogue of $\mathcal{D}(h,k)$ is unaffected by the sign change, and so \eqref{eqn: Dis0swap} still holds with $\mathcal{D}(h,k)$ replaced by its analogue. The sign change does not affect the other bounds in our proof. In evaluating the analogue of $\mathcal{U}^2(h,k)$, instead of using Proposition~\ref{prop: CISProp2}, we use the version of it for
$$
|1-r|^{-\omega} - |1+r|^{-\omega}.
$$
This version has the function
$$
\mathscr{X}(\omega)\mathscr{Y}(1-z)\mathscr{Y}(1-\omega+z)
$$
in place of $\mathcal{H}(z,\omega)$, where $\mathscr{Y}(s)$ is defined by
$$
\mathscr{Y}(s) = \pi^{s-\frac{1}{2}} \frac{\Gamma( 1-\frac{1}{2}s)}{\Gamma(\frac{1}{2}+\frac{1}{2}s)}.
$$

\section{The diagonal terms \texorpdfstring{$\mathcal{D}(h,k)$}{D(h,k)}}\label{sec: diagonal}
In this section, we focus on the sum $\mathcal{D}(h,k)$ of the diagonal terms, defined by \eqref{eqn: Dsum}. We first perform a short analysis of the main contribution $\mathcal{I}^*_0(h,k)$ of the zero-swap term. We will then see that $\mathcal{I}^*_0(h,k)$ coincides exactly with the main contribution of $\mathcal{D}(h,k)$.

\subsection{The prediction for the zero-swap term}
We may simplify $\mathcal{I}^*_0(h,k)$, defined by \eqref{eqn: Istardef} with $\ell=0$, by cancelling the zeta-function factors $\zeta(1+\alpha+\beta+s_1+s_2)$ with the convergent products of the corresponding local factors. We also apply \eqref{eqn: taufactoringidentity}. The result is
\begin{align*}
\begin{split}
\mathcal{I}^*_0(h,k)=  \frac{1}{2(2\pi i )^3} &\int_{(\varepsilon)} \int_{(\varepsilon)} \int_{(2+\varepsilon)} X^{s_1+s_2} Q^w \widetilde{V}(s_1) \widetilde{V}(s_2)\widetilde{W}(w)\\
   &\hspace{.25in}\times\prod_{p|hk} \sum_{ \substack{0\leq m,n<\infty\\ m+\ordp(h) = n+\ordp(k)}}\frac{\tau_{A}(p^m)\tau_{B} (p^n)}{p^{m(1/2+s_1)}p^{n(1/2+s_2)} } \\
&\hspace{.75in}\times \prod_{p \nmid hk} \left( \frac{(1-p^{-w})^2}{1-p^{1-w}}+ \sum_{\ell=1}^{\infty} \frac{\tau_{A} (p^\ell) \tau_{B } (p^{\ell} ) }{p^{\ell(1+s_1+s_2)}}   \right) \,dw \,ds_2\,ds_1.
\end{split}
\end{align*}
To simplify the latter $m,n$-sum, define $H:=h/(h,k)$ and $K:=k/(h,k)$. A given pair $m,n$ is a pair of nonnegative integers with $m+\ordp(h) = n+\ordp(k)$ if and only if there is a nonnegative integer $\ell$ such that $m=\ell + \ordp(K)$ and $n= \ell + \ordp(H)$. Hence we may write the $m,n$ sum as 
\begin{align*}
\frac{1}{p^{\ordp(K)(1/2+s_1)+\ordp(H)(1/2+s_2)}} \sum_{\ell=0}^{\infty}\frac{\tau_A(p^{\ordp(K)+\ell})\tau_B(p^{\ordp(H)+\ell})}{p^{\ell(1+s_1+s_2)}}.
\end{align*}
Thus we predict that
\begin{equation}\label{eqn: zero-swap}
\begin{split}
\mathcal{I}^*_0(h,k)=\frac{1}{2(2\pi i )^3} &\int_{(\varepsilon)} \int_{(\varepsilon)} \int_{(2+\varepsilon)} \frac{X^{s_1+s_2}}{H^{1/2+s_2}K^{1/2+s_1}} Q^w \widetilde{V}(s_1) \widetilde{V}(s_2)\widetilde{W}(w)\\
&\times \prod_{p|hk} \sum_{\ell=0}^{\infty}\frac{\tau_A(p^{\ordp(K)+\ell})\tau_B(p^{\ordp(H)+\ell})}{p^{\ell(1+s_1+s_2)}}\\
&\hspace{.5in}\times\prod_{p\nmid hk} \left( \frac{(1-p^{-w})^2}{1-p^{1-w}}+ \sum_{\ell=1}^{\infty} \frac{\tau_{A} (p^\ell) \tau_{B } (p^{\ell} ) }{p^{\ell(1+s_1+s_2)}}   \right) \,dw \,ds_2\,ds_1.
\end{split}
\end{equation}

\subsection{\texorpdfstring{$\mathcal{D}(h,k)$}{D(h,k)} coincides with the prediction for the zero-swap term}
In this subsection, we show that $\mathcal{D}(h,k)$, defined by \eqref{eqn: Dsum}, is equal to the right-hand side of \eqref{eqn: zero-swap} plus an admissible error term. To this end, we first make a change of variables in the $m,n$ sum. Since  $H:=h/(h,k)$ and $K:=k/(h,k),$ the condition $mh=nk$ is equivalent to the condition that $m=K\ell \text{ and } n=H\ell$ for some positive integer $\ell$. We thus arrive at
\begin{equation*}
\begin{split}
\mathcal{D}(h,k)= \frac{1}{2}\sum_{\substack{1\leq q<\infty \\ (q,hk)=1}}& W\left( \frac{q}{Q}\right)  \sum_{\substack{1\leq \ell <\infty \\ (\ell,q)=1}} \frac{\tau_A(K\ell) \tau_B(H\ell) }{\ell\sqrt{HK}} V\left( \frac{K\ell}{X} \right) 
V\left( \frac{H\ell}{X} \right)  \sum_{\substack{1\leq c\leq C, d \geq 1 \\ cd=q \\ d|K\ell h\pm H\ell k}}\phi(d)\mu(c).
\end{split}
\end{equation*}
Recall that we use the notation $d|K\ell h\pm H\ell k$ to signify that we are adding two copies of the sum: one with $d|K\ell h- H\ell k$ and the other with $d|K\ell h+ H\ell k$.  In the first copy, we are summing over all $d$ because $K h= H k$. In the second copy, the condition that $d$ divides $K\ell h+ H\ell k$  is equivalent to the condition that $d|2$ because $K h= H k$ and $(q,hk\ell)=1$. Thus, the $c,d$-sum in the second copy has at most two terms, and so the second copy is bounded by 
\[
\ll Q\sum_{\ell \ll X} \frac{(HK\ell)^\varepsilon}{\ell \sqrt{HK}} \ll Q\frac{(XHK)^\varepsilon}{\sqrt{HK}}.
\]
Hence
\begin{align*}
\mathcal{D}(h,k)&=\frac{1}{2}\sum_{\substack{1\leq q<\infty \\ (q,hk)=1}} W\left( \frac{q}{Q}\right) \sum_{\substack{1\le \ell < \infty\\(\ell,q)=1}}\frac{\tau_A(K\ell)\tau_B(H\ell)}{\ell\sqrt{HK}}V\left(\frac{K\ell}{X}\right)V\left(\frac{H\ell}{X}\right)\\
&\hspace{.25in}\times\sum_{\substack{1\leq c\leq C, d \geq 1 \\ cd=q}}\phi(d)\mu(c)+ O\left( Q\frac{(XHK)^\varepsilon}{\sqrt{HK}} \right).
\end{align*}
We next extend the $c$-sum to $\infty$.  The error introduced in doing so is
\begin{align*}
&\ll \sum_{q \ll Q}\sum_{\ell \ll X} \frac{(HK\ell)^\varepsilon}{\ell \sqrt{HK}}\sum_{\substack{c>C , d \geq 1 \\ cd=q}}\phi(d) \ \ll \frac{Q^2}{C} \frac{(XHK)^{\varepsilon}}{\sqrt{HK}}.
\end{align*}
Note that we are careful to estimate the $c$-sum in terms of $C$, which is necessary because the main term in Theorem~\ref{thm: main} is of size about $Q^2$. Later, we will choose $C$ as a specific positive power of $Q$ to control this error term. Setting $\phi^\star(q):=\sum_{cd=q}\phi(d)\mu(c)$, we now have
\begin{align*}
    \mathcal{D}(h,k)&=\frac{1}{2}\sum_{\substack{1\le q <\infty \\ (q,hk)=1}}W\left(\frac{q}{Q}\right)\phi^\star(q)\sum_{\substack{1\le \ell < \infty \\ (\ell,q)=1}}\frac{\tau_A(K\ell)\tau_B(H\ell)}{\ell \sqrt{HK}}V\left(\frac{K\ell}{X}\right)V\left(\frac{H\ell}{X}\right)\\
    &\hspace{.25in}+ O\left(\left(Q+\frac{Q^2}{C}\right)\frac{(XHK)^\varepsilon}{\sqrt{HK}}\right).
\end{align*}
Next, write $V,W$ in terms of their Mellin transforms using Mellin inversion to find
\begin{align}\label{eqn: diagonalapplyMellin}
\begin{split}
    \mathcal{D}(h,k)&=\frac{1}{2(2\pi i)^3}\int_{(\varepsilon)}\int_{(\varepsilon)}\frac{X^{s_1+s_2}}{H^{1/2+s_1}K^{1/2+s_2}}\widetilde{V}(s_1)\widetilde{V}(s_2)\int_{(2+\varepsilon)}Q^w\widetilde{W}(w)\\
    &\hspace{.25in}\times\sum_{\substack{1\le q <\infty \\ (q,hk)=1}}q^{-w}\phi^\star(q)\sum_{\substack{1\le \ell < \infty \\ (\ell,q)=1}}\frac{\tau_A(K\ell)\tau_B(H\ell)}{\ell^{1+s_1+s_2}}\,dw\,ds_2\,ds_1\\
    &\hspace{.5in}+O\left(\left(Q+\frac{Q^2}{C}\right)\frac{(XHK)^\varepsilon}{\sqrt{HK}}\right),
\end{split}
\end{align}
where we have chosen the location of the $w$-line to be along Re$(w)=2 + \varepsilon$ to ensure that the $q$-sum is absolutely convergent. We may then rewrite the $q,\ell$-sum in \eqref{eqn: diagonalapplyMellin} as the Euler product
\begin{equation}\label{eqn: diagonalqlsumeuler}
    \prod_{p}\Bigg(\sum_{\substack{0\le q < \infty\\ \min\{ q,\ordp(h)+\ordp(k)\}=0}}\!\!\!\!\!\!\!\!p^{-qw}\phi^{\star}(p^q)\sum_{\substack{0\le \ell < \infty\\\min\{\ell,q\}=0}}\frac{\tau_A(p^{\ordp(K)+\ell})\tau_B(p^{\ordp(H)+\ell})}{p^{\ell(1+s_1+s_2)}}\Bigg).
\end{equation}
If $p|hk$, then $\ordp(h)+\ordp(k)\ge 1$. In this case, for the condition $\min\{ q,\ordp(h)+\ordp(k)\}=0$ to hold, we must have $q=0$. Since $\phi^{\star}(p^0) =1$, it follows that the contribution to the Euler product from the primes dividing $hk$ is
\begin{align*}
    \prod_{p|hk} \Bigg(\sum_{\ell = 0}^{\infty}\frac{\tau_A(p^{\ordp(K)+\ell})\tau_B(p^{\ordp(H)+\ell})}{p^{\ell(1+s_1+s_2)}} \Bigg),
\end{align*}
which we note has no dependence on $w$. Now suppose that $p\nmid hk$. Then $\ordp(h)+\ordp(k)=0$, which means we may drop the condition that $\min\{ q,\ordp(h)+\ordp(k)\}=0$. The contribution to the Euler product from primes not dividing $hk$ is thus
\begin{align*}
    \prod_{p\nmid hk}\sum_{\substack{0\le q,\ell < \infty\\ \min(\ell,q)=0}}p^{-qw}\phi^{\star}(p^q) \frac{\tau_A(p^{\ell})\tau_B(p^{\ell})}{p^{\ell(1+s_1+s_2)}} = \prod_{p\nmid hk}\left( 1 + \sum_{q=1}^{\infty}p^{-qw}\phi^{\star}(p^q)+ \sum_{\ell=1}^{\infty}\frac{\tau_A(p^{\ell})\tau_B(p^{\ell})}{p^{\ell(1+s_1+s_2)}}\right)
\end{align*}
Inserting the definition of $\phi^{\star}$ into the $q$-sum, we directly calculate the $q=1$ term and realize the sum of the terms with $q>1$ as a geometric series to find, after a short calculation, that
\begin{align*}
    1 + \sum_{q=1}^{\infty}p^{-qw}\phi^{\star}(p^q) &=1 + \sum_{q=1}^{\infty}p^{-qw} \sum_{cd=p^q}\phi(d)\mu(c)
    =\left(\frac{1}{1-p^{1-w}}\right)\left( 1-p^{-w}\right)^2.
\end{align*}
Hence, writing the $q,\ell$-sum in \eqref{eqn: diagonalapplyMellin} as the Euler product \eqref{eqn: diagonalqlsumeuler} and applying the above simplifications, we arrive at
\begin{align*}
     \mathcal{D}(h,k)&=\frac{1}{2(2\pi i)^3}\int_{(\varepsilon)}\int_{(\varepsilon)}\int_{(2+\varepsilon)}\frac{X^{s_1+s_2}}{H^{1/2+s_1}K^{1/2+s_2}}Q^w\widetilde{V}(s_1)\widetilde{V}(s_2)\widetilde{W}(w)\\
     &\hspace{.25in}\times\prod_{p|hk} \left(\sum_{\ell = 0}^{\infty}\frac{\tau_A(p^{\ordp(K)+\ell})\tau_B(p^{\ordp(H)+\ell})}{p^{\ell(1+s_1+s_2)}} \right)\\
     &\hspace{1in}\times\prod_{p\nmid hk}\left( \frac{(1-p^{-w})^2}{1-p^{1-w}}  + \sum_{\ell=1}^{\infty}\frac{\tau_A(p^{\ell})\tau_B(p^{\ell})}{p^{\ell(1+s_1+s_2)}}\right)\,dw\,ds_1\,ds_2\\
     &\hspace{1.5in}+ O\left(\left(Q+\frac{Q^2}{C}\right)\frac{(XHK)^\varepsilon}{\sqrt{HK}}\right).
\end{align*}
After relabeling $s_1$ as $s_2$ and vice versa, we see that the integral above exactly matches the right-hand side of \eqref{eqn: zero-swap}. In other words,
\begin{equation}\label{eqn: Dis0swap}
\mathcal{D}(h,k)=\mathcal{I}^*_0(h,k) + O\left(\left(Q+\frac{Q^2}{C}\right)\frac{(XHK)^\varepsilon}{\sqrt{HK}}\right).
\end{equation}

\section{The term \texorpdfstring{$\mathcal{L}(h,k)$}{L(h,k)}}\label{sec: Lsum}
Recall the definition \eqref{eqn: Lsum} of $\mathcal{L}(h,k)$, and recall that we interpret the $d$-sum therein as two sums: one with the condition $d|mh-nk$ and the other with the condition $d|mh+nk$. We first show how to re-express $\mathcal{L}(h,k)$ in terms of characters modulo $d$. For $(mnhk,d)=1$, the orthogonality of character sums implies
\begin{align*}
\frac{1}{\phi(d)}\sum_{\psi \bmod d} \psi(mh)\overline{\psi}(nk) = \left\{\begin{array}{cl} 1 & \text{if } d|mh-nk \\ \\ 0 & \text{else} \end{array}\right.
\end{align*}
and
\begin{align*}
\frac{1}{\phi(d)}\sum_{\psi \bmod d} \psi(mh)\overline{\psi}(-nk) = \left\{\begin{array}{cl} 1 & \text{if } d|mh+nk \\ \\ 0 & \text{else}. \end{array}\right.
\end{align*}
Since $\overline{\psi}(1)+\overline{\psi}(-1)=2$ if $\psi$ is even and $0$ if $\psi$ is odd, it follows that the sum of these two character sums is
$$
\frac{2}{\phi(d)}\sum_{\substack{\psi \bmod d \\ \psi \text{ even}}} \psi(mh)\overline{\psi}( nk).
$$
Therefore, we may recast $\mathcal{L}(h,k)$ as
\begin{align*}
\mathcal{L}(h,k) &=  \sum_{\substack{1\leq q<\infty \\ (q,hk)=1}} W\left( \frac{q}{Q}\right)  \sum_{\substack{1\leq m,n<\infty \\ (mn,q)=1}} \frac{\tau_A(m) \tau_B(n) }{\sqrt{mn}} V\left( \frac{m}{X} \right)  V\left( \frac{n}{X} \right)\\
&\hspace{.5in}\times\sum_{\substack{c>C, d \geq 1 \\ cd=q  }} \mu(c)\sum_{\substack{\psi \bmod d \\ \psi \text{ even}}} \psi(mh)\overline{\psi}( nk).
\end{align*}
Split the right-hand side to write
\begin{equation}\label{eqn: Lsplit}
\mathcal{L}(h,k) = \mathcal{L}^0(h,k)+\mathcal{L}^r(h,k),
\end{equation}
where $\mathcal{L}^0(h,k)$ is the contribution of the principal character modulo $d$ and $\mathcal{L}^r(h,k)$ is the rest. In other words,
\begin{equation}\label{eqn: L0}
\begin{split}
\mathcal{L}^0(h,k) := \sum_{\substack{1\leq q<\infty \\ (q,hk)=1}} W\left( \frac{q}{Q}\right)  \sum_{\substack{1\leq m,n<\infty \\ (mn,q)=1}} \frac{\tau_A(m) \tau_B(n) }{\sqrt{mn}} V\left( \frac{m}{X} \right)  V\left( \frac{n}{X} \right)
\sum_{\substack{c>C, d \geq 1 \\ cd=q  }} \mu(c)
\end{split}
\end{equation}
and
\begin{equation*}
\begin{split}
\mathcal{L}^r(h,k) :&= \sum_{\substack{1\leq q<\infty \\ (q,hk)=1}} W\left( \frac{q}{Q}\right)  \sum_{\substack{1\leq m,n<\infty \\ (mn,q)=1}} \frac{\tau_A(m) \tau_B(n) }{\sqrt{mn}} V\left( \frac{m}{X} \right)  V\left( \frac{n}{X} \right)\\
&\hspace{.5in}\times  \sum_{\substack{c>C, d \geq 1 \\ cd=q  }} \mu(c)\sum_{\substack{\psi \bmod d \\ \psi \text{ even} \\ \psi\neq \psi_0}} \psi(mh)\overline{\psi}( nk),
\end{split}
\end{equation*}
where $\psi_0$ denotes the principal character modulo $d$. 

In this section, we have two goals. First, we will bound the contribution of $\mathcal{L}^r(h,k)$ and show, on average over $h,k$, that it is an acceptable error term. Second, we will rework $\mathcal{L}^0(h,k)$ in preparation to show (later, in Section \ref{sec: U1}) that $\mathcal{L}^0(h,k)$ cancels with a term arising during the analysis of  $\mathcal{U}(h,k)$.

\subsection{Bounding the contribution of \texorpdfstring{$\mathcal{L}^r(h,k)$}{Lr(h,k)} }
We may freely interchange the order of summation because each of $W$ and $V$ has compact support, forcing the sums to be finite. We bring the $m,n$-sum inside and then use Mellin inversion to write
\begin{equation*}
\begin{split}
\mathcal{L}^r(h,k)  &=  \sum_{\substack{1\leq q<\infty \\ (q,hk)=1}} W\left( \frac{q}{Q}\right)   \sum_{\substack{c>C, d \geq 1 \\ cd=q  }} \mu(c)\sum_{\substack{\psi \bmod d \\ \psi \text{ even} \\ \psi\neq \psi_0}} \psi( h)\overline{\psi}(  k)  \sum_{\substack{1\leq m,n<\infty \\ (mn,q)=1}} \frac{\tau_A(m) \tau_B(n) \psi(m )\overline{\psi}( n ) }{\sqrt{mn}}\\
&\hspace{.5in}\times\frac{1}{(2\pi i)^2}\int_{(\frac{1}{2}+\varepsilon)} \int_{(\frac{1}{2}+\varepsilon)}  \frac{X^{s_1+s_2}}{m^{s_1}n^{s_2}}\widetilde{V}(s_1)\widetilde{V}(s_2) \,ds_2 \,ds_1,
\end{split}
\end{equation*}
where we have chosen the lines of integration to be at $\re(s_1)=\re(s_2)=\frac{1}{2}+\varepsilon$ so that in the next step we can interchange the $m,n$-sum and the integrals. Since $q=cd$ and $\psi(\nu)=0$ for $(\nu,d)>1$, the $m,n$-sum is the same as
\begin{equation*}
\begin{split}
\sum_{\substack{1\leq m,n<\infty \\ (mn,c)=1}} \frac{\tau_A(m) \tau_B(n) \psi(m )\overline{\psi}( n ) }{m^{\frac{1}{2}+s_1}n^{\frac{1}{2}+s_2}} &= \prod_{\alpha\in A} L(\tfrac{1}{2}+s_1+\alpha,\psi) \prod_{\beta\in B} L(\tfrac{1}{2}+s_2+\beta,\overline{\psi})\\
&\hspace{.25in}\times \prod_{\alpha\in A}\Bigg(\prod_{p|c}\bigg(1- \frac{\psi(p)}{p^{\frac{1}{2}+s_1+\alpha}} \bigg) \Bigg) \prod_{\beta\in B}\Bigg(\prod_{p|c}\bigg(1- \frac{\overline{\psi}(p)}{p^{\frac{1}{2}+s_2+\beta}} \bigg) \Bigg).
\end{split}
\end{equation*}
Therefore, we have 
\begin{equation*}
\begin{split}
\mathcal{L}^r(h,k)=&\sum_{\substack{1\leq q<\infty \\ (q,hk)=1}} W\left( \frac{q}{Q}\right)   \sum_{\substack{c>C, d \geq 1 \\ cd=q  }} \mu(c)\sum_{\substack{\psi \bmod d \\ \psi \text{ even} \\ \psi\neq \psi_0}} \psi( h)\overline{\psi}(  k)\frac{1}{(2\pi i)^2}\int_{(\frac{1}{2}+\varepsilon)} \int_{(\frac{1}{2}+\varepsilon)}  X^{s_1+s_2} \\
&\hspace{.25in}\times \widetilde{V}(s_1)\widetilde{V}(s_2) 
\prod_{\alpha\in A} L(\tfrac{1}{2}+s_1+\alpha,\psi) \prod_{\beta\in B} L(\tfrac{1}{2}+s_2+\beta,\overline{\psi})\\
&\hspace{.5in} \times    \prod_{\alpha\in A}\Bigg(\prod_{p|c}\bigg(1- \frac{\psi(p)}{p^{\frac{1}{2}+s_1+\alpha}} \bigg) \Bigg) \prod_{\beta\in B}\Bigg(\prod_{p|c}\bigg(1- \frac{\overline{\psi}(p)}{p^{\frac{1}{2}+s_2+\beta}} \bigg) \Bigg) \,ds_2 \,ds_1.
\end{split}
\end{equation*}
We may now move the lines of integration to $\re(s_1)=\re(s_2)=\varepsilon$ by the rapid decay of $\widetilde{V}(s_1)$ and $\widetilde{V}(s_2)$ and the fact that $L(s,\psi)$ has no pole whenever $\psi$ is non-principal. We multiply both sides of the above equation by $ \lambda_h \overline{\lambda_k}(hk)^{-1/2}$, and then sum over all positive integers $h,k\leq Q^{\vartheta}$ to arrive at the quantity we aim to bound:
\begin{equation}\label{eqn: Lrweaimtobound}
\begin{split}
&\sum_{h,k\leq Q^{\vartheta}} \frac{\lambda_h \overline{\lambda_k}}{\sqrt{hk}} \mathcal{L}^r(h,k)  =  \sum_{h,k\leq Q^{\vartheta}} \frac{\lambda_h \overline{\lambda_k}}{\sqrt{hk}}\sum_{\substack{1\leq q<\infty \\ (q,hk)=1}} W\left( \frac{q}{Q}\right)   \sum_{\substack{c>C, d \geq 1 \\ cd=q  }} \mu(c)\sum_{\substack{\psi \bmod d \\ \psi \text{ even} \\ \psi\neq \psi_0}} \psi( h)\overline{\psi}(  k)\\
&\hspace{.25in}\times \frac{1}{(2\pi i)^2}\int_{( \varepsilon)} \int_{( \varepsilon)}  X^{s_1+s_2} \widetilde{V}(s_1)\widetilde{V}(s_2)\prod_{\alpha\in A} L(\tfrac{1}{2}+s_1+\alpha,\psi) \prod_{\beta\in B} L(\tfrac{1}{2}+s_2+\beta,\overline{\psi})  \\
&\hspace{.5in}\times    \prod_{\alpha\in A}\Bigg(\prod_{p|c}\bigg(1- \frac{\psi(p)}{p^{\frac{1}{2}+s_1+\alpha}} \bigg) \Bigg) \prod_{\beta\in B}\Bigg(\prod_{p|c}\bigg(1- \frac{\overline{\psi}(p)}{p^{\frac{1}{2}+s_2+\beta}} \bigg) \Bigg) \,ds_2 \,ds_1.
\end{split}
\end{equation}
Now observe that
$$
\prod_{p|c}\bigg|1- \frac{\psi(p)}{p^{\frac{1}{2}+z}} \bigg| \leq \prod_{p|c}(2) \ll_{\varepsilon} c^{\varepsilon}
$$
for any complex number $z$ with $|z|<1/2$. Moreover, it holds that
$$
\sum_{\substack{h,k\leq Q^{\vartheta} \\ (hk,q)=1} } \frac{\lambda_h \overline{\lambda_k} \psi( h)\overline{\psi}(  k) }{\sqrt{hk}} = \Bigg|\sum_{\substack{h \leq Q^{\vartheta} \\ (h ,q)=1} } \frac{\lambda_h  \psi( h)  }{\sqrt{h }} \Bigg|^2.
$$
We bound the $L$-functions in \eqref{eqn: Lrweaimtobound} by assuming GLH\footnote{We must assume GLH in this step because of the potentially large number of $L(s,\psi)$ factors. This differs from the argument in Conrey et~al.~\cite{CIS}, where they bound the size of the square of the $L$-function using the large sieve and the approximate functional equation (see the argument following equation (4.6) in \cite{CIS}).}. It follows from these and the triangle inequality that
\begin{equation*}
\begin{split}
\sum_{h,k\leq Q^{\vartheta}} \frac{\lambda_h \overline{\lambda_k}}{\sqrt{hk}} \mathcal{L}^r(h,k)  &\ll_{\varepsilon} X^{\varepsilon}  \sum_{ 1\leq q<\infty  } W\left( \frac{q}{Q}\right)   \sum_{\substack{c>C, d \geq 1 \\ cd=q  }} (cd)^{\varepsilon} \sum_{\substack{\psi \bmod d \\ \psi \text{ even} \\ \psi\neq \psi_0}}  \Bigg|\sum_{\substack{h \leq Q^{\vartheta} \\ (h ,q)=1} } \frac{\lambda_h  \psi( h)  }{\sqrt{h }} \Bigg|^2   \\
&\hspace{.5in}\times   \int_{( \varepsilon)} \int_{( \varepsilon)}   |s_1s_2|^{\varepsilon} |\widetilde{V}(s_1)||\widetilde{V}(s_2)|  \,|ds_2 \,ds_1|.
\end{split}
\end{equation*}
The rapid decay of $\widetilde{V}$ implies that the latter double integral is $\ll 1$. We substitute $q=cd$ and write the $q$-sum as a double sum over $c$ and $d$. Furthermore, in preparation to use the large sieve, we express each $\psi$ in terms of the primitive character it is induced by to deduce the upper bound
\begin{equation*}
\begin{split}
\sum_{h,k\leq Q^{\vartheta}} \frac{\lambda_h \overline{\lambda_k}}{\sqrt{hk}} \mathcal{L}^r(h,k)  \ll_{\varepsilon} (XQ)^{\varepsilon} \sum_{c>C} \sum_{d=1}^{\infty} W\left( \frac{cd}{Q}\right) \sum_{u|d} \,\sideset{}{^\flat}\sum_{\substack{\psi \bmod u \\ \psi\neq \psi_0}}\Bigg|\sum_{\substack{h \leq Q^{\vartheta} \\ (h ,q)=1} } \frac{\lambda_h  \psi( h)  }{\sqrt{h }} \Bigg|^2,
\end{split}
\end{equation*}
where we again use $\flat$ to denote that the sum is over even primitive characters. We substitute $d=ru$ and write the $d$-sum as a double sum over $r$ and $u$ to arrive at
\begin{equation*}
\begin{split}
\sum_{h,k\leq Q^{\vartheta}} \frac{\lambda_h \overline{\lambda_k}}{\sqrt{hk}} \mathcal{L}^r(h,k)  \ll_{\varepsilon} (XQ)^{\varepsilon} \sum_{c>C} \sum_{r=1}^{\infty} \sum_{u=1}^{\infty} W\left( \frac{cru}{Q}\right)  \,\sideset{}{^\flat}\sum_{\substack{\psi \bmod u \\ \psi\neq \psi_0}}  \Bigg|\sum_{\substack{h \leq Q^{\vartheta} \\ (h ,q)=1} } \frac{\lambda_h  \psi( h)  }{\sqrt{h }} \Bigg|^2 .
\end{split}
\end{equation*}
Since $W$ is bounded and compactly supported, it follows that
\begin{equation*}
\begin{split}
\sum_{h,k\leq Q^{\vartheta}} \frac{\lambda_h \overline{\lambda_k}}{\sqrt{hk}} \mathcal{L}^r(h,k)  \ll_{\varepsilon} (XQ)^{\varepsilon} \sum_{C<c\ll Q} \sum_{r\ll \frac{Q}{c}} \sum_{u \ll \frac{Q}{cr}}   \,\sideset{}{^\flat}\sum_{\substack{\psi \bmod u \\ \psi\neq \psi_0}}  \Bigg|\sum_{\substack{h \leq Q^{\vartheta} \\ (h ,q)=1} } \frac{\lambda_h  \psi( h)  }{\sqrt{h }} \Bigg|^2 .
\end{split}
\end{equation*}
The large sieve (see, for example, \cite[\S27, Theorem 4]{Davenport}) implies that
\begin{equation*}
\sum_{u \ll \frac{Q}{cr}}   \,\sideset{}{^\flat}\sum_{\substack{\psi \bmod u \\  \psi\neq \psi_0}}  \Bigg|\sum_{\substack{h \leq Q^{\vartheta} \\ (h ,q)=1} } \frac{\lambda_h  \psi( h)  }{\sqrt{h }} \Bigg|^2 \ll \bigg(Q^{\vartheta} + \frac{Q^2}{c^2r^2} \bigg)\sum_{\substack{h \leq Q^{\vartheta} \\ (h ,q)=1} } \frac{|\lambda_h|^2  }{h}.
\end{equation*}
Hence, since $\lambda_h \ll_{\varepsilon} h^{\varepsilon}$, it follows that
\begin{equation}\label{eqn: Lrbound}
\begin{split}
\sum_{h,k\leq Q^{\vartheta}} \frac{\lambda_h \overline{\lambda_k}}{\sqrt{hk}} \mathcal{L}^r(h,k)
& \ll_{\varepsilon} (XQ)^{\varepsilon} \sum_{C<c\ll Q} \sum_{r\ll \frac{Q}{c}}  \bigg(Q^{\vartheta} + \frac{Q^2}{c^2r^2} \bigg)Q^{\varepsilon} \\
& \ll_{\varepsilon} (XQ)^{\varepsilon} \sum_{C<c\ll Q} \bigg(\frac{Q^{1+\vartheta}}{c} + \frac{Q^2}{c^2} \bigg) \\
& \ll_{\varepsilon} (XQ)^{\varepsilon} \bigg(Q^{1+\vartheta+\varepsilon} + \frac{Q^2}{C} \bigg).
\end{split}
\end{equation}
As mentioned in Section~\ref{sec: outline}, we will eventually choose $C$ as a specific positive power of $Q$ to control this error term.

\subsection{Preparing \texorpdfstring{$\mathcal{L}^0(h,k)$}{L0(h,k)} for eventual cancellation}
The goal of this subsection is to put $\mathcal{L}^0(h,k)$ into a form that, as we will eventually see in Section \ref{sec: U1}, cancels with a term arising from our analysis of $\mathcal{U}(h,k)$. To this end, let us first focus on the $c,d$-sum in the definition \eqref{eqn: L0} of $\mathcal{L}^0(h,k)$. We complete the $c$-sum by writing
\[
 \sum_{\substack{c>C, d \geq 1 \\ cd=q  }} \mu(c) = \sum_{c|q}\mu(c) -  \sum_{\substack{c\le C, d \geq 1 \\ cd=q  }} \mu(c) = \left\lfloor \frac{1}{q} \right\rfloor -  \sum_{\substack{c\le C, d \geq 1 \\ cd=q  }} \mu(c).
\]
The latter $c,d$-sum equals $1$ if $q=1$, and so it follows that
\begin{align*}
      \sum_{\substack{c>C, d \geq 1 \\ cd=q  }} \mu(c) 
= \begin{cases}
        \displaystyle -\sum_{\substack{c\leq C, d \geq 1 \\ cd=q  }} \mu(c) & \text{if } q>1\\ \\
        \hphantom{---} 0 &\text{if } q=1.
      \end{cases}
\end{align*}
From this and the definition \eqref{eqn: L0} of $\mathcal{L}^0(h,k)$, we arrive at
\begin{equation*}
    \mathcal{L}^0(h,k) = -\sum_{\substack{1< q<\infty \\ (q,hk)=1}} W\left( \frac{q}{Q}\right)  \sum_{\substack{1\leq m,n<\infty \\ (mn,q)=1}} \frac{\tau_A(m) \tau_B(n) }{\sqrt{mn}} V\left( \frac{m}{X} \right)  V\left( \frac{n}{X} \right)  \sum_{\substack{c\leq C, d \geq 1 \\ cd=q  }} \mu(c).
\end{equation*}
Without loss of generality, we may ignore the condition $q>1$ and simply sum over all $1\leq q<\infty$ because the $q=1$ term is zero for large enough $Q$, as $W$ is supported away from $0$. We substitute $q=cd$ and interchange the order of summation to deduce that
\begin{equation}\label{eqn: L0cancelformalmost}
\mathcal{L}^0(h,k) = - \sum_{\substack{1\leq c\le C\\ (c,hk)=1}} \mu(c)\sum_{\substack{1\leq m,n<\infty \\ (mn,c)=1}} \frac{\tau_A(m) \tau_B(n) }{\sqrt{mn}} V\left( \frac{m}{X} \right)  V\left( \frac{n}{X} \right)\sum_{\substack{1\le d < \infty \\ (d,mhnk)=1}}W\left( \frac{cd}{Q}\right).
\end{equation}
To evaluate the latter $d$-sum, we use Stieltjes integration and the fact that
$$
\sum_{\substack{d\le x \\ (d,m)=1}}1 = x\frac{\phi(m)}{m} + E(x,m)
$$
for some function $E(x,m)$ such that $E(x,m)=O(m^\varepsilon)$ uniformly for all $x>0$ and positive integers $m$. This results to
\begin{align*}
\sum_{\substack{1\le d < \infty\\ (d,mnhk)=1}}W\left( \frac{cd}{Q}\right) = W\left(\frac{c}{Q}\right) +\frac{\phi(mnhk)}{mnhk}\int_{1}^{\infty}W\left( \frac{cx}{Q}\right)\,dx + \int_{1}^{\infty}W\left(\frac{cx}{Q}\right)\,dE.
\end{align*}
Note that $W(c/Q)\ll 1$. Moreover, we may integrate by parts to see that the last integral is $O((mnhk)^\varepsilon)$ by the bound on $E(x,m)$ and the fact that $W$ is compactly supported. By a change of variables, we have
\[
\frac{c}{Q}\int_{1}^{\infty}W\left( \frac{cx}{Q}\right)\,dx = \int_{0}^{\infty}W(x)\,dx - \int_{0}^{c/Q}W(x)\,dx = \int_{0}^{\infty}W(x)\,dx + O\left(\frac{c}{Q} \right).
\]
Combining these estimates with \eqref{eqn: L0cancelformalmost}, we find that
\begin{equation}\label{eqn: L0cancelform}
\begin{split}
\mathcal{L}^0(h,k) &=-Q\sum_{\substack{1\le c\le C\\ (c,hk)=1}} \frac{\mu(c)}{c}\sum_{\substack{1\leq m,n<\infty \\ (mn,c)=1}} \frac{\tau_A(m) \tau_B(n) }{\sqrt{mn}} V\left( \frac{m}{X} \right)  V\left( \frac{n}{X} \right)\frac{\phi(mnhk)}{mnhk}
\int_{0}^{\infty}W(x)\,dx\\
&\hspace{1.5in}+ O\big( (Xhk)^\varepsilon XC \big).
\end{split}
\end{equation}
In Section \ref{sec: U1}, we will show that a part of $\mathcal{U}(h,k)$ cancels with the main term above.

\section{Preparing the term \texorpdfstring{$\mathcal{U}(h,k)$}{U(h,k)} for analysis}\label{sec: U(h,k)split}
There are two goals for this section. The first is to switch to the complementary modulus by making a change of variables in the definition \eqref{eqn: Usum} of $\mathcal{U}(h,k)$ and then express the divisibility condition in terms of character sums. The second goal is to dissect the contribution of the principal characters in order to isolate the part of it containing the predicted one-swap terms.

\subsection{\texorpdfstring{$\mathcal{U}(h,k)$}{U(h,k)}: Switching to the complementary modulus}
Recall the definition \eqref{eqn: Usum} of $\mathcal{U}(h,k)$. We substitute $q=cd$ and rearrange the sum to deduce that
\begin{equation}\label{eqn: Ubeforedsum}
\mathcal{U}(h,k)= \frac{1}{2}\sum_{\substack{1\leq c \leq C  \\ (c ,hk)=1 }} \mu(c)   \sum_{\substack{1\leq m,n<\infty \\ (mn,c )=1\\ mh\neq nk}} \frac{\tau_A(m) \tau_B(n) }{\sqrt{mn}} V\left( \frac{m}{X} \right)  V\left( \frac{n}{X} \right)\sum_{\substack{ 1\leq d<\infty \\ ( d,mhnk)=1\\ d|mh\pm nk}}\phi(d) W\left( \frac{cd}{Q}\right).
\end{equation}
Let $g=(mh,nk)$. Then the condition that $d|mh\pm nk$ and $( d,mhnk)=1$ is equivalent to the condition that $d|\frac{mh}{g}\pm \frac{nk}{g}$ and $(d,g)=1$. From this and the fact that $\phi(d)= \sum_{ef=d} \mu(e)f$, we see that the $d$-sum in \eqref{eqn: Ubeforedsum} equals
\begin{equation*}
\sum_{\substack{ 1\leq d<\infty \\ ( d,g)=1\\ d|\frac{mh}{g}\pm \frac{nk}{g}}}\phi(d) W\left( \frac{cd}{Q}\right) = \sum_{\substack{ 1\leq e <\infty \\ ( e ,g)=1 }} \mu(e)\sum_{\substack{ 1\leq  f<\infty \\ (  f,g)=1\\ ef|\frac{mh}{g}\pm \frac{nk}{g}}} f W\left( \frac{cef}{Q}\right).
\end{equation*}
Use M\"{o}bius inversion to detect the condition $(f,g)=1$ and write the above as 
\begin{equation*}
\sum_{\substack{ 1\leq e <\infty \\ ( e ,g)=1 }} \mu(e)\sum_{\substack{ 1\leq  f<\infty \\   ef|\frac{mh}{g}\pm \frac{nk}{g}}} \sum_{\substack{a|f \\ a|g}} \mu(a)f W\left( \frac{cef}{Q}\right) =\sum_{\substack{ 1\leq e <\infty \\ ( e ,g)=1 }} \mu(e) \sum_{a|g} \mu(a) \sum_{\substack{ 1\leq  f<\infty \\  a|f \\  ef|\frac{mh}{g}\pm \frac{nk}{g}}} f W\left( \frac{cef}{Q}\right).
\end{equation*}
Make a change of variables $f=ab$ in the $f$-sum to see that this equals
\begin{equation}\label{eqn: Udsum}
 \sum_{\substack{ 1\leq e <\infty \\ ( e ,g)=1 }} \mu(e) \sum_{a|g} a\mu(a) \sum_{\substack{ 1\leq  b<\infty   \\  eab|\frac{mh}{g}\pm \frac{nk}{g}}} b W\left( \frac{ceab}{Q}\right).
\end{equation}
Now define the ``complementary modulus'' $\ell$ by
\begin{equation*}
|mh\pm nk| =geab\ell,
\end{equation*}
and use it to make a change of variables in the $b$-sum to write \eqref{eqn: Udsum} as
\begin{equation}\label{eqn: Udsum2}
\begin{split}
 \sum_{\substack{ 1\leq e <\infty \\ ( e ,g)=1 }}& \mu(e) \sum_{a|g} a\mu(a) \sum_{\substack{ 1\leq  \ell <\infty   \\  ea\ell |\frac{mh}{g}\pm \frac{nk}{g}}} \frac{|mh\pm nk|}{gea\ell} W\left( \frac{c|mh\pm nk|}{g\ell Q}\right)\\
&\hspace{.25in}  = \sum_{\substack{ 1\leq e <\infty \\ ( e ,g)=1 }} \frac{\mu(e)}{e} \sum_{a|g}  \mu(a) \sum_{\substack{ 1\leq  \ell <\infty   \\  ea\ell |\frac{mh}{g}\pm \frac{nk}{g}}} \frac{|mh\pm nk|}{g \ell} W\left( \frac{c|mh\pm nk|}{g\ell Q}\right).
\end{split}
\end{equation}
Since $g$ is defined by $g=(mh,nk)$, we must have that $ea\ell$ is coprime to each of $mh/g$ and $nk/g$, because if not then the condition $ea\ell |(mh\pm nk)/g$ would imply that $mh/g$ and $nk/g$ are not coprime, contradicting the definition of $g$. Thus the orthogonality of character sums implies
\begin{equation*}
\frac{1}{\phi(ea \ell)} \sum_{\psi \bmod ea\ell} \psi \left( \frac{mh}{g} \right) \overline{\psi}\left( \mp \frac{nk}{g} \right) = \left\{ \begin{array}{cl} 1 & \text{if } ea\ell |\frac{mh}{g}\pm \frac{nk}{g} \\ \\ 0 & \text{else}. \end{array} \right.
\end{equation*}
Hence, we may replace the condition $ea\ell |(mh\pm nk)/g$ in \eqref{eqn: Udsum2} with the above multiplier to conclude that the $d$-sum appearing in \eqref{eqn: Ubeforedsum} is equal to
\begin{equation*}
\begin{split}
 \sum_{\substack{ 1\leq e <\infty \\ ( e ,g)=1 }}& \frac{\mu(e)}{e} \sum_{a|g}  \mu(a) \sum_{\substack{ 1\leq  \ell <\infty \\ (ea\ell, \frac{mh}{g}\cdot\frac{nk}{g})=1    }} \frac{1}{\phi(ea \ell)} \sum_{\psi \bmod ea\ell} \psi \left( \frac{mh}{g} \right) \overline{\psi}\left( \mp \frac{nk}{g} \right)\\
 &\hspace{.5in}\times\frac{|mh\pm nk|}{g \ell} W\left( \frac{c|mh\pm nk|}{g\ell Q}\right).
\end{split}
\end{equation*}
It follows that
\begin{equation}\label{eqn: Uafterswitch}
\begin{split}
\mathcal{U}(h,k)&= \frac{1}{2} \sum_{\substack{1\leq c \leq C  \\ (c ,hk)=1 }} \mu(c)   \sum_{\substack{1\leq m,n<\infty \\ (mn,c )=1\\ mh\neq nk}} \frac{\tau_A(m) \tau_B(n) }{\sqrt{mn}} V\left( \frac{m}{X} \right)  V\left( \frac{n}{X} \right)\sum_{\substack{ 1\leq e <\infty \\ ( e ,g)=1 }} \frac{\mu(e)}{e}   \sum_{a|g}  \mu(a)\\
&\hspace{.25in} \times\sum_{\substack{ 1\leq  \ell <\infty \\ (ea\ell, \frac{mh}{g}\cdot\frac{nk}{g})=1    }} \frac{1}{\phi(ea \ell)} \sum_{\psi \bmod ea\ell} \psi \left( \frac{mh}{g} \right) \overline{\psi}\left( \mp \frac{nk}{g} \right)\frac{|mh\pm nk|}{g \ell} W\left( \frac{c|mh\pm nk|}{g\ell Q}\right).
\end{split}
\end{equation}
Write this as
\begin{equation}\label{eqn: Usplit00}
\mathcal{U}(h,k)= \mathcal{U}^0(h,k) + \mathcal{U}^r(h,k),
\end{equation}
where $\mathcal{U}^0(h,k)$ is the contribution of the principal character in the $\psi$-sum, and $\mathcal{U}^r(h,k)$ is the contribution of the non-principal characters. In other words, $\mathcal{U}^0(h,k)$ and $\mathcal{U}^r(h,k)$ are defined by
\begin{equation}\label{eqn: U0}
\begin{split}
 \mathcal{U}^0(h,k)&:= \frac{1}{2} \sum_{\substack{1\leq c \leq C  \\ (c ,hk)=1 }} \mu(c)   \sum_{\substack{1\leq m,n<\infty \\ (mn,c )=1\\ mh\neq nk}} \frac{\tau_A(m) \tau_B(n) }{\sqrt{mn}} V\left( \frac{m}{X} \right)  V\left(\frac{n}{X}\right)\sum_{\substack{ 1\leq e <\infty \\ ( e ,g)=1 }} \frac{\mu(e)}{e} \sum_{a|g}  \mu(a)  \\
&\hspace{.25in}  \times  \sum_{\substack{ 1\leq  \ell <\infty \\ (ea\ell, \frac{mh}{g}\cdot\frac{nk}{g})=1    }}   \frac{|mh\pm nk|}{g \ell\phi(ea \ell)} W\left( \frac{c|mh\pm nk|}{g\ell Q}\right)
\end{split}
\end{equation}
and
\begin{equation}\label{eqn: Ur}
\begin{split}
 \mathcal{U}^r(h,k)&:= \frac{1}{2} \sum_{\substack{1\leq c \leq C  \\ (c ,hk)=1 }} \mu(c)   \sum_{\substack{1\leq m,n<\infty \\ (mn,c )=1\\ mh\neq nk}} \frac{\tau_A(m) \tau_B(n) }{\sqrt{mn}} V\left( \frac{m}{X} \right)  V\left( \frac{n}{X} \right)\sum_{\substack{ 1\leq e <\infty \\ ( e ,g)=1 }} \frac{\mu(e)}{e} \sum_{a|g}  \mu(a)  \\
&\hspace{.25in} \times  \sum_{\substack{ 1\leq  \ell <\infty \\ (ea\ell, \frac{mh}{g}\cdot\frac{nk}{g})=1    }} \frac{1}{\phi(ea \ell)} \sum_{\substack{\psi \bmod ea\ell\\ \psi\neq \psi_0}} \psi \left( \frac{mh}{g} \right) \overline{\psi}\left( \mp \frac{nk}{g} \right)\frac{|mh\pm nk|}{g \ell} W\left( \frac{c|mh\pm nk|}{g\ell Q}\right),
\end{split}
\end{equation}
respectively, where $\psi_0$ denotes the principal character mod $ea\ell$.

\subsection{The principal contribution \texorpdfstring{$\mathcal{U}^0(h,k)$}{U0(h,k)}}
Our goal in this subsection is to separate out a part of $\mathcal{U}^0(h,k)$ that we will eventually prove contains the one-swap terms that are predicted by the recipe. We apply Mellin inversion to write
\begin{equation}\label{eqn: mellinW}
\frac{|mh\pm nk|}{g \ell} W\left( \frac{c|mh\pm nk|}{g\ell Q}\right) = \frac{Q}{2\pi i c} \int_{(\varepsilon)} \ell^{-w} \Upsilon_{\pm}(w;mh,nk) \,dw,
\end{equation}
where 
\begin{equation}\label{eqn: Upsilondef}
    \Upsilon_{\pm}(w;mh,nk)= \Upsilon_{\pm}(w;mh,nk;c,Q):=\int_{0}^{\infty}\frac{c|mh\pm nk|}{g x Q}W\left(\frac{c|mh\pm nk|}{g x Q} \right)x^{w-1}\,dx.
\end{equation}
We insert \eqref{eqn: mellinW} into \eqref{eqn: U0}, then interchange the order of summation and write the $\ell$-sum as an Euler product using the following lemma.
\begin{lemma}\label{lem: Lemma6ofCIS}\cite[Lemma 6]{CIS}
Let $s$ be a complex number with $\re(s)>0$, and let $u$ and $v$ be coprime natural numbers. Then
\[
\sum_{\substack{\ell=1 \\ (\ell, v)=1}}^{\infty}\frac{1}{\phi(u\ell) \ell^s} = \frac{1}{\phi(u)}\zeta(1+s)R(s;u,v),
\]
where
\begin{equation}\label{eqn: Rdef}
R(s;u,v) = \prod_{p|v}\left(1-\frac{1}{p^{s+1}} \right)\prod_{p\nmid uv}\left(1+\frac{1}{p^{s+1}(p-1)}\right)
\end{equation}
converges absolutely in $\re(s)>-1$.
\end{lemma}

The result is
\begin{equation}\label{eqn: U0beforesplit}
\begin{split}
\mathcal{U}^0(h,k)&= \frac{Q}{2} \sum_{\substack{1\leq c \leq C  \\ (c ,hk)=1 }} \frac{\mu(c)}{c}   \sum_{\substack{1\leq m,n<\infty \\ (mn,c )=1\\ mh\neq nk}} \frac{\tau_A(m) \tau_B(n) }{\sqrt{mn}} V\left( \frac{m}{X} \right)  V\left( \frac{n}{X} \right)  \sum_{\substack{ 1\leq e <\infty \\ ( e ,g)=1 }} \frac{\mu(e)}{e} \sum_{\substack{ a|g \\ (ea, \frac{mh}{g}\cdot\frac{nk}{g})=1}}  \frac{\mu(a)}{\phi(ea)}\\
&\hspace{.25in} \times  \frac{1}{2\pi i}\int_{(\varepsilon)} \Upsilon_\pm(w;mh,nk)  \zeta(1+w) R(w;ea, mhnk/g^2)\,dw.
\end{split}
\end{equation}
Note that $\Upsilon_\pm(w;mh,nk)$ has rapid decay as $|w|\rightarrow \infty$ by \eqref{eqn: Upsilondef} and a repeated application of integration by parts. Hence, we may move the line of integration in \eqref{eqn: U0beforesplit} to Re$(w)=-\epsilon$. Doing so leaves a residue at $w=0$ from the pole of $\zeta(1+w)$, and we arrive at
\begin{equation}\label{eqn: U0split}
\mathcal{U}^0(h,k) = \mathcal{U}^1(h,k)+\mathcal{U}^2(h,k),
\end{equation}
where $\mathcal{U}^1(h,k)$ is the residue, i.e.,
\begin{equation}\label{eqn: U1}
\begin{split}
\mathcal{U}^1(h,k):=
& \frac{Q}{2} \sum_{\substack{1\leq c \leq C  \\ (c ,hk)=1 }} \frac{\mu(c)}{c}   \sum_{\substack{1\leq m,n<\infty \\ (mn,c )=1\\ mh\neq nk}} \frac{\tau_A(m) \tau_B(n) }{\sqrt{mn}} V\left( \frac{m}{X} \right)  V\left( \frac{n}{X} \right) \sum_{\substack{ 1\leq e <\infty \\ ( e ,g)=1 }} \frac{\mu(e)}{e}  \\
& \times \sum_{\substack{ a|g \\ (ea, \frac{mh}{g}\cdot\frac{nk}{g})=1}}  \frac{\mu(a)}{\phi(ea)} \Upsilon_\pm(0;mh,nk)      R(0;ea, mhnk/g^2),
\end{split}
\end{equation}
and $\mathcal{U}^2(h,k)$ is defined by
\begin{equation}\label{eqn: U2}
\begin{split}
& \mathcal{U}^2(h,k):= \frac{Q}{2} \sum_{\substack{1\leq c \leq C  \\ (c ,hk)=1 }} \frac{\mu(c)}{c}  \sum_{\substack{1\leq m,n<\infty \\ (mn,c )=1\\ mh\neq nk}} \frac{\tau_A(m) \tau_B(n) }{\sqrt{mn}} V\left( \frac{m}{X} \right)  V\left( \frac{n}{X} \right)\sum_{\substack{ 1\leq e <\infty \\ ( e ,g)=1 }} \frac{\mu(e)}{e}   \\
& \times \sum_{\substack{ a|g \\ (ea, \frac{mh}{g}\cdot\frac{nk}{g})=1}}  \frac{\mu(a)}{\phi(ea)} \cdot \frac{1}{2\pi i}\int_{(-\epsilon)} \Upsilon_\pm(w;mh,nk)  \zeta(1+w) R(w;ea, mhnk/g^2)\,dw.
\end{split}
\end{equation}

\subsection{The term \texorpdfstring{$\mathcal{U}^1(h,k)$}{U1(h,k)} approximately cancels with \texorpdfstring{$\mathcal{L}^0(h,k)$}{L0(h,k)}}\label{sec: U1}
In this subsection, we show that the term $\mathcal{U}^1(h,k)$ defined by \eqref{eqn: U1} cancels with the main contribution of $\mathcal{L}^0(h,k)$, which we have evaluated in \eqref{eqn: L0cancelform}. We first focus on the $e,a$-sum in $\eqref{eqn: U1}$. To express it as an Euler product, we observe that Lemma~\ref{lem: sumstoEulerproducts} and the definition \eqref{eqn: Rdef} of $R$ implies for $\re(w)>-1$ that
\begin{equation}\label{eqn: easum}
\begin{split}
&\sum_{\substack{1\leq e<\infty \\ (e, g)=1}}\frac{\mu(e)}{e} \sum_{\substack{a|g \\ (ea,\frac{mh}{g}\cdot \frac{nk}{g})=1}} \frac{\mu(a)}{\phi(ea)}R(w;ae,mnhk/g^2)\\
&\hspace{1in}=\prod_{p|mnhk/g^2}\left(1-\frac{1}{p^{1+w}} \right)\prod_{\substack{p|g \\ p\nmid mnhk/g^2}}\left(1+\frac{1}{p^{w+1}(p-1)} -\frac{1}{p-1} \right)\\
&\hspace{1.5in}\times \prod_{\substack{p\nmid g \\ p\nmid mnhk/g^2}}\left(1+\frac{p^{-w}-1}{p(p-1)}\right)
\end{split}
\end{equation}
(this is the same as (7.7) of \cite{CIS}). It follows from this with $w=0$ that
\begin{equation}\label{eqn: easumw0}
\sum_{\substack{1\leq e<\infty \\ (e, g)=1}}\frac{\mu(e)}{e} \sum_{\substack{a|g \\ (ea,\frac{mh}{g}\cdot \frac{nk}{g})=1}} \frac{\mu(a)}{\phi(ea)}R(0;ae,mnhk/g^2) = \frac{\phi(mnhk)}{mnhk}.
\end{equation}
Now the definition \eqref{eqn: Upsilondef} of $\Upsilon_{\pm}$ and a change of variables gives
\begin{equation*}
\Upsilon_{+}(0;mh,nk) + \Upsilon_{-}(0;mh,nk) = 2\int_0^{\infty} W(u)\,du.
\end{equation*}
From this, \eqref{eqn: U1}, and \eqref{eqn: easumw0}, we deduce that
\begin{equation}\label{eqn: U1simplify}
\mathcal{U}^1(h,k)=
 Q \sum_{\substack{1\leq c \leq C  \\ (c ,hk)=1 }} \frac{\mu(c)}{c}   \sum_{\substack{1\leq m,n<\infty \\ (mn,c )=1\\ mh\neq nk}} \frac{\tau_A(m) \tau_B(n) }{\sqrt{mn}} V\left( \frac{m}{X} \right)  V\left( \frac{n}{X} \right) \frac{\phi(mnhk)}{mnhk} \int_0^{\infty} W(u)\,du.
\end{equation}
In order to show that $\mathcal{U}^1(h,k)$ cancels with the main term of $\mathcal{L}^0(h,k)$ given in \eqref{eqn: L0cancelform}, we must complete the sum above to include the terms $mh=nk$. In order to do this successfully, we must show that the total contribution of the terms with $mh=nk$ is small. By \eqref{eqn: divisorbound} and our assumption that $V$ and $W$ have compact support, the sum of the terms with $mh=nk$ is at most
\begin{equation}\label{eqn: U1boundstep1}
\ll Q\sum_{1\leq c \leq C} \frac{1}{c}   \sum_{\substack{1\leq m,n \ll X \\ mh=nk}} \frac{(mn)^\varepsilon  }{\sqrt{mn}}.
\end{equation}
Observe that $mh=nk$ if and only if there is an integer $\ell$ such that $m=\ell K$ and $n=\ell H$, where, as before, $H$ and $K$ are defined by $H:=h/(h,k)$ and $K:=k/(h,k)$. Thus \eqref{eqn: U1boundstep1} is
\begin{align*}
\ll (HK)^{-1/2+\varepsilon}Q(\log C) \sum_{1\le \ell \ll X}\frac{1}{\ell^{1-\varepsilon}}
 \ll X^{\varepsilon} (HK)^{-1/2+\varepsilon}Q\log C.
\end{align*}
Hence, including the $mh=nk$ terms in \eqref{eqn: U1simplify} gives
\begin{equation}\label{eqn: U1ready}
\begin{split}
\mathcal{U}^1(h,k)&=
 Q \sum_{\substack{1\leq c \leq C  \\ (c ,hk)=1 }} \frac{\mu(c)}{c}   \sum_{\substack{1\leq m,n<\infty \\ (mn,c )=1}} \frac{\tau_A(m) \tau_B(n) }{\sqrt{mn}} V\left( \frac{m}{X} \right)  V\left( \frac{n}{X} \right) \frac{\phi(mnhk)}{mnhk} \int_0^{\infty} W(u)\,du\\
 &\hspace{.5in}+O\bigg(Q\frac{(XCHK)^{\varepsilon}}{\sqrt{HK}} \bigg).
\end{split}
\end{equation}
The main term here cancels with the main term from our analysis of $\mathcal{L}^0(h,k)$, given in \eqref{eqn: L0cancelform}. More precisely, it follows from \eqref{eqn: L0cancelform} and \eqref{eqn: U1ready} that
\begin{equation}\label{eqn: U1ready2}
\mathcal{U}^1(h,k) = -\mathcal{L}^0(h,k) + O \big( (Xhk)^\varepsilon XC \big) + O\bigg(Q\frac{(XCHK)^{\varepsilon}}{\sqrt{HK}} \bigg).
\end{equation}

Summarizing this section, we deduce from \eqref{eqn: Usplit00}, \eqref{eqn: U0split}, and \eqref{eqn: U1ready2} that
\begin{equation}\label{eqn: Uready}
\mathcal{U}(h,k) = -\mathcal{L}^0(h,k)+\mathcal{U}^2(h,k)+\mathcal{U}^r(h,k)+ O \big( (Xhk)^\varepsilon XC \big) + O\bigg(Q\frac{(XCHK)^{\varepsilon}}{\sqrt{HK}} \bigg).
\end{equation}
Looking forward, we show in Section~\ref{sec: U2} that $\mathcal{U}^2(h,k)$ is, up to an admissible error term, equal to the one-swap terms $\mathcal{I}_1(h,k)$ predicted by the recipe. In Section~\ref{sec: Ur}, we bound the average of $(hk)^{-1/2}\mathcal{U}^r(h,k)$ over $h,k$ and show that $\mathcal{U}^r(h,k)$ is an acceptable error term.

\section{The term \texorpdfstring{$\mathcal{U}^2(h,k)$}{U2(h,k)}: extracting the one-swap terms}\label{sec: U2}
Recall that $\mathcal{U}^2(h,k)$, defined by \eqref{eqn: U2}, does not include the diagonal terms $mh=nk$. As in the analysis of $\mathcal{U}^1(h,k)$, we will find it advantageous to add these terms back in, and so we must show that the total contribution of these terms is acceptably small. The analysis that follows is similar to that of $\mathcal{U}^1(h,k)$ in Subsection~\ref{sec: U1}. However, the treatment of $\Upsilon_{\pm}(w;mk,nk)$ is more delicate because the variables $m$ and $n$ are entangled in the factor $|mh\pm nk|$. To ameliorate this challenge, we first introduce a bit of averaging as in Section~7 of \cite{CIS}. This averaging will lead to expressions with absolutely convergent integrals after separating the variables $m$ and $n$ in $\Upsilon_{\pm}(w;mk,nk)$ (Proposition~\ref{prop: CISProp2} below). The absolute convergence, in turn, will allow us to interchange the order of summation in our analysis of $\mathcal{U}^2(h,k)$ and extract the predicted one-swap terms in the subsections that follow.

To begin, we state and prove the averaging result that we will apply as just described.
\begin{lemma}\label{lem: smoothing}
Let $f:[0,\infty)\rightarrow \mathbb{C}$ be a continuously differentiable function of compact support such that $f$ is zero in a neighborhood of zero. Let $x,y,v\in \mathbb{R}$, with $v>0$. Then the function
$$
t\longmapsto f(v|x-t y|)
$$
is continuously differentiable on $\mathbb{R}$. Moreover, if $\,0<\delta<1$, then
\begin{equation*}
 f(v|x-y|) = \frac{1}{2\delta}\int_{-\delta}^{\delta} f(v| x-e^{\xi} y|)\,d\xi +O(|vy|\delta),
\end{equation*}
where the implied constant depends only on $f$.
\end{lemma}
\begin{proof}
That the function $t\longmapsto f(v|x-t y|)$ is continuously differentiable on $\mathbb{R}$ follows by the chain rule and the assumption that $f$ is zero in a neighborhood of zero. Moreover, $f'(x)=O(1)$ uniformly on $\mathbb{R}$ because $f$ has compact support, and so
$$
\frac{d}{dt} f(v|x-t y|)=\pm vy f'(v|x-t y|)  \ll |vy|.
$$
It follows from this and the fundamental theorem of calculus that, for $0<\delta<1$,
\begin{align*}
\int_{-\delta}^{\delta} f(v| x-e^{\xi} y|)\,d\xi -\int_{-\delta}^{\delta} f(v|x-y|)\,d\xi
& = \int_{-\delta}^{\delta} \int_1^{e^{\xi}} \frac{d}{dt} f(v|x-t y|)\,dt\,d\xi \\
& \ll |vy| \int_{-\delta}^{\delta} |\xi|\,d\xi 
 \ll |vy|\delta^2.
\end{align*}
Rearranging the terms gives the lemma.
\end{proof}

Before we apply Lemma~\ref{lem: smoothing} to the sum $\mathcal{U}^2(h,k)$ defined by \eqref{eqn: U2}, we first truncate the $w$-integral in \eqref{eqn: U2}. Doing so will enable us to easily deal with the error term arising from the application of Lemma~\ref{lem: smoothing}. To this end, observe that if $\xi\in \mathbb{R}$, then a change of variables implies
\begin{equation}\label{eqn: Upsilonchangeofvar}
\int_{0}^{\infty }  \frac{c |mh\pm e^{\xi}nk|}{gx Q} W\left( \frac{c |mh\pm e^{\xi}nk|}{gx Q}\right) x^{w-1} \,dx =  \left( \frac{c |mh\pm e^{\xi}nk|}{g Q}\right)^{w} \widetilde{W}(1-w).
\end{equation}
If $w,c,m,h,n,k$ are as in \eqref{eqn: U2}, then $|mh\pm nk|\geq 1$ since $mh\neq nk$, and so the definition \eqref{eqn: Upsilondef} of $\Upsilon_{\pm}(w;mh,nk)$, \eqref{eqn: Upsilonchangeofvar} with $\xi=0$, and \eqref{eqn: mellinrapiddecay} imply that
\begin{equation}\label{eqn: Upsilonrapiddecay}
\Upsilon_{\pm}(w;mh,nk) \ll_{\nu} \frac{(gQ)^{\varepsilon}}{|w|^{\nu}}
\end{equation}
for any positive integer $\nu$. Now the definition \eqref{eqn: Rdef} of $R$ implies that if $\re(w)=-\varepsilon$, then
\begin{equation}\label{eqn: Rbound}
R(w;ea, mhnk/g^2) \ll (mhnk)^{\varepsilon}.
\end{equation}
From this and \eqref{eqn: Upsilonrapiddecay}, we see that the part of the integral in \eqref{eqn: U2} that has $|\text{Im}(w)|\geq (XQ)^{\varepsilon}$ is negligible. Thus, using also \eqref{eqn: divisorbound}, the definition $g=(mh,nk)$, and the assumption that $V$ has compact support, we deduce that
\begin{equation}\label{eqn: U2truncated}
\begin{split}
& \mathcal{U}^2(h,k)= \frac{Q}{2} \sum_{\substack{1\leq c \leq C  \\ (c ,hk)=1 }} \frac{\mu(c)}{c}  \sum_{\substack{1\leq m,n<\infty \\ (mn,c )=1\\ mh\neq nk}} \frac{\tau_A(m) \tau_B(n) }{\sqrt{mn}} V\left( \frac{m}{X} \right)  V\left( \frac{n}{X} \right)\sum_{\substack{ 1\leq e <\infty \\ ( e ,g)=1 }} \frac{\mu(e)}{e}   \\
& \times \sum_{\substack{ a|g \\ (ea, \frac{mh}{g}\cdot\frac{nk}{g})=1}}  \frac{\mu(a)}{\phi(ea)} \cdot \frac{1}{2\pi i}\int_{-\epsilon -i(XQ)^{\varepsilon}}^{-\epsilon +i(XQ)^{\varepsilon}} \Upsilon_\pm(w;mh,nk)  \zeta(1+w) R(w;ea, mhnk/g^2)\,dw \\
& \hspace{1.5in} + O\big( (Chk)^{\varepsilon} Q^{-99}\big).
\end{split}
\end{equation}

Having truncated the integral in \eqref{eqn: U2}, we now apply Lemma~\ref{lem: smoothing}. Recall that the support of $W$ is a compact subset of $(0,\infty)$. Use Lemma~\ref{lem: smoothing} with $f(u)=uW(u)$ and $\delta$ defined by \eqref{eqn: deltadef} to deduce that the integrand in \eqref{eqn: Upsilondef} satisfies
\begin{align*}
\frac{c |mh\pm nk|}{gx Q} W\left( \frac{c |mh\pm nk|}{gx Q}\right) = \frac{1}{2\delta}\int_{-\delta}^{\delta} \frac{c |mh\pm e^{\xi}nk|}{gx Q} W\left( \frac{c |mh\pm e^{\xi}nk|}{gx Q}\right)\,d\xi +O\left(\frac{c nk \delta}{gx Q}\right).
\end{align*}
We insert this into the definition \eqref{eqn: Upsilondef} of $\Upsilon_{\pm}(w;mh,nk)$. The contribution of the error term is
\begin{align*}
\ll \frac{c nk \delta}{g Q} \left(  \frac{c |mh\pm nk|}{g  Q}\right)^{-\varepsilon-1} \ll X k \delta (g Q)^{\varepsilon}
\end{align*}
for $w,c,m,h,n,k$ satisfying the conditions in \eqref{eqn: U2truncated}, because $|mh\pm nk|\geq 1$, $c\geq 1$, $n\ll X$, and, by the support of $W$, the integrand in \eqref{eqn: Upsilondef} is zero unless $x\asymp c|mh\pm nk| /(gQ)$. We arrive at
\begin{equation*}
\Upsilon_{\pm}(w;mh,nk) = \frac{1}{2\delta}\int_0^{\infty}\int_{-\delta}^{\delta} \frac{c |mh\pm e^{\xi}nk|}{gx Q} W\left( \frac{c |mh\pm e^{\xi}nk|}{gx Q}\right)x^{w-1}\,d\xi\,dx +O\big( X k \delta (g Q)^{\varepsilon}\big).
\end{equation*}
This and \eqref{eqn: Upsilonchangeofvar} imply
\begin{equation*}
\Upsilon_{\pm}(w;mh,nk) = \frac{1}{2\delta} \int_{-\delta}^{\delta} \left( \frac{c |mh\pm e^{\xi}nk|}{g Q}\right)^{w} \widetilde{W}(1-w)\,d\xi +O\big( X k \delta (g Q)^{\varepsilon}\big). 
\end{equation*}
We insert this into \eqref{eqn: U2truncated} to deduce that
\begin{equation}\label{eqn: U2before2}
\begin{split}
\mathcal{U}^2(h,k) &= \frac{Q}{2} \sum_{\substack{1\leq c \leq C  \\ (c ,hk)=1 }} \frac{\mu(c)}{c}  \sum_{\substack{1\leq m,n<\infty \\ (mn,c )=1\\ mh\neq nk}} \frac{\tau_A(m) \tau_B(n) }{\sqrt{mn}} V\left( \frac{m}{X} \right)  V\left( \frac{n}{X} \right)  \sum_{\substack{ 1\leq e <\infty \\ ( e ,g)=1 }} \frac{\mu(e)}{e} \sum_{\substack{ a|g \\ (ea, \frac{mh}{g}\cdot\frac{nk}{g})=1}}  \frac{\mu(a)}{\phi(ea)} \\
&\hspace{.25in} \times  \frac{1}{2\pi i}\int_{-\epsilon -i(XQ)^{\varepsilon}}^{-\epsilon +i(XQ)^{\varepsilon}} \zeta(1+w) R(w;ea, mhnk/g^2) \left( \frac{c  }{g Q}\right)^{w} \widetilde{W}(1-w)  \\
&\hspace{.5in}\times \frac{1}{2\delta}\int_{-\delta}^{\delta} |mh\pm e^{\xi}nk|^w \,d\xi\,dw + O\big((XChk)^{\varepsilon} k X^2 Q^{-97}\big),
\end{split}
\end{equation}
where, to bound the error term, we have used \eqref{eqn: divisorbound}, \eqref{eqn: Rbound}, the definition $g=(mh,nk)$, the definition \eqref{eqn: deltadef} of $\delta$, and the assumption that $V$ has compact support.

The following proposition, which is Proposition 2 in \cite{CIS}, enables us to separate the variables $m$ and $n$ in the expression $|mh\pm e^{\xi} nk|$ and thus write the $m,n,e,a$-sum in \eqref{eqn: U2before2} in terms of an Euler product.
\begin{prop}[Proposition 2 of \cite{CIS}]\label{prop: CISProp2}
Let $\omega$ be a complex number with $\re(\omega)>0$. Then for any $0<c<\re(\omega)$, and $r>0$ with $r\ne 1$, we have
\[
|1+r|^{-\omega}+|1-r|^{-\omega}=\frac{1}{2\pi i}\int_{(c)}\mathcal{H}(z,\omega)r^{-z}\,dz.
\]
Therefore, for any $\delta>0$,
\begin{equation}\label{eqn: Prop2Integral}
\frac{1}{2\delta}\int_{-\delta}^{\delta}|1+ e^{\xi}r|^{-\omega} + |1- e^{\xi}r|^{-\omega}\,d\xi = \frac{1}{2\pi i}\int_{(c)}\mathcal{H}(z,\omega)r^{-z}\frac{e^{\delta z}-e^{-\delta z}}{2\delta z}\, dz,
\end{equation}
where $\mathcal{H}(z,\omega)$ is defined by \eqref{eqn: Hdef}. The $z$-integral in \eqref{eqn: Prop2Integral} converges absolutely for $\re(\omega)<1$. 
\end{prop}

We apply Proposition~\ref{prop: CISProp2} with $\omega=-w$, $\re(w)=-\epsilon$, $c=\epsilon/2$, and $r=nk/(mh)$, which is $\neq 1$ in \eqref{eqn: U2before2}, to deduce that
$$
\frac{1}{2\delta}\int_{-\delta}^{\delta} \left|1+  e^{\xi} \frac{nk}{mh}\right|^{w} + \left|1-  e^{\xi} \frac{nk}{mh}\right|^{w}\,d\xi  =\frac{1}{2\pi i} \int_{(\epsilon/2)} \mathcal{H}(z,-w) \left(\frac{nk}{mh}\right)^{-z} \frac{e^{\delta z} - e^{-\delta z}}{2\delta z}\,dz.
$$
Multiply both sides by $(mh)^w$ to find that
$$
\frac{1}{2\delta}\int_{-\delta}^{\delta} |mh +  e^{\xi} nk|^{w} + |mh -  e^{\xi} nk|^{w} \,d\xi  =\frac{1}{2\pi i} \int_{(\epsilon/2)} \mathcal{H}(z,-w) (mh)^{w+z}(nk)^{-z} \frac{e^{\delta z} - e^{-\delta z}}{2\delta z}\,dz.
$$
We insert this into \eqref{eqn: U2before2} and arrive at
\begin{equation}\label{eqn: U2before}
\begin{split}
 \mathcal{U}^2(h,k)&= \frac{Q}{2} \sum_{\substack{1\leq c \leq C  \\ (c ,hk)=1 }} \frac{\mu(c)}{c}  \sum_{\substack{1\leq m,n<\infty \\ (mn,c )=1\\ mh\neq nk}} \frac{\tau_A(m) \tau_B(n) }{\sqrt{mn}} V\left( \frac{m}{X} \right)  V\left( \frac{n}{X} \right)  \sum_{\substack{ 1\leq e <\infty \\ ( e ,g)=1 }} \frac{\mu(e)}{e} \sum_{\substack{ a|g \\ (ea, \frac{mh}{g}\cdot\frac{nk}{g})=1}}  \frac{\mu(a)}{\phi(ea)}\\
&\hspace{.15in}\times \frac{1}{2\pi i}\int_{-\epsilon -i(XQ)^{\varepsilon}}^{ -\epsilon +i(XQ)^{\varepsilon} } \zeta(1+w)R(w;ea, mhnk/g^2) \left( \frac{c}{gQ}\right)^w  \widetilde{W}(1-w)\\
&\hspace{.25in}\times\frac{1}{2\pi i} \int_{(\epsilon/2)} \mathcal{H}(z,-w) (mh)^{w+z}(nk)^{-z} \frac{e^{\delta z} - e^{-\delta z}}{2\delta z}\,dz  \,dw + O\big((XChk)^{\varepsilon} k X^2 Q^{-97}\big).
\end{split}
\end{equation}
By \eqref{eqn: divisorbound}, \eqref{eqn: Hbound}, \eqref{eqn: mellinrapiddecay},  \eqref{eqn: Rbound}, and the assumption that $V$ has compact support, we may extend the $w$-integral in \eqref{eqn: U2before} to infinity by introducing a negligible error. We then insert \eqref{eqn: easum} to deduce that
\begin{align}
 \mathcal{U}^2(h,k) & = \frac{Q}{2} \sum_{\substack{1\leq c \leq C  \\ (c ,hk)=1 }}\frac{\mu(c)}{c} \sum_{\substack{1\le m,n<\infty\\ (mn,c)=1\\mh\neq nk}}\frac{\tau_A(m)\tau_B(n)}{\sqrt{mn}}V\left(\frac{m}{X}\right)V\left(\frac{n}{X}\right)\frac{1}{2\pi i}\int_{(-\epsilon)}\zeta(1+w) \notag\\
 &\hspace{.25in}\times\widetilde{W}(1-w)\left(\frac{c}{gQ}\right)^w \frac{1}{2\pi i}\int_{(\epsilon/2)}\mathcal{H}(z,-w)(mh)^{w+z}(nk)^{-z}\frac{e^{\delta z}-e^{-\delta z}}{2\delta z} \notag\\
 &\hspace{.5in}\times \prod_{p|mnhk/g^2}\left(1-\frac{1}{p^{1+w}}\right) \prod_{\substack{p|g\\p\nmid mnhk/g^2}}\left(1+\frac{1}{p^{w+1}(p-1)}-\frac{1}{p-1} \right) \notag\\
 &\hspace{.75in}\times \prod_{\substack{p\nmid g\\p\nmid mnhk/g^2}}\left(1+\frac{p^{-w}-1}{p(p-1)}\right)\, dz\, dw + O\big((XChk)^{\varepsilon} k X^2 Q^{-97}\big). \label{eqn: U2before3}
\end{align}

We next add the $mh=nk$ terms to complete the $m,n$-sum. Let us first show that their total, which is the above main term expression with the condition $mh\ne nk$ replaced with $mh=nk$, is acceptably small. As we have seen in the discussion below \eqref{eqn: U1boundstep1}, $mh=nk$ if and only if $m=\ell K$ and $n=\ell H$ for some integer $\ell$. For such an $\ell$, the condition $(mn,c)=1$ is equivalent to $(\ell,c)=1$ because $(c,hk)=1$. Moreover, if $mh=nk$, then the definition $g=(mh,nk)$ implies $g=mh=nk$. Thus the total contribution of the $mh=nk$ terms is
\begin{align}
\frac{Q}{2} \sum_{\substack{1\leq c \leq C  \\ (c ,hk)=1 }} \frac{\mu(c)}{c} & \sum_{\substack{1\le \ell <\infty\\ (\ell,c)=1}}\frac{\tau_A(\ell K)\tau_B(\ell H)}{\ell\sqrt{HK}}V\left(\frac{\ell K}{X}\right)V\left(\frac{\ell H}{X}\right) \frac{1}{2\pi i}\int_{(-\epsilon)}\zeta(1+w) \notag\\
&\times \widetilde{W}(1-w)\left(\frac{c}{Q}\right)^w\frac{1}{2\pi i}\int_{(\epsilon/2)}\mathcal{H}(z,-w)\frac{e^{\delta z}-e^{-\delta z}}{2\delta z} \notag\\
&\hspace{.5in}\times\prod_{p|\ell hk}\left(1+\frac{1}{p^{w+1}(p-1)}-\frac{1}{p-1} \right)\prod_{p\nmid \ell hk}\left(1+\frac{p^{-w}-1}{p(p-1)}\right)\, dz\, dw. \label{eqn: U2diagonalHK}
\end{align}
We may restrict the $\ell$ sum to $1\le \ell \ll X$ because $V$ is compactly supported. The product over $p|\ell hk$ is bounded by $(hk\ell)^\varepsilon$, and the infinite product over $p\nmid \ell hk$ is absolutely convergent since $\re(w)=-\epsilon$. Thus \eqref{eqn: U2diagonalHK} is
\begin{equation}\label{eqn: U2diagonalHK2}
\begin{split}
    &\ll Q^{1+\varepsilon}\sum_{1\leq c \leq C}\frac{1}{c^{1+\varepsilon}}\sum_{1\le \ell \ll X}\frac{(hk\ell)^{\varepsilon}}{\ell\sqrt{HK}}\\
    & \hspace{.5in} \times \int_{(-\epsilon)}\int_{(\epsilon/2)}\left|\zeta(1+w)\widetilde{W}(1-w)\mathcal{H}(z,-w)\frac{e^{\delta z}-e^{-\delta z}}{2\delta z}\right|\,|dz|\,|dw|
\end{split}
\end{equation}
To bound the latter $w,z$-integral, observe that if $\re(w)=-\epsilon$ and $\re(z)=\epsilon/2$, then \eqref{eqn: mellinrapiddecay} and \eqref{eqn: Hbound} imply that $\widetilde{W}(1-w)\mathcal{H}(z,-w)$ is $O(|w|^{-99}|z|^{\varepsilon-1})$ for $|w-z|\geq |z|/2$, and is $O(|w|^{-99}|z|^{-99})$ for $|w-z|\leq |z|/2$ since $|w|\asymp |z|$ and $|w-z|\geq \epsilon/2$ in this case. Hence
\begin{equation}\label{eqn: zwbound}
\begin{split}
\int_{(-\epsilon)} & \int_{(\epsilon/2)}\left|\zeta(1+w)\widetilde{W}(1-w)\mathcal{H}(z,-w)\frac{e^{\delta z}-e^{-\delta z}}{2\delta z}\right|\,|dz|\,|dw| \\
& \ll \int_{(-\epsilon)}\int_{(\epsilon/2)} |w|^{-98}|z|^{\varepsilon-1}\min\left\{ 1, \frac{1}{\delta|z|}\right\} \,|dz|\,|dw| \ll \left( \frac{1}{\delta}\right)^{\varepsilon}.
\end{split}
\end{equation}
From this, \eqref{eqn: U2diagonalHK2}, and the definition \eqref{eqn: deltadef} of $\delta$, we deduce that the total contribution of the $mh=nk$ terms is
\begin{equation}\label{eqn: U2diagonalbound}
\ll Q^{1+\varepsilon}\sum_{1\leq c \leq C}\frac{1}{c^{1+\varepsilon}}\sum_{1\le \ell \ll X}\frac{(hk\ell)^{\varepsilon}}{\ell\sqrt{HK}} \ll X^{\varepsilon} Q^{1+\varepsilon}\frac{(hk)^\varepsilon(h,k)}{(hk)^{1/2}}.
\end{equation}

We now complete the $m,n$-sum in \eqref{eqn: U2before3} by including the $mh=nk$ terms. As we have just shown, this introduces an error of size \eqref{eqn: U2diagonalbound}. Then, we apply Mellin inversion to $V(m/X)$ and $V(n/X)$ and arrive at
\begin{equation}\label{eqn: U2before4}
\begin{split}
 \mathcal{U}^2(h,k)&= \frac{Q}{2} \sum_{\substack{1\leq c \leq C  \\ (c ,hk)=1 }} \frac{\mu(c)}{c} \cdot  \frac{1}{(2\pi i)^2}\int_{(2)}\int_{(2)} X^{s_1+s_2}
 \sum_{\substack{1\leq m,n<\infty \\ (mn,c )=1}} \frac{\tau_A(m) \tau_B(n) }{m^{\frac{1}{2}+s_1}n^{\frac{1}{2}+s_2}} \widetilde{V}(s_1)\widetilde{V}(s_2)\\
&\hspace{.25in}\times \frac{1}{2\pi i}\int_{(-\epsilon)} \zeta(1+w) \widetilde{W}(1-w) \left( \frac{c}{gQ}\right)^w \frac{1}{2\pi i} \int_{(\epsilon/2)} \mathcal{H}(z,-w) (mh)^{w+z}(nk)^{-z} \\
&\hspace{.5in}\times \frac{e^{\delta z} - e^{-\delta z}}{2\delta z} \prod_{p|mnhk/g^2}\left(1-\frac{1}{p^{1+w}}\right) \prod_{\substack{p|g\\p\nmid mnhk/g^2}}\left(1+\frac{1}{p^{w+1}(p-1)}-\frac{1}{p-1} \right) \\
&\hspace{.75in}\times \prod_{\substack{p\nmid g\\p\nmid mnhk/g^2}}\left(1+\frac{p^{-w}-1}{p(p-1)}\right) \,dz \,dw \,ds_2 \,ds_1 \\
& \hspace{1in} + O\left(X^{\varepsilon}Q^{1+\varepsilon}\frac{(hk)^\varepsilon(h,k)}{(hk)^{1/2}} + (XChk)^{\varepsilon} k X^2 Q^{-97}\right).
\end{split}
\end{equation}
We have chosen the $s_1$- and $s_2$-lines to be at $\re(s_1)=\re(s_2)=2$ to ensure the absolute convergence of the $m,n$-sum.

Our next task is to express the $m,n$-sum in \eqref{eqn: U2before4} as an Euler product. This sum is
\begin{equation}\label{eqn: U2mnea}
\begin{split}
&  \sum_{\substack{1\leq m,n<\infty \\ (mn,c )=1 }} \frac{\tau_A(m) \tau_B(n) }{ m^{\frac{1}{2}+s_1}n^{\frac{1}{2}+s_2}}    g^{-w} m^{w+z}n^{-z}  \prod_{p| mnhk/g^2} \left( 1-\frac{1}{p^{1+w}}\right)\\
&\hspace{.5in}\times \prod_{\substack{ p|g \\ p\nmid  mnhk/g^2 }} \left( 1+\frac{1}{p^{1+w}(p-1)}-\frac{1}{p-1}\right) \prod_{\substack{ p\nmid g \\ p\nmid  mnhk/g^2} } \left(1+ \frac{p^{-w}-1}{p(p-1)} \right) \\
& = \sum_{ 1\leq m,n<\infty} \prod_p f(m,n,p),
\end{split}
\end{equation}
where $f(m,n,p)$ is defined by
\begin{equation*}
f(m,n,p) := F_1(m,n,p)F_2(m,n,p)F_3(m,n,p)
\end{equation*}
with $F_1,F_2,F_3$ defined by
\begin{equation*}
F_1(m,n,p) := \begin{cases} 1 & \text{if } p\nmid c \\  \\ 1 & \text{if } p|c \text{ and }  \ordp(m)=\ordp(n)=0 \\ \\ 0 & \text{if } p|c \text{ and } \ordp(mn)>0,
\end{cases}
\end{equation*}
\begin{equation*}
F_2(m,n,p) := \frac{\tau_A(p^{\ordp(m)}) \tau_B(p^{\ordp(n)}) }{ p^{(\frac{1}{2}+s_1-w-z)\ordp(m)} p^{(\frac{1}{2}+s_2+z)\ordp(n)} p^{w \min\{\ordp(m)+\ordp(h), \ordp(n) + \ordp(k) \} }}, 
\end{equation*}
and
\begin{equation*}
F_3(m,n,p) := \begin{cases} \displaystyle 1-\frac{1}{p^{1+w}} & \text{if } p|\frac{mnhk}{g^2} \\  \\  \displaystyle 1+\frac{1}{p^{1+w}(p-1)}-\frac{1}{p-1} & \text{if } p|g \text{ and }  p\nmid \frac{mnhk}{g^2} \\ \\ \displaystyle 1+ \frac{p^{-w}-1}{p(p-1)} & \text{if } p\nmid \frac{mnhk}{g}  ,\end{cases}
\end{equation*}
respectively. We can rewrite the conditions in $F_3(m,n,p)$ in terms of $\ordp(m)$, $\ordp(n)$, $\ordp(h)$, and $\ordp(k)$, as follows. Since $g=(mh,nk)$, a prime $p$ divides $mnhk/g^2$ if and only if
$$
\ordp(m)+\ordp(h)+\ordp(n)+\ordp(k)-2\min\{\ordp(m)+\ordp(h), \ordp(n)+\ordp(k)\}>0.
$$
Since two real numbers $x,y$ satisfy $x+y-2\min\{x,y\}>0$ if and only if $x\neq y$, it follows that $p| mnhk/g^2$ if and only if  $\ordp(m)+\ordp(h)\neq \ordp(n)+\ordp(k)$. A similar argument shows that a prime $p$ satisfies $p\nmid mnhk/g$ if and only if $p\nmid mhnk$. Thus the definition of $F_3(m,n,p)$ is equivalent to
\begin{equation*}
F_3(m,n,p) = \begin{cases} \displaystyle 1-\frac{1}{p^{1+w}} & \text{if } \ordp(m)+\ordp(h)\neq \ordp(n)+\ordp(k) \\  \\  \displaystyle 1+\frac{1}{p^{1+w}(p-1)}-\frac{1}{p-1} & \text{if } \ordp(m)+\ordp(h)= \ordp(n)+\ordp(k) >0 \\ \\ \displaystyle 1+ \frac{p^{-w}-1}{p(p-1)} & \text{if } \ordp(m)+\ordp(h)= \ordp(n)+\ordp(k) =0 .\end{cases}
\end{equation*}
If $p|c$, then $F_1(m,n,p)=0$ unless $m=n=1$, in which case
\begin{equation*}
f(1,1,p) = F_1(1,1,p)F_2(1,1,p)F_3(1,1,p) = 1+ \frac{p^{-w}-1}{p(p-1)}
\end{equation*}
because $(c,hk)=1$. Thus, from \eqref{eqn: U2mnea} and Lemma~\ref{lem: sumstoEulerproducts}, we deduce that the $m,n$-sum in \eqref{eqn: U2before4} equals
\begin{align*}
& \prod_p \sum_{0\leq m,n<\infty} f(p^m,p^n,p)  \\
& = \prod_{p|c} \left( 1+ \frac{p^{-w}-1}{p(p-1)} \right) \prod_{p\nmid c} \sum_{0\leq m,n<\infty} F_2(m,n,p) F_3(m,n,p) \\
& = \prod_{p|c} \left( 1+ \frac{p^{-w}-1}{p(p-1)} \right)\\
& \hspace{.25in}\times \prod_{\substack{p\nmid c \\ p|hk}} \Bigg( \sum_{\substack{0\leq m,n<\infty \\ m+\ordp(h)= n + \ordp(k) } }\frac{\tau_A(p^{m}) \tau_B(p^{n}) \left( 1+\frac{1}{p^{1+w}(p-1)}-\frac{1}{p-1} \right)}{ p^{m(\frac{1}{2}+s_1-w-z)} p^{n(\frac{1}{2}+s_2+z)} p^{w \min\{m+\ordp(h), n + \ordp(k) \} }}  \\
& \hspace{.75in} + \sum_{\substack{0\leq m,n<\infty \\ m+\ordp(h)\neq  n + \ordp(k) } } \frac{\tau_A(p^{m}) \tau_B(p^{n}) \left( 1-\frac{1}{p^{1+w}}  \right)}{ p^{m(\frac{1}{2}+s_1-w-z)} p^{n(\frac{1}{2}+s_2+z)} p^{w \min\{m+\ordp(h), n + \ordp(k) \} }}  \Bigg) \\
& \hspace{.25in} \times \prod_{\substack{p\nmid c \\ p \nmid hk}} \Bigg( 1+ \frac{p^{-w}-1}{p(p-1)} + \sum_{m=1 }^{\infty} \frac{\tau_A(p^{m}) \tau_B(p^{m}) \left( 1+\frac{1}{p^{1+w}(p-1)}-\frac{1}{p-1} \right)}{ p^{m(1+s_1+s_2)} }  \\
&  \hspace{.75in} + \sum_{\substack{0\leq m,n<\infty \\ m\neq n  } } \frac{\tau_A(p^{m}) \tau_B(p^{n}) \left( 1-\frac{1}{p^{1+w}}  \right)}{ p^{m(\frac{1}{2}+s_1-w-z)} p^{n(\frac{1}{2}+s_2+z)} p^{w \min\{m, n \} }}  \Bigg).
\end{align*}
We substitute this for the $m,n$-sum in \eqref{eqn: U2before4}. For convenience, we also make a change of variables $w\mapsto 1-w$. The result is
\begin{align}
 \mathcal{U}^2(h,k)&= \frac{Q}{2} \sum_{\substack{1\leq c \leq C  \\ (c ,hk)=1 }} \frac{\mu(c)}{c}  \cdot \frac{1}{(2\pi i)^4 } \int_{(2)} \int_{(2)} X^{s_1+s_2}  \widetilde{V}(s_1)\widetilde{V}(s_2)  \int_{(1+\epsilon)} \zeta(2-w)  \widetilde{W}(w) \left( \frac{c}{ Q}\right)^{1-w}  \notag\\
 &\hspace{.25in}\times \int_{(\epsilon/2)} \mathcal{H}(z,w-1) \frac{e^{\delta z} - e^{-\delta z}}{2\delta z}h^{1-w+z}k^{-z} \prod_{p|c} \left( 1+ \frac{p^{w-1}-1}{p(p-1)} \right)  \notag\\
&\times \prod_{\substack{p\nmid c \\ p|hk}} \Bigg( \sum_{\substack{0\leq m,n<\infty \\ m+\ordp(h)= n + \ordp(k) } }\frac{\tau_A(p^{m}) \tau_B(p^{n}) \left( 1+\frac{p^w}{p^{2}(p-1)}-\frac{1}{p-1} \right)}{ p^{m(-\frac{1}{2}+s_1+w-z)} p^{n(\frac{1}{2}+s_2+z)} p^{(1-w) \min\{m+\ordp(h), n + \ordp(k) \} }}  \notag\\
&\hspace{.5in}+ \sum_{\substack{0\leq m,n<\infty \\ m+\ordp(h)\neq  n + \ordp(k) } } \frac{\tau_A(p^{m}) \tau_B(p^{n}) \left( 1-\frac{p^w}{p^{2}}  \right)}{ p^{m(-\frac{1}{2}+s_1+w-z)} p^{n(\frac{1}{2}+s_2+z)} p^{(1-w) \min\{m+\ordp(h), n + \ordp(k) \} }}  \Bigg) \notag\\
&\times \prod_{\substack{p\nmid c \\ p \nmid hk}} \Bigg(
1+ \frac{p^{w-1}-1}{p(p-1)} + \sum_{m=1}^{\infty} \frac{\tau_A(p^{m}) \tau_B(p^{m}) \left( 1+\frac{p^w}{p^{2}(p-1)}-\frac{1}{p-1} \right)}{ p^{m(1+s_1+s_2  )} } \notag\\
&\hspace{.5in}+ \sum_{\substack{0\leq m,n<\infty \\ m\neq n  } } \frac{\tau_A(p^{m}) \tau_B(p^{n}) \left( 1-\frac{p^w}{p^{2}}  \right)}{ p^{m(-\frac{1}{2}+s_1+w-z)} p^{n(\frac{1}{2}+s_2+z)} p^{(1-w) \min\{m, n \} }}  \Bigg)\,dz  \,dw \,ds_2\,ds_1 \notag\\
&+ O\left(X^{\varepsilon}Q^{1+\varepsilon}\frac{(hk)^\varepsilon(h,k)}{(hk)^{1/2}} + (XChk)^{\varepsilon} k X^2 Q^{-97}\right). \label{eqn: U2ready}
\end{align}

\subsection{Analysis of the predicted one-swap terms from the recipe}\label{subsec: one-swap1}
Before we continue our treatment of $\mathcal{U}^2(h,k)$, we first break down the predicted one-swap terms into several parts via the residue theorem. Afterward, we will show that $\mathcal{U}^2(h,k)$ is equal to the sum of the same parts plus admissible error terms.

Recall that the definition of $\mathcal{I}_1^*(h,k)$ is given by \eqref{eqn: Istardef} with $\ell=1$. For each term in the definition of $\mathcal{I}_1^*(h,k)$, we denote the element of $U$ by $\alpha$ and the element of $V$ by $\beta$, and we multiply the integrand by $ \zeta(w-1+\alpha+s_1+\beta+s_2) $ and divide it by the Euler product of $\zeta(w-1+\alpha+s_1+\beta+s_2)$. This ``factoring out'' of the zeta-function gives us a further analytic continuation of the integrand and allows us to evaluate its residues when shifting contours. With these notations and factorization, we thus write $\mathcal{I}_1^*(h,k)$ as
\begin{equation}\label{eqn: 1swapterms0}
\begin{split}
\mathcal{I}_1^*(h,k) = \sum_{\substack{\alpha\in A \\ \beta\in B}} \frac{1}{2(2\pi i )^3} \int_{(\varepsilon)} \int_{(\varepsilon)} \int_{(2+\varepsilon)} X^{s_1+s_2} Q^w \widetilde{V}(s_1) \widetilde{V}(s_2)\widetilde{W}(w) \\
\times \mathscr{X} (\tfrac{1}{2}+\alpha +s_1 ) \mathscr{X} (\tfrac{1}{2}+\beta+s_2 ) \prod_{\substack{\gamma \in A_{s_1}\smallsetminus \{\alpha+s_1\} \cup \{-\beta-s_2\} \\ \delta \in B_{s_2}\smallsetminus \{\beta+s_2\} \cup \{-\alpha-s_1\}  }} \zeta(1+\gamma+\delta)  \\
\times \zeta(w-1+\alpha+s_1+\beta+s_2) \mathcal{K}(s_1,s_2,w) \,dw \,ds_2\,ds_1,
\end{split}
\end{equation}
where $\mathcal{K}(s_1,s_2,w)$ is defined by
\begin{align}
\mathcal{K}(s_1,s_2,w) =
& \mathcal{K}(s_1,s_2,w;A,B,\alpha,\beta,h,k) \notag\\
: = &  \prod_{p|hk} \Bigg\{ \left( 1 - \frac{1}{p^{w-1+\alpha+s_1+\beta+s_2}}\right) \prod_{\substack{\gamma \in A_{s_1}\smallsetminus \{\alpha+s_1\} \cup \{-\beta-s_2\} \\ \delta \in B_{s_2}\smallsetminus \{\beta+s_2\} \cup \{-\alpha-s_1\}  }}\left( 1- \frac{1}{p^{1+\gamma+\delta}}\right) \notag\\
& \hspace{.5in} \times \sum_{ \substack{0\leq m,n<\infty\\ m+\ordp(h) = n+\ordp(k)}}\frac{\tau_{A_{s_1}\smallsetminus \{\alpha+s_1\} \cup \{-\beta-s_2\}} (p^m) \tau_{B_{s_2}\smallsetminus \{\beta+s_2\} \cup \{-\alpha-s_1\} } (p^n)  }{p^{m/2}p^{n/2} }\Bigg\} \notag\\
&  \times \prod_{p\nmid hk} \Bigg\{ \left( 1 - \frac{1}{p^{w-1+\alpha+s_1+\beta+s_2}}\right)\prod_{\substack{\gamma \in A_{s_1}\smallsetminus \{\alpha+s_1\} \cup \{-\beta-s_2\} \\ \delta \in B_{s_2}\smallsetminus \{\beta+s_2\} \cup \{-\alpha-s_1\}  }}\left( 1- \frac{1}{p^{1+\gamma+\delta}}\right) \notag\\
& \hspace{.5in}  \times \Bigg(1 + \frac{p-2}{p^{w+ \alpha+s_1  + \beta + s_2 } }   + \left( 1-\frac{1}{p}\right)^2 \frac{p^{2(1-w- \alpha-s_1  - \beta - s_2 )}}{1-p^{1-w- \alpha-s_1  - \beta - s_2 }}  \notag\\
& \hspace{.75in}  + \sum_{m=1}^{\infty} \frac{\tau_{A_{s_1}\smallsetminus \{\alpha+s_1\} \cup \{-\beta-s_2\}} (p^m) \tau_{B_{s_2}\smallsetminus \{\beta+s_2\} \cup \{-\alpha-s_1\}} (p^m)  }{p^m}   \Bigg)\Bigg\}. \label{eqn: 1swapeulerKdef}
\end{align}

To facilitate our estimations, we first prove the following lemma, which will allow us to move lines of integration and bound the integrals that remain after applying the residue theorem.
\begin{lemma}\label{lem: 1swapeulerbound}
Suppose that $\epsilon>0$ is arbitrarily small. Let $\alpha\in A$ and $\beta\in B$, and let $h$ and $k$ be positive integers. If $s_1,s_2,w$ are complex numbers such that
\begin{enumerate}
\item[\upshape{(i)}] $\re(w-1+\alpha+s_1+\beta+s_2)\geq \frac{1}{2}+\varepsilon$, \item[\upshape{(ii)}] $-\frac{1}{2}+5\epsilon\leq \re(s_1+s_2) \leq 2\epsilon$, and
\item[\upshape{(iii)}] either $|\re(s_1)|\leq \epsilon$ or $|\re(s_2)|\leq \epsilon$,
\end{enumerate}
then the product \eqref{eqn: 1swapeulerKdef} defining $\mathcal{K}(s_1,s_2,w;A,B,\alpha,\beta,h,k)$ converges absolutely and we have
$$
\mathcal{K}(s_1,s_2,w;A,B,\alpha,\beta,h,k) \ll_{\varepsilon} (hk)^{\varepsilon}.
$$
\end{lemma}
\begin{proof}
Since $\re(w-1+\alpha+s_1+\beta+s_2)\geq \frac{1}{2}+\varepsilon$, we have
\begin{equation}\label{eqn: 1swapeulerbound1}
\frac{1}{p^{w-1+\alpha+s_1+\beta+s_2} }\ll \frac{1}{p^{\frac{1}{2}+\varepsilon}}.
\end{equation}
Moreover, each term of the form $p^{-1-\gamma-\delta}$ in the definition \eqref{eqn: 1swapeulerKdef} of $\mathcal{K}(s_1,s_2,w)$ satisfies $p^{-1-\gamma-\delta}\ll p^{-\frac{1}{2}-\varepsilon}$ because $-\frac{1}{2}+5\epsilon\leq \re(s_1+s_2) \leq 2\epsilon$ and each element of $A\cup B$ is $\ll 1/\log Q$. We may thus multiply out the product and apply the definition \eqref{eqn: taudef} of $\tau_E$ to deduce that
\begin{equation}\label{eqn: 1swapeulerbound2}
\begin{split}
& \left( 1 - \frac{1}{p^{w-1+\alpha+s_1+\beta+s_2}}\right)  \prod_{\substack{\gamma \in A_{s_1}\smallsetminus \{\alpha+s_1\} \cup \{-\beta-s_2\} \\ \delta \in B_{s_2}\smallsetminus \{\beta+s_2\} \cup \{-\alpha-s_1\}  }}\left( 1- \frac{1}{p^{1+\gamma+\delta}}\right)\\
& = 1 - \frac{1}{p^{w-1+\alpha+s_1+\beta+s_2}} - \sum_{\substack{\gamma \in A_{s_1}\smallsetminus \{\alpha+s_1\} \cup \{-\beta-s_2\} \\ \delta \in B_{s_2}\smallsetminus \{\beta+s_2\} \cup \{-\alpha-s_1\}  }}\frac{1}{p^{1+\gamma+\delta}} + O\left( \frac{1}{p^{1+\varepsilon}}\right) \\
& = 1 - \frac{1}{p^{w-1+\alpha+s_1+\beta+s_2}}  - \frac{\tau_{A_{s_1}\smallsetminus \{\alpha+s_1\} \cup \{-\beta-s_2\}} (p ) \tau_{B_{s_2}\smallsetminus \{\beta+s_2\} \cup \{-\alpha-s_1\}} (p )  }{p } + O\left( \frac{1}{p^{1+\varepsilon}}\right).
\end{split}
\end{equation}
We may assume that $|\re(s_2)|\leq \epsilon$ as the proof for the case with $|\re(s_1)|\leq \epsilon$ is similar. Since $ -\frac{1}{2}+5\epsilon \leq \re(s_1+s_2) \leq 2\epsilon$, it then follows that $-\frac{1}{2}+4\epsilon\leq \re(s_1) \leq 3\epsilon $. This, the inequality $|\re(s_2)|\leq \epsilon$, and the bound \eqref{eqn: divisorbound} imply
\begin{equation}\label{eqn: 1swapeulerbound3}
\tau_{A_{s_1}\smallsetminus \{\alpha+s_1\} \cup \{-\beta-s_2\}} (p^m) \ll_{\varepsilon} p^{m(\frac{1}{2}-4\epsilon+\varepsilon)},
\end{equation}
and
\begin{equation}\label{eqn: 1swapeulerbound4}
\tau_{B_{s_2}\smallsetminus \{\beta+s_2\} \cup \{-\alpha-s_1\}} (p^n)\ll_{\varepsilon} p^{n(3\epsilon+\varepsilon)}.
\end{equation}
Therefore
\begin{equation}\label{eqn: 1swapeulerbound6}
\sum_{m=2}^{\infty} \frac{\tau_{A_{s_1}\smallsetminus \{\alpha+s_1\} \cup \{-\beta-s_2\}} (p^m) \tau_{B_{s_2}\smallsetminus \{\beta+s_2\} \cup \{-\alpha-s_1\}} (p^m)  }{p^m} \ll \sum_{m=2}^{\infty} \frac{p^{m(\frac{1}{2}- \epsilon+\varepsilon) }   }{p^m} \ll \frac{1}{p^{1+\varepsilon}}
\end{equation}
and
\begin{equation}\label{eqn: 1swapeulerbound5}
\begin{split}
\sum_{ \substack{0\leq m,n<\infty\\ m+\ordp(h) = n+\ordp(k)}}&\frac{\tau_{A_{s_1}\smallsetminus \{\alpha+s_1\} \cup \{-\beta-s_2\}} (p^m) \tau_{B_{s_2}\smallsetminus \{\beta+s_2\} \cup \{-\alpha-s_1\} } (p^n)  }{p^{m/2}p^{n/2} }\\
&\hspace{.5in}\ll \sum_{0\leq m,n<\infty} \frac{p^{m(\frac{1}{2}-4\epsilon+\varepsilon)} p^{n(3\epsilon+\varepsilon)} }{ p^{m/2}p^{n/2} }\\
&\hspace{.5in}\ll 1.
\end{split}
\end{equation}
From \eqref{eqn: 1swapeulerbound1}, \eqref{eqn: 1swapeulerbound2}, \eqref{eqn: 1swapeulerbound3} with $m=1$, \eqref{eqn: 1swapeulerbound4} with $n=1$, and \eqref{eqn: 1swapeulerbound5}, we deduce that if $p|hk$ then the local factor in \eqref{eqn: 1swapeulerKdef} corresponding to $p$ is $O(1)$. On the other hand, from \eqref{eqn: 1swapeulerbound1}, \eqref{eqn: 1swapeulerbound2}, and \eqref{eqn: 1swapeulerbound6}, we deduce that if $p\nmid hk$ then the local factor in \eqref{eqn: 1swapeulerKdef} corresponding to $p$ is $1+O(p^{-1-\varepsilon})$. It follows that the right-hand side of \eqref{eqn: 1swapeulerKdef} converges absolutely, and is $\ll (hk)^{\varepsilon}$ since $\prod_{p|\nu}O(1) \ll {\nu}^{\varepsilon}$ for any positive integer $\nu$.
\end{proof}

We now move the $w$-line in \eqref{eqn: 1swapterms0} to $\re(w)=\frac{3}{2}+\varepsilon$. This leaves a residue from the pole at $w=2-\alpha-s_1-\beta-s_2$. To bound the new integral that has $\re(w)=\frac{3}{2}+\varepsilon$, we use Lemma~\ref{lem: 1swapeulerbound}, \eqref{eqn: zetaalphaalphabound}, and \eqref{eqn: mellinrapiddecay}. Since the residue of $\zeta(s)$ at $s=1$ is $1$, we arrive at
\begin{equation}\label{eqn: 1swapterms1}
\mathcal{I}_1^*(h,k) = \sum_{\substack{\alpha\in A \\ \beta\in B}} \frac{1}{2(2\pi i )^2} \int_{(\epsilon)} \int_{(\epsilon)} \mathcal{J} \,ds_2\,ds_1 + O\big( X^{\varepsilon} Q^{\frac{3}{2}+\varepsilon} (hk)^{\varepsilon}\big),
\end{equation}
where, for brevity, we define $\mathcal{J}$ by
\begin{equation}\label{eqn: integrandJdef}
\begin{split}
\mathcal{J} := X^{s_1+s_2} Q^{2-\alpha-s_1-\beta-s_2} \widetilde{V}(s_1) \widetilde{V}(s_2)\widetilde{W}(2-\alpha-s_1-\beta-s_2) \\
\times \mathscr{X} (\tfrac{1}{2}+\alpha +s_1 ) \mathscr{X} (\tfrac{1}{2}+\beta+s_2 ) \prod_{\substack{\gamma \in A_{s_1}\smallsetminus \{\alpha+s_1\} \cup \{-\beta-s_2\} \\ \delta \in B_{s_2}\smallsetminus \{\beta+s_2\} \cup \{-\alpha-s_1\}  }} \zeta(1+\gamma+\delta)  \\
\times  \mathcal{K}(s_1,s_2,2-\alpha-s_1-\beta-s_2).
\end{split}
\end{equation}
Notice that we have now specified the lines of integration in \eqref{eqn: 1swapterms1} to be $\re(s_1)=\epsilon$ and $\re(s_2)=\epsilon$, with $\epsilon$ fixed and arbitrarily small. The purpose of this is to make the succeeding estimations more explicit.

Next, we move the $s_2$-line in \eqref{eqn: 1swapterms1} to $\re(s_2)=-\frac{1}{2}+5\epsilon$. This leaves residues from the pole at $s_2=0$ due to the factor $\widetilde{V}(s_2)$, the pole at $s_2=-s_1-\alpha-\beta$ due to the factor $\zeta(1-\alpha-s_1-\beta-s_2)$, and the poles at $s_2=-s_1-\alpha'-\beta'$ due to the factors $\zeta(1+\alpha'+s_1+\beta'+s_2)$, where $\alpha'$ runs through the elements of $A\smallsetminus \{\alpha\}$ and $\beta'$ runs through the elements of $B\smallsetminus \{\beta\}$. To bound the new integral that has $\re(s_2)=-\frac{1}{2}+5\epsilon$, we use Lemma~\ref{lem: 1swapeulerbound}, \eqref{eqn: zetaalphaalphabound}, and \eqref{eqn: mellinrapiddecay}. We arrive at
\begin{equation*}
\begin{split}
\mathcal{I}_1^*(h,k) &= \sum_{\substack{\alpha\in A \\ \beta\in B}} \frac{1}{4\pi i } \int_{(\epsilon)} \underset{s_2=0}{\text{Res}}\ \mathcal{J} \,ds_1  +  \sum_{\substack{\alpha\in A \\ \beta\in B}} \frac{1}{4\pi i } \int_{(\epsilon)} \underset{s_2=-s_1-\alpha-\beta}{\text{Res}}\ \mathcal{J} \,ds_1\\
&\hspace{.25in}+ \sum_{\substack{\alpha\in A \\ \beta\in B}} \sum_{\substack{\alpha'\neq \alpha \\ \beta'\neq \beta}} \frac{1}{4\pi i } \int_{(\epsilon)} \underset{s_2=-s_1-\alpha'-\beta'}{\text{Res}}\ \mathcal{J} \,ds_1\\
&\hspace{.5in}+ O\big( X^{-\frac{1}{2}+\varepsilon} Q^{\frac{5}{2}} (hk)^{\varepsilon}\big) + O\big( X^{\varepsilon} Q^{\frac{3}{2}+\varepsilon} (hk)^{\varepsilon}\big).
\end{split}
\end{equation*}
For brevity, write this as
\begin{equation}\label{eqn: 1swapterms2}
\mathcal{I}_1^*(h,k) = J_1+J_2+J_3 + O\big( X^{-\frac{1}{2}+\varepsilon} Q^{\frac{5}{2}} (hk)^{\varepsilon}\big) + O\big( X^{\varepsilon} Q^{\frac{3}{2}+\varepsilon} (hk)^{\varepsilon}\big).
\end{equation}

We first evaluate the contribution $J_1$ of the residue at $s_2=0$. By \eqref{eqn: mellinVresidue} and the definition \eqref{eqn: integrandJdef} of $\mathcal{J}$, we have
\begin{equation}\label{eqn: Ress20J}
\begin{split}
\underset{s_2=0}{\text{Res}}\ \mathcal{J} &= X^{s_1 } Q^{2-\alpha-s_1-\beta } \widetilde{V}(s_1)  \widetilde{W}(2-\alpha-s_1-\beta )\mathscr{X} (\tfrac{1}{2}+\alpha +s_1 ) \mathscr{X} (\tfrac{1}{2}+\beta  )\\
&\hspace{.25in}\times\prod_{\substack{\gamma \in A_{s_1}\smallsetminus \{\alpha+s_1\} \cup \{-\beta \} \\ \delta \in B \smallsetminus \{\beta \} \cup \{-\alpha-s_1\}  }} \zeta(1+\gamma+\delta)\mathcal{K}(s_1,0,2-\alpha-s_1-\beta).
\end{split}
\end{equation}
We move the line of integration in the definition
\begin{equation*}
J_1 := \sum_{\substack{\alpha\in A \\ \beta\in B}} \frac{1}{4\pi i } \int_{(\epsilon)} \underset{s_2=0}{\text{Res}}\ \mathcal{J} \,ds_1
\end{equation*}
to $\re(s_1)=-\frac{1}{2}+5\epsilon$. We find residues from the pole at $s_1=0$ due to the factor $\widetilde{V}(s_1)$ in \eqref{eqn: Ress20J}, the pole at $s_1=-\alpha-\beta$ due to the factor $\zeta(1-\alpha-s_1-\beta)$, and the poles at $s_1=-\alpha'-\beta'$ due to the factors $\zeta(1+\alpha'+s_1+\beta')$, where $\alpha'$ runs through the elements of $A\smallsetminus \{\alpha\}$ and $\beta'$ runs through the elements of $B\smallsetminus \{\beta\}$. To bound the new integral that has $\re(s_1)=-\frac{1}{2}+5\epsilon$, we use Lemma~\ref{lem: 1swapeulerbound}, \eqref{eqn: zetaalphaalphabound}, and \eqref{eqn: mellinrapiddecay}. We deduce that
\begin{equation*}
\begin{split}
J_1 &=  \frac{1}{2}\sum_{\substack{\alpha\in A \\ \beta\in B}}  \underset{s_1=0}{\text{Res}}\  \underset{s_2=0}{\text{Res}}\ \mathcal{J}   +  \frac{1}{2}\sum_{\substack{\alpha\in A \\ \beta\in B}}    \underset{s_1=-\alpha-\beta}{\text{Res}}\ \underset{s_2=0}{\text{Res}}\ \mathcal{J}\\
&\hspace{.25in}+ \frac{1}{2 }\sum_{\substack{\alpha\in A \\ \beta\in B}} \sum_{\substack{\alpha'\neq \alpha \\ \beta'\neq \beta}}  \underset{s_1=-\alpha'-\beta'}{\text{Res}}\ \underset{s_2=0}{\text{Res}}\ \mathcal{J}  + O\big( X^{-\frac{1}{2}+\varepsilon} Q^{\frac{5}{2}} (hk)^{\varepsilon}\big).
\end{split}
\end{equation*}
For brevity, we write this as
\begin{equation}\label{eqn: J1split}
J_1 = J_{11}+J_{12}+J_{13} + O\big( X^{-\frac{1}{2}+\varepsilon} Q^{\frac{5}{2}} (hk)^{\varepsilon}\big).
\end{equation}
We deduce from \eqref{eqn: mellinVresidue} and \eqref{eqn: Ress20J} that
\begin{equation}\label{eqn: J11evaluated}
\begin{split}
J_{11} &= \sum_{\substack{\alpha\in A \\ \beta\in B}} \frac{1}{2}  Q^{2-\alpha-\beta } \widetilde{W}(2-\alpha-\beta )\mathscr{X} (\tfrac{1}{2}+\alpha ) \mathscr{X} (\tfrac{1}{2}+\beta  )\\
&\hspace{.25in}\times\prod_{\substack{\gamma \in A\smallsetminus \{\alpha\} \cup \{-\beta \} \\ \delta \in B \smallsetminus \{\beta \} \cup \{-\alpha\}  }} \zeta(1+\gamma+\delta) \mathcal{K}(0,0,2-\alpha-\beta).
\end{split}
\end{equation}
Since $\mathscr{X} (\tfrac{1}{2}-\beta ) \mathscr{X} (\tfrac{1}{2}+\beta  ) =1$ by the definition \eqref{eqn: scriptXdef} of $\mathscr{X}$ and the residue of $\zeta(1-\alpha-s_1-\beta)$ at $s_1=-\alpha-\beta$ is $-1$, it follows from \eqref{eqn: Ress20J} that term $J_{12}$ in \eqref{eqn: J1split} equals
\begin{equation}\label{eqn: J12evaluated}
\begin{split}
J_{12} &= - \frac{1}{2}\sum_{\substack{\alpha\in A \\ \beta\in B}} X^{-\alpha-\beta } Q^{2} \widetilde{V}(-\alpha-\beta)  \widetilde{W}(2) \prod_{ \substack{\hat{\alpha} \neq \alpha \\ \hat{\beta}\neq \beta}} \zeta(1-\alpha-\beta+\hat{\alpha} +\hat{\beta} )\\
&\times\prod_{ \hat{\alpha} \neq \alpha} \zeta(1+\hat{\alpha}-\alpha)\prod_{\hat{\beta}\neq \beta} \zeta(1+\hat{\beta}-\beta) \mathcal{K}(-\alpha-\beta,0,2).
\end{split}
\end{equation}
Next, since the residue of $\zeta(1+\alpha'+s_1+\beta')$ at $s_1=-\alpha'-\beta'$ is $1$, it follows from \eqref{eqn: Ress20J} that the term $J_{13}$ in \eqref{eqn: J1split} equals
\begin{equation}\label{eqn: J13evaluated}
\begin{split}
& J_{13} = \frac{1}{2} \sum_{\substack{\alpha\in A \\ \beta \in B }}  \sum_{\substack{\alpha'\neq \alpha \\ \beta'\neq \beta}}      X^{-\alpha'-\beta' } Q^{2-\alpha -\beta +\alpha'+\beta' } \widetilde{V}(-\alpha'-\beta')  \widetilde{W}(2-\alpha-\beta + \alpha'+\beta' )\\
& \hspace{.1in}\times \mathscr{X} (\tfrac{1}{2}+\alpha -\alpha'-\beta' ) \mathscr{X} (\tfrac{1}{2}+\beta  ) \zeta(1-\alpha-\beta+\alpha'+\beta' ) \prod_{ \substack{\hat{\alpha} \neq \alpha \\ \hat{\beta}\neq \beta \\ (\hat{\alpha},\hat{\beta})\neq (\alpha',\beta')}} \zeta(1-\alpha'-\beta'+\hat{\alpha}  +\hat{\beta} )\\
&\hspace{.25in}\times \prod_{ \hat{\alpha} \neq \alpha} \zeta(1+\hat{\alpha}-\alpha)\prod_{\hat{\beta}\neq \beta} \zeta(1+\hat{\beta}-\beta) \mathcal{K}(-\alpha'-\beta',0,2-\alpha-\beta+\alpha'+\beta').
\end{split}
\end{equation}
This, \eqref{eqn: J1split}, \eqref{eqn: J11evaluated}, and \eqref{eqn: J12evaluated} complete our evaluation of $J_1$.

Having estimated $J_1$, we next turn to the integral $J_2$ from \eqref{eqn: 1swapterms2}. Recall its definition
\begin{equation}\label{eqn: J2def}
J_2 := \sum_{\substack{\alpha\in A \\ \beta\in B}} \frac{1}{4\pi i } \int_{(\epsilon)} \underset{s_2=-s_1-\alpha-\beta}{\text{Res}}\ \mathcal{J} \,ds_1.
\end{equation}
Since $\mathscr{X} (\tfrac{1}{2}+\alpha +s_1 ) \mathscr{X} (\tfrac{1}{2}-\alpha-s_1 )=1$ by the definition \eqref{eqn: scriptXdef} of $\mathscr{X}$ and the residue of $\zeta(1-\alpha-s_1-\beta-s_2)$ at $s_2=-s_1-\alpha-\beta$ is $-1$, we see from the definition \eqref{eqn: integrandJdef} of $\mathcal{J}$ that
\begin{equation}\label{eqn: Ress2s1abJ}
\begin{split}
\underset{s_2=-s_1-\alpha-\beta}{\text{Res}}\ \mathcal{J} &= -X^{-\alpha-\beta} Q^{2} \widetilde{V}(s_1) \widetilde{V}(-s_1-\alpha-\beta)\widetilde{W}(2) \\
&\hspace{.25in}\times  \prod_{ \substack{\hat{\alpha} \neq \alpha \\ \hat{\beta}\neq \beta}} \zeta(1+\hat{\alpha} -\alpha+\hat{\beta}-\beta) \prod_{ \hat{\alpha} \neq \alpha} \zeta(1+\hat{\alpha}-\alpha)\prod_{\hat{\beta}\neq \beta} \zeta(1+\hat{\beta}-\beta)  \\
&\hspace{.25in}\times  \mathcal{K}(s_1,-s_1-\alpha-\beta,2).
\end{split}
\end{equation}
We move the line of integration in \eqref{eqn: J2def} to $\re(s_1)=-\epsilon-\re(\alpha)-\re(\beta)$. This leaves residues from the poles at $s_1=0$ and $s_1=-\alpha-\beta$ due to the factors $\widetilde{V}(s_1)$ and $\widetilde{V}(-s_1-\alpha-\beta)$ in \eqref{eqn: Ress2s1abJ}, and we arrive at
\begin{equation*}
\begin{split}
J_2 =  \frac{1}{2}\sum_{\substack{\alpha\in A \\ \beta\in B}}  \underset{s_1=0}{\text{Res}}\  \underset{s_2=-s_1-\alpha-\beta}{\text{Res}}\ \mathcal{J}   +  \frac{1}{2}\sum_{\substack{\alpha\in A \\ \beta\in B}}    \underset{s_1=-\alpha-\beta}{\text{Res}}\  \underset{s_2=-s_1-\alpha-\beta}{\text{Res}}\ \mathcal{J}\\
+ \sum_{\substack{\alpha\in A \\ \beta\in B}} \frac{1}{4\pi i } \int_{(-\epsilon-\re(\alpha)-\re(\beta))} \underset{s_2=-s_1-\alpha-\beta}{\text{Res}}\ \mathcal{J} \,ds_1.
\end{split}
\end{equation*}
For brevity, we write this as
\begin{equation}\label{eqn: J2split}
J_{2} =J_{21} +J_{22} +J_{23}.
\end{equation}
Since the residue of $\widetilde{V}(s)$ at $s=0$ is $1$ by \eqref{eqn: mellinVresidue}, it follows from \eqref{eqn: Ress2s1abJ} that
\begin{equation}\label{eqn: J21evaluated}
\begin{split}
J_{21}&=  - \frac{1}{2}\sum_{\substack{\alpha\in A \\ \beta\in B}} X^{-\alpha-\beta} Q^{2} \widetilde{V}(-\alpha-\beta)\widetilde{W}(2) \\
&\hspace{.25in}\times  \prod_{ \substack{\hat{\alpha} \neq \alpha \\ \hat{\beta}\neq \beta}} \zeta(1+\hat{\alpha} -\alpha+\hat{\beta}-\beta) \prod_{ \hat{\alpha} \neq \alpha} \zeta(1+\hat{\alpha}-\alpha)\prod_{\hat{\beta}\neq \beta} \zeta(1+\hat{\beta}-\beta)\mathcal{K}(0,-\alpha-\beta,2).
\end{split}
\end{equation}
Similarly, since the residue of $\widetilde{V}(-s_1-\alpha-\beta)$ at $s_1=-\alpha-\beta$ is $-1$ by \eqref{eqn: mellinVresidue}, we see from \eqref{eqn: Ress2s1abJ} that the term $J_{22}$ in \eqref{eqn: J2split} equals
\begin{equation}\label{eqn: J22evaluated}
\begin{split}
J_{22}  &= \frac{1}{2}\sum_{\substack{\alpha\in A \\ \beta\in B}} X^{-\alpha-\beta} Q^{2} \widetilde{V}(-\alpha-\beta)\widetilde{W}(2) \\
&\hspace{.25in}\times  \prod_{ \substack{\hat{\alpha} \neq \alpha \\ \hat{\beta}\neq \beta}} \zeta(1+\hat{\alpha} -\alpha+\hat{\beta}-\beta) \prod_{ \hat{\alpha} \neq \alpha} \zeta(1+\hat{\alpha}-\alpha)\prod_{\hat{\beta}\neq \beta} \zeta(1+\hat{\beta}-\beta)\mathcal{K}(-\alpha-\beta,0,2).
\end{split}
\end{equation}
To evaluate the integral $J_{23}$ in \eqref{eqn: J2split}, we first prove the following lemma that gives a functional equation for $\mathcal{K}$.
\begin{lemma}\label{lem: Kfunctionaleqn}
Suppose that $\epsilon>0$ is arbitrarily small. Let $\alpha\in A$ and $\beta\in B$, and let $h$ and $k$ be positive integers. If $s_1,s_2,w$ are complex numbers satisfying the conditions (i)--(iii) in Lemma~\ref{lem: 1swapeulerbound},
then
\begin{equation*}
\mathcal{K}(s_1,s_2,w;A,B,\alpha,\beta,h,k) = \left( \frac{h}{k}\right)^{s_1}\mathcal{K}(0,s_1+s_2,w;A,B,\alpha,\beta,h,k).
\end{equation*}
\end{lemma}
\begin{proof}
Lemma~\ref{lem: 1swapeulerbound} guarantees that the product \eqref{eqn: 1swapeulerKdef}  defining $\mathcal{K}(s_1,s_2,w)$ converges absolutely and thus $\mathcal{K}(s_1,s_2,w)$ is well-defined. Now \eqref{eqn: taufactoringidentity} implies that
\begin{equation}\label{eqn: Kfunctional1}
\begin{split}
& \tau_{A_{s_1}\smallsetminus \{\alpha+s_1\} \cup \{-\beta-s_2\}} (p^m) \tau_{B_{s_2}\smallsetminus \{\beta+s_2\} \cup \{-\alpha-s_1\}} (p^m) \\
& = \tau_{A \smallsetminus \{\alpha \} \cup \{-\beta-s_1-s_2\}} (p^m) \tau_{B_{s_1+s_2}\smallsetminus \{\beta+s_1+s_2\} \cup \{-\alpha\}} (p^m).
\end{split}
\end{equation}
Similarly, if $m+\ordp(h)=n+\ordp(k)$, then \eqref{eqn: taufactoringidentity} implies
\begin{equation}\label{eqn: Kfunctional2}
\begin{split}
\tau_{A_{s_1}\smallsetminus \{\alpha+s_1\} \cup \{-\beta-s_2\}} (p^m) &\tau_{B_{s_2}\smallsetminus \{\beta+s_2\} \cup \{-\alpha-s_1\}} (p^n) \\
&= p^{s_1(n-m)}\tau_{A \smallsetminus \{\alpha \} \cup \{-\beta-s_1-s_2\}} (p^m) \tau_{B_{s_1+s_2}\smallsetminus \{\beta+s_1+s_2\} \cup \{-\alpha\}} (p^n) \\
&= p^{s_1(\ordp(h)-\ordp(k))}\tau_{A \smallsetminus \{\alpha \} \cup \{-\beta-s_1-s_2\}} (p^m) \tau_{B_{s_1+s_2}\smallsetminus \{\beta+s_1+s_2\} \cup \{-\alpha\}} (p^n).
\end{split}
\end{equation}
Also, we have
$$
\gamma+\delta = (\gamma-s_1) + (\delta+s_1).
$$
Lemma~\ref{lem: Kfunctionaleqn} follows from this, \eqref{eqn: Kfunctional1}, \eqref{eqn: Kfunctional2}, the definition \eqref{eqn: 1swapeulerKdef} of $\mathcal{K}$, and the fact that
$$
\prod_{p| hk} p^{s_1(\ordp(h)-\ordp(k))} = \left( \frac{h}{k}\right)^{s_1}.
$$
\end{proof}

We now evaluate the integral $J_{23}$ in \eqref{eqn: J2split}. Lemma~\ref{lem: Kfunctionaleqn} implies for $\re(s_1)=-\epsilon-\re(\alpha)-\re(\beta)$ that
\begin{equation}\label{eqn: KfunctionaleqnJ23}
\mathcal{K}(s_1,-s_1-\alpha-\beta,2) = \left( \frac{h}{k}\right)^{s_1}\mathcal{K}(0,-\alpha-\beta,2).
\end{equation}
Moreover, a change of variables $s_1\mapsto -s-\alpha-\beta$ gives
$$
\int_{(-\epsilon-\re(\alpha)-\re(\beta))} \widetilde{V}(s_1)\widetilde{V}(-s_1-\alpha-\beta) \left( \frac{h}{k}\right)^{s_1}\,ds_1 = \int_{(\epsilon)} \widetilde{V}(-s-\alpha-\beta) \widetilde{V}(s) \left( \frac{h}{k}\right)^{-s-\alpha-\beta} \,ds.
$$
From this, \eqref{eqn: KfunctionaleqnJ23}, and \eqref{eqn: Ress2s1abJ}, we deduce that the integral $J_{23}$ in \eqref{eqn: J2split} equals
\begin{equation}\label{eqn: J23evaluated}
\begin{split}
J_{23}  & = -\sum_{\substack{\alpha\in A \\ \beta\in B}}\frac{1}{4\pi i} \int_{(\epsilon)} X^{-\alpha-\beta} Q^{2}  \widetilde{V}(s) \widetilde{V}(-s-\alpha-\beta)\widetilde{W}(2) \\
& \hspace{.25in}\times  \prod_{ \substack{\hat{\alpha} \neq \alpha \\ \hat{\beta}\neq \beta}} \zeta(1+\hat{\alpha} -\alpha+\hat{\beta}-\beta) \prod_{ \hat{\alpha} \neq \alpha} \zeta(1+\hat{\alpha}-\alpha)\prod_{\hat{\beta}\neq \beta} \zeta(1+\hat{\beta}-\beta)  \\
& \hspace{.25in}\times \left( \frac{h}{k}\right)^{-s-\alpha-\beta} \mathcal{K}(0,-\alpha-\beta,2)\,ds.
\end{split}
\end{equation}
This, \eqref{eqn: J2split}, \eqref{eqn: J21evaluated}, and \eqref{eqn: J22evaluated} complete our calculation of $J_2$.

Now that we have evaluated $J_2$, we next turn our attention to the term $J_3$ in \eqref{eqn: 1swapterms2}. Recall its definition
\begin{equation}\label{eqn: J3def}
J_3 := \sum_{\substack{\alpha\in A \\ \beta\in B}} \sum_{\substack{\alpha'\neq \alpha \\ \beta'\neq \beta}} \frac{1}{4\pi i } \int_{(\epsilon)} \underset{s_2=-s_1-\alpha'-\beta'}{\text{Res}}\ \mathcal{J} \,ds_1.
\end{equation}
Since the residue of $\zeta(1+\alpha'+s_1+\beta'+s_2)$ at $s_2=-s_1-\alpha'-\beta'$ is $1$, it follows from the definition \eqref{eqn: integrandJdef} of $\mathcal{J}$ that if $\alpha'\in A\smallsetminus\{\alpha\}$ and $\beta'\in B\smallsetminus\{\beta\}$, then
\begin{equation}\label{eqn: Ress2s1a'b'J}
\begin{split}
\underset{s_2=-s_1-\alpha'-\beta'}{\text{Res}}\ \mathcal{J} = X^{-\alpha'-\beta'} Q^{2-\alpha-\beta+\alpha'+\beta'} \widetilde{V}(s_1) \widetilde{V}(-s_1-\alpha'-\beta')\widetilde{W}(2-\alpha-\beta+\alpha'+\beta') \\
\times \mathscr{X} (\tfrac{1}{2}+\alpha +s_1 ) \mathscr{X} (\tfrac{1}{2}+\beta-s_1-\alpha'-\beta' ) \zeta(1-\alpha-\beta+\alpha'+\beta') \\
\prod_{ \substack{\hat{\alpha} \neq \alpha \\ \hat{\beta}\neq \beta \\ (\hat{\alpha},\hat{\beta})\neq (\alpha',\beta')}} \zeta(1+\hat{\alpha}  +\hat{\beta} -\alpha'-\beta') \prod_{ \hat{\alpha} \neq \alpha} \zeta(1+\hat{\alpha}-\alpha) \prod_{\hat{\beta}\neq \beta} \zeta(1+\hat{\beta}-\beta)  \\
\times  \mathcal{K}(s_1,-s_1-\alpha'-\beta',2-\alpha-\beta+\alpha'+\beta').
\end{split}
\end{equation}
We move the line of integration in \eqref{eqn: J3def} to $\re(s_1)=-\epsilon$. We find residues from the poles at $s_1=0$ and $s_1=-\alpha'-\beta'$ due to the factors $\widetilde{V}(s_1)$ and $\widetilde{V}(-s_1-\alpha'-\beta')$ in \eqref{eqn: Ress2s1a'b'J}, and thus deduce that
\begin{equation*}
\begin{split}
J_3 &=  \frac{1}{2}\sum_{\substack{\alpha\in A \\ \beta\in B}} \sum_{\substack{\alpha'\neq \alpha \\ \beta'\neq \beta}}  \underset{s_1=0}{\text{Res}}\  \underset{s_2=-s_1-\alpha'-\beta'}{\text{Res}}\ \mathcal{J}   +  \frac{1}{2}\sum_{\substack{\alpha\in A \\ \beta\in B}}  \sum_{\substack{\alpha'\neq \alpha \\ \beta'\neq \beta}}  \underset{s_1=-\alpha'-\beta'}{\text{Res}}\  \underset{s_2=-s_1-\alpha'-\beta'}{\text{Res}}\ \mathcal{J}\\
&\hspace{.25in}+ \sum_{\substack{\alpha\in A \\ \beta\in B}} \sum_{\substack{\alpha'\neq \alpha \\ \beta'\neq \beta}} \frac{1}{4\pi i } \int_{(-\epsilon)} \underset{s_2=-s_1-\alpha'-\beta'}{\text{Res}}\ \mathcal{J} \,ds_1.
\end{split}
\end{equation*}
For brevity, write this as
\begin{equation}\label{eqn: J3split}
J_{3} =J_{31} +J_{32} +J_{33}.
\end{equation}
We see from \eqref{eqn: mellinVresidue} and \eqref{eqn: Ress2s1a'b'J} that
\begin{equation}\label{eqn: J31evaluated}
\begin{split}
J_{31} &= \frac{1}{2}\sum_{\substack{\alpha\in A \\ \beta\in B}} \sum_{\substack{\alpha'\neq \alpha \\ \beta'\neq \beta}} X^{-\alpha'-\beta'} Q^{2-\alpha-\beta+\alpha'+\beta'} \widetilde{V}(-\alpha'-\beta')\widetilde{W}(2-\alpha-\beta+\alpha'+\beta') \\
&\hspace{.25in}\times \mathscr{X} (\tfrac{1}{2}+\alpha ) \mathscr{X} (\tfrac{1}{2}+\beta-\alpha'-\beta' ) \zeta(1-\alpha-\beta+\alpha'+\beta') \\
&\hspace{.25in}\times\prod_{ \substack{\hat{\alpha} \neq \alpha \\ \hat{\beta}\neq \beta \\ (\hat{\alpha},\hat{\beta})\neq (\alpha',\beta')}} \zeta(1+\hat{\alpha}  +\hat{\beta} -\alpha'-\beta') \prod_{ \hat{\alpha} \neq \alpha} \zeta(1+\hat{\alpha}-\alpha) \prod_{\hat{\beta}\neq \beta} \zeta(1+\hat{\beta}-\beta)  \\
&\hspace{.25in}\times  \mathcal{K}(0,-\alpha'-\beta',2-\alpha-\beta+\alpha'+\beta').
\end{split}
\end{equation}
Similarly, since the residue of $\widetilde{V}(-s_1-\alpha'-\beta')$ at $s_1=-\alpha'-\beta'$ is $-1$ by \eqref{eqn: mellinVresidue}, we deduce from \eqref{eqn: Ress2s1a'b'J} that the term $J_{32}$ in \eqref{eqn: J3split} equals
\begin{align}
J_{32} &= -\frac{1}{2}\sum_{\substack{\alpha\in A \\ \beta\in B}} \sum_{\substack{\alpha'\neq \alpha \\ \beta'\neq \beta}} X^{-\alpha'-\beta'} Q^{2-\alpha-\beta+\alpha'+\beta'} \widetilde{V}(-\alpha'-\beta')\widetilde{W}(2-\alpha-\beta+\alpha'+\beta') \notag\\
&\hspace{.25in}\times \mathscr{X} (\tfrac{1}{2}+\alpha -\alpha'-\beta' ) \mathscr{X} (\tfrac{1}{2}+\beta ) \zeta(1-\alpha-\beta+\alpha'+\beta') \notag\\
&\hspace{.25in}\times\prod_{ \substack{\hat{\alpha} \neq \alpha \\ \hat{\beta}\neq \beta \\ (\hat{\alpha},\hat{\beta})\neq (\alpha',\beta')}} \zeta(1+\hat{\alpha}  +\hat{\beta} -\alpha'-\beta') \prod_{ \hat{\alpha} \neq \alpha} \zeta(1+\hat{\alpha}-\alpha) \prod_{\hat{\beta}\neq \beta} \zeta(1+\hat{\beta}-\beta) \notag \\
&\hspace{.25in}\times  \mathcal{K}(-\alpha'-\beta',0,2-\alpha-\beta+\alpha'+\beta'). \label{eqn: J32evaluated}
\end{align}
To simplify the integral $J_{33}$ in \eqref{eqn: J3split}, we apply Lemma~\ref{lem: Kfunctionaleqn} to deduce that
\begin{equation*}
\mathcal{K}(s_1,-s_1-\alpha'-\beta',2-\alpha-\beta+\alpha'+\beta') = \left(\frac{h}{k}\right)^{s_1} \mathcal{K}(0,-\alpha'-\beta',2-\alpha-\beta+\alpha'+\beta').
\end{equation*}
It follows from this and \eqref{eqn: Ress2s1a'b'J} that the integral $J_{33}$ in \eqref{eqn: J3split} equals
\begin{equation}\label{eqn: J33evaluated}
\begin{split}
J_{33} &= \sum_{\substack{\alpha\in A \\ \beta\in B}} \sum_{\substack{\alpha'\neq \alpha \\ \beta'\neq \beta}} \frac{1}{4\pi i } \int_{(-\epsilon)} 
X^{-\alpha'-\beta'} Q^{2-\alpha-\beta+\alpha'+\beta'} \widetilde{V}(s_1) \widetilde{V}(-s_1-\alpha'-\beta')\widetilde{W}(2-\alpha-\beta+\alpha'+\beta') \\
&\hspace{.25in}\times \mathscr{X} (\tfrac{1}{2}+\alpha +s_1 ) \mathscr{X} (\tfrac{1}{2}+\beta-s_1-\alpha'-\beta' ) \zeta(1-\alpha-\beta+\alpha'+\beta') \\
&\hspace{.25in}\times\prod_{ \substack{\hat{\alpha} \neq \alpha \\ \hat{\beta}\neq \beta \\ (\hat{\alpha},\hat{\beta})\neq (\alpha',\beta')}} \zeta(1+\hat{\alpha}  +\hat{\beta} -\alpha'-\beta') \prod_{ \hat{\alpha} \neq \alpha} \zeta(1+\hat{\alpha}-\alpha) \prod_{\hat{\beta}\neq \beta} \zeta(1+\hat{\beta}-\beta)  \\
&\hspace{.25in}\times \left( \frac{h}{k}\right)^{s_1} \mathcal{K}(0,-\alpha'-\beta',2-\alpha-\beta+\alpha'+\beta') \,ds_1.
\end{split}
\end{equation}
This, \eqref{eqn: J3split}, \eqref{eqn: J31evaluated}, and \eqref{eqn: J32evaluated} complete our evaluation of $J_3$.

Putting together our calculations, we deduce from \eqref{eqn: 1swapterms2}, \eqref{eqn: J1split}, \eqref{eqn: J2split}, and \eqref{eqn: J3split} that
\begin{equation*}
\begin{split}
\mathcal{I}_1^*(h,k) 
& = J_{11}+J_{12}+J_{13} + J_{21}+J_{22}+J_{23} + J_{31}+J_{32}+J_{33} \\
& \hspace{.25in}+ O\big( X^{-\frac{1}{2}+\varepsilon} Q^{\frac{5}{2}} (hk)^{\varepsilon}\big) + O\big( X^{\varepsilon} Q^{\frac{3}{2}+\varepsilon} (hk)^{\varepsilon}\big).
\end{split}
\end{equation*}
The terms $J_{12}$ and $J_{22}$ cancel each other by \eqref{eqn: J12evaluated} and \eqref{eqn: J22evaluated}, while $J_{13}$ cancels with $J_{32}$ by \eqref{eqn: J13evaluated} and \eqref{eqn: J32evaluated}. Therefore
\begin{equation}\label{eqn: 1swapsready}
\mathcal{I}_1^*(h,k) = J_{11}   + J_{21} +J_{23} + J_{31} +J_{33} + O\big( X^{-\frac{1}{2}+\varepsilon} Q^{\frac{5}{2}} (hk)^{\varepsilon}\big) + O\big( X^{\varepsilon} Q^{\frac{3}{2}+\varepsilon} (hk)^{\varepsilon}\big),
\end{equation}
and we have evaluated $J_{11}$ in \eqref{eqn: J11evaluated}, $J_{21}$ in \eqref{eqn: J21evaluated}, $J_{23}$ in \eqref{eqn: J23evaluated}, $J_{31}$ in \eqref{eqn: J31evaluated}, and $J_{33}$ in \eqref{eqn: J33evaluated}. Our goal for the rest of this section is to show that $\mathcal{U}^2(h,k)$ is equal to the right-hand side of \eqref{eqn: 1swapsready} up to an admissible error term.

\subsection{Analysis of \texorpdfstring{$\mathcal{U}^2(h,k)$}{U2(h,k)}}\label{subsec: one-swap2}

We now continue our analysis of $\mathcal{U}^2(h,k)$. Our goal for this subsection and the next is to show that $\mathcal{U}^2(h,k)$ is equal to the right-hand side of \eqref{eqn: 1swapsready} up to an admissible error term. We multiply the integrand in \eqref{eqn: U2ready} by
\begin{equation}\label{eqn: U2factoroutzetas}
\prod_{\alpha\in A} \zeta(-\tfrac{1}{2}+\alpha+s_1+w-z)\prod_{\beta\in B} \zeta(\tfrac{1}{2}+\beta+s_2+z)
\end{equation}
and divide it by the Euler product of \eqref{eqn: U2factoroutzetas}. The result is
\begin{equation}\label{eqn: U2analysis0}
\begin{split}
 \mathcal{U}^2(h,k) &=  \sum_{\substack{1\leq c \leq C  \\ (c ,hk)=1 }}     \frac{\mu(c)}{2(2\pi i)^4 } \int_{(2)} \int_{(2)} \int_{(1+\epsilon)}\int_{(\epsilon/2)} X^{s_1+s_2}Q^{w} c^{ -w} \\
&\hspace{.25in} \times \widetilde{V}(s_1)\widetilde{V}(s_2) \widetilde{W}(w)  \mathcal{H}(z,w-1) \frac{e^{\delta z} - e^{-\delta z}}{2\delta z}\\
&\hspace{.25in}\times \zeta(2-w)  \prod_{\alpha\in A} \zeta(-\tfrac{1}{2}+\alpha+s_1+w-z)\prod_{\beta\in B} \zeta(\tfrac{1}{2}+\beta+s_2+z)\\
&\hspace{.25in}\times h^{1-w+z}k^{-z} \mathcal{P}(s_1,s_2,w,z) \,dz  \,dw \,ds_2\,ds_1\\
&\hspace{.25in}+ O\left(X^{\varepsilon}Q^{1+\varepsilon}\frac{(hk)^\varepsilon(h,k)}{(hk)^{1/2}} + (XChk)^{\varepsilon} k X^2 Q^{-97}\right),
\end{split}
\end{equation}
where $\mathcal{P}(s_1,s_2,w,z)$ is defined by
\begin{equation}\label{eqn: U2eulerPdef}
\begin{split}
&\mathcal{P}(s_1,s_2,w,z) = \mathcal{P}(s_1,s_2,w,z;A,B,h,k,c) \\
&\hspace{.25in}: = \prod_{p|c} \Bigg\{ \left( 1+\frac{p^{w-1}-1}{p(p-1)} \right) \prod_{\alpha\in A}\left(1-\frac{1}{p^{-\frac{1}{2}+\alpha+s_1+w-z}} \right) \prod_{\beta\in B} \left(1-\frac{1}{p^{\frac{1}{2}+\beta+s_2+z}} \right) \Bigg\}\\
&\hspace{.25in}\times \prod_{\substack{p\nmid c\\ p|hk}} \Bigg\{ \prod_{\alpha\in A}\left(1-\frac{1}{p^{-\frac{1}{2}+\alpha+s_1+w-z}} \right) \prod_{\beta\in B} \left(1-\frac{1}{p^{\frac{1}{2}+\beta+s_2+z}} \right) \\
&\hspace{.25in}\times \Bigg( \sum_{\substack{0\leq m,n<\infty \\ m+\ordp(h)= n + \ordp(k) } }\frac{\tau_A(p^{m}) \tau_B(p^{n}) \left( 1+\frac{p^w}{p^{2}(p-1)}-\frac{1}{p-1} \right)}{ p^{m(-\frac{1}{2}+s_1+w-z)} p^{n(\frac{1}{2}+s_2+z)} p^{(1-w) \min\{m+\ordp(h), n + \ordp(k) \} }}   \\
&\hspace{.25in}+ \sum_{\substack{0\leq m,n<\infty \\ m+\ordp(h)\neq  n + \ordp(k) } } \frac{\tau_A(p^{m}) \tau_B(p^{n}) \left( 1-\frac{p^w}{p^{2}}  \right)}{ p^{m(-\frac{1}{2}+s_1+w-z)} p^{n(\frac{1}{2}+s_2+z)} p^{(1-w) \min\{m+\ordp(h), n + \ordp(k) \} }} \Bigg)\Bigg\} \\
&\hspace{.25in}\times \prod_{\substack{p\nmid c\\ p\nmid hk}} \Bigg\{ \prod_{\alpha\in A}\left(1-\frac{1}{p^{-\frac{1}{2}+\alpha+s_1+w-z}} \right) \prod_{\beta\in B} \left(1-\frac{1}{p^{\frac{1}{2}+\beta+s_2+z}} \right)  \\ 
&\hspace{.25in}\times \Bigg(  1+ \frac{p^{w-1}-1}{p(p-1)}   + \sum_{m=1}^{\infty} \frac{\tau_A(p^{m}) \tau_B(p^{m}) \left( 1+\frac{p^w}{p^{2}(p-1)}-\frac{1}{p-1} \right)}{ p^{m(1+s_1+s_2  )} }   \\
&\hspace{.25in}+ \sum_{ 0\leq m<n<\infty  } \frac{\tau_A(p^{m}) \tau_B(p^{n}) \left( 1-\frac{p^w}{p^{2}}  \right)}{ p^{m(\frac{1}{2}+s_1 -z)} p^{n(\frac{1}{2}+s_2+z)} } + \sum_{ 0\leq n<m <\infty   } \frac{\tau_A(p^{m}) \tau_B(p^{n}) \left( 1-\frac{p^w}{p^{2}}  \right)}{ p^{m(-\frac{1}{2}+s_1+w-z)} p^{n(\frac{3}{2}+s_2-w+z)}  } \Bigg)\Bigg\}.
\end{split}
\end{equation}
As in our analysis of $\mathcal{I}_1^*(h,k)$ in the previous subsection, this ``factoring out'' of the zeta-functions \eqref{eqn: U2factoroutzetas} gives us the analytic continuation of the integrand and allows us to evaluate its residues when shifting contours.

To facilitate our estimations, we first prove the following lemma, which will allow us to move some lines of integration and bound integrals that remain after applying the residue theorem.
\begin{lemma}\label{lem: U2eulerbound}
Suppose that $\epsilon>0$ is arbitrarily small. Let $h,k,c$ be positive integers with $(c,hk)=1$. If $s_1,s_2,w$ are complex numbers such that
\begin{enumerate}
\item[\upshape{(i)}] $\re(w)\leq 2-\epsilon$,
\item[\upshape{(ii)}] $\re(-\frac{1}{2}+s_1+w-z) \geq \frac{1}{2} +\epsilon$,
\item[\upshape{(iii)}] $\re(\frac{1}{2}+s_2+z) \geq \frac{1}{2}+\epsilon $
\item[\upshape{(iv)}] $\re(1+s_1+s_2)\geq 1+\epsilon$
\item[\upshape{(v)}] $\re(\frac{1}{2}+s_1-z)\geq \epsilon$, and
\item[\upshape{(vi)}] $\re(\frac{3}{2}+ s_2 - w + z)\geq \epsilon$.
\end{enumerate}
then the product \eqref{eqn: U2eulerPdef} defining $\mathcal{P}(s_1,s_2,w,z;A,B,h,k,c)$ converges absolutely and we have
$$
\mathcal{P}(s_1,s_2,w,z;A,B,h,k,c) \ll_{\varepsilon} (chk)^{\varepsilon} (h,k)^{\re(w)-1}.
$$
\end{lemma}
\begin{proof}
For brevity, in this proof we will refer to the conditions in the hypothesis by their respective labels (i), (ii), $\dots$, (vi). We will also repeatedly apply without mention the bounds $\tau_A(m)\ll m^{\varepsilon}$ and $\tau_B(n)\ll n^{\varepsilon}$, which follow from \eqref{eqn: divisorbound} and our assumption that $\alpha, \beta \ll 1/\log Q$ for all $\alpha \in A$ and $\beta \in B$. The condition (i) implies
\begin{equation}\label{eqn: U2eulerbound1}
1+\frac{p^{w-1}-1}{p(p-1)} = 1 + O\bigg( \frac{1}{p^{1+\varepsilon}}\bigg).
\end{equation}
From (ii), (iii), and our assumption that $\alpha, \beta \ll 1/\log Q$ for all $\alpha \in A$ and $\beta \in B$, we see that
$$
\frac{1}{p^{-\frac{1}{2}+\alpha+s_1+w-z}} \ll \frac{1}{p^{\frac{1}{2}+\varepsilon}}
$$
for all $\alpha\in A$ and
$$
\frac{1}{p^{\frac{1}{2}+\beta+s_2+z}} \ll \frac{1}{p^{\frac{1}{2}+\varepsilon}}
$$
for all $\beta\in B$. Thus, multiplying out the product and applying the definition \eqref{eqn: taudef} gives
\begin{equation}\label{eqn: U2eulerbound2}
\begin{split}
& \prod_{\alpha\in A}\left(1-\frac{1}{p^{-\frac{1}{2}+\alpha+s_1+w-z}} \right) \prod_{\beta\in B} \left(1-\frac{1}{p^{\frac{1}{2}+\beta+s_2+z}} \right) \\
& = 1- \sum_{\alpha\in A } \frac{1}{p^{-\frac{1}{2}+\alpha+s_1+w-z}} -\sum_{\beta\in B} \frac{1}{p^{\frac{1}{2}+\beta+s_2+z}} + O\bigg( \frac{1}{p^{1+\varepsilon}}\bigg) \\
& = 1 - \frac{\tau_A(p)}{p^{-\frac{1}{2}+s_1+w-z}} - \frac{\tau_B(p)}{p^{\frac{1}{2}+s_2+z}} + O\bigg( \frac{1}{p^{1+\varepsilon}}\bigg).
\end{split}
\end{equation}
By (i) and (iv), we have
\begin{equation}\label{eqn: U2eulerbound6}
\sum_{m=1}^{\infty} \frac{\tau_A(p^{m}) \tau_B(p^{m}) \left( 1+\frac{p^w}{p^{2}(p-1)}-\frac{1}{p-1} \right)}{ p^{m(1+s_1+s_2  )} } \ll \sum_{m=1}^{\infty} \frac{1}{p^{m(1+\varepsilon)}} \ll \frac{1}{p^{1+\varepsilon}}.
\end{equation}
Next, to estimate the sum
\begin{equation*}
\sum_{ 0\leq m<n<\infty  } \frac{\tau_A(p^{m}) \tau_B(p^{n}) \left( 1-\frac{p^w}{p^{2}}  \right)}{ p^{m(\frac{1}{2}+s_1 -z)} p^{n(\frac{1}{2}+s_2+z)} },
\end{equation*}
we separate it into three parts: the term with $m=0$ and $n=1$, the sum of the terms with $m=0$ and $n\geq 2$, and the sum of the terms with $m\geq 1$. The part with $m=0$ and $n\geq 2$ is at most $O(p^{-1-\varepsilon})$ by (i) and (iii). To bound the part with $m\geq 1$, we evaluate the $n$-sum first and use (i), (iii), and (iv) to write
\begin{equation*}
\begin{split}
\sum_{ 1\leq m<n<\infty  } \frac{\tau_A(p^{m}) \tau_B(p^{n}) \left( 1-\frac{p^w}{p^{2}}  \right)}{ p^{m(\frac{1}{2}+s_1 -z)} p^{n(\frac{1}{2}+s_2+z)} }
& \ll \sum_{m=1}^{\infty} \frac{p^{(m+1)\varepsilon}}{p^{m(\re(\frac{1}{2}+s_1 -z))} p^{(m+1)(\re(\frac{1}{2}+s_2+z))} }\\
& = \sum_{m=1}^{\infty} \frac{p^{(m+1)\varepsilon}}{p^{m(\re(1+s_1 +s_2))+\re(\frac{1}{2}+s_2+z)} } \ll \frac{1}{p^{\frac{3}{2}+\varepsilon}}.
\end{split}
\end{equation*}
We thus arrive at
\begin{equation*}
\sum_{ 0\leq m<n<\infty  } \frac{\tau_A(p^{m}) \tau_B(p^{n}) \left( 1-\frac{p^w}{p^{2}}  \right)}{ p^{m(\frac{1}{2}+s_1 -z)} p^{n(\frac{1}{2}+s_2+z)} }  = \frac{ \tau_B(p ) \left( 1-\frac{p^w}{p^{2}}  \right)}{   p^{ \frac{1}{2}+s_2+z } } +O\bigg( \frac{1}{p^{1+\varepsilon}}\bigg).
\end{equation*}
It follows from this and (vi) that
\begin{equation}\label{eqn: U2eulerbound7}
\sum_{ 0\leq m<n<\infty  } \frac{\tau_A(p^{m}) \tau_B(p^{n}) \left( 1-\frac{p^w}{p^{2}}  \right)}{ p^{m(\frac{1}{2}+s_1 -z)} p^{n(\frac{1}{2}+s_2+z)} } = \frac{ \tau_B(p )  }{   p^{ \frac{1}{2}+s_2+z } } +O\bigg( \frac{1}{p^{1+\varepsilon}}\bigg).
\end{equation}
A similar argument using (i), (ii), (iv), and (v) leads to
\begin{equation}\label{eqn: U2eulerbound8}
\sum_{ 0\leq n<m <\infty   } \frac{\tau_A(p^{m}) \tau_B(p^{n}) \left( 1-\frac{p^w}{p^{2}}  \right)}{ p^{m(-\frac{1}{2}+s_1+w-z)} p^{n(\frac{3}{2}+s_2-w+z)}  } = \frac{\tau_A(p)}{p^{ -\frac{1}{2}+s_1+w-z }} + O\bigg( \frac{1}{p^{1+\varepsilon}}\bigg).
\end{equation}
We next bound the sum
\begin{equation}\label{eqn: U2eulerboundfinite1}
\sum_{\substack{0\leq m,n<\infty \\ m+\ordp(h)= n + \ordp(k) } }\frac{\tau_A(p^{m}) \tau_B(p^{n}) \left( 1+\frac{p^w}{p^{2}(p-1)}-\frac{1}{p-1} \right)}{ p^{m(-\frac{1}{2}+s_1+w-z)} p^{n(\frac{1}{2}+s_2+z)} p^{(1-w) \min\{m+\ordp(h), n + \ordp(k) \} }}.
\end{equation}
For brevity, we denote $h_p:=\ordp(h)$ and $k_p:=\ordp(k)$ for the rest of this proof. We make the change of variable $m\mapsto \nu +k_p$ in \eqref{eqn: U2eulerboundfinite1}, so that $n=\nu+h_p$, to write \eqref{eqn: U2eulerboundfinite1} as
\begin{equation*}
\frac{1}{p^{k_p(\frac{1}{2}+s_1-z)} p^{h_p(\frac{3}{2}+s_2-w+z)} } \sum_{\nu=-\min\{h_p,k_p\} }^{\infty} \frac{\tau_A(p^{\nu+k_p}) \tau_B(p^{\nu+h_p}) \left( 1+\frac{p^w}{p^{2}(p-1)}-\frac{1}{p-1} \right)}{ p^{\nu(1+s_1+s_2)} }.
\end{equation*}
Hence, by (i) and (iv), we see that $\eqref{eqn: U2eulerboundfinite1}$ is at most
\begin{equation*}
\ll \frac{p^{\varepsilon h_p +\varepsilon k_p + \min\{h_p,k_p\} \re(1+s_1+s_2)} }{p^{k_p(\re(\frac{1}{2}+s_1-z))} p^{h_p(\re(\frac{3}{2}+s_2-w+z))} }.
\end{equation*}
The denominator of this bound is $\geq p^{\min\{h_p,k_p\} \re(\frac{1}{2}+s_1-z )} p^{\min\{h_p,k_p\} \re(\frac{3}{2}+s_2-w+z )}  $ by (v) and (vi). It follows that
\begin{equation}\label{eqn: U2eulerbound3}
\begin{split}
\sum_{\substack{0\leq m,n<\infty \\ m+\ordp(h)= n + \ordp(k) } }&\frac{\tau_A(p^{m}) \tau_B(p^{n}) \left( 1+\frac{p^w}{p^{2}(p-1)}-\frac{1}{p-1} \right)}{ p^{m(-\frac{1}{2}+s_1+w-z)} p^{n(\frac{1}{2}+s_2+z)} p^{(1-w) \min\{m+\ordp(h), n + \ordp(k) \} }}\\
&\hspace{.5in}\ll p^{\varepsilon h_p +\varepsilon k_p + \min\{h_p,k_p\} (\re(w)-1)}.
\end{split}
\end{equation}
Next, to bound the sum
\begin{equation*}
\sum_{\substack{0\leq m,n<\infty \\ m+\ordp(h)<  n + \ordp(k) } } \frac{\tau_A(p^{m}) \tau_B(p^{n}) \left( 1-\frac{p^w}{p^{2}}  \right)}{ p^{m(-\frac{1}{2}+s_1+w-z)} p^{n(\frac{1}{2}+s_2+z)} p^{(1-w) \min\{m+\ordp(h), n + \ordp(k) \} }},
\end{equation*}
we split it into the part with $m<k_p-h_p$ and the part with $m\geq k_p-h_p$ to deduce that
\begin{equation}\label{eqn: U2eulerboundfinite2}
\sum_{\substack{0\leq m,n<\infty \\ m+\ordp(h)<  n + \ordp(k) } } \frac{\tau_A(p^{m}) \tau_B(p^{n}) \left( 1-\frac{p^w}{p^{2}}  \right)}{ p^{m(-\frac{1}{2}+s_1+w-z)} p^{n(\frac{1}{2}+s_2+z)} p^{(1-w) \min\{m+\ordp(h), n + \ordp(k) \} }} =\Sigma_1 + \Sigma_2,
\end{equation}
where
\begin{equation}\label{eqn: U2eulerboundSigma1def}
\Sigma_1 : = p^{(w-1)  h_p  } \sum_{m=0 }^{k_p-h_p-1}\sum_{n=0}^{\infty} \frac{\tau_A(p^{m}) \tau_B(p^{n}) \left( 1-\frac{p^w}{p^{2}}  \right)}{ p^{m( \frac{1}{2}+s_1 -z)} p^{n(\frac{1}{2}+s_2+z)} }
\end{equation}
and
\begin{equation}\label{eqn: U2eulerboundSigma2def}
\Sigma_2 : = p^{(w-1)  h_p  } \sum_{m=  \max\{0,k_p-h_p\} }^{\infty}\sum_{n=m+h_p-k_p+1}^{\infty} \frac{\tau_A(p^{m}) \tau_B(p^{n}) \left( 1-\frac{p^w}{p^{2}}  \right)}{ p^{m( \frac{1}{2}+s_1 -z)} p^{n(\frac{1}{2}+s_2+z)} }.
\end{equation}
We use (i) to bound $p^w/p^2$ and apply (iii) to estimate the $n$-sums in \eqref{eqn: U2eulerboundSigma1def} and \eqref{eqn: U2eulerboundSigma2def} to see that
\begin{equation}\label{eqn: U2eulerboundSigma1bound1}
\Sigma_1 \ll p^{(\re(w)-1)  h_p  } \sum_{m=0 }^{k_p-h_p-1}  \frac{p^{m\varepsilon}}{ p^{m( \re(\frac{1}{2}+s_1 -z))} }
\end{equation}
and
\begin{equation}\label{eqn: U2eulerboundSigma2bound1}
\Sigma_2 \ll p^{(\re(w)-1)  h_p  } \sum_{m=  \max\{0,k_p-h_p\} }^{\infty} \frac{ 1 }{ p^{m( \re(1+s_1 +s_2)-\varepsilon)} p^{( h_p-k_p+1)(\re(\frac{1}{2}+s_2+z)-\varepsilon)} }.
\end{equation}
If $h_p\geq k_p$, then the $m$-sum on the right-hand side of \eqref{eqn: U2eulerboundSigma1bound1} is zero. Otherwise, it is $O(1)$ by (v). In either case, we have
\begin{equation}\label{eqn: U2eulerboundSigma1bound2}
\Sigma_1 \ll p^{(\re(w)-1)  \min\{h_p,k_p\}  }.
\end{equation}
If $h_p\geq k_p$, then the $m$-sum in \eqref{eqn: U2eulerboundSigma2bound1} starts at $m=0$ and thus (iii), (iv), and (vi) imply
\begin{equation*}
\Sigma_2 \ll \frac{ p^{(\re(w)-1)  h_p  }  }{ p^{( h_p-k_p+1)(\re(\frac{1}{2}+s_2+z)-\varepsilon)}  }= \frac{ p^{(\re(w)-1)  k_p  }    }{ p^{( h_p-k_p )(\re(\frac{3}{2}+s_2-w+z)-\varepsilon)} p^{ \re(\frac{1}{2}+s_2+z)-\varepsilon } } \ll \frac{ p^{(\re(w)-1)  k_p  } }{p^{\frac{1}{2}+\varepsilon}}
\end{equation*}
On the other hand, if $h_p< k_p$, then the $m$-sum in \eqref{eqn: U2eulerboundSigma2bound1} starts at $m=k_p-h_p$ and it follows from (iii), (iv), and (v) that
\begin{equation*}
\begin{split}
\Sigma_2
& \ll \frac{ p^{(\re(w)-1)  h_p  }  }{ p^{(k_p-h_p) (\re(1+s_1 +s_2)-\varepsilon) } p^{( h_p-k_p+1)(\re(\frac{1}{2}+s_2+z)-\varepsilon)}  } \\
& \leq  \frac{ p^{(\re(w)-1)  h_p  } p^{( k_p-h_p+1)\varepsilon}  }{ p^{(k_p-h_p) (\re(\frac{1}{2}+s_1 -z)) }  p^{ \re(\frac{1}{2}+s_2+z) }  } \leq \frac{p^{(\re(w)-1)  h_p  } p^{( k_p-h_p+1)\varepsilon}. }{p^{\frac{1}{2}+\varepsilon}}
\end{split}
\end{equation*}
In either case, we have
\begin{equation*}
\Sigma_2 \ll p^{(\re(w)-1)  \min\{h_p,k_p\}  + \varepsilon h_p +   \varepsilon k_p + \varepsilon -\frac{1}{2} } . 
\end{equation*}
From this, \eqref{eqn: U2eulerboundSigma1bound2}, and \eqref{eqn: U2eulerboundfinite2}, we arrive at
\begin{equation}\label{eqn: U2eulerbound4}
\begin{split}
\sum_{\substack{0\leq m,n<\infty \\ m+\ordp(h)<  n + \ordp(k) } } &\frac{\tau_A(p^{m}) \tau_B(p^{n}) \left( 1-\frac{p^w}{p^{2}}  \right)}{ p^{m(-\frac{1}{2}+s_1+w-z)} p^{n(\frac{1}{2}+s_2+z)} p^{(1-w) \min\{m+\ordp(h), n + \ordp(k) \} }}\\
&\ll p^{ \varepsilon h_p +   \varepsilon k_p + \varepsilon + (\re(w)-1)  \min\{h_p,k_p\}  }.
\end{split}
\end{equation}
A similar argument using (i), (ii), (iv), (v), and (vi) gives
\begin{equation}\label{eqn: U2eulerbound5}
\begin{split}
\sum_{\substack{0\leq m,n<\infty \\ m+\ordp(h) >   n + \ordp(k) } } &\frac{\tau_A(p^{m}) \tau_B(p^{n}) \left( 1-\frac{p^w}{p^{2}}  \right)}{ p^{m(-\frac{1}{2}+s_1+w-z)} p^{n(\frac{1}{2}+s_2+z)} p^{(1-w) \min\{m+\ordp(h), n + \ordp(k) \} }}\\
&\hspace{.25in}\ll p^{ \varepsilon h_p +   \varepsilon k_p + \varepsilon + (\re(w)-1)  \min\{h_p,k_p\}  }.
\end{split}
\end{equation}
From \eqref{eqn: U2eulerbound1}, \eqref{eqn: U2eulerbound2}, (ii), and (iii), we see that if $p|c$ then the local factor in \eqref{eqn: U2eulerPdef} corresponding to $p$ is $O(1)$. Moreover, from \eqref{eqn: U2eulerbound2}, (ii), (iii), \eqref{eqn: U2eulerbound3}, \eqref{eqn: U2eulerbound4}, and \eqref{eqn: U2eulerbound5}, we deduce that if $p\nmid c$ and $p|hk$ then the local factor in \eqref{eqn: U2eulerPdef} corresponding to $p$ is
$$
\ll p^{(\re(w)-1)  \min\{ \ordp(h),\ordp(k)\}  + \varepsilon \ordp(h) +   \varepsilon \ordp(k) + \varepsilon }.
$$
Finally, from \eqref{eqn: U2eulerbound1}, \eqref{eqn: U2eulerbound2}, \eqref{eqn: U2eulerbound6}, \eqref{eqn: U2eulerbound7}, and \eqref{eqn: U2eulerbound8}, we see that if $p\nmid chk$ then the local factor in \eqref{eqn: U2eulerPdef} corresponding to $p$ is $1+O(p^{-1-\varepsilon})$. We conclude that the right-hand side of \eqref{eqn: U2eulerPdef} converges absolutely, and is
$$
\ll (chk)^{\varepsilon} (h,k)^{\re(w)-1}
$$
because $c$ and $hk$ are coprime, $(h,k)=\prod_{p|hk} p^{  \min\{ \ordp(h),\ordp(k)\}}$, and $\prod_{p|\nu}O(1) \ll {\nu}^{\varepsilon}$ for any positive integer $\nu$.
\end{proof}

We move the $s_2$-line in \eqref{eqn: U2analysis0} to $\re(s_2)=\epsilon$. This leaves a residue from the pole at $s_2=\frac{1}{2}-\beta-z$ for each $\beta\in B$ because of the factors \eqref{eqn: U2factoroutzetas}. Note that we need to assume the Lindel\"{o}f Hypothesis to maintain the absolute convergence of the $z$-integral, as there is an arbitrary number of zeta-functions that depend on $z$ and $\mathcal{H}(z,w-1)$ only decays slowly by \eqref{eqn: Hbound}. The result is
\begin{equation}\label{eqn: U2split}
\mathcal{U}^2(h,k) = I_1 + I_2 + O\left(X^{\varepsilon}Q^{1+\varepsilon}\frac{(hk)^\varepsilon(h,k)}{(hk)^{1/2}} + (XChk)^{\varepsilon} k X^2 Q^{-97}\right),
\end{equation}
where $I_1$ is the integral of the residues at the poles $s_2=\frac{1}{2}-\beta-z$ and $I_2$ is the new integral with $\re(s_2)=\epsilon$. More precisely,
\begin{equation}\label{eqn: U2I1def}
\begin{split}
I_1 &:= \sum_{\beta\in B} \sum_{\substack{1\leq c \leq C  \\ (c ,hk)=1 }}     \frac{\mu(c)}{2(2\pi i)^3 } \int_{(2)}  \int_{(1+\epsilon)}\int_{(\epsilon/2)} X^{s_1+\frac{1}{2}-\beta-z}Q^{w} c^{ -w} \\
&\hspace{.25in} \times \widetilde{V}(s_1)\widetilde{V}(\tfrac{1}{2}-\beta-z) \widetilde{W}(w)  \mathcal{H}(z,w-1) \frac{e^{\delta z} - e^{-\delta z}}{2\delta z}\\
&\hspace{.25in}\times \zeta(2-w) \prod_{\alpha\in A} \zeta(-\tfrac{1}{2}+\alpha+s_1+w-z)\prod_{\hat{\beta}\neq \beta} \zeta(1+\hat{\beta}-\beta )\\
&\hspace{.25in}\times h^{1-w+z}k^{-z} \mathcal{P}(s_1,\tfrac{1}{2}-\beta-z,w,z) \,dz  \,dw  \,ds_1
\end{split}
\end{equation}
and
\begin{equation}\label{eqn: U2I2def}
\begin{split}
I_2 &:=  \sum_{\substack{1\leq c \leq C  \\ (c ,hk)=1 }}     \frac{\mu(c)}{2(2\pi i)^4 } \int_{(2)} \int_{(\epsilon)} \int_{(1+\epsilon)}\int_{(\epsilon/2)} X^{s_1+s_2}Q^{w} c^{ -w} \\
 &\hspace{.25in}\times \widetilde{V}(s_1)\widetilde{V}(s_2) \widetilde{W}(w)  \mathcal{H}(z,w-1) \frac{e^{\delta z} - e^{-\delta z}}{2\delta z}\\
&\hspace{.25in}\times \zeta(2-w) \prod_{\alpha\in A} \zeta(-\tfrac{1}{2}+\alpha+s_1+w-z)\prod_{\beta\in B} \zeta(\tfrac{1}{2}+\beta+s_2+z)\\
&\hspace{.25in}\times h^{1-w+z}k^{-z} \mathcal{P}(s_1,s_2,w,z) \,dz  \,dw \,ds_2\,ds_1.
\end{split}
\end{equation}

We first bound $I_2$. We move the $s_1$-line in \eqref{eqn: U2I2def} to $\re(s_1)=\epsilon$ to deduce that
\begin{equation}\label{eqn: U2I2split}
I_2 = I_{21} + I_{22},
\end{equation}
where $I_{21}$ is the integral of the residues at the poles $s_1=\frac{3}{2}-\alpha-w+z$, where $\alpha$ runs through the elements of $A$, and $I_{22}$ is the new integral with $\re(s_1)=\epsilon$. In other words,
\begin{equation}\label{eqn: U2I21def}
\begin{split}
I_{21} &:=  \sum_{\alpha\in A} \sum_{\substack{1\leq c \leq C  \\ (c ,hk)=1 }}     \frac{\mu(c)}{2(2\pi i)^3 } \int_{(\epsilon)} \int_{(1+\epsilon)}\int_{(\epsilon/2)} X^{\frac{3}{2}-\alpha-w+z +s_2}Q^{w} c^{ -w} \\
&\hspace{.25in} \times \widetilde{V}(\tfrac{3}{2}-\alpha-w+z)\widetilde{V}(s_2) \widetilde{W}(w)  \mathcal{H}(z,w-1) \frac{e^{\delta z} - e^{-\delta z}}{2\delta z}\\
&\hspace{.25in}\times \zeta(2-w) \prod_{\hat{\alpha}\neq \alpha} \zeta(1+\hat{\alpha}-\alpha)\prod_{\beta\in B} \zeta(\tfrac{1}{2}+\beta+s_2+z)\\
&\hspace{.25in}\times h^{1-w+z}k^{-z} \mathcal{P}(\tfrac{3}{2}-\alpha-w+z,s_2,w,z) \,dz  \,dw \,ds_2
\end{split}
\end{equation}
and
\begin{align}
I_{22} &:=  \sum_{\substack{1\leq c \leq C  \\ (c ,hk)=1 }}     \frac{\mu(c)}{2(2\pi i)^4 } \int_{(\epsilon)} \int_{(\epsilon)} \int_{(1+\epsilon)}\int_{(\epsilon/2)} X^{s_1+s_2}Q^{w} c^{ -w} \notag\\
&\hspace{.25in} \times \widetilde{V}(s_1)\widetilde{V}(s_2) \widetilde{W}(w)  \mathcal{H}(z,w-1) \frac{e^{\delta z} - e^{-\delta z}}{2\delta z} \notag\\
&\hspace{.25in}\times \zeta(2-w) \prod_{\alpha\in A} \zeta(-\tfrac{1}{2}+\alpha+s_1+w-z)\prod_{\beta\in B} \zeta(\tfrac{1}{2}+\beta+s_2+z) \notag\\
&\hspace{.25in}\times h^{1-w+z}k^{-z} \mathcal{P}(s_1,s_2,w,z) \,dz  \,dw \,ds_2\,ds_1.\label{eqn: U2I22def}
\end{align}
Note that we need to assume the Lindel\"{o}f Hypothesis to justify \eqref{eqn: U2I2split} like we did to validate \eqref{eqn: U2split}. To estimate $I_{22}$, we again assume the Lindel\"{o}f Hypothesis in order to bound the arbitrary number of zeta-functions in \eqref{eqn: U2I22def} that depend on the variable $z$. We apply \eqref{eqn: mellinrapiddecay}, \eqref{eqn: Hbound}, and Lemma~\ref{lem: U2eulerbound}, and argue as in  \eqref{eqn: zwbound} to deduce from \eqref{eqn: U2I22def} and the definition \eqref{eqn: deltadef} of $\delta$ that
\begin{equation}\label{eqn: U2I22bound}
I_{22} \ll X^{\varepsilon} Q^{1+\varepsilon} h^{\varepsilon}k^{\varepsilon}.
\end{equation}
Next, to bound $I_{21}$, we move the $w$-line in \eqref{eqn: U2I21def} to $\re(w)= \frac{3}{2}-\epsilon$. We traverse no poles in doing so. We then bound the resulting expression by applying \eqref{eqn: zetaalphaalphabound}, \eqref{eqn: mellinrapiddecay}, \eqref{eqn: Hbound}, and Lemma~\ref{lem: U2eulerbound}. The result is
\begin{equation*}
I_{21} \ll X^{\varepsilon} Q^{\frac{3}{2}} h^{\varepsilon}k^{\varepsilon}.
\end{equation*}
From this, \eqref{eqn: U2I22bound}, and \eqref{eqn: U2I2split}, we arrive at
\begin{equation}\label{eqn: U2I2bound}
I_{2} \ll X^{\varepsilon} Q^{\frac{3}{2}} h^{\varepsilon}k^{\varepsilon}.
\end{equation}

Having bounded $I_2$, we now turn our attention to the integral $I_1$ defined by \eqref{eqn: U2I1def}. We move the $s_1$-line in \eqref{eqn: U2I1def} to $\re(s_1)=\epsilon$. This leaves a residue from the pole at $s_1=\frac{3}{2}-\alpha-w+z$ for each $\alpha\in A$, and leads to
\begin{equation}\label{eqn: U2I1split}
I_1=I_{11}+I_{12},
\end{equation}
where
\begin{equation}\label{eqn: U2I11def}
\begin{split}
I_{11} &:= \sum_{\substack{\alpha\in A \\ \beta\in B}} \sum_{\substack{1\leq c \leq C  \\ (c ,hk)=1 }}     \frac{\mu(c)}{2(2\pi i)^2 }  \int_{(1+\epsilon)}\int_{(\epsilon/2)} X^{2-\alpha-\beta-w}Q^{w} c^{ -w} \\
&\hspace{.25in} \times \widetilde{V}(\tfrac{3}{2}-\alpha-w+z)\widetilde{V}(\tfrac{1}{2}-\beta-z) \widetilde{W}(w)  \mathcal{H}(z,w-1) \frac{e^{\delta z} - e^{-\delta z}}{2\delta z}\\
&\hspace{.25in}\times \zeta(2-w)  \prod_{\hat{\alpha}\neq \alpha} \zeta(1+\hat{\alpha}-\alpha)\prod_{\hat{\beta}\neq \beta} \zeta(1+\hat{\beta}-\beta )\\
&\hspace{.25in}\times h^{1-w+z}k^{-z} \mathcal{P}(\tfrac{3}{2}-\alpha-w+z,\tfrac{1}{2}-\beta-z,w,z) \,dz  \,dw 
\end{split}
\end{equation}
and
\begin{equation}\label{eqn: U2I12def}
\begin{split}
I_{12} &:= \sum_{\beta\in B} \sum_{\substack{1\leq c \leq C  \\ (c ,hk)=1 }}     \frac{\mu(c)}{2(2\pi i)^3 } \int_{(\epsilon)}  \int_{(1+\epsilon)}\int_{(\epsilon/2)} X^{s_1+\frac{1}{2}-\beta-z}Q^{w} c^{ -w} \\
 &\hspace{.25in}\times \widetilde{V}(s_1)\widetilde{V}(\tfrac{1}{2}-\beta-z) \widetilde{W}(w)  \mathcal{H}(z,w-1) \frac{e^{\delta z} - e^{-\delta z}}{2\delta z}\\
&\hspace{.25in}\times \zeta(2-w)  \prod_{\alpha\in A} \zeta(-\tfrac{1}{2}+\alpha+s_1+w-z)\prod_{\hat{\beta}\neq \beta} \zeta(1+\hat{\beta}-\beta )\\
&\hspace{.25in}\times h^{1-w+z}k^{-z} \mathcal{P}(s_1,\tfrac{1}{2}-\beta-z,w,z) \,dz  \,dw  \,ds_1.
\end{split}
\end{equation}

To bound $I_{12}$, we move the $w$-line in \eqref{eqn: U2I12def} to the right by a distance of at most $\epsilon/2$, and then move the $z$-line to the right by a distance of at most $\epsilon/2$. We do this in such a way as to maintain the inequality $1+\frac{\epsilon}{2} \leq \re(w-z) \leq 1+\epsilon$, so as to not traverse any pole of $\mathcal{H}(z,w-1)$. We repeat this process until the $w$-line is at $\re(w)=\frac{3}{2}-\epsilon$ and the $z$-line is at $\frac{1}{2}-\frac{3\epsilon}{2}$. This leaves no residues because we do not cross any poles of the integrand. We then bound the resulting integral by applying \eqref{eqn: zetaalphaalphabound}, \eqref{eqn: mellinrapiddecay}, \eqref{eqn: Hbound}, and Lemma~\ref{lem: U2eulerbound}. We arrive at
\begin{equation}\label{eqn: U2I12bound}
I_{12} \ll X^{\varepsilon} Q^{\frac{3}{2}} h^{\varepsilon}k^{\varepsilon}.
\end{equation}

To estimate the integral $I_{11}$ defined by \eqref{eqn: U2I11def}, our first task is to extend the $c$-sum in \eqref{eqn: U2I11def} to infinity. To do this, we need to bound the sum
\begin{equation}\label{eqn: U2E11def}
\begin{split}
\sum_{\substack{\alpha\in A \\ \beta\in B}} \sum_{\substack{c>C  \\ (c ,hk)=1 }}&     \frac{\mu(c)}{2(2\pi i)^2 }  \int_{(1+\epsilon)}\int_{(\epsilon/2)} X^{2-\alpha-\beta-w}Q^{w} c^{ -w} \\
&\hspace{.25in} \times \widetilde{V}(\tfrac{3}{2}-\alpha-w+z)\widetilde{V}(\tfrac{1}{2}-\beta-z) \widetilde{W}(w)  \mathcal{H}(z,w-1) \frac{e^{\delta z} - e^{-\delta z}}{2\delta z}\\
&\hspace{.25in}\times \zeta(2-w)  \prod_{\hat{\alpha}\neq \alpha} \zeta(1+\hat{\alpha}-\alpha)\prod_{\hat{\beta}\neq \beta} \zeta(1+\hat{\beta}-\beta )\\
&\hspace{.25in}\times h^{1-w+z}k^{-z} \mathcal{P}(\tfrac{3}{2}-\alpha-w+z,\tfrac{1}{2}-\beta-z,w,z) \,dz  \,dw.
\end{split}
\end{equation}
We first move the $w$-line in \eqref{eqn: U2E11def} to $\re(w)=\frac{3}{2}$, crossing no poles. Then, we move the $z$-line to $\re(z)=\frac{1}{2}-\epsilon$, again traversing no poles. Afterward, we further move the $w$-line to $\re(w)=2-2\epsilon$. This does not cross any poles since now $\re(z)=\frac{1}{2}-\epsilon$. We bound the new integral that has $\re(w)=2-2\epsilon$ and $\re(z)=\frac{1}{2}-\epsilon$ using \eqref{eqn: zetaalphaalphabound}, \eqref{eqn: mellinrapiddecay}, \eqref{eqn: Hbound}, and Lemma~\ref{lem: U2eulerbound}, and deduce that \eqref{eqn: U2E11def} is at most
\begin{equation*}
\ll \frac{ (XChk)^{\varepsilon} Q^{2} (h,k) }{C\sqrt{hk}}.
\end{equation*}
It follows from this and \eqref{eqn: U2I11def} that
\begin{equation}\label{eqn: I11toR0}
I_{11} = R_0 + O\bigg(\frac{ (XChk)^{\varepsilon} Q^{2} (h,k) }{C\sqrt{hk}}\bigg) ,
\end{equation}
where $R_0$ is defined by
\begin{equation}\label{eqn: R0def}
\begin{split}
R_0 &:= \sum_{\substack{\alpha\in A \\ \beta\in B}} \sum_{\substack{c\geq 1  \\ (c ,hk)=1 }}     \frac{\mu(c)}{2(2\pi i)^2 }  \int_{(1+\epsilon)}\int_{(\epsilon/2)} X^{2-\alpha-\beta-w}Q^{w} c^{ -w} \\
 &\hspace{.25in}\times \widetilde{V}(\tfrac{3}{2}-\alpha-w+z)\widetilde{V}(\tfrac{1}{2}-\beta-z) \widetilde{W}(w)  \mathcal{H}(z,w-1) \frac{e^{\delta z} - e^{-\delta z}}{2\delta z}\\
&\hspace{.25in}\times \zeta(2-w)  \prod_{\hat{\alpha}\neq \alpha} \zeta(1+\hat{\alpha}-\alpha)\prod_{\hat{\beta}\neq \beta} \zeta(1+\hat{\beta}-\beta )\\
&\hspace{.25in}\times h^{1-w+z}k^{-z} \mathcal{P}(\tfrac{3}{2}-\alpha-w+z,\tfrac{1}{2}-\beta-z,w,z) \,dz  \,dw.
\end{split}
\end{equation}
We may evaluate the sum of $\mu(c)c^{ -w} \mathcal{P}(\tfrac{3}{2}-\alpha-w+z,\tfrac{1}{2}-\beta-z,w,z)$ over all $c\geq 1$ with $(c ,hk)=1$ by using the definition \eqref{eqn: U2eulerPdef} of $\mathcal{P}$ and Lemma~\ref{lem: sumstoEulerproducts}, with absolute convergence ensured by Lemma~\ref{lem: U2eulerbound}. This and \eqref{eqn: R0def} lead to
\begin{equation}\label{eqn: R0factored}
\begin{split}
R_0 &= \sum_{\substack{\alpha\in A \\ \beta\in B}}  \frac{1}{2(2\pi i)^2 }  \int_{(1+\epsilon)}\int_{(\epsilon/2)} X^{2-\alpha-\beta-w}Q^{w} \widetilde{V}(\tfrac{3}{2}-\alpha-w+z)\widetilde{V}(\tfrac{1}{2}-\beta-z) \widetilde{W}(w)  \mathcal{H}(z,w-1)\\
&\hspace{.25in}\times\frac{e^{\delta z} - e^{-\delta z}}{2\delta z}\zeta(2-w)  \prod_{\hat{\alpha}\neq \alpha} \zeta(1+\hat{\alpha}-\alpha)\prod_{\hat{\beta}\neq \beta} \zeta(1+\hat{\beta}-\beta )\\
&\hspace{.25in}\times h^{1-w+z}k^{-z}  \prod_{ p|hk } \Bigg\{  \prod_{\hat{\alpha}\in A}\left(1-\frac{1}{p^{1+\hat{\alpha}-\alpha}} \right) \prod_{\hat{\beta}\in B} \left(1-\frac{1}{p^{1+\hat{\beta}-\beta}} \right) \\
&\hspace{.5in}\times \Bigg( \sum_{\substack{0\leq m,n<\infty \\ m+\ordp(h)= n + \ordp(k) } }\frac{\tau_A(p^{m}) \tau_B(p^{n}) \left( 1+\frac{p^w}{p^{2}(p-1)}-\frac{1}{p-1} \right)}{ p^{m(1-\alpha)} p^{n(1-\beta)} p^{(1-w) \min\{m+\ordp(h), n + \ordp(k) \} }}   \\
&\hspace{.5in}+ \sum_{\substack{0\leq m,n<\infty \\ m+\ordp(h)\neq  n + \ordp(k) } } \frac{\tau_A(p^{m}) \tau_B(p^{n}) \left( 1-\frac{p^w}{p^{2}}  \right)}{ p^{m(1-\alpha)} p^{n(1-\beta)} p^{(1-w) \min\{m+\ordp(h), n + \ordp(k) \} }} \Bigg)\Bigg\} \\
&\hspace{.25in}\times \prod_{  p\nmid hk } \Bigg\{  \prod_{\hat{\alpha}\in A}\left(1-\frac{1}{p^{1+\hat{\alpha}-\alpha}} \right) \prod_{\hat{\beta}\in B} \left(1-\frac{1}{p^{1+\hat{\beta}-\beta}} \right)  \\ 
&\hspace{.5in}\times \Bigg( \left(1-\frac{1}{p^w}\right)\left( 1+ \frac{p^{w-1}-1}{p(p-1)} \right)  + \sum_{m=1}^{\infty} \frac{\tau_A(p^{m}) \tau_B(p^{m}) \left( 1+\frac{p^w}{p^{2}(p-1)}-\frac{1}{p-1} \right)}{ p^{m(3-\alpha-\beta-w  )} }   \\
&\hspace{.5in}+ \sum_{ 0\leq m<n<\infty  } \frac{\tau_A(p^{m}) \tau_B(p^{n}) \left( 1-\frac{p^w}{p^{2}}  \right)}{ p^{m(2-\alpha-w)} p^{n(1-\beta)} } + \sum_{ 0\leq n<m <\infty   } \frac{\tau_A(p^{m}) \tau_B(p^{n}) \left( 1-\frac{p^w}{p^{2}}  \right)}{ p^{m(1-\alpha)} p^{n(2-\beta-w)}  } \Bigg)\Bigg\}\,dz  \,dw .
\end{split}
\end{equation}
From \eqref{eqn: U2I1split}, \eqref{eqn: U2I12bound}, \eqref{eqn: I11toR0}, we deduce that
\begin{equation}\label{eqn: I1toR0}
I_1 = R_0 + O\bigg(\frac{ (XChk)^{\varepsilon} Q^{2} (h,k) }{C\sqrt{hk}}\bigg) + O\Big( X^{\varepsilon} Q^{\frac{3}{2}} h^{\varepsilon}k^{\varepsilon} \Big),
\end{equation}
where $R_0$ is expressed as a finite sum of contour integrals in \eqref{eqn: R0factored}.

To be able to shift the contours and evaluate residues, we analytically continue the integrand in \eqref{eqn: R0factored} by multiplying it by
\begin{equation}\label{eqn: U2I11factoroutzetas}
\prod_{\substack{ \hat{\alpha}\neq \alpha \\ \hat{\beta}\neq \beta}} \zeta( 3+\hat{\alpha}+\hat{\beta} -\alpha-\beta-w)
\end{equation}
and dividing it by the Euler product of \eqref{eqn: U2I11factoroutzetas}. The result is
\begin{align}
R_0 &= \sum_{\substack{\alpha\in A \\ \beta\in B}}  \frac{1}{2(2\pi i)^2 }  \int_{(1+\epsilon)}\int_{(\epsilon/2)} X^{2-\alpha-\beta-w}Q^{w} \notag\\
&\hspace{.25in} \times \widetilde{V}(\tfrac{3}{2}-\alpha-w+z)\widetilde{V}(\tfrac{1}{2}-\beta-z) \widetilde{W}(w)  \mathcal{H}(z,w-1) \frac{e^{\delta z} - e^{-\delta z}}{2\delta z} \notag\\
&\hspace{.25in}\times \zeta(2-w)  \prod_{\substack{ \hat{\alpha}\neq \alpha \\ \hat{\beta}\neq \beta}} \zeta( 3+\hat{\alpha}+\hat{\beta} -\alpha-\beta-w) \prod_{\hat{\alpha}\neq \alpha} \zeta(1+\hat{\alpha}-\alpha)\prod_{\hat{\beta}\neq \beta} \zeta(1+\hat{\beta}-\beta ) \notag\\
&\hspace{.25in}\times h^{1-w+z}k^{-z}  \mathcal{G} (w,\alpha,\beta)\,dz  \,dw, \label{eqn: R0def2}
\end{align}
where $\mathcal{G} (w,\alpha,\beta)$ is defined by
\begin{equation}\label{eqn: R0eulerGdef}
\begin{split}
\mathcal{G}(w
& ,\alpha,\beta) = \mathcal{G}(w,\alpha,\beta; A,B,h,k)\\
&:= \prod_{ p|hk } \Bigg\{ \prod_{\substack{ \hat{\alpha}\neq \alpha \\ \hat{\beta}\neq \beta}}\left(1-\frac{1}{p^{3+\hat{\alpha}+\hat{\beta} -\alpha-\beta-w}} \right) \prod_{\hat{\alpha}\in A}\left(1-\frac{1}{p^{1+\hat{\alpha}-\alpha}} \right) \prod_{\hat{\beta}\in B} \left(1-\frac{1}{p^{1+\hat{\beta}-\beta}} \right) \\
&\hspace{.75in}\times \Bigg( \sum_{\substack{0\leq m,n<\infty \\ m+\ordp(h)= n + \ordp(k) } }\frac{\tau_A(p^{m}) \tau_B(p^{n}) \left( 1+\frac{p^w}{p^{2}(p-1)}-\frac{1}{p-1} \right)}{ p^{m(1-\alpha)} p^{n(1-\beta)} p^{(1-w) \min\{m+\ordp(h), n + \ordp(k) \} }}   \\
&\hspace{.75in}+ \sum_{\substack{0\leq m,n<\infty \\ m+\ordp(h)\neq  n + \ordp(k) } } \frac{\tau_A(p^{m}) \tau_B(p^{n}) \left( 1-\frac{p^w}{p^{2}}  \right)}{ p^{m(1-\alpha)} p^{n(1-\beta)} p^{(1-w) \min\{m+\ordp(h), n + \ordp(k) \} }} \Bigg)\Bigg\} \\
&\hspace{.5in}\times \prod_{ p\nmid hk } \Bigg\{ \prod_{\substack{ \hat{\alpha}\neq \alpha \\ \hat{\beta}\neq \beta}}\left(1-\frac{1}{p^{3+\hat{\alpha}+\hat{\beta} -\alpha-\beta-w}} \right) \prod_{\hat{\alpha}\in A}\left(1-\frac{1}{p^{1+\hat{\alpha}-\alpha}} \right) \prod_{\hat{\beta}\in B} \left(1-\frac{1}{p^{1+\hat{\beta}-\beta}} \right)  \\ 
&\hspace{.75in}\times \Bigg(  \left( 1-\frac{1}{p^w}\right)\left(1+ \frac{p^{w-1}-1}{p(p-1)}\right)   + \sum_{m=1}^{\infty} \frac{\tau_A(p^{m}) \tau_B(p^{m}) \left( 1+\frac{p^w}{p^{2}(p-1)}-\frac{1}{p-1} \right)}{ p^{m(3-\alpha-\beta-w  )} }   \\
&\hspace{.75in}+ \sum_{ 0\leq m<n<\infty  } \frac{\tau_A(p^{m}) \tau_B(p^{n}) \left( 1-\frac{p^w}{p^{2}}  \right)}{ p^{m(2-\alpha-w)} p^{n(1-\beta)} } + \sum_{ 0\leq n<m <\infty   } \frac{\tau_A(p^{m}) \tau_B(p^{n}) \left( 1-\frac{p^w}{p^{2}}  \right)}{ p^{m(1-\alpha)} p^{n(2-\beta-w)}  } \Bigg)\Bigg\}.
\end{split}
\end{equation}

We next prove the following lemma, which we will use to justify moving the lines of integration and bound some of the integrals that remain after applying the residue theorem.
\begin{lemma}\label{lem: R0eulerbound}
Suppose that $\epsilon>0$ is arbitrarily small. Let $\alpha\in A$ and $\beta\in B$, and let $h$ and $k$ be positive integers. If $w$ is a complex number such that
$$
1+\epsilon \leq \re(w) \leq \frac{5}{2}-\epsilon,
$$
then the product \eqref{eqn: R0eulerGdef} defining $\mathcal{G}(w,\alpha,\beta; A,B,h,k)$ converges absolutely and we have
$$
\mathcal{G}(w,\alpha,\beta; A,B,h,k) \ll_{\varepsilon} h^{\frac{1}{2}+\varepsilon} k^{\frac{1}{2}+\varepsilon} (h,k)^{\frac{1}{2}+\varepsilon}.
$$
\end{lemma}
\begin{proof}
In this proof, we will repeatedly apply without mention the bounds $\tau_A(m)\ll m^{\varepsilon}$, and $\tau_B(n)\ll n^{\varepsilon}$, which follow from \eqref{eqn: divisorbound} and the assumption that $\alpha,\beta \ll 1/\log Q$ for all $\alpha \in A$ and $\beta \in B$. Since $\re(w)\leq \frac{5}{2}-\epsilon$, we have
$$
\frac{1}{p^{3+\hat{\alpha}+\hat{\beta} -\alpha-\beta-w}} \ll \frac{1}{p^{\frac{1}{2}+\varepsilon}}
$$
for all $\hat{\alpha}\in A$ and $\hat{\beta}\in B$. Also, it holds that
$$
\frac{1}{p^{1+\hat{\alpha}-\alpha}} \ll \frac{1}{p^{1-\varepsilon}}
$$
and
$$
\frac{1}{p^{1+\hat{\beta}-\beta}} \ll \frac{1}{p^{1-\varepsilon}}
$$
for all $\hat{\alpha}\in A$ and $\hat{\beta}\in B$. Hence, multiplying out the product and applying the definition \eqref{eqn: taudef} gives
\begin{equation}\label{eqn: R0eulerbound1}
\begin{split}
\prod_{\substack{ \hat{\alpha}\neq \alpha \\ \hat{\beta}\neq \beta}}&\left(1-\frac{1}{p^{3+\hat{\alpha}+\hat{\beta} -\alpha-\beta-w}} \right) \prod_{\hat{\alpha}\in A}\left(1-\frac{1}{p^{1+\hat{\alpha}-\alpha}} \right) \prod_{\hat{\beta}\in B} \left(1-\frac{1}{p^{1+\hat{\beta}-\beta}} \right) \\
&= 1 -\sum_{\substack{ \hat{\alpha}\neq \alpha \\ \hat{\beta}\neq \beta}} \frac{1}{p^{3+\hat{\alpha}+\hat{\beta} -\alpha-\beta-w}} - \sum_{\hat{\alpha}\in A} \frac{1}{p^{1+\hat{\alpha}-\alpha}}  - \sum_{\hat{\beta}\in B} \frac{1}{p^{1+\hat{\beta}-\beta}} + O\left( \frac{1}{p^{1+\varepsilon}}\right) \\
&= 1 - \frac{\tau_A(p)\tau_B(p)}{p^{3-\alpha-\beta-w}} + \frac{\tau_A(p)}{p^{3-\alpha-w}} + \frac{\tau_B(p)}{p^{3-\beta-w}} - \frac{1}{p^{3-w}} - \frac{ \tau_A(p)}{p^{1-\alpha}} - \frac{ \tau_B(p)}{p^{1-\beta}} + O\left( \frac{1}{p^{1+\varepsilon}}\right).
\end{split}
\end{equation}
Since $1+\epsilon \leq \re(w)\leq \frac{5}{2}-\epsilon$, we have
\begin{equation}\label{eqn: R0eulerbound2}
\left( 1-\frac{1}{p^w}\right)\left(1+ \frac{p^{w-1}-1}{p(p-1)}\right) = 1+\frac{1}{p^{3-w}} + O\left( \frac{1}{p^{1+\varepsilon}}\right).
\end{equation}
The assumption $\re(w)\leq \frac{5}{2}-\epsilon$ also implies $\re(3-\alpha-\beta-w)\geq \frac{1}{2}+\varepsilon$ and thus
\begin{equation}\label{eqn: R0eulerbound3}
\sum_{m=1}^{\infty} \frac{\tau_A(p^{m}) \tau_B(p^{m}) \left( 1+\frac{p^w}{p^{2}(p-1)}-\frac{1}{p-1} \right)}{ p^{m(3-\alpha-\beta-w  )} }  = \frac{\tau_A(p)\tau_B(p)}{p^{3-\alpha-\beta-w}} + O\left( \frac{1}{p^{1+\varepsilon}}\right).
\end{equation}
Next, since $p^w/p^2 \ll p^{\frac{1}{2}-\epsilon}$, the terms with $m\geq 1$ in the sum
$$
\sum_{ 0\leq m<n<\infty  } \frac{\tau_A(p^{m}) \tau_B(p^{n}) \left( 1-\frac{p^w}{p^{2}}  \right)}{ p^{m(2-\alpha-w)} p^{n(1-\beta)} }
$$
add up to at most
\begin{equation*}
\ll \sum_{ 1\leq m<n<\infty  } \frac{p^{\frac{1}{2}-\epsilon}}{p^{m(-\frac{1}{2}-\varepsilon+\epsilon)} p^{n(1-\varepsilon)} } \ll \sum_{m=1}^{\infty} \frac{p^{\frac{1}{2}-\epsilon} }{p^{m(-\frac{1}{2}-\varepsilon+\epsilon)}p^{(m+1)(1-\varepsilon)}} \ll \frac{1}{p^{1+\varepsilon}},
\end{equation*}
while the terms with $m=0$ and $n\geq 2$ add up to at most
\begin{equation*}
\ll \sum_{ n=2 }^{\infty} \frac{p^{\frac{1}{2}-\epsilon}}{p^{n(1-\varepsilon)} } \ll \frac{1}{p^{\frac{3}{2}+\varepsilon}}.
\end{equation*}
Hence
\begin{equation}\label{eqn: R0eulerbound4}
\sum_{ 0\leq m<n<\infty  } \frac{\tau_A(p^{m}) \tau_B(p^{n}) \left( 1-\frac{p^w}{p^{2}}  \right)}{ p^{m(2-\alpha-w)} p^{n(1-\beta)} } = \frac{  \tau_B(p )  }{   p^{ 1-\beta } } - \frac{  \tau_B(p ) }{   p^{ 3-\beta-w } } + O\left( \frac{1}{p^{1+\varepsilon}}\right).
\end{equation}
Similarly, or by symmetry, we have
\begin{equation}\label{eqn: R0eulerbound5}
\sum_{ 0\leq n<m <\infty   } \frac{\tau_A(p^{m}) \tau_B(p^{n}) \left( 1-\frac{p^w}{p^{2}}  \right)}{ p^{m(1-\alpha)} p^{n(2-\beta-w)}  } = \frac{\tau_A(p )}{ p^{ 1-\alpha } } - \frac{\tau_A(p ) }{ p^{ 3-\alpha-w } }+ O\left( \frac{1}{p^{1+\varepsilon}}\right).
\end{equation}
We next bound the sum
\begin{equation}\label{eqn: R0eulerboundfinite1}
\sum_{\substack{0\leq m,n<\infty \\ m+\ordp(h)= n + \ordp(k) } }\frac{\tau_A(p^{m}) \tau_B(p^{n}) \left( 1+\frac{p^w}{p^{2}(p-1)}-\frac{1}{p-1} \right)}{ p^{m(1-\alpha)} p^{n(1-\beta)} p^{(1-w) \min\{m+\ordp(h), n + \ordp(k) \} }}.
\end{equation}
For brevity, we denote $h_p:=\ordp(h)$ and $k_p:=\ordp(k)$ for the rest of this proof. We make the change of variable $m\mapsto \nu +k_p$ in \eqref{eqn: R0eulerboundfinite1}, so that $n=\nu+h_p$, to see that \eqref{eqn: R0eulerboundfinite1} equals
\begin{equation*}
p^{k_p(w+\alpha-2)+h_p(w+\beta-2)} \sum_{\nu=-\min\{h_p,k_p\} }^{\infty} \frac{\tau_A(p^{\nu+k_p}) \tau_B(p^{\nu+h_p}) \left( 1+\frac{p^w}{p^{2}(p-1)}-\frac{1}{p-1} \right)}{ p^{\nu(3-\alpha-\beta-w)} }.
\end{equation*}
This and the inequality $(k_p+h_p -\min\{h_p,k_p\})\re(w)\leq  \frac{5}{2}(k_p+h_p -\min\{h_p,k_p\})$ imply
\begin{equation}\label{eqn: R0eulerbound6}
\begin{split}
\sum_{\substack{0\leq m,n<\infty \\ m+\ordp(h)= n + \ordp(k) } }&\frac{\tau_A(p^{m}) \tau_B(p^{n}) \left( 1+\frac{p^w}{p^{2}(p-1)}-\frac{1}{p-1} \right)}{ p^{m(1-\alpha)} p^{n(1-\beta)} p^{(1-w) \min\{m+\ordp(h), n + \ordp(k) \} }} \\
&\hspace{.25in}\ll p^{k_p(\re(w)-2+\varepsilon)+h_p(\re(w)-2+\varepsilon) + \min\{h_p,k_p\} \re( 3-w) }\\
&\hspace{.25in}\ll p^{(\frac{1}{2}+\varepsilon)( h_p +k_p + \min\{h_p,k_p\}) } .
\end{split}
\end{equation}
Next, to bound the sum
\begin{equation*}
\sum_{\substack{0\leq m,n<\infty \\ m+\ordp(h) <   n + \ordp(k) } } \frac{\tau_A(p^{m}) \tau_B(p^{n}) \left( 1-\frac{p^w}{p^{2}}  \right)}{ p^{m(1-\alpha)} p^{n(1-\beta)} p^{(1-w) \min\{m+\ordp(h), n + \ordp(k) \} }},
\end{equation*}
we split it into the part with $m<k_p-h_p$ and the part with $m\geq k_p-h_p$ to deduce that
\begin{equation}\label{eqn: R0eulerboundfinite2}
\sum_{\substack{0\leq m,n<\infty \\ m+\ordp(h) <   n + \ordp(k) } } \frac{\tau_A(p^{m}) \tau_B(p^{n}) \left( 1-\frac{p^w}{p^{2}}  \right)}{ p^{m(1-\alpha)} p^{n(1-\beta)} p^{(1-w) \min\{m+\ordp(h), n + \ordp(k) \} }} =\Sigma_1 + \Sigma_2,
\end{equation}
where
\begin{equation}\label{eqn: R0eulerboundSigma1def}
\Sigma_1 : = p^{(w-1)  h_p  } \sum_{m=0 }^{k_p-h_p-1}\sum_{n=0}^{\infty} \frac{\tau_A(p^{m}) \tau_B(p^{n}) \left( 1-\frac{p^w}{p^{2}}  \right)}{ p^{m(2-\alpha-w)} p^{n(1-\beta)} }
\end{equation}
and
\begin{equation}\label{eqn: R0eulerboundSigma2def}
\Sigma_2 : = p^{(w-1)  h_p  } \sum_{m=  \max\{0,k_p-h_p\} }^{\infty}\sum_{n=m+h_p-k_p+1}^{\infty} \frac{\tau_A(p^{m}) \tau_B(p^{n}) \left( 1-\frac{p^w}{p^{2}}  \right)}{ p^{m(2-\alpha-w)} p^{n(1-\beta)} }.
\end{equation}
We apply $\re(w)\leq \frac{5}{2}-\epsilon$ and bound the $n$-sums in \eqref{eqn: R0eulerboundSigma1def} and \eqref{eqn: R0eulerboundSigma2def} to deduce that
\begin{equation}\label{eqn: R0eulerboundSigma1bound1}
\Sigma_1 \ll p^{(\frac{3}{2}-\epsilon)  h_p  } \sum_{m=0 }^{k_p-h_p-1} \frac{p^{\frac{1}{2}-\epsilon}}{ p^{m(-\frac{1}{2}-\varepsilon+\epsilon)}}
\end{equation}
and
\begin{equation}\label{eqn: R0eulerboundSigma2bound1}
\Sigma_2 \ll p^{(\frac{3}{2}-\epsilon)  h_p  } \sum_{m=  \max\{0,k_p-h_p\} }^{\infty} \frac{ p^{\frac{1}{2}-\epsilon} }{ p^{m( \frac{1}{2}-\varepsilon+\epsilon)} p^{( h_p-k_p+1)(1-\varepsilon)} }.
\end{equation}
The right-hand side of \eqref{eqn: R0eulerboundSigma1bound1} is zero if $h_p\geq k_p$, and otherwise it is $ \ll p^{(\frac{3}{2}-\epsilon)  h_p  + (k_p-h_p) (\frac{1}{2} +\varepsilon)}$. In either case, we have
\begin{equation}\label{eqn: R0eulerboundSigma1bound2}
\Sigma_1 \ll p^{(\frac{1}{2}+\varepsilon)( h_p +k_p + \min\{h_p,k_p\}) }.
\end{equation}
If $h_p\geq k_p$, then the $m$-sum in \eqref{eqn: R0eulerboundSigma2bound1} starts at $m=0$ and thus
\begin{equation*}
\Sigma_2 \ll \frac{ p^{(\frac{3}{2}-\epsilon)  h_p  + \frac{1}{2}-\epsilon}  }{ p^{( h_p-k_p+1)(1-\varepsilon)}  } \ll p^{(\frac{1}{2}+\varepsilon)  h_p  +k_p - \frac{1}{2} + \varepsilon}.
\end{equation*}
On the other hand, if $h_p< k_p$, then the $m$-sum in \eqref{eqn: R0eulerboundSigma2bound1} starts at $m=k_p-h_p$ and hence
\begin{equation*}
\Sigma_2 \ll \frac{ p^{(\frac{3}{2}-\epsilon)  h_p  + \frac{1}{2}-\epsilon}  }{ p^{(k_p-h_p) (\frac{1}{2}-\varepsilon+\epsilon) } p^{( h_p-k_p+1)(1-\varepsilon)}  }  \ll p^{ (1+\varepsilon) h_p + (\frac{1}{2}+\varepsilon) k_p -\frac{1}{2}+\varepsilon }.
\end{equation*}
In either case, we have
\begin{equation*}
\Sigma_2 \ll  p^{(\frac{1}{2}+\varepsilon)( h_p +k_p + \min\{h_p,k_p\}) -\frac{1}{2}+\varepsilon} . 
\end{equation*}
From this, \eqref{eqn: R0eulerboundSigma1bound2}, and \eqref{eqn: R0eulerboundfinite2}, we arrive at
\begin{equation}\label{eqn: R0eulerbound7}
\sum_{\substack{0\leq m,n<\infty \\ m+\ordp(h) <   n + \ordp(k) } } \frac{\tau_A(p^{m}) \tau_B(p^{n}) \left( 1-\frac{p^w}{p^{2}}  \right)}{ p^{m(1-\alpha)} p^{n(1-\beta)} p^{(1-w) \min\{m+\ordp(h), n + \ordp(k) \} }} \ll  p^{(\frac{1}{2}+\varepsilon)( h_p +k_p + \min\{h_p,k_p\}) }.
\end{equation}
Similarly, or by symmetry, it holds that
\begin{equation}\label{eqn: R0eulerbound8}
\sum_{\substack{0\leq m,n<\infty \\ m+\ordp(h) >   n + \ordp(k) } } \frac{\tau_A(p^{m}) \tau_B(p^{n}) \left( 1-\frac{p^w}{p^{2}}  \right)}{ p^{m(1-\alpha)} p^{n(1-\beta)} p^{(1-w) \min\{m+\ordp(h), n + \ordp(k) \} }} \ll  p^{(\frac{1}{2}+\varepsilon)( h_p +k_p + \min\{h_p,k_p\}) }.
\end{equation}
From \eqref{eqn: R0eulerbound1}, \eqref{eqn: R0eulerbound2}, \eqref{eqn: R0eulerbound3}, \eqref{eqn: R0eulerbound4}, and \eqref{eqn: R0eulerbound5}, we deduce that if $p\nmid hk$ then the local factor in \eqref{eqn: R0eulerGdef} corresponding to $p$ is $1+O(p^{-1-\varepsilon})$. To bound the local factors corresponding to the primes $p|hk$, observe that \eqref{eqn: R0eulerbound1} is $O(1)$ because $\re(w)\leq \frac{5}{2}-\epsilon$. This, \eqref{eqn: R0eulerbound6}, \eqref{eqn: R0eulerbound7}, and \eqref{eqn: R0eulerbound8} imply that if $p|hk$ then the local factor corresponding to $p$ is
$$
\ll p^{(\frac{1}{2}+\varepsilon)( h_p +k_p + \min\{h_p,k_p\}) }.
$$
We conclude that the right-hand side of \eqref{eqn: R0eulerGdef} converges absolutely, and is
$$
\ll h^{\frac{1}{2}+\varepsilon} k^{\frac{1}{2}+\varepsilon} (h,k)^{\frac{1}{2}+\varepsilon}
$$
because $hk(h,k)=\prod_{p|hk} p^{  h_p+k_p+ \min\{h_p,k_p\}}$ and $\prod_{p|\nu}O(1) \ll {\nu}^{\varepsilon}$ for any positive integer $\nu$.
\end{proof}

We now move the $w$-line in \eqref{eqn: R0def2} rightward to $\re(w)= \frac{5}{2}-\epsilon$ to deduce that
\begin{equation}\label{eqn: R0split}
R_0 = R_1 + R_2 + R_3 + R_4,
\end{equation}
where $R_1$ is the integral of the residue at $w=2$, $R_2$ is the integral of the residue at $w=\frac{3}{2}-\alpha+z$, $R_3$ is the integral of the residues at the poles of \eqref{eqn:  U2I11factoroutzetas}, and $R_4$ is the new integral with $\re(w)= \frac{5}{2}-\epsilon$.

We first bound $R_4$, which is defined by
\begin{equation*}
\begin{split}
R_4 &:= \sum_{\substack{\alpha\in A \\ \beta\in B}}  \frac{1}{2(2\pi i)^2 }  \int_{(\frac{5}{2}-\epsilon)}\int_{(\epsilon/2)} X^{2-\alpha-\beta-w}Q^{w} \\
&\hspace{.25in} \times \widetilde{V}(\tfrac{3}{2}-\alpha-w+z)\widetilde{V}(\tfrac{1}{2}-\beta-z) \widetilde{W}(w)  \mathcal{H}(z,w-1) \frac{e^{\delta z} - e^{-\delta z}}{2\delta z}\\
&\hspace{.25in}\times \zeta(2-w)  \prod_{\substack{ \hat{\alpha}\neq \alpha \\ \hat{\beta}\neq \beta}} \zeta( 3+\hat{\alpha}+\hat{\beta} -\alpha-\beta-w) \prod_{\hat{\alpha}\neq \alpha} \zeta(1+\hat{\alpha}-\alpha)\prod_{\hat{\beta}\neq \beta} \zeta(1+\hat{\beta}-\beta )\\
&\hspace{.25in}\times h^{1-w+z}k^{-z}  \mathcal{G} (w,\alpha,\beta)\,dz  \,dw.
\end{split}
\end{equation*}
We move the $z$-line to $\re(z)=\frac{1}{2}-\epsilon$, traversing no poles in the process. We then bound the resulting integral using \eqref{eqn: zetaalphaalphabound}, \eqref{eqn: mellinrapiddecay}, \eqref{eqn: Hbound}, and Lemma~\ref{lem: R0eulerbound}. The result is
\begin{equation}\label{eqn: R4bound}
R_4\ll X^{-\frac{1}{2}+\varepsilon} Q^{\frac{5}{2}} (hk)^{\varepsilon}.
\end{equation}

We next evaluate the integral $R_1$ defined in \eqref{eqn: R0split}. To do this, observe that the winding number in the application of the residue theorem in \eqref{eqn: R0split} is $-1$. Also, the definition \eqref{eqn: Hdef} implies that
\begin{equation*}
\underset{w=2}{\text{Res}}\ \mathcal{H}(z,w-1) = -2
\end{equation*}
because Res$_{s=0}\Gamma(s)=1$ and $\Gamma(1/2)=\sqrt{\pi}$. Furthermore, $\zeta(0)=-1/2$. Hence
\begin{equation}\label{eqn: R1def}
\begin{split}
R_1 &= -\sum_{\substack{\alpha\in A \\ \beta\in B}}  \frac{1}{ 4\pi i  }  \int_{(\epsilon/2)} X^{-\alpha-\beta}Q^{2} \widetilde{V}(-\tfrac{1}{2}-\alpha+z)\widetilde{V}(\tfrac{1}{2}-\beta-z) \widetilde{W}(2) \frac{e^{\delta z} - e^{-\delta z}}{2\delta z}\\
&\hspace{.25in}\times  \prod_{\substack{ \hat{\alpha}\neq \alpha \\ \hat{\beta}\neq \beta}} \zeta( 1+\hat{\alpha}+\hat{\beta} -\alpha-\beta) \prod_{\hat{\alpha}\neq \alpha} \zeta(1+\hat{\alpha}-\alpha)\prod_{\hat{\beta}\neq \beta} \zeta(1+\hat{\beta}-\beta )\\
&\hspace{.25in}\times h^{-1+z}k^{-z}  \mathcal{G} (2,\alpha,\beta)\,dz .
\end{split}
\end{equation}
Some factors here do not depend on $z$, and we only need to evaluate
\begin{equation*}
\int_{(\epsilon/2)} \widetilde{V}(-\tfrac{1}{2}-\alpha+z)\widetilde{V}(\tfrac{1}{2}-\beta-z) \frac{e^{\delta z} - e^{-\delta z}}{2\delta z} h^{z}k^{-z}\,dz.
\end{equation*}
The part of this with $|\text{Im}z|\geq 1/\delta$ is negligible because of \eqref{eqn: mellinrapiddecay} and the definition \eqref{eqn: deltadef} of $\delta$. In the complementary part with $|\text{Im}z|\leq 1/\delta$, we have
\begin{equation}\label{eqn: exppowerseriesapprox}
\frac{e^{\delta z} - e^{-\delta z}}{2\delta z} = 1 +O(\delta|z|).
\end{equation}
Thus
\begin{equation*}
\begin{split}
\int_{(\epsilon/2)} \widetilde{V}(-\tfrac{1}{2}-\alpha+z)&\widetilde{V}(\tfrac{1}{2}-\beta-z) \frac{e^{\delta z} - e^{-\delta z}}{2\delta z} h^{z}k^{-z}\,dz \\
&\hspace{.25in}= \int_{\frac{\epsilon}{2}-\frac{i}{\delta}}^{\frac{\epsilon}{2}+\frac{i}{\delta}} \widetilde{V}(-\tfrac{1}{2}-\alpha+z)\widetilde{V}(\tfrac{1}{2}-\beta-z)h^{z}k^{-z}\,dz +O\big((hk)^{\varepsilon} \delta \big).
\end{split}
\end{equation*}
By \eqref{eqn: deltadef} and \eqref{eqn: mellinrapiddecay}, we may extend the range of Im$(z)$ in the latter integral to $(-\infty,\infty)$ by adding a negligible quantity. We then make the change of variables $s\mapsto \frac{1}{2}-\beta-z$, and afterward move the line of integration to $\re(s)=\epsilon$. We traverse no poles in doing so, and we arrive at
\begin{equation*}
\begin{split}
\int_{(\epsilon/2)} \widetilde{V}(-\tfrac{1}{2}-\alpha+z)&\widetilde{V}(\tfrac{1}{2}-\beta-z) \frac{e^{\delta z} - e^{-\delta z}}{2\delta z} h^{z}k^{-z}\,dz \\
&\hspace{.25in}= \int_{(\epsilon)} \widetilde{V}(-\alpha-\beta-s)\widetilde{V}(s)\left( \frac{h}{k}\right)^{\frac{1}{2}-\beta-s}\,ds +O\big((hk)^{\varepsilon} \delta \big).
\end{split}
\end{equation*}
From this and \eqref{eqn: R1def}, we deduce that
\begin{equation}\label{eqn: R1evaluated}
\begin{split}
R_1 &= -\sum_{\substack{\alpha\in A \\ \beta\in B}}  \frac{1}{ 4\pi i  }  \int_{(\epsilon)} X^{-\alpha-\beta}Q^{2} \widetilde{V}(-\alpha-\beta-s)\widetilde{V}(s) \widetilde{W}(2)\\
&\hspace{.25in}\times  \prod_{\substack{ \hat{\alpha}\neq \alpha \\ \hat{\beta}\neq \beta}} \zeta( 1+\hat{\alpha}+\hat{\beta} -\alpha-\beta) \prod_{\hat{\alpha}\neq \alpha} \zeta(1+\hat{\alpha}-\alpha)\prod_{\hat{\beta}\neq \beta} \zeta(1+\hat{\beta}-\beta )\\
&\hspace{.25in}\times h^{-\frac{1}{2}-\beta-s} k^{-\frac{1}{2}+\beta+s} \mathcal{G} (2,\alpha,\beta)\,ds + O\big( (Xhk)^{\varepsilon}k^{1/2} Q^{-96}\big),
\end{split}
\end{equation}
where we have applied \eqref{eqn: deltadef}, \eqref{eqn: zetaalphaalphabound}, and Lemma~\ref{lem: R0eulerbound} to bound the error term.

Having evaluated $R_1$, we next turn our attention to the integral $R_2$ defined in \eqref{eqn: R0split}. By \eqref{eqn: mellinVresidue}, the residue of $\widetilde{V}(\frac{3}{2}-\alpha-w+z)$ at $w=\frac{3}{2}-\alpha+z$ is $-1$. From this and the fact that the winding number in the application of the residue theorem in \eqref{eqn: R0split} is $-1$, we deduce that
\begin{equation*}
\begin{split}
R_2 &= \sum_{\substack{\alpha\in A \\ \beta\in B}}  \frac{1}{4\pi i } \int_{(\epsilon/2)} X^{\frac{1}{2}-\beta-z}Q^{\frac{3}{2}-\alpha+z} \\
 &\hspace{.25in}\times \widetilde{V}(\tfrac{1}{2}-\beta-z) \widetilde{W}(\tfrac{3}{2}-\alpha+z)  \mathcal{H}(z,\tfrac{1}{2}-\alpha+z) \frac{e^{\delta z} - e^{-\delta z}}{2\delta z}\\
&\hspace{.25in}\times \zeta(\tfrac{1}{2}+\alpha-z)  \prod_{\substack{ \hat{\alpha}\neq \alpha \\ \hat{\beta}\neq \beta}} \zeta( \tfrac{3}{2}+\hat{\alpha}+\hat{\beta} -\beta-z) \prod_{\hat{\alpha}\neq \alpha} \zeta(1+\hat{\alpha}-\alpha)\prod_{\hat{\beta}\neq \beta} \zeta(1+\hat{\beta}-\beta )\\
&\hspace{.25in}\times h^{-\frac{1}{2}+\alpha}k^{-z}  \mathcal{G} (\tfrac{3}{2}-\alpha+z,\alpha,\beta)\,dz.
\end{split}
\end{equation*}
We move the line of integration to $\re(z)=1-2\epsilon$ to deduce that
\begin{equation}\label{eqn: R2split}
R_2 = R_{21} + R_{22} + R_{23} + R_{24},
\end{equation}
where $R_{21}$ is the residue at $z=\frac{1}{2}-\beta$, $R_{22}$ is the residue at $z=\frac{1}{2}+\alpha$, $R_{23}$ is the sum of the residues at the poles $z=\frac{1}{2}+\alpha'+\beta'-\beta$, where $\alpha'$ runs through the elements of $A\smallsetminus \{\alpha\}$ and $\beta'$ runs through the elements of $B\smallsetminus \{\beta\}$, and $R_{24}$ is the new integral with $\re(z)=1-2\epsilon$. To bound $R_{24}$, we apply \eqref{eqn: zetaalphaalphabound}, \eqref{eqn: mellinrapiddecay}, \eqref{eqn: Hbound}, and Lemma~\ref{lem: R0eulerbound}. The result is
\begin{equation}\label{eqn: R24bound}
R_{24} \ll X^{-\frac{1}{2}+\varepsilon} Q^{\frac{5}{2}} (hk)^{\varepsilon}.
\end{equation}

We next estimate the residue $R_{21}$ defined in \eqref{eqn: R2split}. By \eqref{eqn: mellinVresidue}, the residue of $\widetilde{V}(\frac{1}{2}-\beta-z)$ at $z=\frac{1}{2}-\beta$ is $-1$. From this and the fact that the winding number in the application of the residue theorem in \eqref{eqn: R2split} is $-1$, we have
\begin{equation}\label{eqn: R21explicit}
\begin{split}
R_{21} &= \frac{1}{2 }\sum_{\substack{\alpha\in A \\ \beta\in B}}   Q^{2-\alpha-\beta} \widetilde{W}(2-\alpha-\beta)  \mathcal{H}(\tfrac{1}{2}-\beta,1-\alpha-\beta) \frac{e^{\delta (\frac{1}{2}-\beta)} - e^{-\delta (\frac{1}{2}-\beta)}}{\delta (1-2\beta)}\\
&\hspace{.25in}\times \zeta(\alpha+\beta)  \prod_{\substack{ \hat{\alpha}\neq \alpha \\ \hat{\beta}\neq \beta}} \zeta( 1+\hat{\alpha}+\hat{\beta} ) \prod_{\hat{\alpha}\neq \alpha} \zeta(1+\hat{\alpha}-\alpha)\prod_{\hat{\beta}\neq \beta} \zeta(1+\hat{\beta}-\beta )\\
&\hspace{.25in}\times h^{-\frac{1}{2}+\alpha}k^{-\frac{1}{2}+\beta}  \mathcal{G} (2-\alpha-\beta,\alpha,\beta).
\end{split}
\end{equation}
Now \eqref{eqn: Hintermsofchifactor} gives
\begin{equation*}
\mathcal{H}(\tfrac{1}{2}-\beta,1-\alpha-\beta)=\mathscr{X}(1-\alpha-\beta)\mathscr{X}(\tfrac{1}{2}+\beta)\mathscr{X}(\tfrac{1}{2}+\alpha),
\end{equation*}
and thus the functional equation of $\zeta(s)$ implies
$$
\zeta(\alpha+\beta)\mathcal{H}(\tfrac{1}{2}-\beta,1-\alpha-\beta) = \zeta(1-\alpha-\beta) \mathscr{X}(\tfrac{1}{2}+\beta)\mathscr{X}(\tfrac{1}{2}+\alpha).
$$
It follows from this and \eqref{eqn: R21explicit} that
\begin{equation}\label{eqn: R21explicit2}
\begin{split}
R_{21} &= \frac{1}{2 }\sum_{\substack{\alpha\in A \\ \beta\in B}}   Q^{2-\alpha-\beta} \widetilde{W}(2-\alpha-\beta)  \mathscr{X}(\tfrac{1}{2}+\alpha)\mathscr{X}(\tfrac{1}{2}+\beta) \frac{e^{\delta (\frac{1}{2}-\beta)} - e^{-\delta (\frac{1}{2}-\beta)}}{\delta (1-2\beta)}\\
&\hspace{.25in}\times \zeta(1-\alpha-\beta)  \prod_{\substack{ \hat{\alpha}\neq \alpha \\ \hat{\beta}\neq \beta}} \zeta( 1+\hat{\alpha}+\hat{\beta} ) \prod_{\hat{\alpha}\neq \alpha} \zeta(1+\hat{\alpha}-\alpha)\prod_{\hat{\beta}\neq \beta} \zeta(1+\hat{\beta}-\beta )\\
&\hspace{.25in}\times h^{-\frac{1}{2}+\alpha}k^{-\frac{1}{2}+\beta}  \mathcal{G} (2-\alpha-\beta,\alpha,\beta).
\end{split}
\end{equation}
By \eqref{eqn: deltadef} and the assumption that $\alpha,\beta \ll 1/\log Q$ for all $\alpha \in A$ and $\beta \in B$, we have
\begin{equation*}
\frac{e^{\delta (\frac{1}{2}-\beta)} - e^{-\delta (\frac{1}{2}-\beta)}}{\delta (1-2\beta)} = 1 + O\big( Q^{-99}\big).
\end{equation*}
We insert this into \eqref{eqn: R21explicit2} and apply Lemma~\ref{lem: R0eulerbound} and \eqref{eqn: zetaalphaalphabound} to bound the contribution of the error term. The result is
\begin{equation}\label{eqn: R21evaluated}
\begin{split}
R_{21} &= \frac{1}{2 }\sum_{\substack{\alpha\in A \\ \beta\in B}}   Q^{2-\alpha-\beta} \widetilde{W}(2-\alpha-\beta)  \mathscr{X}(\tfrac{1}{2}+\alpha)\mathscr{X}(\tfrac{1}{2}+\beta) \\
&\hspace{.25in}\times \zeta(1-\alpha-\beta)  \prod_{\substack{ \hat{\alpha}\neq \alpha \\ \hat{\beta}\neq \beta}} \zeta( 1+\hat{\alpha}+\hat{\beta} ) \prod_{\hat{\alpha}\neq \alpha} \zeta(1+\hat{\alpha}-\alpha)\prod_{\hat{\beta}\neq \beta} \zeta(1+\hat{\beta}-\beta )\\
&\hspace{.25in}\times h^{-\frac{1}{2}+\alpha}k^{-\frac{1}{2}+\beta}  \mathcal{G} (2-\alpha-\beta,\alpha,\beta) + O\big((hk)^{\varepsilon}(h,k)^{1/2}Q^{-96}\big).
\end{split}
\end{equation}

Our next task is to evaluate the residue $R_{22}$ defined in \eqref{eqn: R2split}. To do this, observe that the winding number in the application of the residue theorem in \eqref{eqn: R2split} is $-1$. Moreover, the definition \eqref{eqn: Hdef} implies that
\begin{equation*}
\underset{z=\frac{1}{2}+\alpha}{\text{Res}}\ \mathcal{H}(z,\tfrac{1}{2}-\alpha+z) = -2
\end{equation*}
because Res$_{s=0}\Gamma(s)=1$ and $\Gamma(1/2)=\sqrt{\pi}$. Furthermore, $\zeta(0)=-1/2$. Hence
\begin{equation}\label{eqn: R22explicit}
\begin{split}
R_{22} &= -\frac{1}{2}\sum_{\substack{\alpha\in A \\ \beta\in B}}   X^{-\alpha-\beta}Q^{2}\widetilde{V}(-\alpha-\beta) \widetilde{W}(2) \frac{e^{\delta (\frac{1}{2}+\alpha)} - e^{-\delta (\frac{1}{2}+\alpha) }}{ \delta (1+2\alpha)}\\
&\hspace{.25in}\times \prod_{\substack{ \hat{\alpha}\neq \alpha \\ \hat{\beta}\neq \beta}} \zeta( 1+\hat{\alpha}+\hat{\beta} -\alpha-\beta) \prod_{\hat{\alpha}\neq \alpha} \zeta(1+\hat{\alpha}-\alpha)\prod_{\hat{\beta}\neq \beta} \zeta(1+\hat{\beta}-\beta )\\
&\hspace{.25in}\times h^{-\frac{1}{2}+\alpha}k^{-\frac{1}{2}-\alpha}  \mathcal{G} (2,\alpha,\beta).
\end{split}
\end{equation}
By \eqref{eqn: deltadef} and the assumption that $\alpha,\beta \ll 1/\log Q$ for all $\alpha \in A$ and $\beta \in B$, we have
\begin{equation*}
\frac{e^{\delta (\frac{1}{2}+\alpha)} - e^{-\delta (\frac{1}{2}+\alpha)}}{\delta (1+2\alpha)} = 1 + O\big( Q^{-99}\big).
\end{equation*}
We insert this into \eqref{eqn: R22explicit} and apply Lemma~\ref{lem: R0eulerbound} and \eqref{eqn: zetaalphaalphabound} to bound the contribution of the error term. The result is
\begin{equation}\label{eqn: R22evaluated}
\begin{split}
R_{22} &= -\frac{1}{2}\sum_{\substack{\alpha\in A \\ \beta\in B}}   X^{-\alpha-\beta}Q^{2} \widetilde{V}(-\alpha-\beta) \widetilde{W}(2) \\
&\hspace{.25in}\times \prod_{\substack{ \hat{\alpha}\neq \alpha \\ \hat{\beta}\neq \beta}} \zeta( 1+\hat{\alpha}+\hat{\beta} -\alpha-\beta) \prod_{\hat{\alpha}\neq \alpha} \zeta(1+\hat{\alpha}-\alpha)\prod_{\hat{\beta}\neq \beta} \zeta(1+\hat{\beta}-\beta )\\
&\hspace{.25in}\times h^{-\frac{1}{2}+\alpha}k^{-\frac{1}{2}-\alpha}  \mathcal{G} (2,\alpha,\beta) + O\big( (Xhk)^{\varepsilon}(h,k)^{1/2} Q^{-96}\big).
\end{split}
\end{equation}

We next estimate the sum $R_{23}$ defined in \eqref{eqn: R2split}. Since the winding number in the application of the residue theorem in \eqref{eqn: R2split} is $-1$ and
\begin{equation*}
\underset{z=\frac{1}{2}+\alpha'+\beta'-\beta}{\text{Res}}\ \zeta( \tfrac{3}{2}+\alpha'+\beta' -\beta-z)=-1,
\end{equation*}
it follows that
\begin{align}
R_{23} &= \sum_{\substack{\alpha\in A \\ \beta\in B}} \sum_{\substack{\alpha'\neq \alpha \\ \beta'\neq \beta }} \frac{1}{2 } X^{-\alpha'-\beta'}Q^{2-\alpha-\beta+\alpha'+\beta'} \notag \\
 &\hspace{.25in}\times \widetilde{V}(-\alpha'-\beta') \widetilde{W}(2-\alpha-\beta+\alpha'+\beta')  \mathcal{H}(\tfrac{1}{2}+\alpha'+\beta'-\beta,1-\alpha-\beta+\alpha'+\beta') \notag \\
 &\hspace{.25in}\times \frac{e^{\delta (\frac{1}{2}+\alpha'+\beta'-\beta) } - e^{-\delta (\frac{1}{2}+\alpha'+\beta'-\beta)}}{2\delta (\frac{1}{2}+\alpha'+\beta'-\beta)}\zeta(\alpha+\beta-\alpha'-\beta') \notag\\
 &\hspace{.25in}\times\prod_{\substack{ \hat{\alpha} \neq \alpha \\ \hat{\beta}\neq \beta \\ (\hat{\alpha}, \hat{\beta})\neq (\alpha',\beta') }} \zeta( 1+\hat{\alpha}+\hat{\beta} -\alpha'-\beta') \prod_{\hat{\alpha}\neq \alpha} \zeta(1+\hat{\alpha}-\alpha)\prod_{\hat{\beta}\neq \beta} \zeta(1+\hat{\beta}-\beta ) \notag\\
&\hspace{.25in}\times h^{-\frac{1}{2}+\alpha}k^{-\frac{1}{2}-\alpha'-\beta'+\beta}  \mathcal{G} (2-\alpha-\beta+\alpha'+\beta',\alpha,\beta). \label{eqn: R23explicit}
\end{align}
Now \eqref{eqn: Hintermsofchifactor} gives
\begin{equation*}
\mathcal{H}(\tfrac{1}{2}+\alpha'+\beta'-\beta,1-\alpha-\beta+\alpha'+\beta')=\mathscr{X}(1-\alpha-\beta+\alpha'+\beta')\mathscr{X}(\tfrac{1}{2}+\beta-\alpha'-\beta')\mathscr{X}(\tfrac{1}{2}+\alpha),
\end{equation*}
and thus the functional equation of $\zeta(s)$ implies
\begin{equation*}
\begin{split}
&\zeta(\alpha+\beta-\alpha'-\beta')\mathcal{H}(\tfrac{1}{2}+\alpha'+\beta'-\beta,1-\alpha-\beta+\alpha'+\beta') \\
&\hspace{.25in}= \zeta(1-\alpha-\beta+\alpha'+\beta') \mathscr{X}(\tfrac{1}{2}+\beta-\alpha'-\beta')\mathscr{X}(\tfrac{1}{2}+\alpha).
\end{split}
\end{equation*}
It follows from this and \eqref{eqn: R23explicit} that
\begin{equation}\label{eqn: R23explicit2}
\begin{split}
R_{23} &= \sum_{\substack{\alpha\in A \\ \beta\in B}} \sum_{\substack{\alpha'\neq \alpha \\ \beta'\neq \beta }} \frac{1}{2 } X^{-\alpha'-\beta'}Q^{2-\alpha-\beta+\alpha'+\beta'} \\
&\hspace{.25in} \times \widetilde{V}(-\alpha'-\beta') \widetilde{W}(2-\alpha-\beta+\alpha'+\beta') \mathscr{X}(\tfrac{1}{2}+\beta-\alpha'-\beta')\mathscr{X}(\tfrac{1}{2}+\alpha)  \\
&\hspace{.25in} \times \frac{e^{\delta (\frac{1}{2}+\alpha'+\beta'-\beta) } - e^{-\delta (\frac{1}{2}+\alpha'+\beta'-\beta)}}{2\delta (\frac{1}{2}+\alpha'+\beta'-\beta)}\zeta(1-\alpha-\beta+\alpha'+\beta') \\
&\hspace{.25in}\times\prod_{\substack{ \hat{\alpha} \neq \alpha \\ \hat{\beta}\neq \beta \\ (\hat{\alpha}, \hat{\beta})\neq (\alpha',\beta') }} \zeta( 1+\hat{\alpha}+\hat{\beta} -\alpha'-\beta') \prod_{\hat{\alpha}\neq \alpha} \zeta(1+\hat{\alpha}-\alpha)\prod_{\hat{\beta}\neq \beta} \zeta(1+\hat{\beta}-\beta )\\
&\hspace{.25in}\times h^{-\frac{1}{2}+\alpha}k^{-\frac{1}{2}-\alpha'-\beta'+\beta}  \mathcal{G} (2-\alpha-\beta+\alpha'+\beta',\alpha,\beta).
\end{split}
\end{equation}
By \eqref{eqn: deltadef} and the assumption that $\alpha,\beta \ll 1/\log Q$ for all $\alpha \in A$ and $\beta \in B$, we have
\begin{equation*}
\frac{e^{\delta (\frac{1}{2}+\alpha'+\beta'-\beta) } - e^{-\delta (\frac{1}{2}+\alpha'+\beta'-\beta)}}{2\delta (\frac{1}{2}+\alpha'+\beta'-\beta)} = 1 + O\big( Q^{-99}\big).
\end{equation*}
We insert this into \eqref{eqn: R23explicit2} and apply Lemma~\ref{lem: R0eulerbound} and \eqref{eqn: zetaalphaalphabound} to bound the contribution of the error term. The result is
\begin{equation}\label{eqn: R23evaluated}
\begin{split}
R_{23} &= \sum_{\substack{\alpha\in A \\ \beta\in B}} \sum_{\substack{\alpha'\neq \alpha \\ \beta'\neq \beta }} \frac{1}{2 } X^{-\alpha'-\beta'}Q^{2-\alpha-\beta+\alpha'+\beta'} \\
 &\hspace{.25in}\times \widetilde{V}(-\alpha'-\beta') \widetilde{W}(2-\alpha-\beta+\alpha'+\beta') \mathscr{X}(\tfrac{1}{2}+\beta-\alpha'-\beta')\mathscr{X}(\tfrac{1}{2}+\alpha)  \\
&\hspace{.25in}\times \zeta(1-\alpha-\beta+\alpha'+\beta')  \prod_{\substack{ \hat{\alpha} \neq \alpha \\ \hat{\beta}\neq \beta \\ (\hat{\alpha}, \hat{\beta})\neq (\alpha',\beta') }} \zeta( 1+\hat{\alpha}+\hat{\beta} -\alpha'-\beta') \prod_{\hat{\alpha}\neq \alpha} \zeta(1+\hat{\alpha}-\alpha)\\
&\hspace{.25in}\prod_{\hat{\beta}\neq \beta} \zeta(1+\hat{\beta}-\beta ) h^{-\frac{1}{2}+\alpha}k^{-\frac{1}{2}-\alpha'-\beta'+\beta}  \mathcal{G} (2-\alpha-\beta+\alpha'+\beta',\alpha,\beta) \\
&\hspace{.25in}+ O\big( (Xhk)^{\varepsilon}(h,k)^{1/2} Q^{-96}\big).
\end{split}
\end{equation}
This, \eqref{eqn: R2split}, \eqref{eqn: R24bound}, \eqref{eqn: R21evaluated}, and \eqref{eqn: R22evaluated} complete our evaluation of $R_2$.

Having estimated $R_2$, we next turn our attention to the integral $R_3$ defined in \eqref{eqn: R0split}. Since
\begin{equation*}
\underset{w=2-\alpha-\beta+\alpha'+\beta'}{\text{Res}}\ \zeta( 3+\alpha'+\beta' -\alpha-\beta-w)=-1
\end{equation*}
and the winding number in the application of the residue theorem in \eqref{eqn: R0split} is $-1$, we may write
\begin{equation}\label{eqn: R3def}
\begin{split}
R_3 &= \sum_{\substack{\alpha\in A \\ \beta\in B}} \sum_{\substack{\alpha'\neq \alpha \\ \beta'\neq \beta}} \frac{1}{4\pi i}  \int_{(\epsilon/2)} X^{-\alpha'-\beta'}Q^{2-\alpha-\beta+\alpha'+\beta'}  \widetilde{V}(-\tfrac{1}{2}+\beta-\alpha'-\beta'+z)\widetilde{V}(\tfrac{1}{2}-\beta-z)\\
&\hspace{.25in}\widetilde{W}(2-\alpha-\beta+\alpha'+\beta')  \mathcal{H}(z,1-\alpha-\beta+\alpha'+\beta') \frac{e^{\delta z} - e^{-\delta z}}{2\delta z}\\
&\hspace{.25in}\times \zeta(\alpha+\beta-\alpha'-\beta')  \prod_{\substack{ \hat{\alpha}\neq \alpha \\ \hat{\beta}\neq \beta \\ (\hat{\alpha},\hat{\beta})\neq (\alpha',\beta')}} \zeta( 1+\hat{\alpha}+\hat{\beta} -\alpha'-\beta') \prod_{\hat{\alpha}\neq \alpha} \zeta(1+\hat{\alpha}-\alpha)\\
&\hspace{.25in}\times\prod_{\hat{\beta}\neq \beta} \zeta(1+\hat{\beta}-\beta )h^{-1+\alpha+\beta-\alpha'-\beta'+z}k^{-z}  \mathcal{G} (2-\alpha-\beta+\alpha'+\beta',\alpha,\beta)\,dz .
\end{split}
\end{equation}
By \eqref{eqn: Hintermsofchifactor}, we have
\begin{equation*}
\mathcal{H}(z,1-\alpha-\beta+\alpha'+\beta')=\mathscr{X}(1-\alpha-\beta+\alpha'+\beta')\mathscr{X}(1-z)\mathscr{X}(\alpha+\beta-\alpha'-\beta'+z),
\end{equation*}
and thus the functional equation of $\zeta(s)$ implies
$$
\zeta(\alpha+\beta-\alpha'-\beta')\mathcal{H}(z,1-\alpha-\beta+\alpha'+\beta') = \zeta(1-\alpha-\beta+\alpha'+\beta') \mathscr{X}(1-z)\mathscr{X}(\alpha+\beta-\alpha'-\beta'+z).
$$
It follows from this and \eqref{eqn: R3def} that
\begin{equation}\label{eqn: R3def2}
\begin{split}
R_3 &= \sum_{\substack{\alpha\in A \\ \beta\in B}} \sum_{\substack{\alpha'\neq \alpha \\ \beta'\neq \beta}} \frac{1}{4\pi i}  \int_{(\epsilon/2)} X^{-\alpha'-\beta'}Q^{2-\alpha-\beta+\alpha'+\beta'}  \widetilde{V}(-\tfrac{1}{2}+\beta-\alpha'-\beta'+z)\widetilde{V}(\tfrac{1}{2}-\beta-z)\\
&\hspace{.25in}\times\widetilde{W}(2-\alpha-\beta+\alpha'+\beta') \mathscr{X}(1-z)\mathscr{X}(\alpha+\beta-\alpha'-\beta'+z) \frac{e^{\delta z} - e^{-\delta z}}{2\delta z}\\
&\hspace{.25in}\times \zeta(1-\alpha-\beta+\alpha'+\beta') \prod_{\substack{ \hat{\alpha}\neq \alpha \\ \hat{\beta}\neq \beta \\ (\hat{\alpha},\hat{\beta})\neq (\alpha',\beta')}} \zeta( 1+\hat{\alpha}+\hat{\beta} -\alpha'-\beta') \prod_{\hat{\alpha}\neq \alpha} \zeta(1+\hat{\alpha}-\alpha)\\
&\hspace{.25in}\times\prod_{\hat{\beta}\neq \beta} \zeta(1+\hat{\beta}-\beta )h^{-1+\alpha+\beta-\alpha'-\beta'+z}k^{-z}  \mathcal{G} (2-\alpha-\beta+\alpha'+\beta',\alpha,\beta)\,dz .
\end{split}
\end{equation}
Some factors here do not depend on $z$, and we only need to evaluate
$$
\int_{(\epsilon/2)} \widetilde{V}(-\tfrac{1}{2}+\beta-\alpha'-\beta'+z)\widetilde{V}(\tfrac{1}{2}-\beta-z)  \mathscr{X}(1-z)\mathscr{X}(\alpha+\beta-\alpha'-\beta'+z)\frac{e^{\delta z} - e^{-\delta z}}{2\delta z} h^{z}k^{-z}\,dz.
$$
The part of this with $|\text{Im}z|\geq 1/\delta$ is negligible because of \eqref{eqn: mellinrapiddecay}, \eqref{eqn: Stirlingchi}, and the definition \eqref{eqn: deltadef} of $\delta$. In the complementary part with $|\text{Im}z|\leq 1/\delta$, we have \eqref{eqn: exppowerseriesapprox} and thus
\begin{equation*}
\begin{split}
\int_{(\epsilon/2)}& \widetilde{V}(-\tfrac{1}{2}+\beta-\alpha'-\beta'+z)\widetilde{V}(\tfrac{1}{2}-\beta-z) \\
&\hspace{.25in}\times\mathscr{X}(1-z)\mathscr{X}(\alpha+\beta-\alpha'-\beta'+z)\frac{e^{\delta z} - e^{-\delta z}}{2\delta z} h^{z}k^{-z}\,dz \\
&= \int_{\frac{\epsilon}{2}-\frac{i}{\delta}}^{\frac{\epsilon}{2}+\frac{i}{\delta}} \widetilde{V}(-\tfrac{1}{2}+\beta-\alpha'-\beta'+z)\widetilde{V}(\tfrac{1}{2}-\beta-z) \\ &\hspace{.25in}\times \mathscr{X}(1-z)\mathscr{X}(\alpha+\beta-\alpha'-\beta'+z) h^{z}k^{-z}\,dz+O\big((hk)^{\varepsilon} \delta \big).
\end{split}
\end{equation*}
By \eqref{eqn: deltadef} and \eqref{eqn: mellinrapiddecay}, we may extend the range of Im$(z)$ in the latter integral to $(-\infty,\infty)$ by adding a negligible quantity. We then make the change of variables
$$
s\longmapsto -\frac{1}{2}+\beta-\alpha'-\beta'+z,
$$
and afterward move the line of integration to $\re(s)=-\epsilon$. We traverse no poles in doing so, and we arrive at
\begin{equation*}
\begin{split}
\int_{(\epsilon/2)}& \widetilde{V}(-\tfrac{1}{2}+\beta-\alpha'-\beta'+z)\widetilde{V}(\tfrac{1}{2}-\beta-z)\\
&\hspace{.25in}\times\mathscr{X}(1-z)\mathscr{X}(\alpha+\beta-\alpha'-\beta'+z)\frac{e^{\delta z} - e^{-\delta z}}{2\delta z} h^{z}k^{-z}\,dz \\
&= \int_{(-\epsilon)} \widetilde{V}(s) \widetilde{V}(-\alpha'-\beta'-s)\\
&\hspace{.25in}\times\mathscr{X}(\tfrac{1}{2}+\beta-\alpha'-\beta'-s)  \mathscr{X}(\tfrac{1}{2}+\alpha+s)\left( \frac{h}{k}\right)^{\frac{1}{2}-\beta+\alpha'+\beta'+s}\,ds +O\big((hk)^{\varepsilon} \delta \big).
\end{split}
\end{equation*}
From this and \eqref{eqn: R3def2}, we deduce that
\begin{equation}\label{eqn: R3evaluated}
\begin{split}
R_3 &= \sum_{\substack{\alpha\in A \\ \beta\in B}} \sum_{\substack{\alpha'\neq \alpha \\ \beta'\neq \beta}} \frac{1}{4\pi i}  \int_{(-\epsilon)} X^{-\alpha'-\beta'}Q^{2-\alpha-\beta+\alpha'+\beta'} \\
 &\hspace{.25in}\times \widetilde{V}(s) \widetilde{V}(-\alpha'-\beta'-s) \widetilde{W}(2-\alpha-\beta+\alpha'+\beta')  \\
&\hspace{.25in}\times \mathscr{X}(\tfrac{1}{2}+\beta-\alpha'-\beta'-s)  \mathscr{X}(\tfrac{1}{2}+\alpha+s)\zeta(1-\alpha-\beta+\alpha'+\beta')\\
&\hspace{.25in}\times\prod_{\substack{ \hat{\alpha}\neq \alpha \\ \hat{\beta}\neq \beta \\ (\hat{\alpha},\hat{\beta})\neq (\alpha',\beta')}} \zeta( 1+\hat{\alpha}+\hat{\beta} -\alpha'-\beta') \prod_{\hat{\alpha}\neq \alpha} \zeta(1+\hat{\alpha}-\alpha)\prod_{\hat{\beta}\neq \beta} \zeta(1+\hat{\beta}-\beta )\\
&\hspace{.25in}\times h^{-\frac{1}{2}+\alpha+ s}k^{-\frac{1}{2}+\beta-\alpha'-\beta'- s}  \mathcal{G} (2-\alpha-\beta+\alpha'+\beta',\alpha,\beta)\,ds  \\
&\hspace{.25in}+ O\big( (Xhk)^{\varepsilon}k^{1/2} Q^{-96}\big),
\end{split}
\end{equation}
where we have applied \eqref{eqn: deltadef}, \eqref{eqn: zetaalphaalphabound}, \eqref{eqn: mellinrapiddecay}, \eqref{eqn: Hbound}, and Lemma~\ref{lem: R0eulerbound} to bound the error term.

Putting together our calculations, we see from \eqref{eqn: U2split}, \eqref{eqn: U2I2bound}, and \eqref{eqn: I1toR0} that
\begin{equation*}
\mathcal{U}^2(h,k) = R_0 + O \bigg( \bigg( Q^{1+\varepsilon} + \frac{Q^2}{C^{1-\varepsilon}}\bigg) \frac{(Xhk)^{\varepsilon}(h,k)}{\sqrt{hk}}  \bigg) + O\Big( X^{\varepsilon} Q^{\frac{3}{2}} h^{\varepsilon}k^{\varepsilon} + (XChk)^{\varepsilon} k X^2 Q^{-97}\Big).
\end{equation*}
From this, \eqref{eqn: R0split}, \eqref{eqn: R4bound}, \eqref{eqn: R2split}, and \eqref{eqn: R24bound}, we arrive at
\begin{equation}\label{eqn: U2residues}
\begin{split}
\mathcal{U}^2(h,k) &= R_1+R_{21}+R_{22}+R_{23}+R_{3} + O \bigg( \bigg( Q^{1+\varepsilon} + \frac{Q^2}{C^{1-\varepsilon}}\bigg) \frac{(Xhk)^{\varepsilon}(h,k)}{\sqrt{hk}}  \bigg)\\
&\hspace{.25in} + O\Big( X^{\varepsilon} Q^{\frac{3}{2}} h^{\varepsilon}k^{\varepsilon} + X^{-\frac{1}{2}+\varepsilon} Q^{\frac{5}{2}} (hk)^{\varepsilon} + (XChk)^{\varepsilon} k X^2 Q^{-97}\Big),
\end{split}
\end{equation}
where we have evaluated the residue $R_1$ in \eqref{eqn: R1evaluated}, $R_{21}$ in \eqref{eqn: R21evaluated}, $R_{22}$ in \eqref{eqn: R22evaluated}, $R_{23}$ in \eqref{eqn: R23evaluated}, and $R_{3}$ in \eqref{eqn: R3evaluated}. In the next subsection, we will match these five residues with the five residues on the right-hand side of \eqref{eqn: 1swapsready} in such a way that corresponding residues are equal, thus showing that $\mathcal{U}^2(h,k)$ is equal to $\mathcal{I}_1^*(h,k)$ up to an admissible error term.

\subsection{Matching the residues: Euler product evaluations}\label{sec: matchresidues}

To be able to show that each of the residues on the right-hand side of \eqref{eqn: U2residues} is equal to some term on the right-hand side of in \eqref{eqn: 1swapsready}, we will prove the following identity involving the Euler products $\mathcal{G}$ and $\mathcal{K}$.

\begin{lemma}\label{lem: GKeuleridentity}
Let $\alpha\in A$ and $\beta\in B$. Suppose that $h$ and $k$ are positive integers. If $\mathcal{G}$ is defined by \eqref{eqn: R0eulerGdef} and $\mathcal{K}$ by \eqref{eqn: 1swapeulerKdef}, then
\begin{equation}\label{eqn: GKeuleriden}
h^{-\frac{1}{2}+\alpha} k^{-\frac{1}{2}+\beta} \mathcal{G}(2-\alpha-\beta,\alpha,\beta;A,B,h,k) = \mathcal{K}(0,0,2-\alpha-\beta;A,B,\alpha,\beta,h,k).
\end{equation}
\end{lemma}

Our proof of Lemma~\ref{lem: GKeuleridentity} will depend on the following three lemmas. The first is a slight generalization of an identity due to Conrey and Keating~\cite{CK3}

\begin{lemma}\label{lem: CK3identity}
Let $\alpha\in A$ and $\beta\in B$. Suppose that $j$ and $\ell$ are nonnegative integers and $p$ is a prime. Then
\begin{equation*}
\begin{split}
\tau_{A\smallsetminus\{\alpha\}\cup \{-\beta\}} (p^j) \tau_{B\smallsetminus\{\beta\}} (p^{\ell}) + \tau_{A\smallsetminus\{\alpha\}} (p^j) \tau_{B\smallsetminus\{\beta\}\cup \{-\alpha\}} (p^{\ell}) -  \tau_{A\smallsetminus\{\alpha\}} (p^j) \tau_{B\smallsetminus\{\beta\}} (p^{\ell}) \\
= \tau_{A\smallsetminus\{\alpha\}\cup \{-\beta\}} (p^j) \tau_{B\smallsetminus\{\beta\} \cup\{-\alpha\}} (p^{\ell}) -p^{\alpha+\beta} \tau_{A\smallsetminus\{\alpha\}\cup \{-\beta\}}(p^{j-1}) \tau_{B\smallsetminus\{\beta\} \cup\{-\alpha\}} (p^{\ell-1}),
\end{split}
\end{equation*}
where $\tau_E(p^{-1})$ is defined to be zero for any multiset $E$.
\end{lemma}
\begin{proof}
We argue as in \cite{CK3}. Observe that the definition \eqref{eqn: taudef} implies that if $m$ is any nonnegative integer, $E$ is any finite multiset, and $\gamma\in E$, then
\begin{equation}\label{eqn: tauremoveelement}
\tau_E(p^m)=\tau_{E\smallsetminus \{\gamma\}}(p^m) + p^{-\gamma}\tau_E(p^{m-1}).
\end{equation}
We apply this, multiply out the resulting products, and then cancel one $\tau_{A\smallsetminus\{\alpha\}} (p^j) \tau_{B\smallsetminus\{\beta\}} (p^{\ell})$ with its negative to deduce that
\begin{equation*}
\begin{split}
&\tau_{A\smallsetminus\{\alpha\}\cup \{-\beta\}} (p^j) \tau_{B\smallsetminus\{\beta\}} (p^{\ell}) + \tau_{A\smallsetminus\{\alpha\}} (p^j) \tau_{B\smallsetminus\{\beta\}\cup \{-\alpha\}} (p^{\ell}) -  \tau_{A\smallsetminus\{\alpha\}} (p^j) \tau_{B\smallsetminus\{\beta\}} (p^{\ell}) \\
&= \Big( \tau_{A\smallsetminus\{\alpha\}} (p^j) + p^{\beta}\tau_{A\smallsetminus\{\alpha\}\cup \{-\beta\}} (p^{j-1}) \Big) \tau_{B\smallsetminus\{\beta\}} (p^{\ell}) \\
&\hspace{.5in}+ \tau_{A\smallsetminus\{\alpha\}} (p^j) \Big( \tau_{B\smallsetminus\{\beta\} } (p^{\ell}) + p^{\alpha} \tau_{B\smallsetminus\{\beta\}\cup \{-\alpha\}} (p^{\ell-1}) \Big)-  \tau_{A\smallsetminus\{\alpha\}} (p^j) \tau_{B\smallsetminus\{\beta\}} (p^{\ell}) \\
&= \tau_{A\smallsetminus\{\alpha\}} (p^j)\tau_{B\smallsetminus\{\beta\}} (p^{\ell}) + p^{\beta}\tau_{A\smallsetminus\{\alpha\}\cup \{-\beta\}} (p^{j-1})\tau_{B\smallsetminus\{\beta\}} (p^{\ell}) + p^{\alpha}\tau_{A\smallsetminus\{\alpha\}}(p^{j }) \tau_{B\smallsetminus\{\beta\}\cup \{-\alpha\}} (p^{\ell-1}).
\end{split}
\end{equation*}
We add and subtract $p^{\alpha+\beta}\tau_{A\smallsetminus\{\alpha\}\cup\{-\beta\}} (p^{j-1}) \tau_{B\smallsetminus\{\beta\}\cup\{-\alpha\}} (p^{\ell-1})$, and then factor part of the resulting expression to arrive at
\begin{equation*}
\begin{split}
\tau_{A\smallsetminus\{\alpha\}\cup \{-\beta\}} (p^j) \tau_{B\smallsetminus\{\beta\}} (p^{\ell}) + \tau_{A\smallsetminus\{\alpha\}} (p^j) \tau_{B\smallsetminus\{\beta\}\cup \{-\alpha\}} (p^{\ell}) -  \tau_{A\smallsetminus\{\alpha\}} (p^j) \tau_{B\smallsetminus\{\beta\}} (p^{\ell}) \\
= \Big( \tau_{A\smallsetminus\{\alpha\}} (p^j) + p^{\beta}\tau_{A\smallsetminus\{\alpha\}\cup \{-\beta\}} (p^{j-1})\Big) \Big( \tau_{B\smallsetminus\{\beta\}} (p^{\ell}) + p^{\alpha} \tau_{B\smallsetminus\{\beta\}\cup \{-\alpha\}} (p^{\ell-1})\Big) \\
- p^{\alpha+\beta}\tau_{A\smallsetminus\{\alpha\}\cup\{-\beta\}} (p^{j-1}) \tau_{B\smallsetminus\{\beta\}\cup\{-\alpha\}} (p^{\ell-1}).
\end{split}
\end{equation*}
The lemma now follows from this and \eqref{eqn: tauremoveelement}.
\end{proof}

\begin{lemma}\label{lem: taulongidentity}
Let $\alpha\in A$ and $\beta\in B$. Suppose that $j$ and $\ell$ are nonnegative integers and $p$ is a prime. Then
\begin{equation}\label{eqn: taulongiden}
\begin{split}
\tau_{A\smallsetminus\{\alpha\}\cup \{-\beta\}} (p^j) \tau_{B\smallsetminus\{\beta\}} (p^{\ell})&=
(1-p^{-\alpha-\beta}) \tau_{A\cup\{-\beta\}} (p^j)\tau_{B  } (p^{\ell})+p^{-\alpha-\beta} \tau_A(p^j)\tau_{B  } (p^{\ell}) \\
&-p^{-\beta} \tau_A(p^j) \tau_{B  } (p^{\ell-1})-(1-p^{-\alpha-\beta}) \tau_{A\cup\{-\beta\}} (p^{j-1})\tau_{B  } (p^{\ell-1})
\end{split}
\end{equation}
where $\tau_E(p^{-1})$ is defined to be zero for any multiset $E$.
\end{lemma}
\begin{proof}
We apply \eqref{eqn: tauremoveelement} and multiply out the resulting expression to deduce that
\begin{equation*}
\begin{split}
(1-p^{-\alpha-\beta}) \tau_{A\cup\{-\beta\}} (p^j)\tau_{B  } (p^{\ell})
&= (1-p^{-\alpha-\beta})\Big(  \tau_{A} (p^j)+p^{\beta} \tau_{A\cup\{-\beta\}} (p^{j-1}) \Big)\tau_{B  } (p^{\ell}) \\
&= \tau_{A} (p^j)\tau_{B  } (p^{\ell}) -p^{-\alpha-\beta} \tau_{A} (p^j)\tau_{B  } (p^{\ell})\\
&\hspace{.5in}+ (p^{\beta}-p^{-\alpha }) \tau_{A\cup\{-\beta\}} (p^{j-1}) \tau_{B  } (p^{\ell})
\end{split}
\end{equation*}
The term $-p^{-\alpha-\beta} \tau_{A} (p^j)\tau_{B  } (p^{\ell})$ cancels with its negative on the left-hand side of \eqref{eqn: taulongiden}, and it follows that
\begin{equation*}
\begin{split}
(1-p^{-\alpha-\beta}) \tau_{A\cup\{-\beta\}} (p^j)\tau_{B  } (p^{\ell}) &+p^{-\alpha-\beta} \tau_A(p^j)\tau_{B  } (p^{\ell}) \\
&-p^{-\beta} \tau_A(p^j) \tau_{B  } (p^{\ell-1}) -(1-p^{-\alpha-\beta}) \tau_{A\cup\{-\beta\}} (p^{j-1})\tau_{B  } (p^{\ell-1}) \\
&\hspace{-.5in}= \tau_{A} (p^j)\tau_{B  } (p^{\ell}) + (p^{\beta}-p^{-\alpha }) \tau_{A\cup\{-\beta\}} (p^{j-1}) \tau_{B  } (p^{\ell})  \\
&\hspace{.25in}-p^{-\beta} \tau_A(p^j) \tau_{B  } (p^{\ell-1}) -(1-p^{-\alpha-\beta}) \tau_{A\cup\{-\beta\}} (p^{j-1})\tau_{B  } (p^{\ell-1}).
\end{split}
\end{equation*}
The right-hand side factors as
\begin{equation*}
\begin{split}
\Big( \tau_{A} (p^j) + (p^{\beta}-p^{-\alpha }) \tau_{A\cup\{-\beta\}} (p^{j-1}) \Big)\Big( \tau_{B  } (p^{\ell})-p^{-\beta} \tau_{B  } (p^{\ell-1})\Big),
\end{split}
\end{equation*}
which, by \eqref{eqn: tauremoveelement}, equals $\tau_{A\smallsetminus\{\alpha\}\cup \{-\beta\}} (p^j) \tau_{B\smallsetminus\{\beta\}} (p^{\ell})$.
\end{proof}

\begin{lemma}\label{lem: tauseriesidentity}
Let $\beta\in B$. Suppose that $j$ and $\ell$ are nonnegative integers and $p$ is a prime. Then
\begin{equation*}
p^{(\frac{1}{2}-\beta)(j-\ell)} \sum_{ \substack{ 0\leq m,n<\infty \\ m+ j< n+\ell}  } \frac{\tau_A(p^{m}) \tau_B(p^{n}) }{ p^{m\beta } p^{n(1-\beta)} }  = \sum_{\substack{ 0\leq m,n<\infty \\ m+j =n+\ell}} \frac{ \tau_{A\cup\{-\beta\}}(p^m) \tau_B(p^n) - \tau_{A}(p^m) \tau_B(p^n) }{p^{\frac{m}{2}+\frac{n}{2}}}
\end{equation*}
\end{lemma}
\begin{proof}
The definition \eqref{eqn: taudef} of $\tau_E$ implies that if $D$ and $E$ are finite multisets, then the Dirichlet convolution $\tau_D*\tau_E$ of $\tau_D$ and $\tau_E$ is $\tau_{D\cup E}$. It follows from this and the definition of Dirichlet convolution that, for each nonnegative integer $m$,
\begin{equation*}
\tau_{A\cup\{-\beta\}} (p^{m-1}) = (\tau_A*\tau_{\{-\beta\}})(p^{m-1}) = \sum_{\nu=0}^{m-1} \tau_A(p^{\nu}) \tau_{\{-\beta\}} (p^{m-1-\nu}) = \sum_{\nu=0}^{m-1} \tau_A(p^{\nu}) p^{\beta(m-1-\nu)}.
\end{equation*}
This and the identity \eqref{eqn: tauremoveelement} imply
\begin{equation*}
 \tau_{A\cup\{-\beta\}}(p^m) - \tau_{A}(p^m) = p^{\beta} \tau_{A\cup\{-\beta\}}(p^{m-1})= \sum_{\nu=0}^{m-1} \tau_A(p^{\nu}) p^{\beta(m-\nu)}.
\end{equation*}
Therefore
\begin{equation*}
\sum_{\substack{ 0\leq m,n<\infty \\ m+j =n+\ell}} \frac{ \tau_{A\cup\{-\beta\}}(p^m) \tau_B(p^n) - \tau_{A}(p^m) \tau_B(p^n) }{p^{\frac{m}{2}+\frac{n}{2}}} = \sum_{\substack{ 0\leq m,n<\infty \\ m+j =n+\ell}} \frac{ \tau_B(p^n)  }{p^{\frac{m}{2}+\frac{n}{2}}}\sum_{\nu=0}^{m-1} \tau_A(p^{\nu}) p^{\beta(m-\nu)}.
\end{equation*}
In the latter sum, we may replace $m$ with $n+\ell-j$ to write the sum as
\begin{equation*}
\sum_{n=0}^{\infty} \frac{ \tau_B(p^n)  }{p^{n + \frac{1}{2}(\ell-j)} }\sum_{\nu=0}^{n+\ell-j-1} \tau_A(p^{\nu}) p^{\beta(n+\ell-j-\nu)}= p^{(\frac{1}{2}-\beta)(j-\ell)} \sum_{ \substack{ 0\leq \nu,n<\infty \\ \nu+ j< n+\ell}  } \frac{\tau_A(p^{\nu}) \tau_B(p^{n}) }{ p^{\nu\beta } p^{n(1-\beta)} }.
\end{equation*}
\end{proof}

\begin{proof}[Proof of Lemma~\ref{lem: GKeuleridentity}]
We may write each side of \eqref{eqn: GKeuleriden} as an Euler product by the definitions \eqref{eqn: 1swapeulerKdef} of $\mathcal{K}$ and  \eqref{eqn: R0eulerGdef} of $\mathcal{G}$. The Euler products converge absolutely by Lemmas~\ref{lem: 1swapeulerbound} and \ref{lem: R0eulerbound}. To prove Lemma~\ref{lem: GKeuleridentity}, it suffices to show for each $p$ that the local factors corresponding to $p$ in these Euler products agree.

We first examine the local factors corresponding to a given prime $p\nmid hk$. For brevity, let $\mathfrak{F}_p$ denote the local factor corresponding to this $p$ in the Euler product expression for the left-hand side of \eqref{eqn: GKeuleriden}. Thus, from the definition \eqref{eqn: R0eulerGdef} of $\mathcal{G}$, we see that $\mathfrak{F}_p$ is defined by
\begin{equation}\label{eqn: mathfrakFdef}
\begin{split}
\mathfrak{F}_p &:= \prod_{\substack{ \hat{\alpha}\neq \alpha \\ \hat{\beta}\neq \beta}}\left(1-\frac{1}{p^{1+\hat{\alpha}+\hat{\beta}}} \right) \prod_{\hat{\alpha}\in A}\left(1-\frac{1}{p^{1+\hat{\alpha}-\alpha}} \right) \prod_{\hat{\beta}\in B} \left(1-\frac{1}{p^{1+\hat{\beta}-\beta}} \right)  \\ 
&\hspace{.25in}\times \Bigg(  \left( 1-\frac{1}{p^{2-\alpha-\beta}}\right)\left(1+ \frac{p^{1-\alpha-\beta}-1}{p(p-1)}\right)   + \left( 1+\frac{p^{-\alpha-\beta}}{p-1}-\frac{1}{p-1} \right) \sum_{m=1}^{\infty} \frac{\tau_A(p^{m}) \tau_B(p^{m}) }{ p^{m} }   \\
&\hspace{.25in}+ ( 1-p^{-\alpha-\beta} )\sum_{ 0\leq m<n<\infty  } \frac{\tau_A(p^{m}) \tau_B(p^{n}) }{ p^{m\beta} p^{n(1-\beta)} } + ( 1-p^{-\alpha-\beta} ) \sum_{ 0\leq n<m <\infty   } \frac{\tau_A(p^{m}) \tau_B(p^{n}) }{ p^{m(1-\alpha)} p^{n\alpha}  } \Bigg).
\end{split}
\end{equation}
Lemma~\ref{lem: tauseriesidentity} with $j=\ell=0$ implies
\begin{equation}\label{eqn: tauseriesidenappl1}
\sum_{ 0\leq m<n<\infty  } \frac{\tau_A(p^{m}) \tau_B(p^{n}) }{ p^{m\beta} p^{n(1-\beta)} } = \sum_{m=0}^{\infty} \frac{ \tau_{A\cup\{-\beta\}}(p^m) \tau_B(p^m) }{p^{m}} - \sum_{m=0}^{\infty} \frac{\tau_{A}(p^m) \tau_B(p^m) }{p^{m}}.
\end{equation}
Similarly, Lemma~\ref{lem: tauseriesidentity} with $A$ and $B$ interchanged, $\beta$ replaced by $\alpha$, and $j=\ell=0$ implies
\begin{equation}\label{eqn: tauseriesidenappl2}
\sum_{ 0\leq n<m <\infty   } \frac{\tau_A(p^{m}) \tau_B(p^{n}) }{ p^{m(1-\alpha)} p^{n\alpha}  } = \sum_{m=0}^{\infty} \frac{ \tau_{A}(p^m) \tau_{B\cup\{-\alpha\}}(p^m) }{p^{m}} - \sum_{m=0}^{\infty} \frac{\tau_{A}(p^m) \tau_B(p^m) }{p^{m}}.
\end{equation}
We complete the first $m$-sum in \eqref{eqn: mathfrakFdef} by adding and subtracting its $m=0$ term, and then insert \eqref{eqn: tauseriesidenappl1} and \eqref{eqn: tauseriesidenappl2} to deduce that
\begin{align}
\mathfrak{F}_p &= \prod_{\substack{ \hat{\alpha}\neq \alpha \\ \hat{\beta}\neq \beta}}\left(1-\frac{1}{p^{1+\hat{\alpha}+\hat{\beta}}} \right) \prod_{\hat{\alpha}\in A}\left(1-\frac{1}{p^{1+\hat{\alpha}-\alpha}} \right) \prod_{\hat{\beta}\in B} \left(1-\frac{1}{p^{1+\hat{\beta}-\beta}} \right)  \notag\\ 
&\hspace{.25in}\times \Bigg(  \left( 1-\frac{1}{p^{2-\alpha-\beta}}\right)\left(1+ \frac{p^{1-\alpha-\beta}-1}{p(p-1)}\right) - \left( 1+\frac{p^{-\alpha-\beta}}{p-1}-\frac{1}{p-1} \right)  \notag\\
&\hspace{.5in}+ \left( 2p^{-\alpha-\beta}-1+\frac{p^{-\alpha-\beta}}{p-1}-\frac{1}{p-1} \right) \sum_{m=0}^{\infty} \frac{\tau_A(p^{m}) \tau_B(p^{m}) }{ p^{m} }   \notag\\
&\hspace{.5in}+ ( 1-p^{-\alpha-\beta} )\sum_{m=0}^{\infty} \frac{ \tau_{A\cup\{-\beta\}}(p^m) \tau_B(p^m) }{p^{m}} + ( 1-p^{-\alpha-\beta} ) \sum_{m=0}^{\infty} \frac{ \tau_{A}(p^m) \tau_{B\cup\{-\alpha\}}(p^m) }{p^{m}} \Bigg). \label{eqn: mathfrakF2}
\end{align}
Observe that there is the factor $(1-1/p)^2$ in \eqref{eqn: mathfrakF2}. This factor is the product of the factor corresponding to $\hat{\alpha}=\alpha$ in the product over $\hat{\alpha}\in A$ and the factor corresponding to $\hat{\beta}=\beta$ in the product over $\hat{\beta}\in B$. We distribute $(1-1/p)$ among the terms in \eqref{eqn: mathfrakF2} and arrive at
\begin{equation}\label{eqn: mathfrakFtoSigma0}
\begin{split}
\mathfrak{F}_p &= \left( 1-\frac{1}{p}\right)\prod_{\substack{ \hat{\alpha}\neq \alpha \\ \hat{\beta}\neq \beta}}\left(1-\frac{1}{p^{1+\hat{\alpha}+\hat{\beta}}} \right) \prod_{\hat{\alpha}\neq \alpha }\left(1-\frac{1}{p^{1+\hat{\alpha}-\alpha}} \right) \prod_{\hat{\beta}\neq \beta} \left(1-\frac{1}{p^{1+\hat{\beta}-\beta}} \right)  \\ 
&\hspace{.25in}\times \Bigg(  \left( 1-\frac{1}{p^{2-\alpha-\beta}}\right)\left(1-\frac{1}{p}+ \frac{1}{p^{1+\alpha+\beta}}-\frac{1}{p^2}\right) - \left( 1-\frac{2}{p}+\frac{1}{p^{1+\alpha+\beta}}\right)  +\Sigma_0\Bigg),
\end{split}
\end{equation}
where $\Sigma_0$ is defined by
\begin{equation*}
\begin{split}
\Sigma_0 := & \left( 2p^{-\alpha-\beta}-1-\frac{p^{-\alpha-\beta}}{p} \right) \sum_{m=0}^{\infty} \frac{\tau_A(p^{m}) \tau_B(p^{m}) }{ p^{m} } \\
& + \bigg( 1-p^{-\alpha-\beta} -\frac{1}{p} +\frac{p^{-\alpha-\beta}}{p} \bigg)\Bigg( \sum_{m=0}^{\infty} \frac{ \tau_{A\cup\{-\beta\}}(p^m) \tau_B(p^m) }{p^{m}} + \sum_{m=0}^{\infty} \frac{ \tau_{A}(p^m) \tau_{B\cup\{-\alpha\}}(p^m) }{p^{m}}\Bigg).
\end{split}
\end{equation*}
Multiply out the products in the latter expression and rearrange the terms to write
\begin{equation*}
\begin{split}
\Sigma_0 &= \sum_{m=0}^{\infty} \frac{(1- p^{-\alpha-\beta})\tau_{A\cup\{-\beta\}}(p^m) \tau_B(p^m) +p^{-\alpha-\beta} \tau_A(p^m) \tau_B(p^m) }{p^m} \\
&\hspace{.25in}+ \sum_{m=0}^{\infty} \frac{(1- p^{-\alpha-\beta})\tau_{A}(p^m) \tau_{B\cup\{-\alpha\}}(p^m) +p^{-\alpha-\beta} \tau_A(p^m) \tau_B(p^m) - \tau_A(p^m)\tau_B(p^m) }{p^m} \\
&\hspace{.25in}- \sum_{m=0}^{\infty} \frac{(1- p^{-\alpha-\beta})\big( \tau_{A\cup\{-\beta\}}(p^m) \tau_B(p^m) + \tau_{A}(p^m) \tau_{B\cup\{-\alpha\}}(p^m)\big) + p^{-\alpha-\beta} \tau_A(p^m)\tau_B(p^m)}{p^{m+1}}.
\end{split}
\end{equation*}
We make a change of variables in the last $m$-sum on the right-hand side by replacing each instance of $m$ with $m-1$. To the resulting expression for $\Sigma_0$, we add
\begin{equation*}
\begin{split}
0 &= \sum_{m=0}^{\infty} \frac{p^{-\alpha} \tau_A(p^{m-1})\tau_B(p^m) + p^{-\beta} \tau_A(p^{m})\tau_B(p^{m-1})  }{p^m} \\
&\hspace{.25in}- \sum_{m=0}^{\infty} \frac{p^{-\alpha} \tau_A(p^{m-1})\tau_B(p^m) + p^{-\beta} \tau_A(p^{m})\tau_B(p^{m-1})  }{p^m}
\end{split}
\end{equation*}
and rearrange the terms to deduce that
\begin{equation}\label{eqn: Sigma0split}
\Sigma_0 = \sum_{m=0}^{\infty} \Big( D_{1,m}+D_{2,m} + D_{3,m}\Big) \frac{1}{p^m},
\end{equation}
where $D_{1,m}$, $D_{2,m}$, and $D_{3,m}$ are defined by
\begin{equation*}
\begin{split}
D_{1,m}&:= (1-p^{-\alpha-\beta}) \tau_{A\cup\{-\beta\}} (p^m)\tau_{B  } (p^m) +p^{-\alpha-\beta} \tau_A(p^m)\tau_{B  } (p^m) \\
&\hspace{.25in}-p^{-\beta} \tau_A(p^m) \tau_{B  } (p^{m-1}) -(1-p^{-\alpha-\beta}) \tau_{A\cup\{-\beta\}} (p^{m-1})\tau_{B  } (p^{m-1}),
\end{split}
\end{equation*}
\begin{equation*}
\begin{split}
D_{2,m} &:= (1-p^{-\alpha-\beta}) \tau_A (p^m) \tau_{B  \cup \{-\alpha \}} (p^m)  +p^{-\alpha-\beta} \tau_A (p^m)\tau_B(p^m)\\
&\hspace{.25in}-p^{-\alpha} \tau_A (p^{m-1})  \tau_B(p^m) -(1-p^{-\alpha-\beta}) \tau_A (p^{m-1}) \tau_{B  \cup \{-\alpha \}} (p^{m-1}),
\end{split}
\end{equation*}
and
\begin{equation*}
\begin{split}
D_{3,m} &:= -\tau_A (p^m)\tau_{B  } (p^m) + p^{-\alpha}\tau_A (p^{m-1})\tau_{B  } (p^m) \\
&\hspace{.25in}+ p^{-\beta}\tau_A (p^m)\tau_{B  } (p^{m-1}) -p^{-\alpha-\beta} \tau_A (p^{m-1})\tau_{B  } (p^{m-1}),
\end{split}
\end{equation*}
where we recall that $\tau_E(p^{-1})$ is defined to be zero for any multiset $E$. Now Lemma~\ref{lem: taulongidentity} with $j=\ell=m$ implies
\begin{equation}\label{eqn: D1msimplify}
D_{1,m} = \tau_{A\smallsetminus\{\alpha\}\cup \{-\beta\}} (p^m) \tau_{B\smallsetminus\{\beta\}} (p^{m}).
\end{equation}
Moreover, Lemma~\ref{lem: taulongidentity} with $A$ and $B$ interchanged and $j=\ell=m$ implies
\begin{equation}\label{eqn: D2msimplify}
D_{2,m} = \tau_{A\smallsetminus\{\alpha\}} (p^{m})\tau_{B\smallsetminus\{\beta\}\cup \{-\alpha\}} (p^m) .
\end{equation}
As for $D_{3,m}$, we may factor it and apply \eqref{eqn: tauremoveelement} to deduce that
\begin{equation*}
\begin{split}
D_{3,m}
& = -\big(\tau_A(p^m) -p^{-\alpha}\tau_A(p^{m-1}) \big) \big( \tau_B(p^m) -p^{-\beta}\tau_B(p^{m-1}) \big) \\
& = -\tau_{A\smallsetminus\{\alpha\}}(p^m) \tau_{B\smallsetminus\{\beta\}}(p^m).
\end{split}
\end{equation*}
From this, \eqref{eqn: D1msimplify}, \eqref{eqn: D2msimplify}, and Lemma~\ref{lem: CK3identity} with $j=\ell=m$, we arrive at
\begin{equation*}
\begin{split}
D_{1,m} &+ D_{2,m} + D_{3,m} \\
&= \tau_{A\smallsetminus\{\alpha\}\cup \{-\beta\}} (p^m) \tau_{B\smallsetminus\{\beta\} \cup\{-\alpha\}} (p^{m}) -p^{\alpha+\beta} \tau_{A\smallsetminus\{\alpha\}\cup \{-\beta\}}(p^{m-1}) \tau_{B\smallsetminus\{\beta\} \cup\{-\alpha\}} (p^{m-1}).
\end{split}
\end{equation*}
This and \eqref{eqn: Sigma0split} imply
\begin{equation*}
\begin{split}
\Sigma_0 &= \sum_{m=0}^{\infty} \frac{\tau_{A\smallsetminus\{\alpha\}\cup \{-\beta\}} (p^m) \tau_{B\smallsetminus\{\beta\} \cup\{-\alpha\}} (p^{m})}{p^m} \\
&\hspace{.25in}-p^{\alpha+\beta} \sum_{m=0}^{\infty} \frac{\tau_{A\smallsetminus\{\alpha\}\cup \{-\beta\}} (p^{m-1}) \tau_{B\smallsetminus\{\beta\} \cup\{-\alpha\}} (p^{m-1})}{p^m}.
\end{split}
\end{equation*}
We make a change of variables in the latter $m$-sum by replacing each instance of $m$ with $m+1$. The result is
\begin{equation*}
\begin{split}
\Sigma_0 = \bigg( 1-\frac{1}{p^{1-\alpha-\beta}}\bigg)\sum_{m=0}^{\infty} \frac{\tau_{A\smallsetminus\{\alpha\}\cup \{-\beta\}} (p^m) \tau_{B\smallsetminus\{\beta\} \cup\{-\alpha\}} (p^{m})}{p^m}.
\end{split}
\end{equation*}
We insert this into \eqref{eqn: mathfrakFtoSigma0} and arrive at
\begin{equation}\label{eqn: Sigma0evaluated}
\begin{split}
\mathfrak{F}_p &= \left( 1-\frac{1}{p}\right)\prod_{\substack{ \hat{\alpha}\neq \alpha \\ \hat{\beta}\neq \beta}}\left(1-\frac{1}{p^{1+\hat{\alpha}+\hat{\beta}}} \right) \prod_{\hat{\alpha}\neq \alpha }\left(1-\frac{1}{p^{1+\hat{\alpha}-\alpha}} \right) \prod_{\hat{\beta}\neq \beta} \left(1-\frac{1}{p^{1+\hat{\beta}-\beta}} \right)  \\ 
&\hspace{.25in}\times \Bigg(  \left( 1-\frac{1}{p^{2-\alpha-\beta}}\right)\left(1-\frac{1}{p}+ \frac{1}{p^{1+\alpha+\beta}}-\frac{1}{p^2}\right) - \left( 1-\frac{2}{p}+\frac{1}{p^{1+\alpha+\beta}}\right)  \\
&\hspace{.25in}+ \bigg( 1-\frac{1}{p^{1-\alpha-\beta}}\bigg) + \bigg( 1-\frac{1}{p^{1-\alpha-\beta}}\bigg)\sum_{m=1}^{\infty} \frac{\tau_{A\smallsetminus\{\alpha\}\cup \{-\beta\}} (p^m) \tau_{B\smallsetminus\{\beta\} \cup\{-\alpha\}} (p^{m})}{p^m}\Bigg),
\end{split}
\end{equation}
where we have separated the $m=0$ term from the $m$-sum. A direct calculation gives
\begin{equation*}
\begin{split}
\left( 1-\frac{1}{p^{2-\alpha-\beta}}\right)\left(1-\frac{1}{p}+ \frac{1}{p^{1+\alpha+\beta}}-\frac{1}{p^2}\right) - \left( 1-\frac{2}{p}+\frac{1}{p^{1+\alpha+\beta}}\right) + \bigg( 1-\frac{1}{p^{1-\alpha-\beta}}\bigg) \\
= \left( 1- \frac{1}{p^{1-\alpha-\beta}}\right)\left(1 +\frac{1}{p} \right)\left(1-\frac{1}{p^2} \right).
\end{split}
\end{equation*}
We insert this into \eqref{eqn: Sigma0evaluated} and then factor out $(1-p^{-1+\alpha+\beta})$ to deduce that
\begin{equation*}
\begin{split}
\mathfrak{F}_p = \left( 1-\frac{1}{p}\right)\bigg( 1-\frac{1}{p^{1-\alpha-\beta}}\bigg)\prod_{\substack{ \hat{\alpha}\neq \alpha \\ \hat{\beta}\neq \beta}}\left(1-\frac{1}{p^{1+\hat{\alpha}+\hat{\beta}}} \right) \prod_{\hat{\alpha}\neq \alpha }\left(1-\frac{1}{p^{1+\hat{\alpha}-\alpha}} \right) \prod_{\hat{\beta}\neq \beta} \left(1-\frac{1}{p^{1+\hat{\beta}-\beta}} \right)  \\ 
\times \Bigg( \left(1 +\frac{1}{p} \right)\left(1-\frac{1}{p^2} \right) + \sum_{m=1}^{\infty} \frac{\tau_{A\smallsetminus\{\alpha\}\cup \{-\beta\}} (p^m) \tau_{B\smallsetminus\{\beta\} \cup\{-\alpha\}} (p^{m})}{p^m}\Bigg).
\end{split}
\end{equation*}
The right-hand side is exactly the local factor corresponding to $p$ in the Euler product expression for $\mathcal{K}(0,0,2-\alpha-\beta)$ by the definition \eqref{eqn: 1swapeulerKdef}, because we are assuming that $p\nmid hk$.

We have now shown for each $p\nmid hk$ that the local factors corresponding to $p$ in the Euler product expressions of both sides of \eqref{eqn: GKeuleriden} agree. Our next task is to do the same for each $p|hk$. To this end, let $p|hk$ be given, and let $\mathfrak{G}_p$ denote the local factor corresponding to this $p$ in the Euler product expression for the left-hand side of \eqref{eqn: GKeuleriden}. Also, for brevity, for the rest of this proof we denote $h_p:=\ordp(h)$ and $k_p:=\ordp(k)$. With these notations, we see from the definition \eqref{eqn: R0eulerGdef} of $\mathcal{G}$ that $\mathfrak{G}_p$ is defined by
\begin{equation}\label{eqn: mathfrakGdef}
\begin{split}
\mathfrak{G}_p := p^{-(\frac{1}{2}-\alpha)h_p -(\frac{1}{2}-\beta)k_p }\prod_{\substack{ \hat{\alpha}\neq \alpha \\ \hat{\beta}\neq \beta}}\left(1-\frac{1}{p^{1+\hat{\alpha}+\hat{\beta} }} \right) \prod_{\hat{\alpha}\in A}\left(1-\frac{1}{p^{1+\hat{\alpha}-\alpha}} \right) \prod_{\hat{\beta}\in B} \left(1-\frac{1}{p^{1+\hat{\beta}-\beta}} \right) \\
\times \Bigg( \left( 1+\frac{p^{-\alpha-\beta}}{p-1}-\frac{1}{p-1} \right) \sum_{\substack{0\leq m,n<\infty \\ m+h_p= n + k_p } }\frac{\tau_A(p^{m}) \tau_B(p^{n}) }{ p^{m(1-\alpha)} p^{n(1-\beta)} p^{(-1+\alpha+\beta) \min\{m+h_p, n + k_p \} }}   \\
+ \left( 1-p^{-\alpha-\beta}  \right) \sum_{\substack{0\leq m,n<\infty \\ m+h_p\neq  n + k_p } } \frac{\tau_A(p^{m}) \tau_B(p^{n}) }{ p^{m(1-\alpha)} p^{n(1-\beta)} p^{(-1+\alpha+\beta) \min\{m+h_p, n + k_p \} }} \Bigg).
\end{split}
\end{equation}
If $m+h_p=n+k_p$, then
$$
(-1+\alpha+\beta) \min\{m+h_p,n+k_p \} = \left(-\frac{1}{2}+\alpha\right) (m+h_p) +\left( -\frac{1}{2}+\beta\right)(n+k_p),
$$
and so
\begin{equation}\label{eqn: tauhkseriessimplify}
\begin{split}
p^{-(\frac{1}{2}-\alpha)h_p -(\frac{1}{2}-\beta)k_p }
& \sum_{\substack{0\leq m,n<\infty \\ m+h_p= n + k_p } }\frac{\tau_A(p^{m}) \tau_B(p^{n}) }{ p^{m(1-\alpha)} p^{n(1-\beta)} p^{(-1+\alpha+\beta) \min\{m+h_p, n + k_p \} }}  \\
& = \sum_{\substack{0\leq m,n<\infty \\ m+h_p= n + k_p } }\frac{\tau_A(p^{m}) \tau_B(p^{n}) }{ p^{\frac{m}{2}+\frac{n}{2}}}.
\end{split}
\end{equation}
If $m+h_p<n+k_p$, then $\min\{m+h_p,n+k_p \}=m+h_p$ and it follows from Lemma~\ref{lem: tauseriesidentity} with $j=h_p$ and $\ell=k_p$ that
\begin{equation}\label{eqn: tauseriesidenappl3}
\begin{split}
p^{-(\frac{1}{2}-\alpha)h_p -(\frac{1}{2}-\beta)k_p } \sum_{\substack{0\leq m,n<\infty \\ m+h_p< n + k_p } }\frac{\tau_A(p^{m}) \tau_B(p^{n}) }{ p^{m(1-\alpha)} p^{n(1-\beta)} p^{(-1+\alpha+\beta) \min\{m+h_p, n + k_p \} }} \\
= \sum_{\substack{ 0\leq m,n<\infty \\ m+h_p =n+k_p}} \frac{ \tau_{A\cup\{-\beta\}}(p^m) \tau_B(p^n) - \tau_{A}(p^m) \tau_B(p^n) }{p^{\frac{m}{2}+\frac{n}{2}}}.
\end{split}
\end{equation}
Similarly, Lemma~\ref{lem: tauseriesidentity} with $A$ and $B$ interchanged, $\beta$ replaced by $\alpha$, $j=k_p$, and $\ell=h_p$ implies
\begin{equation*}
\begin{split}
p^{-(\frac{1}{2}-\alpha)h_p -(\frac{1}{2}-\beta)k_p } \sum_{\substack{0\leq m,n<\infty \\ m+h_p > n + k_p } }\frac{\tau_A(p^{m}) \tau_B(p^{n}) }{ p^{m(1-\alpha)} p^{n(1-\beta)} p^{(-1+\alpha+\beta) \min\{m+h_p, n + k_p \} }}\\
= \sum_{\substack{ 0\leq m,n<\infty \\ m+h_p =n+k_p}} \frac{ \tau_A(p^m) \tau_{B\cup\{-\alpha\}}(p^n)  - \tau_{A}(p^m) \tau_B(p^n) }{p^{\frac{m}{2}+\frac{n}{2}}}.
\end{split}
\end{equation*}
It follows from this, \eqref{eqn: mathfrakGdef}, \eqref{eqn: tauhkseriessimplify}, and \eqref{eqn: tauseriesidenappl3} that
\begin{equation}\label{eqn: mathfrakG2}
\begin{split}
\mathfrak{G}_p &= \prod_{\substack{ \hat{\alpha}\neq \alpha \\ \hat{\beta}\neq \beta}}\left(1-\frac{1}{p^{1+\hat{\alpha}+\hat{\beta} }} \right) \prod_{\hat{\alpha}\in A}\left(1-\frac{1}{p^{1+\hat{\alpha}-\alpha}} \right) \prod_{\hat{\beta}\in B} \left(1-\frac{1}{p^{1+\hat{\beta}-\beta}} \right) \\
&\hspace{.25in}\times \Bigg( \left( 2p^{-\alpha-\beta}- 1+\frac{p^{-\alpha-\beta}}{p-1}-\frac{1}{p-1} \right) \sum_{\substack{0\leq m,n<\infty \\ m+h_p= n + k_p } }\frac{\tau_A(p^{m}) \tau_B(p^{n}) }{ p^{\frac{m}{2} + \frac{n}{2} }}   \\
&\hspace{.25in}+ \left( 1-p^{-\alpha-\beta}  \right)  \sum_{\substack{ 0\leq m,n<\infty \\ m+h_p =n+k_p}} \frac{ \tau_{A\cup\{-\beta\}}(p^m) \tau_B(p^n) + \tau_{A}(p^m) \tau_{B\cup\{-\alpha\}}(p^n) }{p^{\frac{m}{2}+\frac{n}{2}}} \Bigg).
\end{split}
\end{equation}
There is the factor $(1-1/p)^2$ in \eqref{eqn: mathfrakG2} by the same reason mentioned below \eqref{eqn: mathfrakF2}. We distribute $(1-1/p)$ among the terms in \eqref{eqn: mathfrakG2} and deduce that
\begin{equation}\label{eqn: mathfrakGtoSigma1}
\mathfrak{G}_p = \Sigma_1\times \left( 1-\frac{1}{p}\right)\prod_{\substack{ \hat{\alpha}\neq \alpha \\ \hat{\beta}\neq \beta}}\left(1-\frac{1}{p^{1+\hat{\alpha}+\hat{\beta} }} \right) \prod_{\hat{\alpha}\neq \alpha}\left(1-\frac{1}{p^{1+\hat{\alpha}-\alpha}} \right) \prod_{\hat{\beta}\neq \beta} \left(1-\frac{1}{p^{1+\hat{\beta}-\beta}} \right),
\end{equation}
where $\Sigma_1$ is defined by
\begin{equation*}
\begin{split}
\Sigma_1 &:= \left( 2p^{-\alpha-\beta}- 1-\frac{p^{-\alpha-\beta}}{p} \right) \sum_{\substack{0\leq m,n<\infty \\ m+h_p= n + k_p } }\frac{\tau_A(p^{m}) \tau_B(p^{n}) }{ p^{\frac{m}{2} + \frac{n}{2} }}   \\
&\hspace{.25in}+ \left( 1-p^{-\alpha-\beta}  -\frac{1}{p} + \frac{ p^{-\alpha-\beta} }{p} \right)  \sum_{\substack{ 0\leq m,n<\infty \\ m+h_p =n+k_p}} \frac{ \tau_{A\cup\{-\beta\}}(p^m) \tau_B(p^n) + \tau_{A}(p^m) \tau_{B\cup\{-\alpha\}}(p^n) }{p^{\frac{m}{2}+\frac{n}{2}}}.
\end{split}
\end{equation*}
Multiply out the products and rearrange the terms to write $\Sigma_1$ as
\begin{equation*}
\begin{split}
&\Sigma_1 = \sum_{\substack{0\leq m,n<\infty \\ m+h_p= n + k_p } } \frac{(1- p^{-\alpha-\beta})\tau_{A\cup\{-\beta\}}(p^m) \tau_B(p^n) +p^{-\alpha-\beta} \tau_A(p^m) \tau_B(p^n) }{p^{\frac{m}{2}+\frac{n}{2} }} \\
&+ \sum_{\substack{0\leq m,n<\infty \\ m+h_p= n + k_p } } \frac{(1- p^{-\alpha-\beta})\tau_{A}(p^m) \tau_{B\cup\{-\alpha\}}(p^n) +p^{-\alpha-\beta} \tau_A(p^m) \tau_B(p^n) - \tau_A(p^m)\tau_B(p^n) }{p^{\frac{m}{2}+\frac{n}{2} }} \\
&- \sum_{\substack{0\leq m,n<\infty \\ m+h_p= n + k_p } } \frac{(1- p^{-\alpha-\beta})\big( \tau_{A\cup\{-\beta\}}(p^m) \tau_B(p^n) + \tau_{A}(p^m) \tau_{B\cup\{-\alpha\}}(p^n)\big) + p^{-\alpha-\beta} \tau_A(p^m)\tau_B(p^n)}{p^{1 + \frac{m}{2}+\frac{n}{2} }}.
\end{split}
\end{equation*}
We make changes of variables in the last $m,n$-sum on the right-hand side by replacing each instance of $m$ with $m-1$ and each instance of $n$ with $n-1$. To the resulting expression for $\Sigma_1$, we add
\begin{equation*}
\begin{split}
0 &= \sum_{\substack{0\leq m,n<\infty \\ m+h_p= n + k_p } } \frac{p^{-\alpha} \tau_A(p^{m-1})\tau_B(p^n) + p^{-\beta} \tau_A(p^{m})\tau_B(p^{n-1})  }{p^{\frac{m}{2}+\frac{n}{2}}} \\
&\hspace{.25in}- \sum_{\substack{0\leq m,n<\infty \\ m+h_p= n + k_p } } \frac{p^{-\alpha} \tau_A(p^{m-1})\tau_B(p^n) + p^{-\beta} \tau_A(p^{m})\tau_B(p^{n-1})  }{p^{\frac{m}{2}+\frac{n}{2}}}
\end{split}
\end{equation*}
and rearrange the terms to deduce that
\begin{equation}\label{eqn: Sigma1split}
\Sigma_1 = \sum_{\substack{0\leq m,n<\infty \\ m+h_p= n + k_p } } \Big( D_{1,m,n}+D_{2,m,n} + D_{3,m,n}\Big) \frac{1}{p^{\frac{m}{2}+\frac{n}{2}}},
\end{equation}
where $D_{1,m,n}$, $D_{2,m,n}$, and $D_{3,m,n}$ are defined by
\begin{equation*}
\begin{split}
D_{1,m,n}&:= (1-p^{-\alpha-\beta}) \tau_{A\cup\{-\beta\}} (p^m)\tau_{B  } (p^n) +p^{-\alpha-\beta} \tau_A(p^m)\tau_{B  } (p^n) \\
&\hspace{.25in}-p^{-\beta} \tau_A(p^m) \tau_{B  } (p^{n-1}) -(1-p^{-\alpha-\beta}) \tau_{A\cup\{-\beta\}} (p^{m-1})\tau_{B  } (p^{n-1}),
\end{split}
\end{equation*}
\begin{equation*}
\begin{split}
D_{2,m,n} &:= (1-p^{-\alpha-\beta}) \tau_A (p^m) \tau_{B  \cup \{-\alpha \}} (p^n)  +p^{-\alpha-\beta} \tau_A (p^m)\tau_B(p^n)\\
&\hspace{.25in}-p^{-\alpha} \tau_A (p^{m-1})  \tau_B(p^n) -(1-p^{-\alpha-\beta}) \tau_A (p^{m-1}) \tau_{B  \cup \{-\alpha \}} (p^{n-1}),
\end{split}
\end{equation*}
and
\begin{equation*}
\begin{split}
D_{3,m,n} &:= -\tau_A (p^m)\tau_{B  } (p^n) + p^{-\alpha}\tau_A (p^{m-1})\tau_{B  } (p^n) \\
&\hspace{.25in}+ p^{-\beta}\tau_A (p^m)\tau_{B  } (p^{n-1}) -p^{-\alpha-\beta} \tau_A (p^{m-1})\tau_{B  } (p^{n-1}),
\end{split}
\end{equation*}
where we recall that $\tau_E(p^{-1})$ is defined to be zero for any multiset $E$. Now Lemma~\ref{lem: taulongidentity} with $j=m$ and $\ell=n$ implies
\begin{equation}\label{eqn: D1mnsimplify}
D_{1,m,n} = \tau_{A\smallsetminus \{\alpha\}\cup\{-\beta\}} (p^m) \tau_{B\smallsetminus\{\beta\}} (p^n).
\end{equation}
Moreover, Lemma~\ref{lem: taulongidentity} with $A$ and $B$ interchanged, $j=n$, and $\ell=m$ implies
\begin{equation}\label{eqn: D2mnsimplify}
D_{2,m,n} = \tau_{A\smallsetminus\{\alpha\}} (p^m) \tau_{B\smallsetminus \{\beta\}\cup\{-\alpha\}} (p^n).
\end{equation}
As for $D_{3,m,n}$, we may factor it and apply \eqref{eqn: tauremoveelement} to deduce that
\begin{equation*}
\begin{split}
D_{3,m,n}
& = -\big(\tau_A(p^m) -p^{-\alpha}\tau_A(p^{m-1}) \big) \big( \tau_B(p^n) -p^{-\beta}\tau_B(p^{n-1}) \big) \\
& = -\tau_{A\smallsetminus\{\alpha\}}(p^m) \tau_{B\smallsetminus\{\beta\}}(p^n).
\end{split}
\end{equation*}
From this, \eqref{eqn: D1mnsimplify}, \eqref{eqn: D2mnsimplify}, and Lemma~\ref{lem: CK3identity} with $j=m$ and $\ell=n$, we arrive at
\begin{equation*}
\begin{split}
D_{1,m,n} &+ D_{2,m,n} + D_{3,m,n} \\
&= \tau_{A\smallsetminus\{\alpha\}\cup \{-\beta\}} (p^m) \tau_{B\smallsetminus\{\beta\} \cup\{-\alpha\}} (p^{n}) -p^{\alpha+\beta} \tau_{A\smallsetminus\{\alpha\}\cup \{-\beta\}}(p^{m-1}) \tau_{B\smallsetminus\{\beta\} \cup\{-\alpha\}} (p^{n-1}).
\end{split}
\end{equation*}
This and \eqref{eqn: Sigma1split} imply
\begin{equation*}
\begin{split}
\Sigma_1 &= \sum_{\substack{0\leq m,n<\infty \\ m+h_p= n + k_p } } \frac{\tau_{A\smallsetminus\{\alpha\}\cup \{-\beta\}} (p^m) \tau_{B\smallsetminus\{\beta\} \cup\{-\alpha\}} (p^{n})}{p^{\frac{m}{2}+\frac{n}{2}}} \\
&\hspace{.25in}- p^{\alpha+\beta} \sum_{\substack{0\leq m,n<\infty \\ m+h_p= n + k_p } } \frac{\tau_{A\smallsetminus\{\alpha\}\cup \{-\beta\}}(p^{m-1}) \tau_{B\smallsetminus\{\beta\} \cup\{-\alpha\}} (p^{n-1}) }{p^{\frac{m}{2}+\frac{n}{2}}}.
\end{split}
\end{equation*}
We make a change of variables in the latter $m,n$-sum by replacing each instance of $m$ with $m+1$ and each instance of $n$ with $n+1$. The result is
\begin{equation*}
\Sigma_1 = \bigg( 1-\frac{1}{p^{1-\alpha-\beta}}\bigg)\sum_{\substack{0\leq m,n<\infty \\ m+h_p= n + k_p } } \frac{\tau_{A\smallsetminus\{\alpha\}\cup \{-\beta\}} (p^m) \tau_{B\smallsetminus\{\beta\} \cup\{-\alpha\}} (p^{n})}{p^{\frac{m}{2}+\frac{n}{2}}}.
\end{equation*}
We insert this into \eqref{eqn: mathfrakGtoSigma1} and arrive at
\begin{equation*}
\begin{split}
\mathfrak{G}_p &= \left( 1-\frac{1}{p}\right)\bigg( 1-\frac{1}{p^{1-\alpha-\beta}}\bigg) \prod_{\substack{ \hat{\alpha}\neq \alpha \\ \hat{\beta}\neq \beta}}\left(1-\frac{1}{p^{1+\hat{\alpha}+\hat{\beta} }} \right) \prod_{\hat{\alpha}\neq \alpha}\left(1-\frac{1}{p^{1+\hat{\alpha}-\alpha}} \right) \prod_{\hat{\beta}\neq \beta} \left(1-\frac{1}{p^{1+\hat{\beta}-\beta}} \right) \\
&\hspace{.25in}\times \sum_{\substack{0\leq m,n<\infty \\ m+h_p= n + k_p } } \frac{\tau_{A\smallsetminus\{\alpha\}\cup \{-\beta\}} (p^m) \tau_{B\smallsetminus\{\beta\} \cup\{-\alpha\}} (p^{n})}{p^{\frac{m}{2}+\frac{n}{2}}}.
\end{split}
\end{equation*}
The right-hand side is exactly the local factor corresponding to $p$ in the Euler product expression for $\mathcal{K}(0,0,2-\alpha-\beta)$ by the definition \eqref{eqn: 1swapeulerKdef}, because we are assuming that $p|hk$.

We have now shown for each $p$ that the local factors corresponding to $p$ in the Euler product expressions of both sides of \eqref{eqn: GKeuleriden} agree. This completes the proof of Lemma~\ref{lem: GKeuleridentity}.
\end{proof}

We will also use the following variant and consequence of Lemma~\ref{lem: GKeuleridentity}.

\begin{lemma}\label{lem: GKeuleridentity2}
Let $\alpha,\alpha^*\in A$ and $\beta,\beta^*\in B$. Suppose that $h$ and $k$ are positive integers. If $\mathcal{G}$ is defined by \eqref{eqn: R0eulerGdef} and $\mathcal{K}$ by \eqref{eqn: 1swapeulerKdef}, then
\begin{equation*}
\begin{split}
h^{-\frac{1}{2}+\alpha} k^{-\frac{1}{2}+\beta-\alpha^*-\beta^*} &\mathcal{G}(2-\alpha-\beta+\alpha^*+\beta^*,\alpha,\beta;A,B,h,k) \\
&\hspace{.25in}= \mathcal{K}(0,-\alpha^*-\beta^*,2-\alpha-\beta+\alpha^*+\beta^*;A,B,\alpha,\beta,h,k).
\end{split}
\end{equation*}
\end{lemma}
\begin{proof}
The definition \eqref{eqn: R0eulerGdef} implies
\begin{align}
\mathcal{G}(2
& -\alpha-\beta+\alpha^*+\beta^*,\alpha,\beta;A,B,h,k) \notag\\
& = \prod_{ p|hk } \Bigg\{ \prod_{\substack{ \hat{\alpha}\neq \alpha \\ \hat{\beta}\neq \beta}}\left(1-\frac{1}{p^{1+\hat{\alpha}+\hat{\beta} -\alpha^*-\beta^*}} \right) \prod_{\hat{\alpha}\in A}\left(1-\frac{1}{p^{1+\hat{\alpha}-\alpha}} \right) \prod_{\hat{\beta}\in B} \left(1-\frac{1}{p^{1+\hat{\beta}-\beta}} \right) \notag\\
&\hspace{.25in} \times \Bigg( \sum_{\substack{0\leq m,n<\infty \\ m+\ordp(h)= n + \ordp(k) } }\frac{\tau_A(p^{m}) \tau_B(p^{n}) \left( 1+\frac{p^{-\alpha-\beta+\alpha^*+\beta^*}}{p-1}-\frac{1}{p-1} \right)}{ p^{m(1-\alpha)} p^{n(1-\beta)} p^{(-1+\alpha+\beta-\alpha^*-\beta^*) \min\{m+\ordp(h), n + \ordp(k) \} }}   \notag\\
&\hspace{.25in} + \sum_{\substack{0\leq m,n<\infty \\ m+\ordp(h)\neq  n + \ordp(k) } } \frac{\tau_A(p^{m}) \tau_B(p^{n}) \big( 1-p^{-\alpha-\beta+\alpha^*+\beta^*}  \big)}{ p^{m(1-\alpha)} p^{n(1-\beta)} p^{(-1+\alpha+\beta-\alpha^*-\beta^*) \min\{m+\ordp(h), n + \ordp(k) \} }} \Bigg)\Bigg\} \notag\\
& \times \prod_{ p\nmid hk } \Bigg\{ \prod_{\substack{ \hat{\alpha}\neq \alpha \\ \hat{\beta}\neq \beta}}\left(1-\frac{1}{p^{1+\hat{\alpha}+\hat{\beta} -\alpha^*-\beta^*}} \right) \prod_{\hat{\alpha}\in A}\left(1-\frac{1}{p^{1+\hat{\alpha}-\alpha}} \right) \prod_{\hat{\beta}\in B} \left(1-\frac{1}{p^{1+\hat{\beta}-\beta}} \right)  \notag\\ 
&\hspace{.25in} \times \Bigg(  \left( 1-\frac{1}{p^{2-\alpha-\beta+\alpha^*+\beta^*}}\right)\left(1+ \frac{p^{1-\alpha-\beta+\alpha^*+\beta^*}-1}{(p-1)}\right)  \notag\\
&\hspace{.5in} + \sum_{m=1}^{\infty} \frac{\tau_A(p^{m}) \tau_B(p^{m}) \left( 1+\frac{p^{-\alpha-\beta+\alpha^*+\beta^*}}{ p-1 }-\frac{1}{p-1} \right)}{ p^{m(1-\alpha^*-\beta^*  )} }   \notag\\
&\hspace{.5in} + \sum_{ 0\leq m<n<\infty  } \frac{\tau_A(p^{m}) \tau_B(p^{n}) \big( 1-p^{-\alpha-\beta+\alpha^*+\beta^*}  \big)}{ p^{m(\beta-\alpha^*-\beta^*)} p^{n(1-\beta)} } \notag\\
&\hspace{.5in} + \sum_{ 0\leq n<m <\infty   } \frac{\tau_A(p^{m}) \tau_B(p^{n}) \big( 1-p^{-\alpha-\beta+\alpha^*+\beta^*}  \big)}{ p^{m(1-\alpha)} p^{n(\alpha-\alpha^*-\beta^*)}  } \Bigg)\Bigg\}, \label{eqn: G2aba'b'}
\end{align}
with the product absolutely convergent by Lemma~\ref{lem: R0eulerbound}. Now
\begin{equation*}
\prod_{\hat{\beta}\in B } \left(1-\frac{1}{p^{1+\hat{\beta}-\beta}} \right) = \prod_{\gamma\in B_{-\alpha^*-\beta^*}} \left(1-\frac{1}{p^{1+\gamma-\beta+\alpha^*+\beta^*}} \right),
\end{equation*}
while \eqref{eqn: taufactoringidentity} implies
\begin{equation*}
\frac{ \tau_B(p^{n}) }{ p^{n(1-\beta)} } = \frac{ \tau_{B_{-\alpha^*-\beta^*}}(p^{n}) }{ p^{n(1-\beta+\alpha^*+\beta^*)} }
\end{equation*}
and
\begin{equation*}
\frac{ \tau_B(p^{m}) }{ p^{m(1-\alpha^*-\beta^*)} } = \frac{ \tau_{B_{-\alpha^*-\beta^*}}(p^{m}) }{ p^{m}}.
\end{equation*}
It follows from these, \eqref{eqn: G2aba'b'}, and the definition \eqref{eqn: R0eulerGdef} of $\mathcal{G}$ that
\begin{equation}\label{eqn: GKeuleriden2b}
\begin{split}
\mathcal{G}&(2-\alpha-\beta+\alpha^*+\beta^*,\alpha,\beta;A,B,h,k) \\
&\hspace{.25in}= \mathcal{G}(2-\alpha-\beta+\alpha^*+\beta^*,\alpha,\beta-\alpha^*-\beta^*;A,B_{-\alpha^*-\beta^*},h,k).
\end{split}
\end{equation}
Lemma~\ref{lem: GKeuleridentity} with $B$ replaced by $B_{-\alpha^*-\beta^*}$ and $\beta$ replaced by $\beta-\alpha^*-\beta^*$ implies
\begin{equation}\label{eqn: GKeuleriden2a}
\begin{split}
h^{-\frac{1}{2}+\alpha} k^{-\frac{1}{2}+\beta-\alpha^*-\beta^*} \mathcal{G}(2-\alpha-\beta+\alpha^*+\beta^*,\alpha,\beta-\alpha^*-\beta^*;A,B_{-\alpha^*-\beta^*},h,k) \\
= \mathcal{K}(0,0,2-\alpha-\beta+\alpha^*+\beta^*;A,B_{-\alpha^*-\beta^*},\alpha,\beta-\alpha^*-\beta^*,h,k).
\end{split}
\end{equation}
To see that the right-hand side is the same as
$$
\mathcal{K}(0,-\alpha^*-\beta^*,2-\alpha-\beta+\alpha^*+\beta^*;A,B,\alpha,\beta,h,k),
$$
we make the following observations. If $w=2-\alpha-\beta+\alpha^*+\beta^*$ and $s_1=s_2=0$, then
$$
w-1+\alpha+s_1+(\beta-\alpha^*-\beta^*)+s_2= 1,
$$
$$
A_{s_1} \smallsetminus\{\alpha+s_1\} \cup\{-\beta+\alpha^*+\beta^*-s_2\} = A  \smallsetminus\{\alpha \} \cup\{-\beta+\alpha^*+\beta^* \},
$$
and
$$
\big(B_{-\alpha^*-\beta^*}\big)_{s_2} \smallsetminus \{\beta-\alpha^*-\beta^*+s_2\} \cup\{-\alpha-s_1\} = B_{-\alpha^*-\beta^*}\smallsetminus \{\beta-\alpha^*-\beta^*\} \cup\{-\alpha\}.
$$
On the other hand, if $w=2-\alpha-\beta+\alpha^*+\beta^*$, $s_1=0$, and $s_2=-\alpha^*-\beta^*$, then
$$
w-1+\alpha+s_1+\beta+s_2= 1,
$$
$$
A_{s_1} \smallsetminus\{\alpha+s_1\} \cup\{-\beta-s_2\} = A  \smallsetminus\{\alpha \} \cup\{-\beta+\alpha^*+\beta^* \},
$$
and
$$
B_{s_2} \smallsetminus \{\beta+s_2\} \cup\{-\alpha-s_1\} = B_{-\alpha^*-\beta^*}\smallsetminus \{\beta-\alpha^*-\beta^*\} \cup\{-\alpha\}.
$$
These observations and the definition \eqref{eqn: 1swapeulerKdef} of $\mathcal{K}$ imply that
\begin{equation*}
\begin{split}
\mathcal{K}(0,0,2-\alpha-\beta+\alpha^*+\beta^*;A,B_{-\alpha^*-\beta^*},\alpha,\beta-\alpha^*-\beta^*,h,k) \\
= \mathcal{K}(0,-\alpha^*-\beta^*,2-\alpha-\beta+\alpha^*+\beta^*;A,B,\alpha,\beta,h,k).
\end{split}
\end{equation*}
From this, \eqref{eqn: GKeuleriden2b}, and \eqref{eqn: GKeuleriden2a}, we arrive at Lemma~\ref{lem: GKeuleridentity2}.
\end{proof}

The special case of Lemma~\ref{lem: GKeuleridentity2} with $\alpha^*=\alpha$ and $\beta^*=\beta$ implies that
\begin{equation}\label{eqn: GKeuleriden3}
\begin{split}
h^{-\frac{1}{2}+\alpha} k^{-\frac{1}{2}-\alpha} \mathcal{G}(2,\alpha,\beta) 
= \mathcal{K}(0,-\alpha -\beta ,2).
\end{split}
\end{equation}
As a side note, we mention that \eqref{eqn: GKeuleriden3} may be proved directly from the definitions \eqref{eqn: R0eulerGdef} of $\mathcal{G}$ and \eqref{eqn: 1swapeulerKdef} of $\mathcal{K}$ by using the identity
$$
\frac{\tau_B(p^n)}{p^{n(-\alpha-\beta)}} = \tau_{B_{-\alpha-\beta}}(p^n),
$$
which follows from \eqref{eqn: taufactoringidentity}, and observing that if $m+\ordp(h)=n+\ordp(k)$ then
\begin{equation*}
\begin{split}
p^{m(1-\alpha)}p^{n(1-\beta)}p^{-\min\{m+\ordp(h),n+\ordp(k)\}}
& = p^{m(1-\alpha)}p^{n(1-\beta)}p^{-\frac{1}{2}(m+\ordp(h))-\frac{1}{2}(n+\ordp(k))} \\
& = p^{-(\frac{1}{2}-\alpha)\ordp(h) -(\frac{1}{2}+\alpha)\ordp(k)+\frac{m}{2}+n(\frac{1}{2}-\alpha-\beta)}
\end{split}
\end{equation*}
because $\alpha(\ordp(h)-\ordp(k))=\alpha(n-m)$.

We are now ready to match each residue on the right-hand side of \eqref{eqn: U2residues} with a residue on the right-hand side of \eqref{eqn: 1swapsready} in such a way that corresponding residues are equal. The identity \eqref{eqn: GKeuleriden3} implies that
\begin{equation*}
\begin{split}
h^{-\frac{1}{2}-\beta-s} k^{-\frac{1}{2}+\beta+s} \mathcal{G}(2,\alpha,\beta)
&= \left( \frac{h}{k}\right)^{-s-\alpha-\beta}h^{-\frac{1}{2}+\alpha} k^{-\frac{1}{2}-\alpha} \mathcal{G}(2,\alpha,\beta)\\
& = \left( \frac{h}{k}\right)^{-s-\alpha-\beta}\mathcal{K}(0,-\alpha -\beta ,2).
\end{split}
\end{equation*}
From this, \eqref{eqn: J23evaluated}, and \eqref{eqn: R1evaluated}, we deduce that
\begin{equation}\label{eqn: R1andJ23}
R_1 = J_{23}+ O\big( (Xhk)^{\varepsilon}k^{1/2} Q^{-96}\big).
\end{equation}
Now from \eqref{eqn: J11evaluated}, \eqref{eqn: R21evaluated}, and Lemma~\ref{lem: GKeuleridentity}, we immediately see that
\begin{equation}\label{eqn: R21andJ11}
R_{21} = J_{11}+ O\big( (hk)^{\varepsilon}(h,k)^{1/2} Q^{-96}\big).
\end{equation}
Next, \eqref{eqn: J21evaluated}, \eqref{eqn: R22evaluated}, and \eqref{eqn: GKeuleriden3} imply
\begin{equation}\label{eqn: R22andJ21}
R_{22} = J_{21}+ O\big( (Xhk)^{\varepsilon}(h,k)^{1/2} Q^{-96}\big).
\end{equation}
From \eqref{eqn: J31evaluated}, \eqref{eqn: R23evaluated}, and Lemma~\ref{lem: GKeuleridentity2} with $\alpha^*=\alpha'$ and $\beta^*=\beta'$, we deduce that
\begin{equation}\label{eqn: R23andJ31}
R_{23} = J_{31}+ O\big( (Xhk)^{\varepsilon}(h,k)^{1/2} Q^{-96}\big).
\end{equation}
Finally, \eqref{eqn: J33evaluated}, \eqref{eqn: R3evaluated}, and Lemma~\ref{lem: GKeuleridentity2} with $\alpha^*=\alpha'$ and $\beta^*=\beta'$ imply
\begin{equation*}
R_{3} = J_{33}+ O\big( (Xhk)^{\varepsilon}k^{1/2} Q^{-96}\big).
\end{equation*}
From this, \eqref{eqn: R1andJ23}, \eqref{eqn: R21andJ11}, \eqref{eqn: R22andJ21}, \eqref{eqn: R23andJ31}, and \eqref{eqn: U2residues}, we arrive at
\begin{equation*}
\begin{split}
\mathcal{U}^2(h,k) &= J_{23}+J_{11}+J_{21}+J_{31}+J_{33} + O \bigg( \bigg( Q^{1+\varepsilon} + \frac{Q^2}{C^{1-\varepsilon}}\bigg) \frac{(Xhk)^{\varepsilon}(h,k)}{\sqrt{hk}}  \bigg)  \\
&\hspace{.25in}+ O\Big( X^{\varepsilon} Q^{\frac{3}{2}} h^{\varepsilon}k^{\varepsilon} + X^{-\frac{1}{2}+\varepsilon} Q^{\frac{5}{2}} (hk)^{\varepsilon} + (XChk)^{\varepsilon} hk X^2 Q^{-96}\Big).
\end{split}
\end{equation*}
From this and \eqref{eqn: 1swapsready}, we conclude that
\begin{equation}\label{eqn: U2is1swap}
\begin{split}
\mathcal{U}^2(h,k) &= \mathcal{I}_1^*(h,k) + O \bigg( \bigg( Q^{1+\varepsilon} + \frac{Q^2}{C^{1-\varepsilon}}\bigg) \frac{(Xhk)^{\varepsilon}(h,k)}{\sqrt{hk}}  \bigg)   \\
&\hspace{.25in}+ O\big( X^{-\frac{1}{2}+\varepsilon} Q^{\frac{5}{2}} (hk)^{\varepsilon} + X^{\varepsilon} Q^{\frac{3}{2}+\varepsilon} (hk)^{\varepsilon} + (XChk)^{\varepsilon} hk X^2 Q^{-96}\big).
\end{split}
\end{equation}

\section{The error term \texorpdfstring{$\mathcal{U}^r(h,k)$}{Ur(h,k)}}\label{sec: Ur}
Recall that $\lambda_1,\lambda_2,\dots$ are arbitrary complex numbers such that $\lambda_h \ll_{\varepsilon} h^{\varepsilon}$ for all $\varepsilon>0$. In this section, we bound the sum 
$$
\sum_{h,k\leq Q^{\vartheta}} \frac{\lambda_h \overline{\lambda_k}}{\sqrt{hk}}\mathcal{U}^r(h,k),
$$
where $\mathcal{U}^r(h,k)$ is defined by \eqref{eqn: Ur}. The majority of the work that follows consists of preparing the above sum for an eventual application of the large sieve.

We begin by showing that the terms in \eqref{eqn: Ur} that have sufficiently large $a\ell$ are zero. Since the support of $W$ is compact and contained in $(0,\infty)$, the summand in the definition \eqref{eqn: Ur} of $\mathcal{U}^r(h,k)$ is zero unless $|mh\pm nk|\asymp  g\ell Q/c$, which implies that either $mh \gg g\ell Q/c$ or $nk \gg g\ell Q/c$. Since $a|g$, $c\leq C$, and $h,k\leq Q^{\vartheta}$, this means that the summand in \eqref{eqn: Ur} is zero unless
$$
m \gg \frac{Qg \ell}{hc} \geq \frac{Qa\ell}{CQ^{\vartheta}} \ \ \ \ \text{or} \ \ \ \ n  \gg \frac{Qg\ell}{kc}\geq \frac{Qa\ell}{CQ^{\vartheta}} .
$$
Now $V(m/X)V(n/X)=0$ except if $m,n\ll X$. Thus the summand in \eqref{eqn: Ur} is zero unless
\begin{equation}\label{eqn: largeal}
X \gg \frac{Qa\ell}{CQ^{\vartheta}}.
\end{equation}
In other words, the terms in the definition \eqref{eqn: Ur} of $\mathcal{U}^r(h,k)$ are zero unless $a\ell \ll  XCQ^{\vartheta-1}$.

We next show that the terms in \eqref{eqn: Ur} that have sufficiently large $ae\ell$ are negligible. We first consider the terms that have $mh/{g} \equiv \mp {nk}/{g}$ (mod~$ae\ell$). In this case, $|mh \pm nk|/g$ is a multiple of $ae\ell$ that is not zero because $mh\neq nk$. Thus $|mh \pm nk|/g \geq ae\ell$, and the triangle inequality implies that either $mh/g \geq ae\ell/2$ or $nk/g \geq ae\ell/2$. Since $h,k\leq Q^{\vartheta}$ and $g=(mh,nk)\geq 1$, these lower bounds imply that either $ae\ell \ll m Q^{\vartheta}$ or $ae\ell \ll n Q^{\vartheta}$. Hence, using the support of $V$ in the same way we deduced \eqref{eqn: largeal}, we see that the terms in \eqref{eqn: Ur} that have $mh/{g} \equiv \mp {nk}/{g}$ (mod~$ae\ell$) are zero unless
$ae\ell \ll XQ^{\vartheta}$.

Next, we consider the terms in \eqref{eqn: Ur} that have $mh/{g} \not\equiv \mp {nk}/{g}$ (mod~$ae\ell$) and $ae\ell \gg Y$, where $Y$ is a large parameter that we will choose later (in Section~\ref{sec: proofofthm}). For these terms, the orthogonality of Dirichlet characters implies that the $\psi$-sum in \eqref{eqn: Ur} is $O(1)$. Moreover, we have shown that these terms are zero unless \eqref{eqn: largeal} holds, and thus we may assume that $e\gg YQ^{\vartheta-1}/(XC)$. It follows from these and \eqref{eqn: divisorbound} that the sum of the terms in \eqref{eqn: Ur} that have $mh/{g} \not\equiv \mp {nk}/{g}$ (mod~$ae\ell$) and $ae\ell \gg Y$ is bounded by
\begin{equation}\label{eqn: aelggYterms}
\ll \sum_{1\leq c\leq C} \sum_{1\leq m,n\ll X} \frac{(mn)^{\varepsilon}}{\sqrt{mn}} \sum_{ \frac{Y Q^{1-\vartheta}}{XC} \ll e<\infty} \frac{1}{e} \sum_{a|g}  \sum_{\substack{1\leq \ell <\infty \\ a\ell \ll XCQ^{\vartheta -1} } } \frac{(ae\ell)^{\varepsilon}}{ae\ell}\cdot \frac{|mh\pm nk|}{g \ell} W\left( \frac{c|mh\pm nk|}{g\ell Q}\right).
\end{equation}
Since the support of $W$ is compact and contained in $(0,\infty)$, we have $|mh\pm nk|/(g\ell) \ll Q/c$ in \eqref{eqn: aelggYterms}, and so \eqref{eqn: aelggYterms} is
\begin{align*}
\ll & (XCQ)^{\varepsilon} Q \sum_{1\leq c\leq C}\frac{1}{c} \sum_{1\leq m,n\ll X} \frac{(mn)^{\varepsilon}}{\sqrt{mn}} \left(\frac{Y Q^{1-\vartheta}}{XC} \right)^{-1+\varepsilon} \ll   (XCQY)^{\varepsilon} \frac{X^2 CQ^{\vartheta}}{Y}.
\end{align*}
This bound is small if $Y$ is, say, a large power of $Q$. We have thus shown that the terms in \eqref{eqn: Ur} that have $mh/{g} \not\equiv \mp {nk}/{g}$ (mod~$ae\ell$) and $ae\ell \gg Y$ are negligible for large enough $Y$.

From all these observations, we deduce for $h,k\leq Q^{\vartheta}$ and $Y\geq XQ^{\vartheta}$ that the total contribution of the terms in the definition \eqref{eqn: Ur} of $\mathcal{U}^r(h,k)$ that have $a\ell \gg  XCQ^{\vartheta-1}$ or $ae\ell \gg Y$ is
$$
\ll (XCQY)^{\varepsilon} \frac{X^2 CQ^{\vartheta}}{Y}.
$$
Thus
\begin{align*}
 \mathcal{U}^r(h,k)&= \frac{1}{2} \sum_{\substack{1\leq c \leq C  \\ (c ,hk)=1 }} \mu(c)   \sum_{\substack{1\leq m,n<\infty \\ (mn,c )=1\\ mh\neq nk}} \frac{\tau_A(m) \tau_B(n) }{\sqrt{mn}} V\left( \frac{m}{X} \right)  V\left( \frac{n}{X} \right)  \sum_{\substack{ 1\leq e <\infty \\ ( e ,g)=1 }}\frac{\mu(e)}{e}\notag\\
 &\hspace{.5in} \times \sum_{a|g}  \mu(a)\sum_{\substack{ 1\leq  \ell <\infty \\ (ea\ell, \frac{mh}{g}\cdot\frac{nk}{g})=1  \\  a\ell \ll  XCQ^{\vartheta-1} \\ ae\ell \ll Y }} \frac{1}{\phi(ea \ell)}\sum_{\substack{\psi \bmod ae\ell\\ \psi\neq \psi_0}} \psi \left( \frac{mh}{g} \right) \overline{\psi}\left( \mp \frac{nk}{g} \right) \notag\\
 &\hspace{.75in} \times\frac{|mh\pm nk|}{g \ell} W\left( \frac{c|mh\pm nk|}{g\ell Q}\right) \ + \ O\left( (XCQY)^{\varepsilon} \frac{X^2 CQ^{\vartheta}}{Y}\right).
\end{align*}
We multiply both sides by $\lambda_h\overline{\lambda_k} (hk)^{-1/2}$ and then sum over all $h,k\leq Q^{\vartheta}$ to arrive at
\begin{equation}\label{eqn: Urshort}
\begin{split}
&\sum_{h,k\leq Q^{\vartheta}} \frac{ \lambda_h \overline{\lambda_k}}{\sqrt{hk}}\mathcal{U}^r(h,k)= \frac{1}{2}\sum_{1\leq c \leq C} \mu(c)  \sum_{1\leq e<\infty} \frac{\mu(e)}{e} \sum_{1\leq \ell<\infty} \sum_{\substack{1\leq a<\infty \\ a\ell \ll XCQ^{\vartheta-1} \\ ae\ell \ll Y }} \frac{\mu(a)}{\phi(ae\ell) \ell}\\
&\hspace{.25in}\times \sum_{\substack{\psi \bmod ae\ell\\ \psi\neq \psi_0}} \Big\{ \mathcal{U}^+ (c,a,e,\ell,\psi) + \mathcal{U}^- (c,a,e,\ell,\psi)\Big\} 
\ + \ O\left( (XCQY)^{\varepsilon} \frac{X^2 CQ^{2\vartheta}}{Y}\right),
\end{split}
\end{equation}
where $\mathcal{U}^{\pm} (c,a,e,\ell,\psi)$ is defined by
\begin{equation}\label{eqn: Upmdef}
\begin{split}
\mathcal{U}^{\pm} (c,a,e,\ell,\psi) &:= \sum_{\substack{h,k\leq Q^{\vartheta} \\ (c,hk)=1}}  \frac{ \lambda_h \overline{\lambda_k}}{\sqrt{hk}} \sideset{}{'}\sum_{m,n} \frac{\tau_A(m)\tau_B(n)}{\sqrt{mn}} V\left( \frac{m}{X}\right)V\left( \frac{n}{X}\right)\\
&\hspace{.25in}\times\psi \left( \frac{mh}{g} \right) \overline{\psi}\left( \mp \frac{nk}{g} \right) \frac{|mh\pm nk|}{g  } W\left( \frac{c|mh\pm nk|}{g\ell Q}\right),
\end{split}
\end{equation}
with the symbol $\sum'$ denoting summation over all positive integers $m,n$ such that $(mn,c)=1$, $mh\neq nk$, $(e,g)=1$, $a|g$, and $(ea\ell, mhnk/g^2)=1$, where $g=(mh,nk)$. We split the $a,e,\ell$-sum in \eqref{eqn: Urshort} into dyadic blocks and deduce that
\begin{equation}\label{eqn: Urshortbound}
\begin{split}
& \sum_{h,k\leq Q^{\vartheta}} \frac{ \lambda_h \overline{\lambda_k}}{\sqrt{hk}}\mathcal{U}^r(h,k) \ll 
\sum_{1\leq c \leq C} \sum_{\substack{A,E,L \\  AL \ll XCQ^{\vartheta-1} \\ AEL \ll Y }} \sum_{A<a\leq 2A} \sum_{E<e\leq 2E} \sum_{L<\ell\leq 2L} \frac{(ae\ell)^{\varepsilon}}{ ae^2\ell^2} \\
& \hspace{.25in} \times \sum_{\substack{\psi \bmod ae\ell\\ \psi\neq \psi_0}} \Big\{ |\mathcal{U}^+ (c,a,e,\ell,\psi)| + |\mathcal{U}^- (c,a,e,\ell,\psi)|\Big\}  \ + \  (XCQY)^{\varepsilon} \frac{X^2 CQ^{2\vartheta}}{Y},
\end{split}
\end{equation}
where each of the summation variables $A,E,L$ runs through the set $\{2^{\nu}: \nu\in \mathbb{Z}, \nu \geq -1\}$. Note that we are abusing notation here and using the symbol $A$ to denote both the summation variable in \eqref{eqn: Urshortbound} and the set in $\tau_A$ in \eqref{eqn: Upmdef}. However, this will not cause confusion.

To remove the interdependence of the summation variables in \eqref{eqn: Upmdef}, we let $g_1=(h,k)$, $g_2=(m,n)$, $g_3=(m/g_2,k/g_1)$, and $g_4=(n/g_2,h/g_1)$, and make the change of variables $h=g_1g_4H$, $k=g_1g_3K$, $m=g_2g_3M$, and $n=g_2g_4N$. Recalling the definition $g=(mh,nk)$, we note that $g=g_1g_2g_3g_4$. By their definitions, the new variables satisfy the coprimality conditions $(g_3,g_4)=1$, $(H,g_3)=1$, $(K,g_4)=1$, $(H,K)=1$, $(M,g_4)=1$, $(N,g_3)=1$, $(M,N)=1$, $(M,K)=1$, and $(N,H)=1$. Furthermore, the properties of $m,h,n,k$ in \eqref{eqn: Upmdef} are equivalent to $(c,MNHKg_1g_2g_3g_4)=1$, $MH\neq NK$, $(e,g_1g_2g_3g_4)=1$, $a|g_1g_2g_3g_4$, $(ea\ell, MNHK)=1$, $g_1g_4 H\leq Q^{\vartheta}$, $g_1g_3 K\leq Q^{\vartheta}$, and $1\leq M,N<\infty$. Since $V$ has compact support, we may assume that $m,n\ll X$ in \eqref{eqn: Upmdef} and hence $g_2\ll X$. Thus, the result of this change of variables is
\begin{equation}\label{eqn: Upm}
\begin{split}
\mathcal{U}^{\pm} (c,a,e,\ell,\psi) = \sum_{\substack{g_1,g_2,g_3,g_4 \\ M, N, H,K}}^*  \frac{ \lambda_{g_1g_4H} \overline{\lambda_{g_1g_3K}}}{g_1\sqrt{g_3g_4HK}} \frac{\tau_A(g_2g_3M)\tau_B(g_2g_4N)}{g_2\sqrt{g_3g_4MN}} V\left( \frac{g_2g_3M}{X}\right)V\left( \frac{g_2g_4N}{X}\right)\\
\times \psi (MH) \overline{\psi}( \mp NK) |MH\pm NK| W\left( \frac{c|MH\pm NK|}{\ell Q}\right),
\end{split}
\end{equation}
where $*$ denotes the conditions for $g_1,g_2,g_3,g_4, M, N, H,K$ listed above.

Our next task is to write \eqref{eqn: Upm} in terms of an Euler product. To this end, we apply Mellin inversion twice to write
\begin{equation}\label{eqn: mellintwice}
\begin{split}
V\left( \frac{g_2g_3x}{X}\right)
& V\left( \frac{g_2g_4y}{X}\right) |xH\pm yK| W\left( \frac{c|xH\pm yK|}{\ell Q}\right) \\
& = \frac{1}{(2\pi i)^2}  \int_{(\frac{1}{2}+\varepsilon)} \int_{(\frac{1}{2}+\varepsilon)} (xH)^{-s_1}(yK)^{-s_2}   \int_0^{\infty}\int_0^{\infty} u^{s_1-1} v^{s_2-1}\\
&\hspace{.25in}\times V\left( \frac{g_2g_3u}{HX}\right)V\left( \frac{g_2g_4v}{KX}\right)|u\pm v| W\left( \frac{c|u\pm v|}{\ell Q}\right) \,dv \,du\,ds_2\,ds_1.
\end{split}
\end{equation}
We have chosen the lines of integration to be at $\re(s_1)=\re(s_2)=\frac{1}{2}+\varepsilon$ to facilitate later estimations. We let $\Psi:[0,\infty)\rightarrow \mathbb{R}$ be a smooth nonnegative function of compact support such that $\Psi(\xi)=1$ for all $\xi$ in the support of $V$. Then
\begin{align*}
V\left( \frac{g_2g_3u}{HX}\right) = \Psi\Big( \frac{ u}{XQ^{\vartheta}}\Big)V\left( \frac{g_2g_3u}{HX}\right)
\end{align*}
for all $u\geq 0$, and applying Mellin inversion on the right-hand side gives
\begin{align*}
V\left( \frac{g_2g_3u}{HX}\right) = \frac{1}{2\pi i } \Psi\Big( \frac{ u}{XQ^{\vartheta}}\Big) \int_{(\varepsilon)} \left(\frac{XH}{g_2g_3 u} \right)^{s_3} \widetilde{V}(s_3) \,ds_3.
\end{align*}
Similarly,
\begin{align*}
V\left( \frac{g_2g_4v}{KX}\right) =\frac{1}{2\pi i } \Psi\Big( \frac{ v}{XQ^{\vartheta}}\Big)\int_{(\varepsilon)} \left(\frac{XK}{g_2g_4 v} \right)^{s_4} \widetilde{V}(s_4) \,ds_4.
\end{align*}
It follows from these and \eqref{eqn: mellintwice} that
\begin{align*}
 V\left( \frac{g_2g_3x}{X}\right)&V\left( \frac{g_2g_4y}{X}\right) |xH\pm yK| W\left( \frac{c|xH\pm yK|}{\ell Q}\right) \\
& = \frac{1}{(2\pi i)^4} \int_{(\frac{1}{2}+\varepsilon)} \int_{(\frac{1}{2}+\varepsilon)} \int_{(\varepsilon)}  \int_{(\varepsilon)} H^{-s_1+s_3}K^{-s_2+s_4} \left(\frac{X}{g_2g_3} \right)^{s_3} \left(\frac{X }{g_2g_4} \right)^{s_4} x^{-s_1}y^{-s_2} \\
&\hspace{.25in} \times\int_0^{\infty} \int_0^{\infty}  u^{s_1-s_3-1} v^{s_2-s_4-1}  \Psi\Big( \frac{ u}{XQ^{\vartheta}}\Big)\Psi\Big( \frac{ v}{XQ^{\vartheta}}\Big)\widetilde{V}(s_3)\widetilde{V}(s_4) |u\pm v|\\
&\hspace{.5in} \times W\left( \frac{c|u\pm v|}{\ell Q}\right)  \,dv \,du\,ds_4\,ds_3\,ds_2\,ds_1.
\end{align*}
Now take $x=M$ and $y=N$, and insert the result into \eqref{eqn: Upm} to deduce that
\begin{equation}\label{eqn: UpmMN}
\begin{split}
\mathcal{U}^{\pm} (c,a,e,\ell,\psi) &= \sum_{\substack{g_1,g_2,g_3,g_4 \\ M, N, H,K}}^*  \frac{ \lambda_{g_1g_4H} \overline{\lambda_{g_1g_3K}}}{g_1\sqrt{g_3g_4HK}} \frac{\tau_A(g_2g_3M)\tau_B(g_2g_4N)}{g_2\sqrt{g_3g_4MN}} \psi (MH) \overline{\psi}( \mp NK)\\
&\hspace{.25in}\times \frac{1}{(2\pi i)^4} \int_{(\frac{1}{2}+\varepsilon)} \int_{(\frac{1}{2}+\varepsilon)} \int_{(\epsilon)}  \int_{(\epsilon)} H^{-s_1+s_3}K^{-s_2+s_4} \left(\frac{X}{g_2g_3} \right)^{s_3} \left(\frac{X }{g_2g_4} \right)^{s_4}\\
&\hspace{.5in}\times M^{-s_1}N^{-s_2}\mathcal{V}(s_1,s_2,s_3,s_4 )\,ds_4\,ds_3\,ds_2\,ds_1,
\end{split}
\end{equation}
where $\mathcal{V}(s_1,s_2,s_3,s_4 )$ is defined by
\begin{equation}\label{eqn: Vs}
\begin{split}
\mathcal{V}(s_1,s_2,s_3,s_4 )
&= \mathcal{V}(s_1,s_2,s_3,s_4;X,Q, \vartheta,c,\ell) \\
& : = \widetilde{V}(s_3)\widetilde{V}(s_4)\int_0^{\infty} \int_0^{\infty}  u^{s_1-s_3-1} v^{s_2-s_4-1}  \Psi\Big( \frac{ u}{XQ^{\vartheta}}\Big)\Psi\Big( \frac{ v}{XQ^{\vartheta}}\Big)\\
&\hspace{.5in}\times |u\pm v| W\left( \frac{c|u\pm v|}{\ell Q}\right)  \,dv \,du.
\end{split}
\end{equation}
The following lemma gives a bound for $\mathcal{V}(s_1,s_2,s_3,s_4 )$, and is analogous to \eqref{eqn: mellinrapiddecay}.

\begin{lemma}\label{lem: uvintegralbound}
If $j_1,j_2$ are nonnegative integers and $s_1,s_2,s_3,s_4$ are complex numbers such that $j_1+j_2\geq 1$, $\re(s_1-s_3)>0$, and $\re(s_2-s_4)>0$, then
\begin{align*}
\int_0^{\infty} \int_0^{\infty}  u^{s_1-s_3-1} v^{s_2-s_4-1}  \Psi\Big( \frac{ u}{XQ^{\vartheta}}\Big)\Psi\Big( \frac{ v}{XQ^{\vartheta}}\Big) |u\pm v| W\left( \frac{c|u\pm v|}{\ell Q}\right)  \,dv \,du \\
\ll   \frac{ (X Q^{\vartheta } )^{\re(s_1+s_2-s_3-s_4)  } }{|s_1-s_3|^{j_1} |s_2-s_4|^{j_2}} \left( \frac{ \ell Q}{c}\right) \bigg(1+\frac{ XcQ^{\vartheta-1}}{\ell} \bigg)^{j_1+j_2-1},
\end{align*}
where the implied constant may depend only on $\Psi$, $W$, $\re(s_1-s_3)$, $\re(s_2-s_4)$, $j_1$, or $j_2$.
\end{lemma}
\begin{proof}
For brevity, let $\mathcal{D}$ denote the double integral in question, and let $W_0$ denote the function $W_0(\xi):=\xi W(\xi)$. Make the change of variables $u\mapsto {u\ell Q}/{c}$ and $v\mapsto {v\ell Q}/{c}$, then integrate by parts with respect to $u$ $j_1$ times and with respect to $v$ $j_2$ times to deduce that
\begin{align*}
\mathcal{D} &= (-1)^{j_1+j_2} \left( \frac{ \ell Q}{c}\right)^{s_1+s_2-s_3-s_4+1}\int_0^{\infty} \int_0^{\infty}   \frac{u^{s_1-s_3+j_1-1}}{(s_1-s_3)\cdots (s_1-s_3+j_1-1)}\\
&\times \frac{v^{s_2-s_4+j_2-1}}{(s_2-s_4)\cdots (s_2-s_4+j_2-1)}\frac{\partial^{j_1}}{\partial u^{j_1}}\frac{\partial^{j_2}}{\partial v^{j_2}} \bigg\{ \Psi\Big( \frac{ u \ell }{XcQ^{\vartheta-1}}\Big)\Psi\Big( \frac{ v \ell }{XcQ^{\vartheta-1}}\Big)   W_0(|u\pm v|)  \bigg\} \,dv \,du.
\end{align*}
We may use the product rule and chain rule to bound the derivatives in the integrand. We also observe that the integrand is zero unless $u,v\ll XcQ^{\vartheta-1}/\ell$ and $|u\pm v|\asymp 1$, because $\Psi$ is supported on a compact subset of $[0,\infty)$ and $W$ is supported on a compact subset of $(0,\infty)$. Thus
\begin{equation}\label{eqn: uvintegralboundD}
\begin{split}
\mathcal{D}  \ll  \left( \frac{ \ell Q}{c}\right)^{\re(s_1+s_2-s_3-s_4)+1}\bigg(1+ \frac{ \ell }{XcQ^{\vartheta-1}} \bigg)^{j_1+j_2}\frac{1}{|s_1-s_3|^{j_1} |s_2-s_4|^{j_2}} \\
\times  \mathop{\iint}_{\substack{0\leq u,v\ll { XcQ^{\vartheta-1}}/{\ell} \\ |u \pm v| \asymp 1 }}    u^{\re(s_1-s_3)+j_1-1} v^{\re(s_2-s_4)+j_2-1}  \,dv \,du.
\end{split}
\end{equation}
Since $j_1,j_2$ are nonnegative integers with $j_1+j_2\geq 1$, it holds that either $j_1\geq 1$ or $j_2\geq 1$. By renaming the variables $u$ and $v$ if necessary, we may suppose, without loss of generality, that $j_2\geq 1$. Then $v^{j_2-1}\ll (XcQ^{\vartheta-1}/{\ell})^{j_2-1}$. Moreover, for each $u$, the $v$-integral is over an interval of length $\ll \min\{1,XcQ^{\vartheta-1}/{\ell}\}$. Hence the $u,v$-integral in \eqref{eqn: uvintegralboundD} is at most
\begin{equation*}
\ll \bigg( \frac{ XcQ^{\vartheta-1}}{\ell} \bigg)^{\re(s_1+s_2-s_3-s_4) + j_1+j_2-1} \min\left\{ 1, \frac{ XcQ^{\vartheta-1}}{\ell}\right\}.
\end{equation*}
Since $\min\{1,1/x\}\asymp 1/(1+x)$ for $x>0$, this proves the lemma.
\end{proof}

Now \eqref{eqn: mellinrapiddecay}, \eqref{eqn: Vs}, and Lemma~\ref{lem: uvintegralbound} imply that
\begin{equation}\label{eqn: Vbound}
\begin{split}
 \mathcal{V}(s_1,s_2,s_3,s_4 ) & \ll_{\varepsilon,j_1,j_2,j_3,j_4} \frac{(X Q^{\vartheta } )^{\re(s_1+s_2-s_3-s_4)  }}{|s_1-s_3|^{j_1}|s_2-s_4|^{j_2}|s_3|^{j_3} |s_4|^{j_4}} \left( \frac{ \ell Q}{c}\right)  \bigg(1+\frac{ XcQ^{\vartheta-1}}{\ell} \bigg)^{j_1+j_2-1}  
 \end{split}
\end{equation}
for any nonnegative integers $j_1,j_2,j_3,j_4$ with $j_1+j_2\geq 1$ and any complex numbers $s_1,s_2,s_3,s_4$ such that each of Re$(s_1-s_3)$, Re$(s_2-s_4)$, Re$(s_3)$, and Re$(s_4)$ is $\geq \varepsilon$. It follows that \eqref{eqn: UpmMN} is absolutely convergent, and we may interchange the order of summation to deduce that, recalling the conditions indicated by $*$ and listed before \eqref{eqn: Upm}, we have
\begin{equation}\label{eqn: Upm2}
\begin{split}
\mathcal{U}^{\pm} (c,a,e,\ell,\psi) =
& \frac{1}{(2\pi i)^4} \int_{(\frac{1}{2}+\varepsilon)} \int_{(\frac{1}{2}+\varepsilon)} \int_{(\epsilon)}  \int_{(\epsilon)} X^{s_3+s_4}\\
&\times \sum_{\substack{1\leq g_1,g_2,g_3,g_4, H,K \ll \max\{Q^{\vartheta},X\}\\ g_1g_4H \leq Q^{\vartheta}, \ g_1g_3K \leq Q^{\vartheta}, \ g_2\ll X \\ (g_3,g_4)=(H,g_3)=(K,g_4)=(H,K)=1 \\ (ec,g_1g_2g_3g_4)=(cae\ell,HK)=1 \\ a|g_1g_2g_3g_4  }}  \frac{ \lambda_{g_1g_4H} \overline{\lambda_{g_1g_3K}}   \,  \psi ( H) \overline{\psi}( \mp   K) }{g_1 g_2^{1+s_3+s_4} g_3^{1+s_3}g_4^{1+s_4} H^{\frac{1}{2}+s_1-s_3}K^{\frac{1}{2}+s_2-s_4}}    \\
& \times \sum_{\substack{1\leq M,N<\infty \\ (M,g_4)=(N,g_3)=(M,K)=(N,H)=1 \\ (M,N)=(MN,cae\ell)=1 \\ MH\neq NK }} \frac{\tau_A(g_2g_3M)\tau_B(g_2g_4N)\psi (M ) \overline{\psi}( N) }{M^{\frac{1}{2}+s_1}N^{\frac{1}{2}+s_2}}\\
&\times \mathcal{V}(s_1,s_2,s_3,s_4 )\,ds_4\,ds_3\,ds_2\,ds_1.
\end{split}
\end{equation}

We next write the $M,N$-sum in terms of an Euler product. To do this, we first add and subtract the terms with $MH=NK$ and write
\begin{align}\label{eqn: MNdifference}
\sum_{\substack{1\leq M,N<\infty \\ (M,g_4)=(N,g_3)=(M,K)=(N,H)=1 \\ (M,N)=(MN,cae\ell)=1 \\ MH\neq NK }} \frac{\tau_A(g_2g_3M)\tau_B(g_2g_4N)\psi (M ) \overline{\psi}( N) }{M^{\frac{1}{2}+s_1}N^{\frac{1}{2}+s_2}} =\mathscr{P}_1 - \mathscr{P}_2,
\end{align}
where $\mathscr{P}_1$ is the sum on the left-hand side, except without the condition $MH\neq NK$, and $\mathscr{P}_2$ is the sum with the condition $MH =NK$ instead of $MH\neq NK$. To evaluate $\mathscr{P}_2$, observe that the conditions $(H,K)=1$ and $(M,N)=1$ imply that $MH=NK$ if and only if $M=K$ and $N=H$. Since $(M,K)=(N,H)=1$, this is only possible if $M=N=H=K=1$. Thus
\begin{equation}\label{eqn: P2evaluated}
\mathscr{P}_2=\tau_A(g_2g_3)\tau_B(g_2g_4).
\end{equation}
Next, we express the sum $\mathscr{P}_1$ defined in \eqref{eqn: MNdifference} as an Euler product and write
\begin{equation}\label{eqn: P1euler}
\mathscr{P}_1 = \prod_{\alpha\in A} L(\tfrac{1}{2}+s_1+\alpha,\psi)  \prod_{\beta\in B} L(\tfrac{1}{2}+s_2+\beta,\overline{\psi}) \mathcal{R}(s_1,s_2),
\end{equation}
where $\mathcal{R}(s_1,s_2)$ is defined by
\begin{equation}\label{eqn: mathcalRdef}
\begin{split}
\mathcal{R}(s_1,s_2) =
& \mathcal{R}(s_1,s_2; g_2,g_3,g_4,H,K,cae\ell) \\
:= & \prod_p \Bigg\{ \prod_{\alpha\in A}\left( 1-\frac{\psi(p)}{p^{\frac{1}{2}+s_1+\alpha}} \right) \prod_{\beta\in B}\left( 1-\frac{\overline{\psi}(p)}{p^{\frac{1}{2}+s_2+\beta}} \right)\\
& \times \sum_{\substack{0\leq m,n<\infty \\ \min\{m,\ordp(g_4K) \} =  \min\{n,\ordp(g_3H) \} =0  \\ \min\{m,n\}=\min\{mn,\ordp(cae\ell) \}=0}}  \frac{\tau_A(p^{m+\ordp(g_2g_3) }) \tau_B(p^{n+\ordp(g_2g_4) })\psi(p^m)\overline{\psi}(p^n)}{p^{m(\frac{1}{2}+s_1)+n(\frac{1}{2}+s_2)}}\Bigg\}.
\end{split}
\end{equation}
If $\re(s_1),\re(s_2)\geq \varepsilon$ and $p | g_2g_3g_4 HKcae\ell$, then the local factor in \eqref{eqn: mathcalRdef} corresponding to $p$ is $O(p^{\varepsilon \ordp(g_2g_3g_4)})$ by \eqref{eqn: divisorbound}. Moreover, if $\re(s_1),\re(s_2)\geq \varepsilon$ and $p\nmid g_2g_3g_4 HKcae\ell$, then it follows from \eqref{eqn: taudef} and \eqref{eqn: divisorbound} that the local factor in \eqref{eqn: mathcalRdef} corresponding to $p$ is
\begin{align*}
\prod_{\alpha\in A}
& \left( 1-\frac{\psi(p)}{p^{\frac{1}{2}+s_1+\alpha}} \right) \prod_{\beta\in B}\left( 1-\frac{\overline{\psi}(p)}{p^{\frac{1}{2}+s_2+\beta}} \right) \sum_{\substack{0\leq m,n<\infty \\ \min\{m,n\}=0}}  \frac{\tau_A(p^m) \tau_B(p^n)\psi(p^m)\overline{\psi}(p^n)}{p^{m(\frac{1}{2}+s_1)+n(\frac{1}{2}+s_2)}} \\
& = \left( 1-\frac{\tau_A(p)\psi(p)}{p^{\frac{1}{2}+s_1}} + O\left( \frac{1}{p^{1+\varepsilon}}\right)\right) \left( 1-\frac{\tau_B(p)\overline{\psi}(p)}{p^{\frac{1}{2}+s_2}} + O\left( \frac{1}{p^{1+\varepsilon}}\right)\right)\\
& \hspace{.25in} \times \left( 1+\frac{\tau_A(p)\psi(p)}{p^{\frac{1}{2}+s_1}} + \frac{\tau_B(p)\overline{\psi}(p)}{p^{\frac{1}{2}+s_2}} + O\left( \frac{1}{p^{1+\varepsilon}}\right)\right) \\
& = 1 + O\left( \frac{1}{p^{1+\varepsilon}}\right).
\end{align*}
Thus, if $\re(s_1),\re(s_2)\geq \varepsilon$, then the product in \eqref{eqn: mathcalRdef} converges absolutely and we have
\begin{equation}\label{eqn: mathcalRbound}
\mathcal{R}(s_1,s_2) \ll (g_2g_3g_4HKcae\ell)^{\varepsilon}
\end{equation}
because $\prod_{p|\nu}O(1) \ll \nu^{\varepsilon}$ for any positive integer $\nu$. Hence, \eqref{eqn: MNdifference} with \eqref{eqn: P2evaluated} and \eqref{eqn: P1euler} gives an analytic continuation of the $M,N$-sum in \eqref{eqn: Upm2} to the region with $\re(s_1),\re(s_2)\geq \varepsilon$. If $\psi$ is non-principal, then this analytic continuation has no poles in the region, and \eqref{eqn: divisorbound} and \eqref{eqn: mathcalRbound} imply that it is bounded by
$$
\ll (g_2g_3g_4HKcae\ell)^{\varepsilon}\Bigg\{ 1+ \prod_{\alpha\in A} |L(\tfrac{1}{2}+s_1+\alpha,\psi) | \prod_{\beta\in B} |L(\tfrac{1}{2}+s_2+\beta,\overline{\psi})| \Bigg\}
$$
for $\re(s_1),\re(s_2)\geq \varepsilon$. This fact together with \eqref{eqn: Vbound} implies that if $\psi$ is non-principal, then we may move the $s_1$- and $s_2$-lines in \eqref{eqn: Upm2} to $\re(s_1)=\re(s_2)=2\epsilon$ and deduce that
\begin{equation}\label{eqn: Upm3}
\begin{split}
\mathcal{U}^{\pm} (c,a,e,\ell,\psi) \ll
& (XQcae\ell)^{\varepsilon}\int_{(2\epsilon)}  \int_{(2\epsilon)} \int_{(\epsilon)}  \int_{(\epsilon)}\\
&\times\Bigg| \sum_{\substack{1\leq g_1,g_2,g_3,g_4, H,K \ll \max\{Q^{\vartheta},X\}\\ g_1g_4H \leq Q^{\vartheta}, \ g_1g_3K \leq Q^{\vartheta}, \ g_2\ll X \\ (g_3,g_4)=(H,g_3)=(K,g_4)=(H,K)=1 \\ (ec,g_1g_2g_3g_4)=(cae\ell,HK)=1 \\ a|g_1g_2g_3g_4  }}  \frac{ \lambda_{g_1g_4H} \overline{\lambda_{g_1g_3K}}   \,  \psi ( H) \overline{\psi}( \mp   K) }{g_1 g_2^{1+s_3+s_4} g_3^{1+s_3}g_4^{1+s_4} H^{\frac{1}{2}+s_1-s_3}K^{\frac{1}{2}+s_2-s_4}}  \Bigg|  \\
& \times \Bigg\{ 1+\prod_{\alpha\in A} |L(\tfrac{1}{2}+s_1+\alpha,\psi)|  \prod_{\beta\in B} |L(\tfrac{1}{2}+s_2+\beta,\overline{\psi})|\Bigg\}\\ 
&\times |\mathcal{V}(s_1,s_2,s_3,s_4 )|\,|ds_4\,ds_3\,ds_2\,ds_1|.
\end{split}
\end{equation}

We apply M\"{o}bius inversion to remove the interdependence of the variables $H$ and $K$ and write
\begin{align*}
& \sum_{\substack{ 1\leq H,K \ll \max\{Q^{\vartheta},X\}\\ g_1g_4H \leq Q^{\vartheta}, \ g_1g_3K \leq Q^{\vartheta} \\ (H,g_3)=(K,g_4)=(H,K)=1 \\  (cae\ell,HK)=1   }}  \frac{ \lambda_{g_1g_4H} \overline{\lambda_{g_1g_3K}}   \,  \psi ( H) \overline{\psi}( \mp   K) }{ H^{\frac{1}{2}+s_1-s_3}K^{\frac{1}{2}+s_2-s_4}}\\
&\hspace{.5in} = \sum_{\substack{ 1\leq H,K \ll \max\{Q^{\vartheta},X\} \\ g_1g_4H \leq Q^{\vartheta}, \ g_1g_3K \leq Q^{\vartheta}  \\ (H,g_3)=(K,g_4)= 1 \\  (cae\ell,HK)=1   }} \sum_{\substack{d|H \\ d|K}} \mu(d) \frac{ \lambda_{g_1g_4H} \overline{\lambda_{g_1g_3K}}   \,  \psi ( H) \overline{\psi}( \mp   K) }{ H^{\frac{1}{2}+s_1-s_3}K^{\frac{1}{2}+s_2-s_4}} \\
&\hspace{.5in} = \sum_{\substack{d\leq Q^{\vartheta} \\ (d,g_3g_4cae\ell)=1}} \frac{\mu(d)|\psi(d)|^2}{d^{1+s_1+s_2-s_3-s_4} } \sum_{\substack{ H\leq Q^{\vartheta}/(dg_1g_4)  \\ (H,g_3cae\ell) = 1  }}   \frac{ \lambda_{dg_1g_4H}   \,  \psi ( H)   }{ H^{\frac{1}{2}+s_1-s_3} } \sum_{\substack{ K\leq Q^{\vartheta}/(dg_1g_3) \\  (K,g_4cae\ell)= 1   }}   \frac{  \overline{\lambda_{dg_1g_3K}}   \,   \overline{\psi}( \mp   K) }{  K^{\frac{1}{2}+s_2-s_4}},
\end{align*}
where in the last line we have made the change of variables $H\mapsto dH$ and $K\mapsto dK$. From this, \eqref{eqn: Upm3}, the triangle inequality, and the fact that $\psi(\mp K) = \psi(\mp 1) \psi(K)$, we deduce that
\begin{align*}
\mathcal{U}^{\pm} (c,a,e,\ell,\psi) \ll
& (XQcae\ell)^{\varepsilon}\sum_{\substack{1\leq g_1,g_2,g_3,g_4 \ll \max\{Q^{\vartheta},X\}\\ g_1g_4  \leq Q^{\vartheta}, \ g_1g_3  \leq Q^{\vartheta}, \ g_2\ll X \\ (g_3,g_4)=(ec,g_1 g_2 g_3 g_4)=1 \\ a|g_1g_2g_3g_4  }} \frac{1}{ g_1 g_2^{1+\varepsilon} g_3^{1+\varepsilon}g_4^{1+\varepsilon} } \sum_{\substack{d\leq Q^{\vartheta} \\ (d,g_3g_4cae\ell)=1}} \frac{1}{d^{1+\varepsilon}}\\
& \times \int_{(2\epsilon)}  \int_{(2\epsilon)} \int_{(\epsilon)}  \int_{(\epsilon)} \Bigg| \sum_{\substack{ H \leq Q^{\vartheta}/(dg_1g_4)  \\ (H,g_3cae\ell) = 1 }}   \frac{ \lambda_{dg_1g_4H}   \,  \psi ( H)   }{ H^{\frac{1}{2}+s_1-s_3} } \Bigg|\Bigg| \sum_{\substack{ K\leq Q^{\vartheta}/(dg_1g_3)  \\  (K,g_4cae\ell)= 1     }}   \frac{  \overline{\lambda_{dg_1g_3K}}   \,   \overline{\psi}(K) }{  K^{\frac{1}{2}+s_2-s_4}} \Bigg|  \\
&  \times  \Bigg\{ 1+\prod_{\alpha\in A} |L(\tfrac{1}{2}+s_1+\alpha,\psi)|  \prod_{\beta\in B} |L(\tfrac{1}{2}+s_2+\beta,\overline{\psi})|\Bigg\}\\
& \times |\mathcal{V}(s_1,s_2,s_3,s_4 )| \,|ds_4\,ds_3\,ds_2\,ds_1|.
\end{align*}
From this and \eqref{eqn: Urshortbound}, we arrive at
\begin{equation}\label{eqn: Urbound}
\begin{split}
\sum_{h,k\leq Q^{\vartheta}} \frac{ \lambda_h \overline{\lambda_k}}{\sqrt{hk}}\mathcal{U}^r(h,k) &\ll \sum_{1\leq c \leq C} \sum_{\substack{A,E,L \\  AL \ll XCQ^{\vartheta-1} \\ AEL \ll Y }} \sum_{A<a\leq 2A} \sum_{E<e\leq 2E}  \frac{( CXQY)^{\varepsilon}}{ AE^2L^2}\\
& \times \sum_{\substack{1\leq g_1,g_2,g_3,g_4\ll \max\{Q^{\vartheta},X\}\\ g_1g_4  \leq Q^{\vartheta}, \ g_1g_3  \leq Q^{\vartheta}, \ g_2\ll X \\ (g_3,g_4)=(ec,g_1 g_2 g_3 g_4)=1 \\ a|g_1g_2g_3g_4  }} \frac{1}{ g_1 g_2^{1+\varepsilon} g_3^{1+\varepsilon}g_4^{1+\varepsilon} }  \sum_{\substack{d\leq Q^{\vartheta} \\ (d,g_3g_4cae )=1}} \frac{1}{d^{1+\varepsilon}}\\
& \times   \Sigma_{c,a,e,d,g_1,g_2,g_3,g_4} \ + \ (XCQY)^{\varepsilon} \frac{X^2 CQ^{2\vartheta}}{Y},
\end{split}
\end{equation}
where $\Sigma_{c,a,e,d,g_1,g_2,g_3,g_4}$ is defined by
\begin{align*}
\Sigma_{c,a,e,d,g_1,g_2,g_3,g_4}
 :=&  \sum_{\substack{L<\ell\leq 2L \\ (d,\ell)=1}}\sum_{\substack{\psi \bmod ae\ell\\ \psi\neq \psi_0}} \int_{(2\epsilon)}  \int_{(2\epsilon)} \int_{(\epsilon)}  \int_{(\epsilon)}\\
&\Bigg| \sum_{\substack{ H \leq Q^{\vartheta}/(dg_1g_4)  \\ (H,g_3cae\ell) = 1  }}   \frac{ \lambda_{dg_1g_4H}   \,  \psi ( H)   }{ H^{\frac{1}{2}+s_1-s_3} } \Bigg|\Bigg| \sum_{\substack{ K \leq Q^{\vartheta}/(dg_1g_3) \\  (K,g_4cae\ell)= 1     }}   \frac{  \overline{\lambda_{dg_1g_3K}}   \,   \overline{\psi}(K) }{  K^{\frac{1}{2}+s_2-s_4}} \Bigg| \\
&\Bigg\{ 1+\prod_{\alpha\in A} |L(\tfrac{1}{2}+s_1+\alpha,\psi)|  \prod_{\beta\in B} |L(\tfrac{1}{2}+s_2+\beta,\overline{\psi})|\Bigg\}\\
&\times |\mathcal{V}(s_1,s_2,s_3,s_4 )|  \,|ds_4\,ds_3\,ds_2\,ds_1|.
\end{align*}
We interchange the order of integration and then make the change of variables $s_5=s_1-s_3$ and $s_6=s_2-s_4$ to write
\begin{equation}\label{eqn: ellpsisum2}
\begin{split}
\Sigma_{c,a,e,d,g_1,g_2,g_3,g_4} = &  \sum_{\substack{L<\ell\leq 2L \\ (d,\ell)=1}}\sum_{\substack{\psi \bmod ae\ell\\ \psi\neq \psi_0}} \int_{(\epsilon)}  \int_{(\epsilon)} \int_{(\epsilon)}  \int_{(\epsilon)} |\mathcal{V}(s_3+s_5,s_4+s_6,s_3,s_4)|\\
& \times    \Bigg| \sum_{\substack{ H \leq Q^{\vartheta}/(dg_1g_4) \\ (H,g_3cae\ell) = 1   }}   \frac{ \lambda_{dg_1g_4H}   \,  \psi ( H)   }{ H^{\frac{1}{2}+s_5} } \Bigg| \Bigg| \sum_{\substack{ K\leq Q^{\vartheta}/(dg_1g_3) \\  (K,g_4cae\ell)= 1    }}   \frac{  \overline{\lambda_{dg_1g_3K}}   \,   \overline{\psi}(   K) }{  K^{\frac{1}{2}+s_6}} \Bigg| \\
& \times \Bigg\{ 1 + \prod_{\alpha\in A} |L(\tfrac{1}{2}+s_3+s_5+\alpha,\psi)| \prod_{\beta\in B} |L(\tfrac{1}{2}+s_4+s_6+\beta,\overline{\psi})|\Bigg\}\\
&\times\,|ds_6\,ds_5\,ds_4\,ds_3|.
\end{split}
\end{equation}
Now GLH and the Phragm\'{e}n-Lindel\"{o}f principle together imply that if $\varepsilon>0$ then
$$
L(s,\psi) \ll_{\varepsilon} (q(1+|t|))^{\varepsilon}
$$
for all $s=\sigma+it$ with $\frac{1}{2}\leq \sigma\leq 1$ and real $t$ and all non-principal Dirichlet characters $\psi$ modulo $q$, where the implied constant depends only on $\varepsilon$. It follows from this and \eqref{eqn: ellpsisum2} that
\begin{equation}\label{eqn: ellpsisum3}
\begin{split}
\Sigma_{c,a,e,d,g_1,g_2,g_3,g_4} \ll  &  \sum_{\substack{L<\ell\leq 2L \\ (d,\ell)=1}} (ae\ell)^{\varepsilon} \sum_{\substack{\psi \bmod ae\ell\\ \psi\neq \psi_0}} \int_{(\epsilon)}  \int_{(\epsilon)} \int_{(\epsilon)}  \int_{(\epsilon)} |\mathcal{V}(s_3+s_5,s_4+s_6,s_3,s_4)| \\
&\times |s_3s_4s_5s_6|^{ \varepsilon} \Bigg| \sum_{\substack{ H \leq Q^{\vartheta}/(dg_1g_4) \\ (H,g_3cae\ell) = 1  }}   \frac{ \lambda_{dg_1g_4H}   \,  \psi ( H)   }{ H^{\frac{1}{2}+s_5} } \Bigg| \Bigg| \sum_{\substack{ K\leq Q^{\vartheta}/(dg_1g_3) \\  (K,g_4cae\ell)= 1     }}   \frac{  \overline{\lambda_{dg_1g_3K}}   \,   \overline{\psi}(   K) }{  K^{\frac{1}{2}+s_6}} \Bigg|\\
&\times \,|ds_6\,ds_5\,ds_4\,ds_3|.
\end{split}
\end{equation}

Our next task is to apply the bound \eqref{eqn: Vbound} for $\mathcal{V}$. We will facilitate later estimations by choosing particular values of $j_1,j_2,j_3,j_4$ in \eqref{eqn: Vbound} for specific ranges of $s_5$ and $s_6$. To this end, we split the range of integration of the $s_5$- and $s_6$-integrals in \eqref{eqn: ellpsisum3} into dyadic segments to write
\begin{equation}\label{eqn: ellpsisum4}
\begin{split}
&\Sigma_{c,a,e,d,g_1,g_2,g_3,g_4}\\
&\hspace{.25in}\ll   \sum_{\substack{L<\ell\leq 2L \\ (d,\ell)=1}} (ae\ell)^{\varepsilon} \sum_{\substack{\psi \bmod ae\ell\\ \psi\neq \psi_0}} \sum_{S_5} \sum_{S_6} \mathop{\int}_{\substack{S_5\le |s_5| \le 2S_5\\ \re(s_5)=\epsilon} }\mathop{\int}_{\substack{S_6\le |s_6| \le 2S_6\\ \re(s_6)=\epsilon}}\\
&\hspace{.25in}\times   \int_{(\epsilon)}  \int_{(\epsilon)} |s_3 s_4 s_5 s_6|^{ \varepsilon} |\mathcal{V}(s_3+s_5,s_4+s_6,s_3,s_4)|   \\
&\hspace{.25in} \times \Bigg| \sum_{\substack{ H\leq Q^{\vartheta}/(dg_1g_4)  \\ (H,g_3cae\ell) = 1  }}   \frac{ \lambda_{dg_1g_4H}   \,  \psi ( H)   }{ H^{\frac{1}{2}+s_5} } \Bigg| \Bigg| \sum_{\substack{ K\leq Q^{\vartheta}/(dg_1g_3) \\  (K,g_4cae\ell)= 1    }}   \frac{  \overline{\lambda_{dg_1g_3K}}   \,   \overline{\psi}(   K) }{  K^{\frac{1}{2}+s_6}} \Bigg| \,|ds_4\,ds_3\,ds_6\,ds_5|,
\end{split}
\end{equation}
where each of $S_5$ and $S_6$ runs through the set $\{0\} \cup \{2^\nu : \nu\in\mathbb{Z}, \nu\geq 0\}$. Here, we make an abuse of notation and interpret the condition $S_5\le |s_5| \le 2S_5$ to mean $\epsilon\le |s_5|\le 1$ when $S_5=0$, and similarly for $S_6$. We now apply \eqref{eqn: Vbound}. We choose $j_3=j_4=2$ in every situation, while we choose $j_1$ and $j_2$ depending on $S_5$ and $S_6$, as specified in the following table.

\begin{center}
\begin{longtable}{| C | C || C| C |}
\hline
\multicolumn{2}{|c||}{conditions on} & \multicolumn{2}{c|}{\,\,\,\,\,choices of\,\,\,\,\,} \\
\multicolumn{1}{|c}{$S_5$}  &  \multicolumn{1}{c||}{$S_6$}  &   \multicolumn{1}{c} {\,\,\,\,\,$j_1$\,\,\,\,\,} & \multicolumn{1}{c|} {$j_2$} \\
\Xcline{1-4}{5\arrayrulewidth}

\endfirsthead

\multicolumn{4}{l} {{(table continued from previous page)}} \\
\hline
\multicolumn{2}{|c||}{conditions on} & \multicolumn{2}{c|}{\,\,\,\,\,choices of\,\,\,\,\,} \\
\multicolumn{1}{|c}{$S_5$}  &  \multicolumn{1}{c||}{$S_6$}  &   \multicolumn{1}{c} {\,\,\,\,\,$j_1$\,\,\,\,\,} & \multicolumn{1}{c|} {$j_2$} \\
\Xcline{1-4}{5\arrayrulewidth} 
\endhead

\hline \multicolumn{4}{r}{{(table continued on next page)}} \\ 
\endfoot

\endlastfoot

    S_5 = 0     &       \vphantom{ \displaystyle\sum_{1}^{2}}S_6 = 0     &       1           &       0\\
\hline
    S_5 = 0     &       \vphantom{ \displaystyle\sum_{1}^{2}}0 < S_6\leq 1+\displaystyle  \frac{XcQ^{\vartheta-1}}{L} &      0       &       1\\
\hline 
    S_5 = 0     &       \vphantom{ \displaystyle\sum_{1}^{2}} S_6>1+\displaystyle \frac{XcQ^{\vartheta-1}}{L}         &       0       &       3\\
\hline
    \vphantom{ \displaystyle\sum_{1}^{2}} 0<S_5\leq 1+\displaystyle \frac{XcQ^{\vartheta-1}}{L} &         S_6=0       & 1         &       0\\
\hline
   \vphantom{ \displaystyle\sum_{1}^{2}} S_5> 1+\displaystyle\frac{XcQ^{\vartheta-1}}{L}         &       S_6=0       & 3         &       0\\
\hline
    0<S_5< S_6    &      \vphantom{ \displaystyle\sum_{1}^{2}} 0<S_6\leq 1+\displaystyle\frac{XcQ^{\vartheta-1}}{L} &   0 &     1\\
\hline
    0<S_5< S_6    &      \vphantom{ \displaystyle\sum_{1}^{2}} S_6> 1+\displaystyle\frac{XcQ^{\vartheta-1}}{L} &  0 &   3\\
\hline
   \vphantom{ \displaystyle\sum_{1}^{2}} 0<S_5\leq 1+\displaystyle\frac{XcQ^{\vartheta-1}}{L}     &       0<S_6\leq S_5       &     1   &     0\\
\hline
  \vphantom{ \displaystyle\sum_{1}^{2}}  S_5> 1+\displaystyle\displaystyle\frac{XcQ^{\vartheta-1}}{L} &  0<S_6 \leq S_5  & 3 &0\\
\hline
\caption{\label{tab: S5S6j1j2}Our choices of the values of $j_1$ and $j_2$ depend on the ranges of the variables of integration $s_5$ and $s_6$. }
\end{longtable}
\end{center}

We arrive at
\begin{equation}\label{eqn: beforelargesieve}
\begin{split}
& \Sigma_{c,a,e,d,g_1,g_2,g_3,g_4} \ll (XQ)^{\varepsilon}\left( \frac{ L Q}{c}\right) \sum_{S_5>0} \sum_{S_6>0} \frac{1}{ S_5^{j_1-\varepsilon} S_6^{j_2-\varepsilon} } \bigg(1+\frac{ XcQ^{\vartheta-1}}{L} \bigg)^{j_1+j_2-1} \sum_{\substack{L<\ell\leq 2L \\ (d,\ell)=1}} (ae\ell)^{\varepsilon}  \\
& \times \sum_{\substack{\psi \bmod ae\ell\\ \psi\neq \psi_0}} \int_{\epsilon-i2S_5}^{\epsilon+i2S_5}  \int_{\epsilon-i2S_6}^{\epsilon+i2S_6} \Bigg| \sum_{\substack{ H \leq Q^{\vartheta}/(dg_1g_4) \\ (H,g_3cae\ell) = 1    }}   \frac{ \lambda_{dg_1g_4H}   \,  \psi ( H)   }{ H^{\frac{1}{2}+s_5} } \Bigg| \Bigg| \sum_{\substack{ K\leq Q^{\vartheta}/(dg_1g_3) \\  (K,g_4cae\ell)= 1     }}   \frac{  \overline{\lambda_{dg_1g_3K}}   \,   \overline{\psi}(   K) }{  K^{\frac{1}{2}+s_6}} \Bigg| \,|ds_6\,ds_5|,
\end{split}
\end{equation}
where the values of $j_1$ and $j_2$ depend on $S_5$ and $S_6$ as described in Table~\ref{tab: S5S6j1j2}. Note that, for conciseness, we have bounded the term with $S_5=S_6=0$ in \eqref{eqn: ellpsisum4} by the term with $S_5=S_6=1$ in \eqref{eqn: beforelargesieve}. We may do this because both terms have the same value of $j_1+j_2$ by Table~\ref{tab: S5S6j1j2}. Similarly, we have bounded the sum of the terms with $S_5=0$ and $S_6>0$ in \eqref{eqn: ellpsisum4} by the sum of the terms with $S_5=1$ and $S_6>0$ in \eqref{eqn: beforelargesieve}, and we have bounded the sum of the terms with $S_5>0$ and $S_6=0$ in \eqref{eqn: ellpsisum4} by the sum of the terms with $S_5>0$ and $S_6=1$ in \eqref{eqn: beforelargesieve}.

In order to be able to apply the large sieve inequality, we use the Cauchy-Schwarz inequality to deduce from \eqref{eqn: beforelargesieve} that
\begin{align}
\Sigma_{c,a,e,d,g_1,g_2,g_3,g_4} \ll 
& (X Q )^{\varepsilon} \left( \frac{ L Q}{c}\right)  \sum_{S_5>0} \sum_{S_6>0} \frac{(aeLS_5S_6)^{\varepsilon}}{ S_5^{j_1 } S_6^{j_2 } } \bigg(1+\frac{ XcQ^{\vartheta-1}}{L} \bigg)^{j_1+j_2-1}  \notag\\
& \times \Bigg( \sum_{\substack{L<\ell\leq 2L \\ (d,\ell)=1}}\sum_{\substack{\psi \bmod ae\ell\\ \psi\neq \psi_0}} \Bigg\{ \int_{\varepsilon-i2S_5}^{\varepsilon+i2S_5}   \Bigg|  \sum_{\substack{ H\leq Q^{\vartheta}/(dg_1g_4)  \\ (H,g_3cae\ell) = 1   }}   \frac{ \lambda_{dg_1g_4H}   \,  \psi ( H)   }{ H^{\frac{1}{2}+s_5} } \Bigg| \,|ds_5|\Bigg\}^2 \Bigg)^{1/2}\notag\\
& \times \Bigg( \sum_{\substack{L<\ell\leq 2L \\ (d,\ell)=1}}\sum_{\substack{\psi \bmod ae\ell\\ \psi\neq \psi_0}} \Bigg\{ \int_{\varepsilon-i2S_6}^{\varepsilon+i2S_6} \Bigg| \sum_{\substack{ K\leq Q^{\vartheta}/(dg_1g_3) \\  (K,g_4cae\ell)= 1      }}   \frac{  \overline{\lambda_{dg_1g_3K}}   \,   \overline{\psi}(   K) }{  K^{\frac{1}{2}+s_6}} \Bigg| \,|ds_6| \Bigg\}^2 \Bigg)^{1/2} \label{eqn: SigmaCauchySchwarz}.
\end{align}
We now apply the hybrid large sieve inequality in the form of the following lemma.
\begin{lemma}\label{hybridlargesieve}
Let $R,T,N,\sigma$ be real numbers with $T\geq 3$, $R,N\geq 1$, and $\sigma\geq 1/2$, and let $j$ be a positive integer. If $\{a_n\}$ is any sequence of complex numbers, then
$$
\sum_{q\leq R} \sum_{\substack{\chi \bmod{qj} \\ \chi\neq \chi_0}} \Bigg( \int_{-T}^T \Bigg|\sum_{n\leq N} \frac{a_n \chi(n)}{n^{\sigma+it}} \Bigg| \,dt \Bigg)^2 \ll_{\varepsilon}  (jRNT)^{\varepsilon} (RNT+jR^2T^2 ) \sum_{ n\leq N } \frac{|a_n|^2}{n^{2\sigma}},
$$
where the $\chi$-sum is over all non-principal Dirichlet characters $\chi$ mod~$qj$.
\end{lemma}
\begin{proof}
The proof of the lemma is contained within the proof of Proposition~1 of \cite{CIS}. For full details, see Appendix \ref{hybridlargesievedetails}.
\end{proof}

From Lemma~\ref{hybridlargesieve} with $R=2L$, $T=2S_5$, $N=Q^{\vartheta}/dg_1g_4$, $\sigma=\frac{1}{2}+\varepsilon$, and $j=ae$, we deduce that
\begin{equation}\label{eqn: applyhybridlargesieve}
\begin{split}
\sum_{\substack{L<\ell\leq 2L \\ (d,\ell)=1}}
& \sum_{\substack{\psi \bmod ae\ell\\ \psi\neq \psi_0}} \Bigg\{ \int_{\varepsilon-i2S_5}^{\varepsilon+i2S_5}   \Bigg|  \sum_{\substack{ H  \leq Q^{\vartheta} / (dg_1g_4) \\ (H,g_3cae\ell) = 1  }}   \frac{ \lambda_{dg_1g_4H}   \,  \psi ( H)   }{ H^{\frac{1}{2}+s_5} } \Bigg| \,|ds_5|\Bigg\}^2 \\
& \ll  (aeLQS_5)^{\varepsilon} \bigg( \frac{Q^{\vartheta}LS_5}{dg_1g_4}  +aeL^2S_5^2 \bigg)\sum_{\substack{ H  \leq Q^{\vartheta} / (dg_1g_4) \\ (H,g_3cae\ell) = 1  }}   \frac{ |\lambda_{dg_1g_4H}|^2   }{ H^{1+\varepsilon} } \\
& \ll (dg_1g_4aeL QS_5)^{\varepsilon} (  Q^{\vartheta} LS_5 +aeL^2S_5^2 ),
\end{split}
\end{equation}
where the last line follows from the assumption $\lambda_h \ll_{\varepsilon} h^{\varepsilon}$. Note that, in using Lemma~\ref{hybridlargesieve} here, we may assume without loss of generality that $2S_5\geq 3$ since if not, then we may extend the interval of integration because the integrand is nonnegative. Similarly, Lemma~\ref{hybridlargesieve} implies
\begin{align*}
\sum_{\substack{L<\ell\leq 2L \\ (d,\ell)=1}}
& \sum_{\substack{\psi \bmod ae\ell\\ \psi\neq \psi_0}}\Bigg\{ \int_{\varepsilon-i2S_6}^{\varepsilon+i2S_6} \Bigg| \sum_{\substack{ K\leq Q^{\vartheta}/(dg_1g_3) \\  (K,g_4cae\ell)= 1   }}   \frac{  \overline{\lambda_{dg_1g_3K}}   \,   \overline{\psi}(   K) }{  K^{\frac{1}{2}+s_6}} \Bigg| \,|ds_6| \Bigg\}^2 \\
& \ll (dg_1g_3aeLQS_6)^{\varepsilon} ( Q^{\vartheta} L S_6 +aeL^2S_6^2 ).
\end{align*}
From this, \eqref{eqn: applyhybridlargesieve}, and \eqref{eqn: SigmaCauchySchwarz}, we arrive at
\begin{equation}\label{eqn: applyhybridlargesieve2}
\begin{split}
\Sigma_{c,a,e,d,g_1,g_2,g_3,g_4} \ll
& (X Q L aedg_1g_3g_4 )^{\varepsilon}\left( \frac{ L Q}{c}\right)  \sum_{S_5>0} \sum_{S_6>0} \frac{1}{ S_5^{j_1-\varepsilon} S_6^{j_2-\varepsilon} } \bigg(1+\frac{ XcQ^{\vartheta-1}}{L} \bigg)^{j_1+j_2-1}  \\
& \times \Big( Q^{\vartheta} LS_5 +aeL^2S_5^2  \Big)^{1/2} \Big( Q^{\vartheta} LS_6 +aeL^2S_6^2 \Big)^{1/2}.
\end{split}
\end{equation}
By our choices of the values of $j_1$ and $j_2$ described in Table~\ref{tab: S5S6j1j2}, if $M,N\in \{2^{\nu}: \nu\in \mathbb{Z},\nu\geq 0\}$ are given, then the term on the right-hand side of \eqref{eqn: applyhybridlargesieve2} that corresponds to the pair $(S_5,S_6)=(M,N)$ is equal to the term that corresponds to the pair $(S_5,S_6)=(N,M)$. Thus, the part of the right-hand side of \eqref{eqn: applyhybridlargesieve2} that has $S_6\leq S_5$ is a bound for the left-hand side. In that part, we have $j_2=0$ by Table~\ref{tab: S5S6j1j2}. Hence
\begin{equation}\label{eqn: applyhybridlargesieve3}
\begin{split}
\Sigma_{c,a,e,d,g_1,g_2,g_3,g_4} \ll (X Q L aedg_1g_3g_4 )^{\varepsilon}\left( \frac{ L Q}{c}\right)  \mathop{\sum_{S_5>0} \sum_{S_6>0}}_{S_6\leq S_5} \frac{S_6^{\varepsilon}}{ S_5^{j_1-\varepsilon}  } \bigg(1+\frac{ XcQ^{\vartheta-1}}{L} \bigg)^{j_1-1}  \\
\times \Big( Q^{\vartheta} LS_5 +aeL^2S_5^2  \Big).
\end{split}
\end{equation}
Recall that, as stated below \eqref{eqn: ellpsisum4}, the variables $S_5$ and $S_6$ in \eqref{eqn: applyhybridlargesieve3} each run through the set $\{2^{\nu}: \nu\in\mathbb{Z},\nu \geq 0\}$. Moreover, as described in Table~\ref{tab: S5S6j1j2}, we have $j_1=1$ for the terms in \eqref{eqn: applyhybridlargesieve3} that have $S_5 \leq 1+ XcQ^{\vartheta-1}/L$ and $j_1=3$ for the terms with $S_5 > 1+ XcQ^{\vartheta-1}/L$. We may thus evaluate the $S_5$- and $S_6$-sums in \eqref{eqn: applyhybridlargesieve3} by writing
\begin{equation}\label{eqn: geometricsum1}
\sum_{0<S_6 \leq S_5} S_6^{\varepsilon} \ll S_5^{\varepsilon}
\end{equation}
for each $S_5$,
\begin{equation}\label{eqn: geometricsum2}
\sum_{0<S_5 \leq 1+ XcQ^{\vartheta-1}/L} \frac{ Q^{\vartheta} LS_5 +aeL^2S_5^2 }{S_5^{1-\varepsilon}} \ll (XcQ)^{\varepsilon} \bigg(Q^{\vartheta} L + aeL^2 \bigg(1 + \frac{ XcQ^{\vartheta-1} }{L} \bigg) \bigg),
\end{equation}
and
\begin{align*}
\sum_{S_5 > 1+ XcQ^{\vartheta-1}/L}
& \frac{ Q^{\vartheta} LS_5 +aeL^2S_5^2 }{S_5^{3-\varepsilon}}\bigg(1 + \frac{ XcQ^{\vartheta-1} }{L} \bigg)^2 \\
& \ll (XcQ)^{\varepsilon} \bigg(Q^{\vartheta} L + aeL^2 \bigg(1 + \frac{ XcQ^{\vartheta-1} }{L} \bigg) \bigg).
\end{align*}
From this, \eqref{eqn: applyhybridlargesieve3}, \eqref{eqn: geometricsum1}, and \eqref{eqn: geometricsum2}, we deduce that
\begin{equation*}
\Sigma_{c,a,e,d,g_1,g_2,g_3,g_4} \ll (X QLcaedg_1g_3g_4)^{\varepsilon}\left( \frac{ L Q}{c}\right)\big( Q^{\vartheta}L + aeL^2 +aeL XcQ^{\vartheta-1} \big) .
\end{equation*}
From this and \eqref{eqn: Urbound}, we arrive at
\begin{align}
\sum_{h,k\leq Q^{\vartheta}} \frac{ \lambda_h \overline{\lambda_k}}{\sqrt{hk}}\mathcal{U}^r(h,k) \ll
& \sum_{1\leq c \leq C} \sum_{\substack{A,E,L \\  AL \ll XCQ^{\vartheta-1} \\ AEL \ll Y }} \sum_{A<a\leq 2A} \sum_{E<e\leq 2E}  \frac{( CXQY)^{\varepsilon}}{ AE^2L^2} \notag\\
& \times \sum_{\substack{1\leq g_1,g_2,g_3,g_4\ll \max\{Q^{\vartheta},X\}\\ g_1g_4  \leq Q^{\vartheta}, \ g_1g_3  \leq Q^{\vartheta}, \ g_2\ll X \\ (g_3,g_4)=(ec,g_1 g_2 g_3 g_4)=1 \\ a|g_1g_2g_3g_4  }} \frac{1}{ (g_1g_2g_3g_4)^{1-\varepsilon} }  \sum_{\substack{d\leq Q^{\vartheta} \\ (d,g_3g_4cae )=1}} \frac{1}{d^{1-\varepsilon}} \notag\\
& \times\left( \frac{ L Q}{c}\right)\big( Q^{\vartheta}L + AEL^2 +AEL XCQ^{\vartheta-1} \big) \ + \ (XCQY)^{\varepsilon} \frac{X^2 CQ^{2\vartheta}}{Y}.\label{eqn: Urbound2}
\end{align}

Our final task for this section is to evaluate the right-hand side of \eqref{eqn: Urbound2}. Observe that 
\begin{equation}\label{eqn: dsumbound}
\sum_{\substack{d\leq Q^{\vartheta} \\ (d,g_3g_4cae )=1}} \frac{1}{d^{1-\varepsilon}} \cdot d \ll Q^{\varepsilon}.
\end{equation}
To evaluate the $g_1,g_2,g_3,g_4$-sum in \eqref{eqn: Urbound2}, we group together terms with the same product $g_1 g_2 g_3 g_4$ and use the divisor bound to write
\begin{align*}
\sum_{\substack{1\leq g_1,g_2,g_3,g_4\ll \max\{Q^{\vartheta},X\}\\ g_1g_4  \leq Q^{\vartheta}, \ g_1g_3  \leq Q^{\vartheta}, \ g_2\ll X \\ (g_3,g_4)=(ec,g_1 g_2 g_3 g_4)=1 \\ a|g_1g_2g_3g_4  }} \frac{1}{ (g_1g_2g_3g_4)^{1-\varepsilon} } 
& \ll \sum_{\substack{\nu\ll XQ^{2\vartheta} \\ a|\nu }} \frac{1}{\nu^{1-\varepsilon}} \\
& \ll \frac{(XQa)^{\varepsilon}}{a} \ll \frac{(XQA)^{\varepsilon}}{A}.
\end{align*}
From this, \eqref{eqn: dsumbound}, and \eqref{eqn: Urbound2}, we deduce that
\begin{align*}
\sum_{h,k\leq Q^{\vartheta}} \frac{ \lambda_h \overline{\lambda_k}}{\sqrt{hk}}\mathcal{U}^r(h,k) &\ll Q\sum_{1\leq c \leq C}\frac{1}{c} \sum_{\substack{A,E,L \\  AL \ll XCQ^{\vartheta-1} \\ AEL \ll Y }}   \frac{(CXQY)^{\varepsilon}}{ A E L}\\
&\hspace{.25in}\times\big( Q^{\vartheta}L + AEL^2  +AEL XCQ^{\vartheta-1} \big) + \  (XCQY)^{\varepsilon} \frac{X^2 CQ^{2\vartheta}}{Y}.
\end{align*}
The condition $AL\ll XCQ^{\vartheta-1}$ implies that $AEL^2 \ll AEL XCQ^{\vartheta-1} $ because $A\gg 1$. Moreover, we have $\sum_{c\leq C} (1/c)\ll C^{\varepsilon}$. Hence
\begin{align}
\sum_{h,k\leq Q^{\vartheta}} \frac{ \lambda_h \overline{\lambda_k}}{\sqrt{hk}}\mathcal{U}^r(h,k)
\ll (CXQY)^{\varepsilon}  Q \sum_{\substack{A,E,L \\  AL \ll XCQ^{\vartheta-1} \\ AEL \ll Y }}   \frac{1}{ A E L }   \big( LQ^{\vartheta}  +AEL XCQ^{\vartheta-1} \big) \notag\\
+ \  (XCQY)^{\varepsilon} \frac{X^2 CQ^{2\vartheta}}{Y}. \label{eqn: Urbound3}
\end{align}
Recall that, as stated below \eqref{eqn: Urshortbound}, each of the summation variables $A,E,L$ in \eqref{eqn: Urbound3} runs through the set $\{2^{\nu}:\nu\in \mathbb{Z}, \nu\geq -1\}$. We may thus evaluate the $A,E,L$-sum in \eqref{eqn: Urbound3} by writing
\begin{equation*}
\sum_{\substack{A,E,L \\  AL \ll XCQ^{\vartheta-1} \\ AEL \ll Y }} \frac{Q^{\vartheta}}{ A E } \leq \sum_{\substack{A,E,L \\  AEL \ll Y }}   4Q^{\vartheta}  
\ll  Y^{\varepsilon} Q^{\vartheta} 
\end{equation*}
and
\begin{equation*}
\sum_{\substack{A,E,L \\  AL \ll XCQ^{\vartheta-1} \\ AEL \ll Y }} XCQ^{\vartheta-1} \leq XCQ^{\vartheta-1}\sum_{\substack{A,E,L \\ AEL \ll Y }} 1
\ll   Y^{\varepsilon} XCQ^{\vartheta-1}.
\end{equation*}
We conclude that
\begin{equation}\label{eqn: Urboundfinal}
\sum_{h,k\leq Q^{\vartheta}} \frac{ \lambda_h \overline{\lambda_k}}{\sqrt{hk}}\mathcal{U}^r(h,k)
\ll_{\varepsilon} (XCQY)^{\varepsilon}  Q ( Q^{\vartheta} +  XCQ^{\vartheta-1}) \ +  \  (XCQY)^{\varepsilon} \frac{X^2 CQ^{2\vartheta}}{Y}.
\end{equation}

\section{Finishing the proof of Theorem~\ref{thm: main}}\label{sec: proofofthm}
We put together our estimates and deduce from \eqref{eqn: IandIstar}, \eqref{eqn: Ssplit}, \eqref{eqn: Dis0swap}, \eqref{eqn: Lsplit}, \eqref{eqn: Uready}, and \eqref{eqn: U2is1swap} that
\begin{equation}\label{eqn: SI0I1E}
\mathcal{S}(h,k) = \mathcal{I}_0(h,k) + \mathcal{I}_1(h,k) + \mathcal{E}(h,k),
\end{equation}
where
\begin{equation}\label{eqn: mathcalEdef}
\begin{split}
\mathcal{E}(h,k) =  \mathcal{L}^r(h,k) + \mathcal{U}^r(h,k) +  O \bigg( \bigg( Q  + \frac{Q^2}{C}\bigg) \frac{(XCQhk)^{\varepsilon}(h,k)}{\sqrt{hk}}  \bigg) \\
+O\Big((XCQhk)^{\varepsilon} \big( XC+ X^{-\frac{1}{2}} Q^{\frac{5}{2}} + Q^{\frac{3}{2}} +X^2hk Q^{-96} \big) \Big).
\end{split}
\end{equation}
For any $\vartheta>0$, we have
\begin{equation*}
\sum_{h,k\leq Q^{\vartheta}} \frac{(hk)^{\varepsilon} (h,k)}{hk} =  \sum_{h,k\leq Q^{\vartheta}} \frac{(hk)^{\varepsilon} }{hk} \sum_{\substack{d|h \\ d|k}} \phi(d) = \sum_{d \leq Q^{\vartheta}} \frac{\phi(d)}{d^{2-\varepsilon}} \Bigg( \sum_{j\leq Q^{\vartheta}/d} \frac{1}{j^{1-\varepsilon}} \Bigg)^2 \ll Q^{\varepsilon},
\end{equation*}
\begin{equation*}
\sum_{h,k\leq Q^{\vartheta}} \frac{(hk)^{\varepsilon} }{\sqrt{hk}} \ll Q^{\vartheta+\varepsilon},
\end{equation*}
and
\begin{equation*}
\sum_{h,k\leq Q^{\vartheta}}\frac{(hk)^{\varepsilon} hk}{\sqrt{hk}} \ll Q^{3\vartheta+\varepsilon}.
\end{equation*}
From these bounds, \eqref{eqn: Lrbound}, \eqref{eqn: Urboundfinal}, and \eqref{eqn: mathcalEdef}, we deduce that if $\vartheta>0$ and $\{\lambda_h\}_{h=1}^{\infty}$ is any sequence of complex numbers such that $\lambda_h \ll_{\varepsilon} h^{\varepsilon}$ for all positive integers $h$, then
\begin{equation}\label{eqn: mathcalEbound}
\begin{split}
\sum_{h,k\leq Q^{\vartheta}} \frac{ \lambda_h \overline{\lambda_k}}{\sqrt{hk}}\mathcal{E}(h,k) \ll (XCQ)^{\varepsilon} \bigg(Q^{1+\vartheta} + \frac{Q^2}{C} \bigg) + (XCQY)^{\varepsilon}  ( Q^{1+\vartheta} +  XCQ^{\vartheta}) \\
+   (XCQY)^{\varepsilon} \frac{X^2 CQ^{2\vartheta}}{Y}  + (XCQ)^{\varepsilon} \big( XCQ^{\vartheta} + X^{-\frac{1}{2}} Q^{\frac{5}{2}+\vartheta} + Q^{\frac{3}{2}+\vartheta} + X^2Q^{-96+3\vartheta}\big).
\end{split}
\end{equation}
Recall our assumption that $X=Q^{\eta}$ with $1<\eta<2$. We optimize the upper bound \eqref{eqn: mathcalEbound} by choosing
$$
C = Q^{1-\frac{\vartheta}{2}-\frac{\eta}{2}},
$$
which implies $Q^2/C = XCQ^{\vartheta} = Q^{1+\frac{\vartheta}{2}+\frac{\eta}{2}}$. We impose the condition
\begin{equation*}
\vartheta<2-\eta
\end{equation*}
so that $C\gg Q^{\varepsilon}$. Note that $\vartheta<2-\eta$ implies $\vartheta<\eta$ since $\eta>1$. We also choose $Y$ to be a large power of $Q$, say $Y=Q^{99}$. With these choices for $C$ and $Y$ and the condition $\vartheta<2-\eta$, we deduce from \eqref{eqn: mathcalEbound} that
\begin{equation}\label{eqn: mathcalEbound2}
\sum_{h,k\leq Q^{\vartheta}} \frac{ \lambda_h \overline{\lambda_k}}{\sqrt{hk}}\mathcal{E}(h,k) \ll Q^{1+\frac{\vartheta}{2}+\frac{\eta}{2}+\varepsilon} + Q^{\frac{5}{2}-\frac{\eta}{2}+\vartheta+\varepsilon} .
\end{equation}

We have thus proved that the conclusion of Theorem~\ref{thm: main} holds under the additional assumption \eqref{eqn: orbitals}. To complete the proof of Theorem~\ref{thm: main}, it is left to show that \eqref{eqn: mathcalEbound2} holds for any multisets $A$ and $B$ of complex numbers with moduli $\leq C_1/\log Q$, where $C_1$ is an arbitrary fixed positive constant. We do this by showing for each $\ell=0,1$ that $\mathcal{I}_\ell(h,k)$ is holomorphic in each of the variables $\alpha\in A$ and $\beta\in B$ in the region where $|\alpha|,|\beta|\leq C_1/\log Q$ for all $\alpha\in A$ and $\beta\in B$ (or, more precisely, that the only singularities of $\mathcal{I}_\ell(h,k)$ in this region are removable singularities). The holomorphy of $\mathcal{I}_0(h,k)$ is immediate from \eqref{eqn: I_l(h,k)def2} with $\ell=0$: if $\ell=0$ then the integrand on the right-hand side of \eqref{eqn: I_l(h,k)def2} is holomorphic in each of the variables $\alpha\in A$ and $\beta\in B$ so long as $\alpha,\beta\ll \varepsilon$ for each $\alpha\in A$ and $\beta\in B$. To prove the holomorphy of $\mathcal{I}_1(h,k)$, define $I_{E,F}(n)$ for finite multisets $E,F$ of complex numbers by the Dirichlet series expression
\begin{equation*}
\frac{\prod_{\xi\in E}\zeta(\xi+s)}{\prod_{\rho\in F} \zeta(\rho+s)} = \sum_{n=1}^{\infty} \frac{I_{E,F}(n)}{n^s}.
\end{equation*}
This definition implies that if $\alpha\in A$ and $\re(s)$ is sufficiently large, then
\begin{equation*}
\sum_{n=1}^{\infty}\frac{I_{A  \cup \{-\beta \}, \{\alpha \} }(n)}{n^s} = \zeta(-\beta +s) \prod_{\hat{\alpha}\neq \alpha} \zeta(\hat{\alpha} +s).
\end{equation*}
From this and the uniqueness of Dirichlet coefficients, we deduce that if $\alpha\in A$, then
\begin{equation}\label{eqn: IABidentity1}
I_{A  \cup \{-\beta \}, \{\alpha \} }(n) = {\tau}_{A\smallsetminus \{\alpha \} \cup \{-\beta \} }(n)
\end{equation}
for every positive integer $n$. Similarly, if $\beta\in B$, then
\begin{equation}\label{eqn: IABidentity2}
I_{B  \cup \{-\alpha \}, \{\beta \}  } (n) =  {\tau}_{B\smallsetminus \{\beta \} \cup \{-\alpha\}  }(n)
\end{equation}
for every positive integer $n$. Now we claim that if $A$ and $B$ have no repeated elements and the elements of $A\cup B$ are distinct from each other and are $\ll 1/\log Q$, then
\begin{equation}\label{eqn: vandermonde1swap}
\begin{split}
\mathcal{I}_{1}(h,k) & = \sum_{\substack{q=1 \\ (q,hk)=1} }^{\infty} W\left( \frac{q}{Q}\right)\sideset{}{^\flat}\sum_{\chi \bmod q} \frac{1}{(2\pi i )^4} \int_{(\epsilon)} \int_{(\epsilon)} \oint_{|z|= \epsilon/4} \oint_{|y|=\epsilon/4} X^{s_1+s_2} \widetilde{V}(s_1) \widetilde{V}(s_2)   \\
&\hspace{.25in} \times \mathscr{X} (\tfrac{1}{2}-z +s_1 )\mathscr{X} (\tfrac{1}{2}-y+s_2 ) q^{z-s_1+y-s_2} \\
&\hspace{.25in} \times \frac{ \prod_{\substack{\alpha\in A \\ \beta\in B}} \zeta(1+\alpha+\beta+s_1+s_2) \prod_{\alpha\in A} \zeta(1+\alpha+z) \prod_{\beta\in B} \zeta(1+\beta+y)  }{ \prod_{\alpha\in A} \zeta(1+\alpha+s_1-y+s_2) \prod_{\beta\in B} \zeta(1-z+s_1+\beta+s_2) } \\
&\hspace{.25in} \times  \zeta(1+y+z-s_1-s_2) \zeta(1-y-z+s_1+s_2) \prod_{p|q} P_0 \\
&\hspace{.05in} \times \prod_{p|hk}\Bigg\{ P_0 \!\!\!\!\!\! \sum_{ \substack{0\leq m,n<\infty\\ m+\ordp(h) = n+\ordp(k)}}  \!\!\!\!\!\! \frac{I_{A_{s_1} \cup \{y-s_2\}, \{-z+s_1\} } (p^m) I_{B_{s_2} \cup \{z-s_1\}, \{-y+s_2\}  } (p^n)  }{p^{m/2}p^{n/2} }  \Bigg\}\\
&\hspace{.05in} \times \prod_{p\nmid qhk} \Bigg\{P_0 \  \sum_{m=0}^{\infty} \frac{I_{A_{s_1} \cup \{y-s_2\}, \{-z+s_1\} } (p^m) I_{B_{s_2} \cup \{z-s_1\}, \{-y+s_2\}  } (p^m)  }{p^{m} } \Bigg\}\\
&\hspace{.25in} \times \,dy\,dz\,ds_2\,ds_1,
\end{split}
\end{equation}
where $P_0$ is defined by
\begin{align*}
P_0 = P_0(z,y,s_1,s_2;A,B)
& := \left( 1-\frac{1}{p}\right)^{-2} \left( 1 - \frac{1}{p^{ 1+y+z-s_1-s_2 }}\right) \left( 1 - \frac{1}{p^{ 1-y-z+s_1+s_2 }}\right) \\
& \times \prod_{  \substack{\alpha\in A \\ \beta\in B}} \left( 1 - \frac{1}{p^{1+\alpha+\beta+s_1+s_2 }}\right) \prod_{\alpha\in A} \left( 1 - \frac{1}{p^{1+\alpha+z }}\right) \prod_{\beta\in B}\left( 1 - \frac{1}{p^{ 1+\beta+y }}\right) \\
& \times  \prod_{\alpha\in A} \left( 1 - \frac{1}{p^{1+\alpha+s_1-y+s_2 }}\right)^{-1} \prod_{\beta\in B} \left( 1 - \frac{1}{p^{ 1-z+s_1+\beta+s_2 }}\right)^{-1} .
\end{align*}
To see this, we use the residue theorem to evaluate the $z$- and $y$-integrals. The Euler product on the right-hand side of \eqref{eqn: vandermonde1swap} converges absolutely by an argument similar to the proof of Lemma~\ref{lem: 1swapeulerbound}. Thus the poles of the integrand that are enclosed by the circles $|z|=\epsilon/4$ and $|y|=\epsilon/4$ are precisely the poles of the factors
$$
\prod_{\alpha\in A}\zeta(1+\alpha+z)\prod_{\beta\in B} \zeta(1+\beta+y).
$$
After evaluating the $z$- and $y$-integrals using the residue theorem, we may simplify each residue by using \eqref{eqn: IABidentity1} and \eqref{eqn: IABidentity2} to see that the right-hand side of \eqref{eqn: vandermonde1swap} is equal to the right-hand side of \eqref{eqn: I_l(h,k)def2} with $\ell=1$. This proves our claim that \eqref{eqn: vandermonde1swap} holds if $A$ and $B$ have no repeated elements and the elements of $A\cup B$ are distinct from each other. Now the right-hand side of \eqref{eqn: vandermonde1swap} is holomorphic in each of the variables $\alpha\in A$ and $\beta\in B$ in any region with $\alpha,\beta\ll 1/\log Q$ for each $\alpha\in A$ and $\beta\in B$ because the Euler product in its integrand converges absolutely. Hence, by analytic continuation, it follows that $\mathcal{I}_1(h,k)$ is holomorphic in each of the variables $\alpha\in A$ and $\beta\in B$ in the region. As a side note, we remark that this argument can be generalized to show the holomorphy of $\mathcal{I}_{\ell}(h,k)$ for each $\ell $ with $0\leq \ell\leq \min\{|A|,|B|\}$.

We have now shown that $\mathcal{I}_0(h,k)$ and $\mathcal{I}_1(h,k)$ are each holomorphic in each of the variables $\alpha\in A$ and $\beta\in B$ in any given region such that $\alpha,\beta\ll 1/\log Q$ for each $\alpha\in A$ and $\beta\in B$. Now $\mathcal{S}(h,k)$ is holomorphic in the same region since its definition \eqref{eqn: S(h,k)def} has only finitely many nonzero terms by the assumption that $W$ and $V$ are compactly supported. It follows from these and \eqref{eqn: SI0I1E} that $\mathcal{E}(h,k)$ is also holomorphic in the same region. Thus, since \eqref{eqn: mathcalEbound} holds for $A,B$ satisfying the condition \eqref{eqn: orbitals}, the maximum modulus principle implies that \eqref{eqn: mathcalEbound} also holds for finite multisets $A,B$ satisfying $|\alpha|,|\beta|\leq C_0/\log Q$ for all $\alpha\in A$ and $\beta\in B$, where $C_0$ is the arbitrary positive constant in \eqref{eqn: orbitals}. This completes the proof of Theorem~\ref{thm: main}.

\appendix

\section{Proof of Lemma \ref{hybridlargesieve}}\label{hybridlargesievedetails}
In this section, we give the details of the proof of Lemma \ref{hybridlargesieve}, which is an analogue of Proposition~1 of \cite{CIS} and likewise a consequence of the hybrid large sieve in the form of Theorem~9.12 of \cite{IK}.

\begin{proof}[Proof of Lemma \ref{hybridlargesieve}]
To apply Theorem~9.12 of \cite{IK}, we need to express each $\chi$ mod~$qj$ in terms of a product of two characters, one with modulus $\tilde{q}$ and the other with modulus $\tilde{\jmath}$, where $\tilde{q}$ and $\tilde{\jmath}$ are factors of $qj$ such that $(\tilde{q}, \tilde{\jmath})=1 $. To this end, recall that each Dirichlet character $\chi$ mod~$qj$ is induced by a unique primitive Dirichlet character modulo some divisor of $qj$.  We may write this divisor uniquely as $\tilde{q}\tilde{\jmath}$, where $(\tilde{q},j)=1$ and $\tilde{\jmath}$ is composed only of primes that divide $j$. Note that if $\chi$ is non-principal, then $\tilde{q}\tilde{\jmath}>1$. Since $\tilde{q}\tilde{\jmath}$ is a divisor of $qj$, it holds that $qj=D\tilde{q}\tilde{\jmath}$ for some positive integer $D$, and dividing both sides by $(j,\tilde{\jmath})$ implies
$$
q\frac{j}{(j,\tilde{\jmath})} = D\tilde{q}\frac{\tilde{\jmath}}{(j,\tilde{\jmath})}.
$$
It follows that $j/(j,\tilde{\jmath})$ divides $D$ because $j/(j,\tilde{\jmath})$ is relatively prime to both $\tilde{q}$ and $\tilde{\jmath}/(j,\tilde{\jmath})$. Thus we may write $D=dj/(j,\tilde{\jmath})$ for some positive integer $d$. Hence $q= d\tilde{q}\tilde{\jmath}/(j,\tilde{\jmath})$. We have thus shown that for each non-principal $\chi$ mod $qj$, there is a unique quadruple $(\tilde{\jmath},d,\tilde{q},\tilde{\chi})$ such that $\tilde{\jmath}$ is a positive integer composed only of the primes dividing $j$, $\tilde{q}$ is a positive integer with $(\tilde{q},j)=1$ and $\tilde{q}\tilde{\jmath}>1$, $d$ is a positive integer such that $q= d\tilde{q}\tilde{\jmath}/(j,\tilde{\jmath})$, and $\tilde{\chi}$ is a primitive character modulo $\tilde{q}\tilde{\jmath}$ such that $\chi=\tilde{\chi}\chi_0$, where $\chi_0$ is the principal character modulo $qj$. Therefore we have
\begin{align*}
\sum_{q\leq R} \sum_{\substack{\chi \bmod{qj} \\ \chi\neq \chi_0}} &\Bigg( \int_{-T}^T \Bigg|\sum_{n\leq N} \frac{a_n \chi(n)}{n^{\sigma+it}} \Bigg| \,dt \Bigg)^2\\
&\leq \sum_{q\leq R} \sum_{\substack{1\leq \tilde{\jmath}<\infty \\ p|\tilde{\jmath}\Rightarrow p|j}} \sum_{\substack{ 1\leq d,\tilde{q} <\infty \\ (\tilde{q},j)=1 \\ \tilde{q}\tilde{j}>1 \\ q=d\tilde{q} {\tilde{\jmath}}/{(j,\tilde{\jmath})} }} \,\sideset{}{^*}\sum_{ \tilde{\chi} \bmod{\tilde{q}\tilde{\jmath}}  } \Bigg( \int_{-T}^T \Bigg|\sum_{n\leq N} \frac{a_n \tilde{\chi}(n)\chi_0(n) }{n^{\sigma+it}} \Bigg| \,dt \Bigg)^2
\end{align*}
because the summand is nonnegative, where the * notation indicates that the sum is over primitive characters. We substitute $q=d\tilde{q} \tilde{\jmath}/(j,\tilde{\jmath})$ to write
\begin{align*}
\sum_{q\leq R} \sum_{\substack{\chi \bmod{qj} \\ \chi\neq \chi_0}} &\Bigg( \int_{-T}^T \Bigg|\sum_{n\leq N} \frac{a_n \chi(n)}{n^{\sigma+it}} \Bigg| \,dt \Bigg)^2\\
&\leq \sum_{\substack{1\leq \tilde{\jmath}<\infty \\ p|\tilde{\jmath}\Rightarrow p|j}} \sum_{d \leq R \frac{(j,\tilde{\jmath})}{\tilde{\jmath}}} \sum_{ \substack{\tilde{q} \leq R  \frac{(j,\tilde{\jmath})}{d\tilde{\jmath}} \\ (\tilde{q},j)=1 \\ \tilde{q}\tilde{\jmath}>1 }}  \,\sideset{}{^*}\sum_{ \tilde{\chi} \bmod{\tilde{q}\tilde{\jmath}}  } \Bigg( \int_{-T}^T \Bigg|\sum_{n\leq N} \frac{a_n \tilde{\chi}(n)\chi_0(n) }{n^{\sigma+it}} \Bigg| \,dt \Bigg)^2,
\end{align*}
where $\chi_0$ denotes the principal character modulo $qj=d\tilde{q}j\tilde{\jmath} /(j,\tilde{\jmath}) $. Now we may replace the function $\chi_0$ on the right-hand side with the characteristic function of the condition $(n,dj)=1$. Indeed, if $(n,dj)>1$, then $n$ and $d\tilde{q}j\tilde{\jmath}/(j,\tilde{\jmath})$ are not relatively prime, and so $\chi_0(n)=0$. If $(n,dj)=1$ and $(n,\tilde{q})>1$, then $\tilde{\chi}(n)\chi_0(n)=\tilde{\chi}(n)$ because both quantities are zero. If $(n,dj)=1$ and $(n,\tilde{q})=1$, then $n$ and $d\tilde{q}j\tilde{\jmath}/(j,\tilde{\jmath})$ are relatively prime, and so $\chi_0(n)=1$. Hence
\begin{align*}
\sum_{q\leq R} \sum_{\substack{\chi \bmod{qj} \\ \chi\neq \chi_0}}& \Bigg( \int_{-T}^T \Bigg|\sum_{n\leq N} \frac{a_n \chi(n)}{n^{\sigma+it}} \Bigg| \,dt \Bigg)^2\\
&\leq \sum_{\substack{1\leq \tilde{\jmath}<\infty \\ p|\tilde{\jmath}\Rightarrow p|j}} \sum_{d \leq R \frac{(j,\tilde{\jmath})}{\tilde{\jmath}}} \sum_{ \substack{\tilde{q} \leq R \frac{(j,\tilde{\jmath})}{d\tilde{\jmath}} \\ (\tilde{q},j)=1 \\ \tilde{q}\tilde{\jmath}>1 }}  \,\sideset{}{^*}\sum_{ \tilde{\chi} \bmod{\tilde{q}\tilde{\jmath}}  } \Bigg( \int_{-T}^T \Bigg|\sum_{\substack{n\leq N\\ (n,dj)=1 }} \frac{a_n \tilde{\chi}(n) }{n^{\sigma+it}} \Bigg| \,dt \Bigg)^2.
\end{align*}
To bound the $\tilde{q},\tilde{\chi}$-sum, we apply the Cauchy-Schwarz inequality and then Theorem~9.12 of \cite{IK}. (There, take $k=\tilde{\jmath}$, $Q=R(j,\tilde{\jmath})/(d\tilde{\jmath})$, $T=T$, $N=N$, $a_n=a_n/n^{\sigma}$ if $(n,dj)=1$, and $a_n=0$ if $(n,dj)>1$. Note that we may apply the theorem because if $\tilde{\chi}$ is a primitive Dirichlet character modulo $\tilde{q}\tilde{\jmath}$, then $\tilde{\chi}$ equals the product of a primitive Dirichlet character modulo $\tilde{q}$ and a primitive Dirichlet character modulo $\tilde{\jmath}$ since $(\tilde{q},\tilde{\jmath})=1$.) This gives
\begin{align*}
\sum_{ \substack{\tilde{q} \leq R {(j,\tilde{\jmath})}/{(d\tilde{\jmath})} \\ (\tilde{q},j)=1 \\ \tilde{q}\tilde{\jmath}>1 }}  \,\sideset{}{^*}\sum_{ \tilde{\chi} \bmod{\tilde{q}\tilde{\jmath}}  } &\Bigg( \int_{-T}^T \Bigg|\sum_{\substack{n\leq N\\ (n,dj)=1 }} \frac{a_n \tilde{\chi}(n) }{n^{\sigma+it}} \Bigg| \,dt \Bigg)^2 \\
& \leq 2T \sum_{ \substack{\tilde{q} \leq R {(j,\tilde{\jmath})}/{(d\tilde{\jmath})} \\ (\tilde{q},j)=1 \\ \tilde{q}\tilde{\jmath}>1 }}  \,\sideset{}{^*}\sum_{ \tilde{\chi} \bmod{\tilde{q}\tilde{\jmath}}  }  \int_{-T}^T \Bigg|\sum_{\substack{n\leq N\\ (n,dj)=1 }} \frac{a_n \tilde{\chi}(n) }{n^{\sigma+it}} \Bigg|^2 \,dt  \\
& \ll  T(\log (jRTN))^3 \bigg(N+ \frac{(j,\tilde{\jmath})^2R^2T}{d^2\tilde{\jmath}}  \bigg) \sum_{\substack{n\leq N\\ (n,dj)=1 }} \frac{|a_n|^2}{n^{2\sigma}},
\end{align*}
where the implied constant is absolute. Therefore
\begin{align*}
& \sum_{q\leq R} \sum_{\substack{\chi \bmod{qj} \\ \chi\neq \chi_0}} \Bigg( \int_{-T}^T \Bigg|\sum_{n\leq N} \frac{a_n \chi(n)}{n^{\sigma+it}} \Bigg| \,dt \Bigg)^2\\
& \ll \sum_{\substack{1\leq \tilde{\jmath}<\infty \\ p|\tilde{\jmath}\Rightarrow p|j}} \sum_{d \leq R {(j,\tilde{\jmath})}/{\tilde{\jmath}}}T(\log (jRTN))^3 \bigg(N+  \frac{(j,\tilde{\jmath})^2R^2T}{d^2\tilde{\jmath}} \bigg) \sum_{\substack{n\leq N\\ (n,dj)=1 }} \frac{|a_n|^2}{n^{2\sigma}}.
\end{align*}
We may ignore the condition $(n,dj)=1$ and then evaluate the $d$-sum to deduce that
\begin{align*}
\sum_{q\leq R} \sum_{\substack{\chi \bmod{qj} \\ \chi\neq \chi_0}}& \Bigg( \int_{-T}^T \Bigg|\sum_{n\leq N} \frac{a_n \chi(n)}{n^{\sigma+it}} \Bigg| \,dt \Bigg)^2 \\
&\ll  T(\log (jRTN))^3\sum_{\substack{1\leq \tilde{\jmath}<\infty \\ p|\tilde{\jmath}\Rightarrow p|j}}   \bigg(\frac{(j,\tilde{\jmath})RN}{ \tilde{\jmath}}+  \frac{(j,\tilde{\jmath})^2R^2T}{\tilde{\jmath}} \bigg) \sum_{ n\leq N } \frac{|a_n|^2}{n^{2\sigma}}.
\end{align*}
Now let $j=\prod_{p|j}p^{j_p}$ be the prime factorization of $j$. Multiplicativity implies
\begin{align*}
\sum_{\substack{1\leq \tilde{\jmath}<\infty \\ p|\tilde{\jmath}\Rightarrow p|j}} \frac{(j,\tilde{\jmath})}{\tilde{\jmath}}
& =\prod_{p|j} \sum_{\nu=0}^{\infty} \frac{p^{\min\{j_p,\nu\}}}{p^{\nu}}
= \prod_{p|j} \bigg(j_p + \frac{1}{1-\frac{1}{p}} \bigg) \ll j^{\varepsilon}
\end{align*}
and
\begin{align*}
\sum_{\substack{1\leq \tilde{\jmath}<\infty \\ p|\tilde{\jmath}\Rightarrow p|j}} \frac{(j,\tilde{\jmath})^2}{\tilde{\jmath} }
& \leq j \sum_{\substack{1\leq \tilde{\jmath}<\infty \\ p|\tilde{\jmath}\Rightarrow p|j}} \frac{(j,\tilde{\jmath})}{\tilde{\jmath}} \ll j^{1+\varepsilon}.
\end{align*}
Hence
\begin{align*}
\sum_{q\leq R} \sum_{\substack{\chi \bmod{qj} \\ \chi\neq \chi_0}} \Bigg( \int_{-T}^T \Bigg|\sum_{n\leq N} \frac{a_n \chi(n)}{n^{\sigma+it}} \Bigg| \,dt \Bigg)^2
\ll  (jRNT)^{\varepsilon} (RNT+jR^2T^2 ) \sum_{ n\leq N } \frac{|a_n|^2}{n^{2\sigma}}.
\end{align*}
\end{proof}

\printbibliography

\end{document}